\documentclass[12pt,leqno]{article}

\usepackage{amsmath,amssymb,amsthm}
\usepackage[all]{xy}
\usepackage{lscape}

\makeatletter

\newtheorem{theorem}{Theorem}[section]
\newtheorem{proposition}[theorem]{Proposition}
\newtheorem{lemma}[theorem]{Lemma}
\newtheorem{corollary}[theorem]{Corollary}

\theoremstyle{definition}

\newtheorem{example}[theorem]{Example}
\newtheorem{definition}[theorem]{Definition}

\theoremstyle{remark}

\theoremstyle{remark}
\newtheorem{remark}[theorem]{Remark}

\def\({{\rm (}}
\def\){{\rm )}}

\let\Mathrm\operator@font
\let\Cal\mathcal
\let\Bbb\mathbb
\newcommand{\fm}{\ensuremath{\mathfrak m}}

\newcommand{\fp}{\ensuremath{\mathfrak p}}

\def\standop#1{\mathop{\Mathrm #1}\nolimits}
\def\difstop#1#2{\expandafter\def\csname #1\endcsname{\standop{#2}}}
\def\defstop#1{\difstop{#1}{#1}}

\defstop{AB}
\defstop{ann}
\defstop{Ass}
\defstop{Add}
\defstop{Alt}
\defstop{Ass}
\difstop{adim}{AFHdim}
\def\alg{_{\Mathrm{alg}}}

\defstop{CMFI}
\defstop{codim}
\defstop{Coh}
\defstop{coht}
\defstop{Coker}
\defstop{Cone}
\defstop{Cl}
\defstop{Cox}
\defstop{cosk}
\defstop{cd}
\defstop{cmd}

\defstop{depth}
\defstop{Der}
\defstop{Div}
\defstop{div}

\defstop{EM}
\defstop{embdim}
\defstop{End}
\defstop{ev}
\defstop{Ext}

\difstop{fdim}{flat.dim}
\defstop{Flat}
\defstop{Func}
\defstop{Fpqc}
\def\FSch{\mathop{\text{$F$-Sch}}\nolimits}

\def\GL{\text{\sl{GL}}}
\defstop{GD}
\defstop{Good}
\defstop{Gal}

\difstop{height}{ht}
\defstop{Hom}

\def\id{\mathord{\Mathrm{id}}}
\def\Id{\mathord{\Mathrm{Id}}}
\difstop{Image}{Im}
\defstop{ind}

\defstop{Ker}

\defstop{Lch}
\defstop{length}
\defstop{Lin}
\defstop{Lqc}
\defstop{lqc}
\defstop{LQ}
\defstop{LM}
\defstop{Lie}
\defstop{Loc}

\defstop{Mat}
\defstop{Max}
\defstop{Min}
\defstop{Mod}
\defstop{Mor}
\defstop{MCM}
\defstop{Map}
\difstop{mred}{red}

\defstop{Nerve}
\defstop{NonSFR}
\defstop{NonCMFI}
\defstop{NonNor}
\defstop{Nor}

\defstop{ob}
\defstop{Ob}
\def\op{^{\standop{op}}}

\defstop{PA}
\defstop{PM}
\defstop{PR}
\defstop{Proj}
\defstop{Prin}
\defstop{Pic}

\defstop{Qch}
\defstop{qch}

\defstop{rad}
\defstop{rank}
\def\red{_{\Mathrm{red}}}
\def\reg{_{\standop{reg}}}
\defstop{res}
\defstop{Reg}
\defstop{Ref}

\def\Sch{\underline{\Mathrm Sch}}
\def\SL{\text{\sl{SL}}}

\def\Sp{\text{\sl{Sp}}}
\defstop{Spec}
\defstop{supp}
\defstop{Supp}
\defstop{Sym}
\defstop{Sing}
\defstop{SFR}
\defstop{Soc}
\defstop{Sh}

\def\sep{_{\Mathrm{sep}}}
\def\extop{\textstyle{\bigwedge}^{\Mathrm{top}}}

\difstop{tdeg}{trans.deg}
\def\tor{_{\Mathrm{tor}}}

\defstop{Tor}
\defstop{TRD}
\difstop{trace}{tr}

\defstop{Zar}

\def\an#1{^{\langle #1\rangle}}
\def\Ki{K\text{\rm-inj}}

\def\C{\Cal C}
\def\E{\Cal E}
\def\F{\Cal F}

\def\g{\mathfrak g}
\def\H{\Cal H}
\def\I{\Cal I}

\def\K{\Cal K}
\def\L{\Cal L}
\def\M{\Cal M}
\def\N{\Cal N}
\def\O{\Cal O}

\def\Q{\Cal Q}
\def\R{\Cal R}
\def\V{\Cal V}

\def\fm{\mathfrak{m}}
\def\fp{\mathfrak{p}}

\let\indlim\varinjlim
\let\projlim\varprojlim

\def\uHom{\mathop{\text{\underline{$\Mathrm Hom$}}}\nolimits}
\def\uProj{\mathop{\text{\underline{$\Mathrm Proj$}}}\nolimits}
\def\uExt{\mathop{\text{\underline{$\Mathrm Ext$}}}\nolimits}
\def\uTor{\mathop{\text{\underline{$\Mathrm Tor$}}}\nolimits}
\def\uGamma{\mathop{\text{\underline{$\Gamma$}}}\nolimits}
\def\uH{\mathop{\text{\underline{$H$}}}\nolimits}

\def\uSpec{\mathop{\text{\underline{$\Mathrm Spec$}}}\nolimits}

\def\sdarrow#1{\downarrow\hbox to 0pt{\scriptsize$#1$\hss}}
\def\suarrow#1{\uparrow\hbox to 0pt{\scriptsize$#1$\hss}}
\def\ssearrow#1{\searrow\hbox to 0pt{\scriptsize$#1$\hss}}

\def\ext{{\textstyle\bigwedge}}

\def\section{\@startsection{section}{1}{\z@ }%
  {-3.5ex plus -1ex minus -.2ex}{2.3ex plus .2ex}{\bf }}

\long\def\refname{\par\kern -3ex
  \begin{center}\rm R\sc{eferences}\end{center}\par\kern 
  -2ex}

\def\@seccntformat#1{\csname the#1\endcsname.\quad}

\def\@@@sect#1#2#3#4#5#6[#7]#8{%
  \ifnum #2>\c@secnumdepth 
  \def \@svsec {}\else \refstepcounter {#1}%
  \def\@svsec{}
  \fi 
  \@tempskipa #5\relax 
  \ifdim \@tempskipa >\z@ 
  \begingroup #6\relax \@hangfrom {\hskip #3\relax 
    \@svsec}{\interlinepenalty \@M #8\par }\endgroup 
  \csname #1mark\endcsname {#7}
  \else 
  \def \@svsechd {#6\hskip #3\@svsec #8\csname #1mark\endcsname {#7}}
  \fi \@xsect {#5}}

\def\@@@startsection#1#2#3#4#5#6{%
  \if@noskipsec \leavevmode \fi \par \@tempskipa #4\relax \@afterindenttrue 
  \ifdim \@tempskipa <\z@ \@tempskipa -\@tempskipa \@afterindentfalse 
  \fi \if@nobreak \everypar {}\else \addpenalty {\@secpenalty }\addvspace 
  {\@tempskipa }\fi \@ifstar {\@ssect {#3}{#4}{#5}{#6}}{\@dblarg 
    {\@@@sect {#1}{#2}{#3}{#4}{#5}{#6}}}}

\def\theparagraph{\thesection.\arabic{paragraph}}
\def\aparagraph{\@@@startsection{paragraph}{2}{\z@ }%
  {1.75ex plus .2ex minus .15ex}{-1em}{\bf(\theparagraph) } }
\def\paragraph{\@@@startsection{paragraph}{2}{\z@ }%
  {1.75ex plus .2ex minus .15ex}{-1em}{}{\bf(\theparagraph)} }

\let\c@theorem\c@paragraph

\title{Equivariant class group. III. \\
Almost principal fiber bundles}
\author{M{\sc itsuyasu} H{\sc ashimoto}}
\date{\normalsize
  Department of Mathematics, Okayama University\\
  Okayama 700--8530, JAPAN\\
  {\small \tt mh@okayama-u.ac.jp}\\
~\\
Dedicated to Professor Yuji Yoshino on the occasion of his 60th birthday
}

\begin{document}

\maketitle

\begin{abstract}
As a formulation of \lq codimension-two arguments' in invariant theory,
we define a (rational) almost principal bundle.
It is a principal bundle off closed subsets of codimension two or more.
We discuss the behavior of 
the category of reflexive modules over locally Krull schemes, 
the category of the coherent
sheaves which satisfy Serre's condition $(S'_2)$ over
Noetherian $(S_2)$ schemes with dualizing complexes,
the class group, 
the canonical module,
the Frobenius pushforwards, and global $F$-regularity, of a rational 
almost principal bundle.
We give examples of finite group schemes, multisection rings, surjectively
graded rings, and determinantal rings, and give unified treatment and 
new proofs to known results in invariant theory, algebraic geometry, and 
commutative algebra, and generalize some of them.
In particular, 
we generalize the result on the canonical module of the multisection ring
of a sequence of divisors by Kurano and the author.
We also give a new proof of a generalization of
Thomsen's result on the Frobenius 
pushforwards of the structure sheaf of a toric variety.
\end{abstract}

\tableofcontents

\setcounter{section}{-1}
\section{Introduction}

\paragraph
This paper is a continuation of \cite{Hashimoto4} and \cite{Hashimoto5}.

\paragraph
Let $k$ be a field, and $G$ an algebraic group scheme over $k$.
In geometric invariant theory, 
categorical quotients, geometric quotients, and principal fiber bundles 
play important roles.
Among them, principal fiber bundles behave well with respect to 
quasi-coherent sheaves, and they are interesting from the viewpoint of 
algebraic invariant theory.
Grothendieck's descent theorem tells us that if $\psi:X\rightarrow Y$ is
a principal $G$-bundle, then $\psi^*:\Qch(Y)\rightarrow \Qch(G,X)$ is
an equivalence of categories, see \cite[(3.13)]{Hashimoto4}.
This fact has played central role in our treatment of equivariant class groups
and Picard groups in \cite{Hashimoto4} and \cite{Hashimoto5}.

\paragraph
A $G$-invariant morphism $\varphi:X\rightarrow Y$ is said to be an
algebraic quotient (or an affine quotient) by $G$ if it is an
affine morphism, and the canonical map $\O_Y\rightarrow (\varphi_*\O_X)^G$
is an isomorphism.
If $B$ is a $G$-algebra (a $k$-algebra on which $G$-acts), then the canonical
map 
$\varphi:X=\Spec B\rightarrow \Spec B^G=Y$ is an algebraic quotient.
It is the central object in algebraic invariant theory.
This is not even a categorical quotient in general 
(see Example~\ref{masuda.ex}), and it seems that imposing 
geometric conditions should yield a good class of algebraic quotients.
However, the algebraic quotient $\varphi$ is rarely a principal
$G$-bundle.
For example, if $B=k[x_1,\ldots,x_n]$ is a polynomial ring and $G$
acts on $B$ linearly, then the origin is the fixed point, and hence
$\varphi$ is not a principal $G$-bundle unless $G$ is the trivial group.

\paragraph
However, when we remove closed subsets of codimension two or more
both from $X$ and $Y$, sometimes the remaining part is a principal bundle,
and this is sometimes useful enough in invariant theory.
We define that the diagram of $G$-schemes
\begin{equation}\label{almost-intro.eq}
\xymatrix{
X & U \ar@{_{(}->}[l]_i \ar[r]^\rho & V \ar@{^{(}->}[r]^j & Y
}
\end{equation}
is a {\em rational almost principal $G$-bundle} if $G$ acts on $V$ and $Y$
trivially, $i$ and $j$ are open immersions, $\codim_Y(Y\setminus j(V))
\geq 2$, $\codim_X(X\setminus i(U))\geq 2$, and $\rho:U\rightarrow V$ is
a principal $G$-bundle (cf.~Definition~\ref{rational-almost-pb.def}).
The name \lq rational' comes from the fact that $\xymatrix{X\ar@{.>}[r] & Y}$ 
is a rational map.
A $G$-invariant morphism $\varphi:X\rightarrow Y$ is called an
{\em almost principal $G$-bundle} if there exist some $U\subset X$ and
$V\subset Y$ such that (\ref{almost-intro.eq}) is a rational
almost principal $G$-bundle, where $i$ and $j$ are inclusions, and
$\rho$ is the restriction of $\varphi$ (cf.~Definition~\ref{almost-pb.def}).

\paragraph
If $X$ and $Y$ are locally Noetherian and normal, then the categories of 
reflexive modules $\Ref(V)$ and $\Ref(Y)$ are equivalent under the 
equivalences $j_*$ and $j^*$.
Similarly, the categories of $G$-equivariant reflexive modules
$\Ref(G,U)$ and $\Ref(G,X)$ are equivalent under the functors $i_*$ and $i^*$.
Finally, by Grothendieck's descent theorem, $\Ref(V)$ and $\Ref(G,U)$ are 
equivalent under the equivalences $\rho^*$ and $(?)^G\circ \rho_*$.
Combining them, we have that $\Ref(Y)$ and $\Ref(G,X)$ are equivalent 
under the equivalences $i_*\rho^*j^*$ and $(?)^G j_* \rho_*  i^*$ 
(note that $(?)^Gj_*$ and $j_*(?)^G$ are equivalent) 
(cf.~Theorem~\ref{main.thm}).
This is the main observation of this paper.

\paragraph
The first 
purpose of this paper is to show that properties of $X$ and those of 
$Y$ are deeply connected,
and we can get information of $Y$ from that on the action of $G$
on $X$, and vice versa.

\paragraph
This kind of argument, sometimes called codimension-two argument, is 
not new here.
We give one example which shows that our construction is a generalization
of a very important notion in invariant theory.

Let $G$ be a finite group acting on a variety $X$ over a field $k$.
Then for $g\in G$, we define $X_g:=\{x\in X\mid gx=x\}$.
If $\codim_X X_g=1$, then we say that $g$ is a pseudoreflection.
If $\varphi:X\rightarrow Y$ 
is an algebraic quotient by $G$, then 
we have that $\varphi$ is an almost principal
$G$-bundle if and only if the action of $G$ is small, that is, 
$G$ does not have a pseudoreflection (Proposition~\ref{finite-main.prop}).
We will show that some of the known important results on the
invariant subrings under the action of finite groups without pseudoreflections
can be generalized to the results on (rational) almost principal bundles.

\paragraph
The second purpose of this paper is to show that (rational)
almost principal bundles are so ubiquitous in invariant theory,
algebraic geometry, and commutative algebra.
As an application, we give new and short proofs to various known results.
Some of them are generalized using our approach.

\paragraph
In what follows, we list the results on the first purpose, that is,
the general results on the comparison of properties of $X$ and $Y$,
for a rational almost principal $G$-bundle (\ref{almost-intro.eq}).
These are proved in Chapter~1 
(sections~\ref{main.sec}--\ref{global-freg.sec}).

For simplicity, in the following list, we 
assume that $X$ and $Y$ are normal 
varieties over an algebraically closed base field $k$.

\paragraph
$\Ref(Y)$ and $\Ref(G,X)$ are equivalent, as we mentioned.

\paragraph
$\Cl(Y)\cong \Cl(G,X)$ (Theorem~\ref{main.thm}).
This is a consequence of the equivalence $\Ref(Y)\cong \Ref(G,X)$.
This result has essential overlap
with the work of Waterhouse \cite[Theorem~4]{Waterhouse} in the affine case.
Our approach also enables us to establish an exact sequence
\[
0\rightarrow H^1\alg(G,\O_X^\times)\rightarrow 
\Cl(Y)\rightarrow \Cl(X)^G\rightarrow H^2\alg(G,\O_X^\times),
\]
see Theorem~\ref{four-term.thm}.
The proof depends on the corresponding exact sequence for the Picard groups
developed in \cite{Hashimoto4}.

\paragraph\label{canonical-intro.par}
Let $\omega_Y$ be the canonical module of $Y$, and
$\omega_X$ be the $G$-canonical module of $X$.
Then $\omega_X\cong i_*\rho^*j^*\omega_X\otimes_k \Theta_G^*$, and
$\omega_Y\cong (j_*\rho_*i^*\omega_X)^G\otimes_k \Theta_G$,
where $\Theta_G=\Theta_{G,k}$ is a certain 
one-dimensional representation of $G$ determined only by $G$
(Theorem~\ref{canonical.thm}).
If $G$ is smooth, then $\Theta_G=\extop\Lie G$.
If $G$ is connected reductive or abelian, then $\Theta_G$ is trivial.
This kind of relationship between $\omega_X$ and $\omega_Y$ can be found in
the work of Knop \cite{Knop} on an action of an algebraic group over an
algebraically closed field of characteristic zero.
See also \cite{Peskin}.
As we also work over characteristic $p>0$ and treat non-reduced group
schemes too, the description of $\omega_X$ and 
$\Theta_G$ depends on the theory of 
equivariant twisted inverse developed in \cite{ETI}.

Although the situation is different, $\Theta_G$ plays a similar role as the
differential character (or different character) 
$\chi_{B,A}^{-1}$ (in the notation of \cite{FW},
where $X=\Spec B$ and $Y=\Spec A$)
played in the study of finite group actions in \cite{Broer} and \cite{FW}.
In the case that
the group $G$ is \'etale, $X=\Spec B$ and $Y$ are affine, and
$\varphi:X\rightarrow Y$ is 
an almost principal $G$-bundle with $B$ a UFD,
where our settings overlap with theirs,
$\Theta_G=\chi_{B,A}^{-1}$ is trivial (this case is treated in \cite{Braun}).
In these papers, the assumption that $B$ is a UFD (or a polynomial ring) 
was important in order
to make the different module $\Cal D_{B/A}$ rank-one free.
We are free from the assumption that $B$ is a UFD, and
we can treat higher dimensional and non-reduced group schemes.

\paragraph\label{Frobenius-pushforwards-intro.par}
Let the characteristic of $k$ be $p>0$.
If $\M\in\Ref(G,X)$, then the Frobenius pushforward $F^e_*({}^e\M)$ 
also lies in 
$\Ref(G,X)$, see for the notation, section~\ref{frob-twist.sec}.
This simple observation suggests that (rational) almost principal bundles
are useful in studying Frobenius pushforwards and related properties and
invariants of algebraic varieties in characteristic $p>0$.
Let $\N\in\Ref(Y)$ corresponds to $\M\in\Ref(G,X)$ under the equivalence
$\Ref(Y)\cong \Ref(G,X)$ above, and $e\geq 1$.
If $G$ is $k$-smooth, then the Frobenius pushforward $F^e_*({}^e\N)$ 
corresponds to the invariance $(F^e_*({}^e\M))^{G_e}$, where $G_e$ is the
$e$th Frobenius kernel of $G$ (Theorem~\ref{Frobenius-pushforward.thm}).
If $G$ is a finite group, then $G_e$ is trivial.
Applying this result, Nakajima and the author recently 
gave a description of the
generalized $F$-signatures of maximal Cohen--Macaulay modules over 
the invariant subrings under the action of
finite groups without pseudoreflections \cite{HN}.

\paragraph
Using the correspondence in (\ref{Frobenius-pushforwards-intro.par}) 
of Frobenius pushforwards,
we get information on the direct-sum decomposition of the Frobenius 
pushforwards $F^e_*({}^e\O_Y)$ from the information of the decomposition of
$(F^e_*({}^e\O_X))^{G_e}$.
The author expects that this observation will be useful in studying 
the problem of 
finite $F$-representation type defined by Smith and van den Bergh
\cite{SmVdB}.
As an application, we will give a short proof of a 
generalization of Thomsen's result \cite{Thomsen} on the decomposition of 
Frobenius pushforwards of the structure sheaf of 
toric varieties (Theorem~\ref{thomsen.cor}).

\paragraph Another related result on Frobenius maps is the heredity of global
$F$-regularity.
We say that an  integral Noetherian $\Bbb F_p$-scheme $X$ 
is globally $F$-regular if for any invertible sheaf $\L$ 
on $X$ and any nonzero section $s:\O_X\rightarrow \L$ of $\L$, 
there exists some $e\geq 1$ such that $sF^e:\O_{X^{(e)}}\rightarrow \L$ 
splits as an $\O_{X^{(e)}}$-linear map.
This was defined by Smith \cite{Smith} for projective varieties, and this
definition is its obvious extension.
She applied this notion to prove a vanishing theorem on a GIT quotient
of a complex Fano variety with rational Gorenstein singularities.
We prove that when $G$ is linearly reductive and both $X$ and $Y$ have
ample invertible sheaves,
$X$ is globally $F$-regular if and only if $Y$ is globally $F$-regular.
Note that as we work over characteristic $p>0$, if $G$ is affine and 
linearly reductive,
then the identity component $G^\circ$ of $G$ is diagonalizable 
\cite{Sweedler3}, and $G/G^\circ$ is a finite group whose order is not
divisible by the characteristic of $k$.

This is the end of the list of our general results.

\paragraph\label{group-scheme-intro.par}
Some of the results above are proved under more general settings,
see the text.
We tried to study not only smooth algebraic groups but also non-reduced
group schemes, as long as possible.
In fact, our base scheme $S$ is basically general, and our group scheme $G$
is basically general, except that it is almost always assumed to be flat.
Some additional assumptions are added case-by-case.

When we consider a torus action, our main construction has some applications
to algebraic geometry and commutative algebra 
(sections~\ref{multisection.sec}--\ref{surjective.sec}).
For a finitely generated torsion-free abelian group $\Lambda\cong\Bbb Z^s$, 
a ring $B$ with the action of the torus $G=\Spec\Bbb Z\Lambda$ 
is nothing but a $\Lambda$-graded ring.
Then the Veronese subring $B_\Gamma$ with respect to
a subgroup $\Gamma\subset\Lambda$ is nothing but $B^N$, where 
$N$ is the diagonalizable group scheme $\Spec\Bbb Z(\Lambda/\Gamma)$.
If the characteristic of the base field $k$ is $p>0$ and $\Lambda/\Gamma$ 
has a $p$-torsion, this is a non-reduced group scheme.
So the invariant theory of non-reduced group schemes arises naturally
in algebraic geometry and commutative algebra in characteristic $p>0$.
As an application of our main construction, we give a characterization of 
a standard graded $k$-algebra $B$ of dimension greater than or equal to two
whose Veronese subalgebra $B_{d\Bbb Z}$ is quasi-Gorenstein
(Proposition~\ref{degree-one.prop}).
Although their (seeming) statement was a little bit weaker,
Goto and K.-i.~Watanabe \cite[(3.2.1)]{GW} already proved it
(exactly the same proof works).

Recently, the author gave a classification of the linearly reductive finite
subgroup schemes of $\SL_2$ \cite{Hashimoto8}.
This enables us to write any complete local $F$-rational Gorenstein ring
of dimension two over an algebraically closed field of positive characteristic
as an invariant subring of such a subgroup scheme.
This has known to be true for sufficiently large characteristic,
but now the bad characteristics have also be covered, using 
non-reduced group schemes.
In this paper, we define the smallness of the action of a group scheme,
generalizing the action of finite groups without pseudoreflections, 
and show that the canonical action of a linearly reductive finite 
subgroup scheme of $\SL_n$ on $k[[x_1,\ldots,x_n]]$ is small
(Proposition~\ref{SL_n-small.prop}).
Applying the general results listed above to the action of $\SL_2$ 
on $k[[x,y]]$, we get
some basic and known results on the two-dimensional $F$-rational Gorenstein 
singularities (such as finite representation type property),
see Theorem~\ref{SL_2.thm}.

The author expects that 
the generalization from groups to group schemes will give interesting 
new aspects to invariant theory.

\paragraph
The codimension-two argument on reflexive sheaves works comfortably 
on Noetherian normal schemes.
However, as applications to commutative algebra are important motivations,
we mainly work on locally Krull schemes when we discuss class groups,
as in \cite{Hashimoto4} and \cite{Hashimoto5}.
A generalization to different direction is the coherent sheaves $\M$ which 
satisfy Serre's $(S'_2)$ condition (that is, 
$\depth_{\O_{Z,z}}\M_z\geq \min(2,\dim \O_{Z,z})$ for $z\in Z$) on a 
quasi-normal locally Noetherian scheme $Z$.
Quasi-normality is a notion which generalizes a 
normal scheme (a little bit more generally, schemes which satisfy
$(T_1)+(S_2)$ ($(T_1)$ is \lq Gorenstein in codimension one,'
it is also written as $(G_1)$ by some authors.
Our notation is after \cite{GM})) and a locally equidimensional 
scheme with a dualizing complex simultaneously, see (\ref{quasinormal.par}).
In particular, we generalize \cite[(1.12)]{Hartshorne4}.
A quasi-Gorenstein locally Noetherian scheme is quasi-normal, and the
generalization to this direction is used to reprove 
the theorem of Goto--Watanabe mentioned in (\ref{group-scheme-intro.par}).

\paragraph
As in \cite{Hashimoto5}, instead of working on a single group scheme $G$,
we mostly work on a short exact sequence
\[
1\rightarrow N\rightarrow G\xrightarrow f H\rightarrow 1
\]
of group schemes.
That is, $f:G\rightarrow H$ is a qfpqc homomorphism of group schemes, and
$N=\Ker f$.
For qfpqc morphisms, see \cite[section~2]{Hashimoto5}.
We say that a rational almost principal $N$-bundle 
(\ref{almost-intro.eq}) is {\em $G$-enriched} if it is also a diagram of 
$G$-morphisms.

For example, let $B=k[x_1,\ldots,x_n]$ be a polynomial ring over a field $k$,
and $N$ a finite subgroup of $\GL_n$ without pseudoreflection, 
acting on $X=\Spec B$ in a natural way.
Let $H=\Bbb G_m$ be the one-dimensional torus, and $G=N\times H$.
As the action of $N$ on $B$ preserves grading, $G$ acts on $B$.
As the inclusion $A\hookrightarrow B$ preserves grading, the
almost principal $N$-bundle $\varphi:X=\Spec B\rightarrow\Spec B^N=Y$ is
$G$-enriched.
By our general result, 
$\varphi^*:\Ref(H,Y)\rightarrow \Ref(G,X)$ is an equivalence
whose quasi-inverse is $(?)^N\varphi_*$ (Theorem~\ref{main.thm}).
Note that $\Ref(H,Y)$ is the category of graded reflexive $\O_Y$-modules, and
$\Ref(G,X)$ is the category of graded reflexive $(N,\O_X)$-modules.
So the auxiliary action of $H$ gives us the graded version of the
invariant theory.

Another example of an auxiliary action is that of Galois groups.
Let $k$ be a field, $N_1$ an \'etale $k$-group scheme, and $\varphi:
X\rightarrow Y$ an almost principal $N_1$-bundle.
Let $k'$ be a finite Galois extension of $k$ with the Galois group $H$
such that the base change $k'\otimes_k N_1$ is a constant finite group $N$
(such $k'$ always exists).
Then $H$ acts on $N$ by group automorphisms, and we can define the 
semidirect product $G:=N\rtimes H$, and the base change
$\varphi':X'\rightarrow Y'$, where the base field is still $k$, not $k'$, 
is a $G$-enriched almost principal $N$-bundle ($H$-equivariant almost 
principal $N$-bundle).
Even if the original $N_1$ is not constant, $G$ and $H$ are 
constant finite groups, 
and we can utilize the usual group theory to study $\varphi'$.

Yet another example can be found in the study of the Cox rings of 
toric varieties, see Proposition~\ref{cox-main.prop}.
See also Lemma~\ref{Veronese-main.lem}.

\paragraph
In Chapter~2 
(sections~\ref{finite.sec}--\ref{determinantal.sec}), we show
various examples of (rational) almost principal bundles and give applications.

\paragraph
The first example is the finite group schemes.
As we have mentioned, an action of a finite group $G$ on an affine
algebraic variety $X=\Spec B$ yields an almost principal $G$-bundle
$\varphi:X=\Spec B\rightarrow \Spec B^G=Y$ if and only if the action
is small (that is, $G$ does not have a pseudoreflection).
For a general finite group scheme action,
we defined the smallness of the action via the
largeness of the free locus of the action, see (\ref{definition-small.par}).
The author does not know how to redefine the smallness using the
non-existence of pseudoreflections for general finite subgroup 
schemes of $\GL_n$, see Remark~\ref{ps-small.rem}.

We call a group scheme $h:G\rightarrow S$ 
on a scheme $S$ is {\em locally finite free} (LFF for short) 
if it is finite and the structure sheaf $h_*\O_G$ is a locally free
sheaf on $S$.
We work on LFF group schemes (as a generalization of finite group
schemes over a field), and prove that an algebraic quotient
$\varphi:X=\Spec B\rightarrow Y=\Spec B^G$ is an almost principal 
$G$-bundle if and only if the action is small 
(Proposition~\ref{finite-main.prop}).
Thus we know that the action of a finite group $G$ without pseudoreflection
on an affine variety yields an almost principal $G$-bundle, as we have
already mentioned.
This fact is also useful in finding examples of (rational) almost principal
bundles with respect to non-reduced finite group schemes in 
(\ref{group-scheme-intro.par}).

As an application, assuming that $X$ satisfies the $(S_2)$ condition, 
we give a characterization of an algebraic 
quotient $\varphi:X\rightarrow Y$ by the action of an LFF group scheme $G$
such that $Y$ is (connected Noetherian with a dualizing complex and)
quasi-Gorenstein (Theorem~\ref{finite-resume.thm}, {\bf 6}).
If, moreover, $G$ is \'etale; abelian group scheme over a field; or 
a linearly reductive group scheme over a 
field (more generally, a Reynolds group scheme over $S$, see below),
then 
we have a very simple relationship: $\omega_Y\cong(\varphi_*\omega_X)^G$.
If, moreover, $X$ is normal, then we have 
$\omega_X\cong(\varphi^*\omega_Y)^{**}$ 
(Theorem~\ref{finite-resume.thm}, {\bf 4}).
When the group scheme is \'etale, there are considerable overlaps with the 
results of \cite{Peskin}, \cite{Broer}, \cite{Braun}, and \cite{FW}.

Also, well-known formula for the class group of $A$ is generalized to
the action of non-reduced finite group schemes, see
Example~\ref{finite-polynomial.ex}.

We point out that if the base scheme $S$, the group scheme $G$, and
the scheme $X=\Spec B$ are affine, to say that $\varphi:X=\Spec B\rightarrow
\Spec B^G$ is an almost principal $G$ bundle is the same as to say that
$B^G\rightarrow B$ is a pseudo-Galois extension in the sense of 
Waterhouse \cite{Waterhouse2} by definition.
His study on the class group is applicable to finite group schemes also,
and our work has many overlaps with his.

\paragraph\label{multisection-intro.par}
Next example is a rational almost principal $G$-bundle arising from
a sequence of divisors $D_1,\ldots,D_s$ over a Noetherian normal variety 
(more generally, a quasi-compact quasi-separated locally Krull 
integral scheme) $Y$, where $G$ is the torus $\Bbb G_m^s$, and
we assume that $\sum_i \Bbb Z D_i$ contains an ample Cartier divisor.
Let $j:V=Y\reg\hookrightarrow Y$ be the regular locus of $Y$ 
(for simplicity, assume that 
$Y$ is Noetherian), and let
\[
\rho:
U=\uSpec_V(\bigoplus_{\lambda\in\Bbb Z^s}\O_Y(\sum_i\lambda_i D_i)|_V 
\cdot t^\lambda)
\rightarrow V
\]
be the canonical map.
Also, let 
\[
X:=\Spec\Gamma(U,\O_U)
=\Spec (\bigoplus_{\lambda\in\Bbb Z^s}\Gamma(Y,\O_Y(\sum_i\lambda_i D_i))\cdot
t^\lambda),
\]
and $i:U\rightarrow X$ be the 
canonical map.
Then
\[
\xymatrix{
X & U \ar@{_{(}->}[l]_i \ar[r]^\rho & V \ar@{^{(}->}[r]^j & Y
}
\]
is a rational almost principal $G$-bundle (Theorem~\ref{kurano-main.thm}).
No map is defined from $X$ to $Y$ here, and this gives an example of
a rational almost principal bundle which is not an almost principal bundle.
This construction already essentially appeared in \cite{HK}
without the formal definition of rational almost principal bundles.
By our main theorem (Theorem~\ref{main.thm}), 
we get an equivalence between the categories $\Ref(Y)$ and
$\Ref(G,X)$ in an explicit way (Corollary~\ref{kurano-correspondence.cor}).
Also, the description of the canonical module of
a multisection ring (where $Y$ is a projective normal variety over a 
field) in \cite{HK} is generalized to a result on Noetherian normal 
integral schemes (Proposition~\ref{multi-canonical.prop}).
A part of the results on the class group of the multisection ring in \cite{EKW}
is also reproved as a theorem on locally Krull schemes
(Proposition~\ref{multi-class.prop}).

Let $\Lambda=\Bbb Z^s$ so that $G=\Spec \Bbb Z\Lambda\times_{\Spec\Bbb Z}S$.
Let $\Gamma$ be a subgroup of $\Lambda$, and set $H=
\Spec \Bbb Z\Gamma\times_{\Spec\Bbb Z}S$, and $f:G\rightarrow H$ the
canonical map.
Then $N:=\Ker f$ is nothing but $\Spec\Bbb Z(\Lambda/\Gamma)\times_{\Spec
\Bbb Z}S$.
Let $B$ be the multisection ring 
$\bigoplus_{\lambda\in\Lambda}\Gamma(Y,\O_Y(\sum_i\lambda_i D_i))\cdot t^\lambda$.
Then $B^N$ is the Veronese subring
$B_\Gamma=
\bigoplus_{\lambda\in\Gamma}\Gamma(Y,\O_Y(\sum_i\lambda_i D_i))\cdot t^\lambda$.
We show that the canonical algebraic quotient 
$\theta:X=\Spec B\rightarrow \Spec B^N=X'$ is a $G$-enriched 
almost principal $N$-bundle (Lemma~\ref{Veronese-main.lem}).
Consequently, we can prove some results which connect $X$ and $X'$.

\paragraph 
When we apply the construction explained 
in (\ref{multisection-intro.par}) to the Cox ring of a toric variety
(with a torsion-free class group), then we get some basic information
on toric varieties, such as the description of the canonical module
(Corollary~\ref{toric-canonical.cor}), and the global $F$-regularity
(Proposition~\ref{toric-globally-freg.prop}).
Also, we prove that for a toric variety $Y$ over a perfect field $k$, 
there exist finitely many equivariant 
rank-one reflexive modules $\M_1,\ldots,\M_u$
on $Y$ (equivariant with respect to the torus action) such that 
any Frobenius pushforward $F^e_*({}^e\O_Y)$ is a finite direct sum of
copies of them, as $\O_Y$-modules, generalizing the theorem of 
Thomsen \cite{Thomsen} 
on non-singular toric varieties.
This has been known also for affine toric
varieties \cite{Bruns}.

\paragraph
Although we can construct a rational almost principal bundle from 
a set of divisors on a normal variety, it seems difficult to
find a rational almost principal bundle from a given multigraded ring $B$.
But this is relatively easy when $B$ is {\em surjectively graded}.
This notion first appeared in \cite{Hashimoto9} for the case that
$B$ is a domain.
We modify this to a usable definition for the case that $B$ is not a domain,
and give an easy way to get many rational almost principal bundles
from multigraded rings (Lemma~\ref{surjective-almost.lem}).

The most typical example is a standard graded algebra $B=\bigoplus_{n\geq 0}
B_n$ with $\dim B\geq 2$.
Then letting $X=\Spec B$, $U=X\setminus 0$, and $Y=\Proj B$, we get a 
rational almost principal bundle
\[
\xymatrix{
X & U \ar@{_{(}->}[l]_i \ar[r]^\rho & Y \ar@{^{(}->}[r]^{1_Y} & Y
},
\]
where $0$ is the origin of the affine cone $X$.
From this construction, we get a very short proof of 
Grothendieck's theorem which tells that
any locally free sheaf on $\Bbb P^1$ is uniquely a direct sum of 
$\O(n)$ (Example~\ref{Gro.ex}).
This is simply because $\Ref(\Bbb P^1)$ and $\Ref(\Bbb G_m,k[x,y])$ are
equivalent by our main theorem.

\paragraph
There are also examples where the group scheme $G$ is not a torus or 
a finite group scheme.
We point out that determinantal and Pfaffian varieties yield 
examples of almost principal $G$-bundles where $G$ is a connected
reductive group which are not finite (section~\ref{determinantal.sec}).

Given a $G$-algebra Krull domain $B$ and a candidate Krull domain $A\subset
B^G$ of $B^G$, 
it is sufficient to prove that $\varphi:\Spec B\rightarrow \Spec A$ is
an almost principal $G$-bundle 
in order to show $A=B^G$ (Theorem~\ref{alg-quot-krull-equiv.thm}).
Thus, proving that $A=B^G$ is 
reduced to proving that $A$ is a Krull domain, when we know
geometric information that $\varphi$ is an almost principal bundle.
This technique essentially appeared in \cite{Hashimoto11}, and
applied to the  same examples.

An example of an action of the additive group $\Bbb G_a$ is also given
(Example~\ref{masuda.ex}).

\paragraph
Some miscellaneous problems are also discussed in this paper, in order to 
overcome technical difficulties to discuss main ingredients.
Most of them are contained in Chapter~0 
(sections~\ref{quotients.sec}--\ref{semireductive.sec}).
We will see them below.

\paragraph
We overview the contents of this paper section by section.

Chapter~0 (sections~\ref{quotients.sec}--\ref{semireductive.sec})
is preliminaries.

In section~\ref{quotients.sec}, we review some basic definitions and 
facts on quotients.
In section~\ref{invariance.sec}, we discuss the compatibility of the
invariance functor and some other operations on sheaves.
In section~\ref{functorial.sec},
we discuss the problem of functorial resolutions.
In section~\ref{restriction.sec}, we discuss the compatibility of the
restriction functor and some other operations on sheaves.
In section~\ref{Reynolds.sec}, generalizing linearly reductive group
schemes over a field, we define Reynolds group schemes and discuss
basic properties.
This class of group schemes 
includes 
linearly reductive group schemes over a field, 
finite groups with the order invertible in the base 
ring, and diagonalizable group schemes (including split tori) over
an arbitrary base ring 
(Examples~\ref{linearly-reductive-Reynolds.ex}--\ref{diagonalizable-Reynolds.ex}).
In section~\ref{base-change.sec}, we discuss the base-change map of twisted inverse
pseudofunctors.
In section~\ref{Serre-canonical.sec}, we discuss the (equivariant) canonical 
modules.
As a generalization of a special case of Knop's work over a field of 
characteristic zero, we give the correspondence between $\omega_X$ and
$\omega_Y$ for a principal $G$-bundle $X\rightarrow Y$.
Also, generalizing the facts on the category of reflexive sheaves on
normal varieties, we show that the category of sheaves $\M$ which satisfy 
the $(S'_2)$ condition (that is, $\depth \M_z\geq \min(2,\dim \O_{Z,z})$
for $z\in Z$) behaves very similarly on quasi-normal Noetherian schemes
$Z$, and we show that codimension-two argument works.
We generalize \cite[(1.12)]{Hartshorne4}.
On the way, we generalize the well-knwon 
result on the equivalence of $(S'_2)$,
reflexive, and being a second syzygy due to 
Evans and Griffith \cite[Theorem~3.6]{EG}, using the new notion of 
$2$-canonical modules (Lemma~\ref{2-canonical-double-dual.lem}).
Recently, similar results in slightly different contexts are obtained by
Dibaei--Sadeghi \cite{DS} and Araya--Iima \cite{AI}.
In section~\ref{frob-twist.sec}, we define a new category
to treat Frobenius twists and Frobenius kernels effectively.
This enables us to discuss the Frobenius kernels of a group scheme 
over an arbitrary $\Bbb F_p$-scheme.
In section~\ref{semireductive.sec}, we discuss when $B$ is finite over 
$B^G$, and when $B^G$ is $F$-finite for an affine algebraic group scheme 
$G$ over a field $k$ of characteristic $p>0$ and a $G$-algebra $B$.
The main result is Lemma~\ref{F-finite-finite.thm}.
If $G$ is a constant finite group, then the lemma is a special case of 
\cite[Theorem]{Fogarty}.

\paragraph
After these preliminaries, in Chapter~1 
(sections~\ref{main.sec}--\ref{global-freg.sec}), 
we do the main definitions and discuss
general properties of rational almost principal bundles.

In section~\ref{main.sec}, we give main definitions, and discuss the
problem of base change.
We also prove Theorem~\ref{alg-quot-krull-equiv.thm} mentioned
above.
In section~\ref{class-canonical.sec}, we discuss the behavior of the
(equivariant) class group, using the results obtained in 
\cite{Hashimoto4} and \cite{Hashimoto5}.
We also discuss the behavior of the canonical modules, using the
results in sections~\ref{base-change.sec} and \ref{Serre-canonical.sec}.
We discuss the behavior of the Frobenius pushforwards with respect to 
rational almost principal bundles in section~\ref{Frobenius-pushforwards.sec}.
The author expects 
some future applications on
the problems in invariant theory related to 
the characteristic $p$ commutative algebra.
The paper \cite{HN} is a trial toward this direction.
Section~\ref{global-freg.sec} is on the global $F$-regularity.
As our construction uses open subschemes, we discuss global
$F$-regularity of schemes which may not be projective.

\paragraph
After proving general results, we give examples and applications of
rational almost principal bundles in Chapter~2 
(sections~\ref{finite.sec}--\ref{determinantal.sec}).
In section~\ref{finite.sec}, we discuss finite group schemes 
(more precisely, LFF group schemes).
In particular, we prove a similar result to the results on the
canonical modules for the finite group actions due to 
Broer \cite{Broer} and Fleischmann--Woodcock \cite{FW}.
In section~\ref{multisection.sec}, we construct a rational almost
principal bundle from a sequence of divisors on a locally Krull scheme, 
and prove a generalization of the theorem of Kurano and the author 
which describes the canonical module of the multisection ring.
As an application, we prove some known and new results on toric 
varieties, using the Cox ring in section~\ref{toric.sec}.
In section~\ref{surjective.sec}, we give a way to construct 
rational almost principal bundles from a multigraded rings.
This enables us to study the Veronese subring using our approach.
In section~\ref{determinantal.sec}, we show that determinantal and
Pfaffian varieties treated by De Concini and Procesi \cite{DP} give
examples of almost principal bundles.

\paragraph
Acknowledgments:
The author is grateful to Professor Kayo Masuda for
kindly showing me that Example~\ref{masuda.ex} is an example of 
an algebraic quotient which is non-surjective.
Special thanks are also due to 
Professor Kazuhiko Kurano,
Dr.~Yusuke Nakajima,
Dr.~Akiyoshi Sannai,
and
Professor Takafumi Shibuta
for stimulating discussion.
The author also thanks 
Professor Tokuji Araya,
Professor Shiro Goto,
Professor Nobuo Hara,
Dr.~Ryo Kanda,
Professor Takesi Kawasaki,
Professor Shunsuke Takagi,
Professor Ryo Takahashi,
Professor Kohji Yanagawa,
and
Professor Yuji Yoshino
for valuable advice.
He is also grateful to 
Professor Kei-ichiro Iima and
Professor Gregor Kemper for valuable 
comments on an earlier version of this paper.

\section*{\large Chapter~0. Preliminaries}\label{preliminaries.chap}

\section{Actions and quotients}\label{quotients.sec}

\paragraph
This paper is a continuation of \cite{Hashimoto4} and \cite{Hashimoto5}.
We follow the notation and terminology of these papers.
In particular, for the notation and terminology on sheaves over
diagrams of schemes and equivariant modules, we follow
\cite{ETI}, \cite{HO2}, and \cite{HO}, unless otherwise specified.
Unexplained notation and terminologies on commutative algebra,
algebraic geometry, algebraic groups, representations of algebraic groups, and
Hopf algebras that are not in these
should be found in \cite{CRT}, \cite{Hartshorne},
\cite{EGA}, \cite{Borel}, 
\cite{Jantzen}, \cite{Sweedler2}, or \cite{Hashimoto7}.
Throughout this paper, $S$ denotes a (base) scheme.

\paragraph
Let $G$ be an $S$-group scheme, and $\varphi:X\rightarrow Y$ a $G$-invariant
morphism.
The {\em secondary map} associated with $G$ and $\varphi$ is the map
\[
\Psi=\Psi_{G,\varphi}:G\times X\rightarrow X\times_Y X
\]
given by $(g,x)\mapsto (gx,x)$.
This map is independent of the choice of $S$ in the sense that when we replace
$S$ by $Y$ and $G$ by $G_Y=G\times Y$, then we get the same map (over the
base scheme $Y$).
If $h:Y'\rightarrow Y$ is an $S$-morphism between $S$-schemes with trivial
$G$-actions, then $\Psi_{G,\varphi'}:G\times X'\rightarrow X'\times_{Y'}X'$ 
is identified with $1_{Y'}\times \Psi_{G,\varphi}:Y'\times_Y(G\times X')
\rightarrow Y'\times_Y(X\times_Y X)$.

\paragraph
Let $G$ be an $S$-group scheme.
A $G$-invariant morphism
$\varphi:X\rightarrow Y$ is said to be a {\em categorical quotient} 
if for any $G$-invariant morphism $\psi:X\rightarrow Z$, 
there exists some unique $S$-morphism $\theta:Y\rightarrow Z$ such that
$\psi=\theta\varphi$.
The categorical quotient is unique (in the category of $G$-schemes under $X$).

\paragraph
A $G$-invariant morphism
$\varphi:X\rightarrow Y$ is said to be an {\em algebraic quotient} 
or {\em affine quotient} by 
the action of $G$
if it is an affine morphism, and $\bar\eta:\O_Y\rightarrow(\varphi_*\O_Y)^G$ 
is an isomorphism.
An algebraic quotient need not be surjective.
It need not be a categorical quotient either in general, see
Example~\ref{masuda.ex} below.
However, if $S=\Spec k$ is a field and $G$ is a semireductive
$k$-group scheme (see (\ref{semireductive.par}) below), 
it is a categorical quotient 
(see Lemma~\ref{universally-submersive.lem}).

\paragraph
A morphism of schemes $h:Z\rightarrow W$ is said to be submersive if 
$h$ is surjective, and for any subset $U$ of $W$, $U$ is open if and only if 
$h^{-1}(U)$ is open in $Z$.
A $G$-invariant morphism 
$\varphi:X\rightarrow Y$ is said to be a {\em geometric quotient} 
if it is submersive, $\O_Y\rightarrow (\varphi_*\O_X)^G$ is an isomorphism,
and $\Psi_{G,\varphi}$ is surjective.
A geometric quotient is a categorical quotient \cite[(0.0.1)]{GIT}.
By definition, an affine geometric quotient is an algebraic quotient.
However, a geometric quotient need not be an algebraic quotient in general.
For example, let $G=\SL_n$ with $n\geq 2$ over an algebraically closed field 
$k$,
and consider the structure map $\varphi:X=G/B\rightarrow \Spec k=Y$,
where $B$ is the subgroup of the upper triangular matrices in $G$.
It is a geometric quotient by $G$, but is not an affine morphism,
since $G/B$ is a projective variety of dimension one or more 
\cite[(11.1)]{Borel}.
So in particular, we have an example of a categorical quotient which is
not an affine morphism.

\paragraph We say that $\varphi:X\rightarrow Y$ is a universal 
(resp.\ uniform) categorical quotient by $G$ if for any $S$-morphism 
(resp.\ any flat $S$-morphism) $Y'\rightarrow
Y$, the base change $\varphi':X'\rightarrow Y'$ is a categorical quotient
by $G$.
A similar definition is done for algebraic and geometric quotients.
An algebraic quotient is uniform under very mild conditions,
see Corollary~\ref{uniform-algebraic-quotient.cor} below.

\paragraph
Let $G$ be an $S$-group scheme acting on an $S$-scheme $X$.
Let us consider $\Psi=\Psi_{G,h_X}:
G\times X\rightarrow X\times X$, where $h_X:X\rightarrow S$ is the structure 
map.
It is easy to see that $\phi:\Psi^{-1}(X)\rightarrow X$ induced by 
$\Psi$ is an $X$-subgroup scheme of $G\times X$, where $X$ is embedded in 
$X\times X$ via the diagonal map (if $h_X$ is separated, then it is a
closed subgroup).
$\Cal S_X:=\Psi^{-1}(X)$  is called the {\em stabilizer} 
of the action of $G$ on $X$.
If $\Cal S_X$ is trivial as an $X$-group scheme, then we say that the action of
$G$ on $X$ is {\em free}.
We say that the action of $G$ on $X$ is {\em GIT-free} if $\Psi$ is a 
closed immersion.
Obviously, a GIT-free action is free.

\begin{lemma}\label{stabilizer-basics.lem}
Let $G$ be an $S$-group scheme, and $\psi:X'\rightarrow X$ be a $G$-morphism
which is a morphism of schemes.
Then there is an inclusion $Q:\Cal S_{X'}\rightarrow \Cal S_X\times_X X'$
of $X'$-subgroup schemes of $G\times X'$.
In particular, if the action of $G$ on $X$ is free, then so is the
action of $G$ on $X'$.
If $\psi$ is a monomorphism \(e.g., an immersion\), 
then $Q$ is an isomorphism.
\end{lemma}

\begin{proof}
Note that both $\Cal S_{X'}$ and $\Cal S_X\times_X X'$ are
$X'$-subgroup schemes of $G\times X'$.
For an $S$-scheme $W$, 
\[
\Cal S_{X'}(W)=\{(g,x')\in G(W)\times X'(W)\mid gx'=x'\}
\]
and 
\[
(\Cal S_X\times_X X')(W)
=\{(g,x')\in G(W)\times X'(W)\mid \psi(gx')=\psi(x')\}.
\]
So
$\Cal S_X\times_X X'$ contains $\Cal S_{X'}$.

If the action of $G$ on $X$ is free, then $\Cal S_X$ is trivial.
So $\Cal S_{X'}$ is also trivial by the discussion above, and the
action of $G$ on $X'$ is free.
The argument above also shows that if $\psi$ is a monomorphism, then
$\Cal S_{X'}=\Cal S_X\times_X X'$.
\end{proof}

\paragraph
Let $G$ be an $S$-group scheme acting on $X$, and $\psi:X'\rightarrow X$ any
monomorphism of $S$-schemes ($X'$ need not be a $G$-scheme).
Then we define the stabilizer at $X'$ of the action of $G$ on $X$ by
the $X'$-group scheme $\Cal S_{X'}:=\Cal S_X\times_X X'$.
This definition does not cause a confusion by 
Lemma~\ref{stabilizer-basics.lem}.
If $x$ is a point of $X$, the stabilizer 
$\Cal S_x$ is a $\kappa(x)$-subgroup scheme of $G\times x$.

\paragraph
A $G$-invariant morphism $\varphi:X\rightarrow Y$ is a principal $G$-bundle
if and only if 
it is qfpqc and $\Psi_{G,\varphi}$ is an isomorphism \cite[(2.8)]{Hashimoto5}.
A principal $G$-bundle is a universal geometric quotient 
\cite[(6.2)]{Hashimoto5}.
If, moreover, $G$ is a normal subgroup scheme of an $S$-group scheme
$\tilde G$ and $\varphi$ is also a $\tilde G$-morphism, then we say
that $\varphi$ is {\em $\tilde G$-enriched}.

\begin{lemma}\label{GIT-remark.lem}
Let $G$ be an $S$-group scheme,
and $\varphi:X\rightarrow Y$ a $G$-invariant submersive 
\(resp.\ universally submersive\) morphism such that 
$\Psi:G\times X\rightarrow X\times_Y X$ given by $\Psi(g,x)=(gx,x)$ is
surjective.
If $U$ is a $G$-stable open subset of $X$, then $\varphi(U)$ is an open 
subset, and $\varphi^{-1}(\varphi(U))=U$.
If, moreover, $G$ is universally open, then $\varphi$ is open
\(resp.\ universally open\).
\end{lemma}

\begin{proof}
This is essentially \cite[(0.2), Remark~(4)]{GIT}.
\end{proof}

\section{Compatibility of $G$-invariance and direct and inverse images}
\label{invariance.sec}

\paragraph\label{G-trivial.par}
Let $G$ be an $S$-group scheme and $Z$ an $S$-scheme on which $G$ acts
trivially.
Set $(?)^G=(?)_{-1}R_{\Delta_M}:\Mod(G,Z)\rightarrow\Mod(Z)$,
and
$\L=(?)_{\Delta_M}L_{-1}:\Mod(Z)\rightarrow \Mod(G,Z)$.
Note that $\L$ is left adjoint to $(?)^G$.

Using the description of $R_{\Delta_M}$ \cite[(6.14)]{ETI},
$(?)^G\M=\M^G$ is the kernel of the map
\[
\M_0\xrightarrow{\beta_{\delta_0}-\beta_{\delta_1}}p_*\M_1,
\]
where $p:G\times Z\rightarrow Z$ is the second projection, which
equals the action (because the action is assumed to be trivial),
and $\delta_i:[0]=\{0\}\rightarrow[1]=\{0,1\}$ is the map given by
$\delta_i(0)=1-i$ (for the notation on simplicial objects, see 
\cite[Chapter~9]{ETI}).
We call $\M^G\in\Mod(Z)$ the {\em $G$-invariance} of $\M$ 
\cite[(30.1)]{ETI}, \cite[(5.30)]{Hashimoto5}.
The natural inclusion
\[
\M^G\hookrightarrow \M_0
\]
is denoted by $\gamma$.

\begin{lemma}\label{G-trivial.thm}
For $\M\in\Mod(G,Z)$, the following are equivalent.
\begin{enumerate}
\item[\bf 1] $\M\cong \L\N$ for some $\N\in\Mod(Z)$.
\item[\bf 2] $\M$ is equivariant, and $\beta_{\delta_0}=\beta_{\delta_1}$.
\item[\bf 3] $\M$ is $G$-trivial, that is, 
$\M$ is equivariant, and $
\gamma:\M^G\rightarrow \M_0$ is an isomorphism \(see {\rm\cite[(30.4)]{ETI}}\).
\item[\bf 4] The counit of adjunction 
$\varepsilon:\L\M^G\rightarrow \M$ is an isomorphism.
\end{enumerate}
\end{lemma}

\begin{proof}
{\bf 1$\Rightarrow$2}.
Since $[-1]$ is the initial object of $\Delta^+_M$, $\M\cong
\L\N$ is equivariant by \cite[(6.38)]{ETI}.
By definition, $(\L\N)_{n}=(\tilde B_G^M(Z))_{\varepsilon(n)}^*\N$,
where $\varepsilon(n):[-1]=\emptyset\rightarrow[n]$ is the unique map.
The map $\beta_i:(\L\N)_0=\N\rightarrow p_*(\L\N)_1=p_*p^*\N$ is the
unit map, and is independent of $i$.

{\bf 2$\Leftrightarrow$3} is trivial.

{\bf 3$\Rightarrow$4}.
By {\bf1$\Rightarrow$2}, $\L\M^G$ is equivariant, and $\M$ is
assumed to be equivariant.
Hence it suffices to prove that $\varepsilon_0:(\L\M^G)_0\rightarrow \M_0$ is
an isomorphism, since the restriction $(?)_0:\EM(G,Z)\rightarrow \Mod(Z)$ 
is faithful.
However, this map is identified with $\gamma:\M^G\rightarrow \M_0$.

{\bf 4$\Rightarrow$1}.
This is trivial.
\end{proof}

\paragraph\label{epsilon.par}
Let $G$ be an $S$-group scheme and $h:Z'\rightarrow Z$ be a morphism of
$S$-schemes on which $G$ acts trivially.
Then the canonical map
\[
\epsilon: h^*\M^G\rightarrow (h^*\M)^G
\]
is defined to
be the composite
\begin{multline*}
\epsilon: h^*\M^G=h^*(?)_{-1}R_{\Delta_M}\M
\xrightarrow\theta
(?)_{-1}\tilde B_G^M(h)^*R_{\Delta_M}\M\\
\xrightarrow\mu
(?)_{-1}R_{\Delta_M}B^M_G(h)^*\M
=
(?)^GB^M_G(h)^*\M=(h^*\M)^G,
\end{multline*}
see \cite[(7.4)]{HO2}.
It is an isomorphism between functors from $\Lqc(G,Z)$ to $\Qch(Z')$ 
if $G$ is quasi-compact quasi-separated and 
$h$ is flat \cite[(7.5)]{HO2} (the flatness of $G$ is assumed there,
but that assumption is unnecessary).
Note that as in \cite[(2.18)]{HM}, $B_G^M(h)^*\M$ is abbreviated as $h^*\M$, by
abuse of notation (although $\M$ may not be quasi-coherent here).
Note that $\epsilon$ is a 
natural transformation between the functors from $\Mod(G,Z)$ to
$\Mod(Z')$.

\begin{lemma}\label{eta-bar-eta.thm}
The diagram
\[
\xymatrix{
h^*(?)^G \ar[r]^\gamma \ar[d]^\epsilon & 
h^*(?)_0 \ar[d]^{\theta} \\
(?)^Gh^* \ar[r]^\gamma & (?)_0 h^*
}
\]
is commutative.
\end{lemma}

\begin{proof}
Follows easily from the commutative diagram in \cite[(6.27)]{ETI}.
\end{proof}

\begin{corollary}\label{epsilon-composition.cor}
Let $h':Z''\rightarrow Z'$ and $h:Z'\rightarrow Z$ be a sequence of
$G$-morphisms.
Then the composite
\[
(hh')^*(?)^G\xrightarrow{d^{-1}}(h')^*h^*(?)^G\xrightarrow\epsilon
(h')^*(?)^Gh^*\xrightarrow\epsilon (?)^G(h')^*h^*
\xrightarrow{d}(?)^G(hh')^G
\]
agrees with $\epsilon$.
\end{corollary}

\begin{proof}
This follows easily from Lemma~\ref{eta-bar-eta.thm} and
\cite[(1.23)]{ETI}.
\end{proof}

\paragraph Let $\varphi:X\rightarrow Y$ and $\psi:Y\rightarrow Z$ be
morphisms of ringed sites.
Then the two functors $(\psi\varphi)_*$ and $\psi_*\varphi_*$ are equal, and
the standard natural isomorphism $c:(\psi\varphi)_*\rightarrow
\psi_*\varphi_*$ is nothing but the identity map.

\paragraph
Let $I$ be a small category, and $X$ an
$I\op$-diagrams of schemes.
For $\M\in\Mod(X)$ and $i\in I$, $\M_i\in\Mod(X_i)$.
For each $\phi\in I(i,j)$, $\beta_\phi:\M_i\rightarrow (X_\phi)_*\M_j$ is
induced, and the composite
\begin{equation}\label{beta.eq}
\M_i\xrightarrow{\beta_\phi}
(X_\phi)_*\M_j
\xrightarrow{\beta_\psi}
(X_\phi)_*(X_\psi)_*\M_k
\xrightarrow{c}
(X_{\psi\phi})_*\M_k
\end{equation}
agrees with $\beta_{\psi\phi}$ for any sequence of morphisms
\begin{equation}\label{composite.eq}
i\xrightarrow{\phi}j\xrightarrow{\psi}k
\end{equation}
in $I$, see \cite[(4.10)]{ETI}.
We call the collection $((\M_i)_{i\in I},(\beta_\phi)_{\phi\in\Mor(I)})$ 
the {\em structure data} of $\M$.
This data exactly determines $\M$ (not up to isomorphisms).

Conversely, if $\M_i\in\Mod(X_i)$ for $i\in I$, 
$\beta_\phi:\Hom_{\O_X}(\M_i,(X_\phi)_*\M_j)$ for $\phi:i\rightarrow j$,
and the composite (\ref{beta.eq}) agrees with $\beta_{\psi\phi}$ for
(\ref{composite.eq}), then it is a structure data for a 
unique $\M\in \Mod(X)$.

\paragraph
Let $f:\M\rightarrow \N$ be a morphism in $\Mod(X)$.
Then $f_i:\M_i\rightarrow \N_i$ is a morphism in $\Mod(X_i)$ such that
$\beta_\phi \circ f_i=(X_\phi)_*f_j \circ \beta_\phi$ for each $\phi:i
\rightarrow j$.
Conversely, such a collection $(f_i)$ gives a unique morphism 
$f:\M\rightarrow \N$.
We call $(f_i)$ the structure data of $f$.

\paragraph
Let $\varphi:X\rightarrow Y$ be a morphism of $I\op$-diagrams of schemes,
and $\M\in\Mod(X)$.
Then the structure data of $\varphi_*\M$ is as follows.
$(\varphi_*\M)_i=(\varphi_i)_*\M_i$, and
$\beta_\phi(\varphi_*\M)$ is the composite
\[
(\varphi_i)_*\M_i\xrightarrow{\beta_\phi}
(\varphi_i)_*(X_\phi)_*\M_j
\xrightarrow{c}
(Y_\phi)_*(\varphi_j)_*\M_j,
\]
as can be seen easily from the direct computation (note that $c$ is the
identity map).

\paragraph
Let $I$ be a small category, $X$ be an $I\op$-diagram of schemes, 
$J$ be a subcategory of $I$, and $\N\in\Mod(X_J)$.
Then $R_J\N\in\Mod(X)$ is given by its structure data.
$(R_J\N)_i=\projlim (X_\phi)_*\N_j$, where the limit is taken over the 
comma category $(i\downarrow J)$, see \cite[(II.6)]{Mac Lane} 
(it is $I_i^J$ in \cite{ETI}).
Here, for a morphism $\psi:j\rightarrow j'$ in $(i\downarrow J)$ from
$\phi:i\rightarrow j$ to $\psi\phi$,
the map
\[
(X_\phi)_*\N_j\xrightarrow {\beta_\psi}
(X_\phi)_*(X_\psi)_*\N_{j'}
\xrightarrow c
(X_{\psi\phi})_*\N_{j'}
\]
is the structure map.

For a morphism $\phi:i\rightarrow i'$ in $I$,
$\beta_\phi:(R_J\N)_i\rightarrow (X_\phi)_*(R_J\N)_{i'}$ is
given by 
\[
\projlim_{\psi\in (i\downarrow J)} (X_\psi)_*\N_j
\rightarrow
(X_{\psi'\phi})_*\N_{j'}
\xrightarrow{c}
(X_\phi)_*(X_{\psi'})_*\N_{j'}
\]
for an  object $\psi':i'\rightarrow j'$ in $(i'\downarrow J)$.

\paragraph
As in \cite[Chapter~5,6]{ETI}, 
$R_J$ is right adjoint to the restriction functor $(?)_J$.
The unit of adjunction $u:\Id\rightarrow R_J(?)_J$ is given by the map
$\M_i\rightarrow \projlim_{\phi\in(i\downarrow J)}(X_\phi)_*\M_j$
induced by $\beta_\phi:\M_i\rightarrow (X_\phi)_*\M_j$ for each $j$.
The counit of adjunction $\varepsilon:(?)_JR_J\rightarrow \Id$ is the projection
$\projlim_{(\psi:j\rightarrow j' )\in (j\downarrow J)}(X_\psi)_*\M_{j'}
\rightarrow \M_j=(X_{\id_j})_*\M_j$
(thus if $J$ is a full subcategory, then $\varepsilon$ is an 
isomorphism \cite[(6.15)]{ETI}).

\paragraph
For a morphism $f:X\rightarrow Y$ and $J\subset I$, we have that
$(?)_Jf_*$ and $(f_J)_*(?)_J$ are identical, and
$c:(?)_Jf_*\rightarrow(f_J)_*(?)_J$ is the identity.

\paragraph\label{xi-structure-data.par}
Combining these, it is easy to describe the canonical isomorphism
\[
\xi:f_*R_J\rightarrow R_J(f_J)_*
\]
(see \cite[(6.26)]{ETI}) via the structure data.
\[
\xi_i:(f_*R_J\N)_i\rightarrow (R_J(f_J)_*\N)_i
\]
is given by
\[
(f_i)_*\projlim (X_\phi)_*\N_j
\cong
\projlim
(f_i)_*(X_\phi)_*\N_j
\xrightarrow c
\projlim
(Y_\phi)_*(f_j)_*\N_j.
\]

\paragraph\label{delta.par}
Let $G$ be an $S$-group scheme, and $h:Z'\rightarrow Z$ an $S$-morphism
between $S$-schemes on which $G$ acts trivially.
We denote the composite isomorphism
\[
h_*(?)^G=h_*(?)_{-1}R_{\Delta_M}
\xrightarrow {c^{-1}}
(?)_{-1}\tilde B_G^M(h)_* R_{\Delta_M}
\xrightarrow\xi
(?)_{-1}R_{\Delta_M}B_G^M(h)_*=(?)^Gh_*
\]
by $e:h_*(?)^G\rightarrow (?)^Gh_*$ as in \cite[(7.3)]{HO2}.
(here $B_G^M(h)_*$ is abbreviated to be $h_*$, by abuse of notation).
By (\ref{xi-structure-data.par}) and the fact that $c$ is the identity,
we have that $e$ is nothing but the canonical isomorphism from 
$h_*\Ker\gamma\rightarrow \Ker h_*\gamma$.
In particular,

\begin{lemma}
Let the notation be as in {\rm(\ref{delta.par})}.
Then the diagram
\[
\xymatrix{
h_*(?)^G \ar[d]^e \ar[r]^\gamma & h_*(?)_0 \ar[d]^{c^{-1}} \\
(?)^Gh_* \ar[r]^\gamma & (?)_0 h_*
}
\]
is commutative.
\qed
\end{lemma}

\begin{corollary}
Let $h':Z''\rightarrow Z'$ and $h:Z'\rightarrow Z$ be a sequence of 
$G$-morphisms.
Then the composite
\[
(hh')_*(?)^G\xrightarrow c h_*h'_*(?)^G\xrightarrow e h_*(?)^G h'_*
\xrightarrow e (?)^G h_*h'_* \xrightarrow{c^{-1}}(?)^G(hh')_*
\]
agrees with $e$.
\qed
\end{corollary}

\begin{lemma}\label{unit-right-induction.thm}
Let $I$ be a small category, $J$ its subcategory, 
$f:X\rightarrow Y$ a morphism of $I\op$-diagrams of schemes.
Then the composite
\[
R_J\xrightarrow u f_*f^*R_J\xrightarrow\mu
f_*R_Jf_J^*\xrightarrow\xi R_J(f_J)_*f_J^*
\]
is the unit map $u$.
\end{lemma}

\begin{proof}
Follows easily from the commutativity of the diagram
\[
\xymatrix{
R_J \ar[r]^-u \ar[d]^u & 
R_J(f_J)_*f_J^* \ar[d]^u \ar[r]^{\xi^{-1}} &
f_*R_Jf_J^* \ar[d]^u \ar `r[rdd] `[dd]^{\id} [dd] & \\
f_*f^*R_J \ar[r]^-u \ar `d[drr] [drr]^\mu &
f_*f^*R_J(f_J)_*f_J^* \ar[r]^-{\xi^{-1}} &
f_*f^*f_*R_Jf_J^* \ar[d]^\varepsilon & \\
 & & f_*R_Jf_J^* & 
}.
\]
\end{proof}

\begin{lemma}\label{invariance-unit.thm}
Let $G$ be an $S$-group scheme, and $h:Z'\rightarrow Z$ an $S$-morphism
between $S$-schemes on which $G$ acts trivially.
Then the composite
\[
(?)^G\xrightarrow u
h_*h^*(?)^G \xrightarrow\epsilon
h_*(?)^G h^*\xrightarrow e
(?)^Gh_*h^*
\]
is the unit map $u$.
\end{lemma}

\begin{proof}
Follows easily from Lemma~\ref{unit-right-induction.thm} and
\cite[(1.24)]{ETI}.
\end{proof}

\begin{lemma}\label{eta-inverse.thm}
Let $\Cal S$ be a category, $(?)_*$ be a covariant symmetric monoidal
almost pseudofunctor on $\Cal S$ {\rm\cite[(1.28)]{ETI}}, and
$(?)^*$ its left adjoint.
Let 
\begin{equation}\label{fiber.eq}
\xymatrix{
X' \ar[r]^g \ar[d]^{\varphi'} &
X \ar[d]^\varphi \\
Y' \ar[r]^h & Y
}
\end{equation}
be a commutative diagram on $\Cal S$.
Then the diagram
\[
\xymatrix{
h^*\O_Y \ar[rr]^C \ar[d]^\eta & & \O_{Y'} \ar[d]^\eta \\
h^*\varphi_*\O_X \ar[r]^\theta & \varphi'_*g^*\O_X \ar[r]^C &
\varphi_*'\O_{X'}
}
\]
is commutative, where we use the notation in {\rm\cite[Chapter~1]{ETI}}.
\end{lemma}

\begin{proof}
Prove the commutativity of the diagram
\[
\xymatrix{
h^*\O_Y \ar[dd]^C \ar[dr]^\eta \ar[rrr] & & & 
h^*\varphi_*\O_X \ar[dl]^\eta \ar[dd]^\theta \\
 & h^*h_*\O_Y \ar[dl]^\varepsilon \ar[d]^\eta \ar[r]^\eta &
h^*\varphi_*g_*\O_X' \ar[d]^\theta \ar[dl]_c & \\
\O_{Y'} \ar `d [drr] [drr]^{\eta} & 
h^*h_*\varphi_*'\O_{X'} \ar[dr]^\varepsilon &
\varphi'_*g^*g_*\O_{X'} \ar[d]^\varepsilon &
\varphi'_*g^*\O_X \ar `d[dl] [dl]_-C \ar[l]_-\eta \\
 & & \varphi'\O_{X'}
}.
\]
The details are left to the reader.
\end{proof}

\paragraph\label{bar-eta.par}
Let $G$ be an $S$-group scheme, and $\varphi:X\rightarrow Y$ a
$G$-invariant morphism.
Then the canonical map $\bar\eta:\O_Y\rightarrow (\varphi_*\O_X)^G$ is
nothing but the composite
\[
\O_Y\xrightarrow{\gamma^{-1}}\O_Y^G\xrightarrow{\eta}
(\varphi_*\O_X)^G,
\]
where $\gamma:\O_Y^G\rightarrow \O_Y$ is an isomorphism as
$\O_Y$ is $G$-trivial, and $\eta:\O_Y\rightarrow \varphi_*\O_X$ is the
standard map.
As $\gamma$ is a natural map, it is easy to see that the composite
\[
\O_Y\xrightarrow{\bar\eta}(\varphi_*\O_X)^G\xrightarrow{\gamma}\varphi_*\O_X
\]
is $\eta$.

\begin{lemma}\label{bar-eta-isom.thm}
Let $G$ be an $S$-group scheme, and {\rm(\ref{fiber.eq})} a commutative
diagram of $G$-schemes such that $G$ acts on $Y$ and $Y'$ trivially.
Then 
\begin{enumerate}
\item[\bf 1] The diagram
\begin{equation}\label{bar-eta-inverse.eq}
\xymatrix{
h^*\O_Y \ar[d]^{\bar\eta} \ar[rrr]^C & & & \O_{Y'} \ar[d]^{\bar\eta} \\
h^*(\varphi_*\O_X)^G \ar[r]^\epsilon &
(h^*\varphi_*\O_X)^G \ar[r]^\theta & 
(\varphi_*'g^*\O_X)^G \ar[r]^C &
(\varphi_*'\O_{X'})^G
}
\end{equation}
is commutative.
\item[\bf 2] Assume that 
{\rm(\ref{fiber.eq})} is cartesian, $\varphi$ is quasi-compact
quasi-separated, $h$ is flat, and $G\rightarrow S$ is quasi-compact 
quasi-separated.
Then $\bar \eta:\O_{Y'}\rightarrow (\varphi'_*\O_{X'})^G$ is
an isomorphism if and only if $h^*\bar\eta: h^*\O_Y\rightarrow 
h^*(\varphi_*\O_X)^G$ is an isomorphism.
\item[\bf 3] In addition to the assumption of {\bf 2}, 
assume that $\eta:\O_Y\rightarrow h_*\O_{Y'}$ and
$\eta:\O_X\rightarrow g_*\O_{X'}$ are isomorphisms.
Then $\bar\eta:\O_Y\rightarrow (\varphi_*\O_X)^G$ is an isomorphism
if and only if $\bar\eta: h_*\O_{Y'}\rightarrow h_*(\varphi'_*\O_{X'})^G$ 
is an isomorphism.
\end{enumerate}
\end{lemma}

\begin{proof}
{\bf 1}.
By (\ref{eta-bar-eta.thm}), (\ref{eta-inverse.thm}), and 
(\ref{bar-eta.par}), the diagram
\[
\xymatrix{
 & h^*\O_Y \ar[rr]^C \ar[d]^\eta \ar[ld]_{\bar\eta} & & \O_{Y'} \ar[d]^{\eta}
\ar `r[ddr] `[dd]^{\bar\eta} [dd] & \\
h^*(\varphi_*\O_X)^G \ar[r]^\gamma \ar[dr]^\epsilon & 
h^*\varphi_*\O_X \ar[r]^\theta & 
\varphi'_*g^*\O_X \ar[r]^C &
\varphi'_*\O_{X'} & \\
 & (h^*\varphi_*\O_X)^G \ar[u]_\gamma \ar[r]^\theta &
(\varphi_*'g^*\O_X)^G \ar[r]^C \ar[u]_\gamma &
(\varphi_*'\O_{X'})^G \ar[u]_\gamma & 
}
\]
is commutative.
As $\gamma:(\varphi_*'\O_{X'})^G\rightarrow \varphi_*'\O_{X'}$ is a 
monomorphism, the result follows.

{\bf 2}.
Note that $\epsilon$ in (\ref{bar-eta-inverse.eq}) is an isomorphism
by \cite[(7.5)]{HO2} (note that in \cite[section~7]{HO2}, $G$ is assumed
to be flat, but this assumption is unnecessary in proving \cite[(7.5)]{HO2}).
$\theta $ in (\ref{bar-eta-inverse.eq}) is also an isomorphism by
\cite[(7.12)]{ETI}.
As the two $C$'s are isomorphisms, the result follows from {\bf 1}.

{\bf 3} follows easily from {\bf 1}, the proof of {\bf 2}, 
Lemma~\ref{invariance-unit.thm}, and \cite[(1.24)]{ETI}.
\end{proof}

\begin{corollary}\label{uniform-algebraic-quotient.cor}
Let $G$ be a quasi-compact quasi-separated $S$-group scheme, 
and $\varphi:X\rightarrow Y$ an algebraic quotient by $G$.
Then $\varphi$ is a uniform algebraic quotient.
\end{corollary}

\begin{proof}
Obvious by Lemma~\ref{bar-eta-isom.thm}, {\bf 2}.
\end{proof}

\begin{lemma}\label{invariance-qcqs.lem}
Let $G$ be a quasi-compact quasi-separated $S$-group scheme, 
$X$ an $S$-scheme on which $G$ acts trivially, and $\M$ a locally
quasi-coherent $(G,\O_X)$-module.
Then $\M^G\in\Qch(X)$.
Moreover, $(?)^G:\Mod(G,X)\rightarrow \Mod(X)$ preserves direct sums.
\end{lemma}

\begin{proof}
Let $p:G\times X\rightarrow X$ be the second projection.
Then $p_*$ preserves quasi-coherence.
So if $\M$ is locally quasi-coherent, then the kernel $\M^G$ of 
$\M_0\rightarrow p_*\M_1$ is quasi-coherent.
Moreover, $p_*:\Mod(G\times X)\rightarrow \Mod(X)$ preserves the direct sums
\cite[Theorem~8]{Kempf}.
As the kernel also preserves the direct sums, $(?)^G$ preserves the
direct sums.
\end{proof}

\paragraph
Let $G$ be an $S$-group scheme, 
and $X$ an $S$-scheme on which $G$ acts trivially.
Then for $\M\in\Mod(G,X)$, we have
\begin{multline*}
\Gamma(X,\M^G)=\Hom_{\O_X}(\O_X,(?)_{-1}R_{\Delta_M}\M)
\cong 
\\
\Hom_{\O_{B_G^M(X)}}(\O_{B_G^M(X)},\M)=\Gamma(\Zar(B_G^M(X)),\M).
\end{multline*}

For $\M,\N\in \Mod(G,X)$,
we denote $\uHom_{\O_{B_G^M(X)}}(\M,\N)$ by
$\uHom_{\O_X}(\M,\N)$, and
$\uHom_{\O_X}(\M,\N)^G$ by $\uHom_{G,\O_X}(\M,\N)$.
In particular, 
\begin{multline*}
\Gamma(X,\uHom_{G,\O_X}(\M,\N))\cong
\Gamma(\Zar(B_G^M(X)),\uHom_{\O_{B_G^M(X)}}(\M,\N))\\
=\Hom_{\O_{B_G^M(X)}}(\M,\N),
\end{multline*}
which we denote by $\Hom_{G,\O_X}(\M,\N)$.

\section{Functorial resolutions}\label{functorial.sec}

\begin{lemma}
Let $\Cal A$ be an abelian category, and assume that for each complex
$\Bbb F\in C(\Cal A)$, a $K$-injective resolution 
$i_{\Bbb F}: \Bbb F\rightarrow \Bbb I_{\Bbb F}$ is chosen.
Then there is a unique functor $\Bbb I:K(\Cal A)\rightarrow\Ki(\Cal A)$,
to the thick subcategory of $K$-injective objects, such that
$\Bbb I(\Bbb F)=\Bbb I_{\Bbb F}$ for each $\Bbb F$, and that
$i:\Id\rightarrow j\Bbb I$ is a natural transformation, where 
$j:\Ki(\Cal A)\hookrightarrow K(\Cal A)$ is the inclusion.
\end{lemma}

We call such a pair $(\Bbb I,i)$ of a functor and a natural map a
{\em functorial $K$-injective resolution}.
Note that the dual assertion of the lemma is the existence of a functorial
$K$-projective resolution.

\begin{proof}
Let $h:\Bbb F\rightarrow \Bbb G$ be a morphism in $C(\Cal A)$.
As $i_{\Bbb F}$ is a quasi-isomorphism and $\Bbb I_{\Bbb G}$ is $K$-injective,
\[
i_{\Bbb F}^*: K(\Bbb I_{\Bbb F},\Bbb I_{\Bbb G})\rightarrow 
K(\Bbb F,\Bbb I_{\Bbb G})
\]
is an isomorphism.
So it is necessary to define 
$\Bbb I(h)$ to be $((i_{\Bbb F})^*)^{-1}(i_{\Bbb G}h)$ to make
$i$ a natural transformation, and the uniqueness follows.

The fact that $\Bbb I$ is a functor and $i$ is a natural transformation
with this definition is easy, and is left to the reader.
\end{proof}

\paragraph
Let $\C$ be a Grothendieck category.
Then for each $\Bbb F\in C(\C)$, the category of complexes in $\Cal C$,
there is an injective 
strictly injective resolution $i_{\Bbb F}:\Bbb F\rightarrow \Bbb I$
\cite{Franke}.
That is, $i_{\Bbb F}$ is a monomorphism and is a quasi-isomorphism, 
$\Bbb I_{\Bbb F}$ is $K$-injective, and $\Bbb I_{\Bbb F}^i$ 
is an injective object for each $i$.
Thus we have

\begin{lemma}
Let $\C$ be a Grothendieck category.
Then there is a functorial 
$K$-injective resolution $i_{\Bbb F}:\Bbb F\rightarrow
\Bbb I_{\Bbb F}$ for $\Bbb F\in K(\C)$ which is a monomorphism
for each $\Bbb F$.
\qed
\end{lemma}

\begin{lemma}\label{K-inj-uniqueness.thm}
Let $\C$ be an abelian category, and $(\Bbb I,i)$ and $(\Bbb I',i')$ be
functorial $K$-injective resolutions of $\C$.
Then there is a unique natural isomorphism $\lambda:\Bbb I\rightarrow \Bbb I'$
such that $(j\lambda)\circ i=i'$.
\end{lemma}

\begin{proof}
For each $\Bbb F\in K(\C)$, 
\[
K(\C)(\Bbb I\Bbb F,\Bbb I'\Bbb F)
\xrightarrow{j}
K(\C)(j\Bbb I\Bbb F,j\Bbb I'\Bbb F)
\xrightarrow{i^*}
K(\C)(\Bbb F,j\Bbb I'\Bbb F)
\]
are isomorphisms.
\end{proof}

\paragraph 
Let $\Cal T$ be a triangulated category.
A triangulated subcategory 
(see \cite[Definition~1.5.1]{Neeman} for the definition)
$\Cal T'$ is said to be {\em localizing} if
it is closed under small direct sums.
For a set of objects (or a full subcategory) $\F$ of $\Cal T$, there is
a smallest localizing subcategory $\Loc(\F)$ of $\Cal T$ containing $\F$.

\paragraph
Let $\C$ be an abelian category which satisfies the {\rm(AB3)} condition.
For $n\in\Bbb Z$, 
let $C(\C)_{\leq n}$ be the full subcategory of $C(\C)$ of the
category of complexes in $\C$ consisting of complexes $\Bbb F$ with 
$\Bbb F^i=0$ for $i>n$.

\begin{lemma}\label{functorial-resolution.thm}
Let $\Cal F$ be a
full subcategory of $\Cal C$ closed under small direct sums.
Assume that there is a 
pair $(\frak F,\frak f)$ such that $\frak F:\Cal C\rightarrow \Cal F$ is
a functor, and $\frak f: j'\frak F\rightarrow \Id$ is a natural map
which is epic objectwise, where $j':\Cal F\hookrightarrow \C$ is
the inclusion.
Assume also that $\frak F(0)=0$.
Then
\begin{enumerate}
\item[\bf 1] 
For $\Bbb F\in C(\C)_{\leq n}$,
there is a functorial resolution $(\frak G,\frak g)$ such that
$\frak G(\Bbb F)$ is in $C(\Cal F)\cap 
C(\C)_{\leq n}$.
\item[\bf 2] There is a functorial inductive system $\frak G_n(\Bbb F)$ 
of complexes,
functorial on $\Bbb F\in C(\C)$, and a quasi-isomorphism
$\frak g_n:\frak G_n(\Bbb F)\rightarrow \Bbb F_{\leq n}$
such that
\begin{enumerate}
\item[\bf a] $\frak G_{-1}(\Bbb F)=0$;
\item[\bf b] For $n\in\Bbb Z$, 
$s_n:\frak G_n(\Bbb F)\rightarrow \frak G_{n-1}(\Bbb F)$ is
a semisplit epi \(that is, $s_n^i:
\frak G_n(\Bbb F)^i\rightarrow \frak G_{n-1}(\Bbb F)^i$ is a split epimorphism
for each $i\in\Bbb Z$\);
\item[\bf c] $\frak H_n:=\Ker s_n$ is in 
$C(\Cal F)\cap C(\C)_{\leq n}$ for each $n\geq 0$.
\end{enumerate}
\item[\bf 3] If $\C$ satisfies the {\rm(AB4)} condition, then
there is a functorial resolution $\frak g:\frak G(\Bbb F)\rightarrow\Bbb F$,
functorial on $\Bbb F\in C(\C)$
with $\frak G(\Bbb F)\in \Ob(\Loc(\Cal F'))$,
where $\Cal F'$ is the full subcategory of $K(\C)$ whose object set is the
same as $\Cal F$.
\end{enumerate}
\end{lemma}

\begin{proof}
For
\[
\Bbb F: \cdots \rightarrow \Bbb F^i\xrightarrow{\partial_{\Bbb F}^i}
\Bbb F^{i+1}\rightarrow\cdots
\]
in $C(\C)$, 
let $\frak F(\Bbb F)$ be
\[
\cdots \rightarrow \frak F(\Bbb F^i)\xrightarrow{\frak F(\partial_{\Bbb F}^i)}
\frak F(\Bbb F^{i+1})\rightarrow\cdots.
\]
Note that
\[
\frak F(\partial_{\Bbb F}^{i+1})\frak F(\partial_{\Bbb F}^i)
=
\frak F(\partial_{\Bbb F}^{i+1}\partial_{\Bbb F}^i)
=\frak F(0)=0,
\]
since $0: \Bbb F^i\rightarrow \Bbb F^{i+2}$ factors through the
null object $0$, and 
hence $\frak F(0):\frak F(\Bbb F^i)\rightarrow \frak F(\Bbb F^{i+2})$ also
factors through the null object 
$\frak F(0)=0$, and hence $\frak F(0)$ is the zero map.

Let $\frak f(\Bbb F):\frak F(\Bbb F)\rightarrow \Bbb F$ be the obvious
natural map.
It is an epic chain map.
Let $\frak K(\Bbb F):=\Ker \frak f(\Bbb F)$.
Then defining $\frak K(0):=\Bbb F$, 
$\frak G(m):=\frak F(\frak K(m))$, and $\frak K(m+1):=\frak K(\frak K(m))$,
we have a resolution
\[
\cdots\rightarrow \frak G(m+1)\rightarrow \frak G(m)\rightarrow
\cdots \rightarrow \frak G(0)\rightarrow 0
\]
of $\Bbb F$.

Assume that $\Bbb F\in C(\C)_{\leq n}$ for some $n$,
and let $\frak G(\Bbb F)$ be the total complex of this resolution, and
$\frak g:\frak G(\Bbb F)\rightarrow \Bbb F$ be the canonical map.
Then $\frak g$ is the desired resolution, and we have proved {\bf 1}.

{\bf 2} is proved by the same proof as in \cite[Lemma~3.3]{Spaltenstein},
except that everything here is functorial.

{\bf 3} This is proved similarly to 
(the dual assertion of) \cite[Application~2.4]{BN}, using {\bf 2}.
\end{proof}

\paragraph Let $(\Bbb X,\O_{\Bbb X})$ be a ringed site with a small basis
of topology $B$.
For $\M\in\Mod(\Bbb X)$, $x\in B$ and $c\in\Gamma(x,\M)$, 
there corresponds a map
\[
(e(c):\O_x\rightarrow \M)\in\Hom_{\O_{\Bbb X}}(\O_x,\M)
\cong\Hom_{\O_{\Bbb X}|_x}(\O_{\Bbb X}|_x,\M|_x)\cong
\Gamma(x,\M),
\]
corresponding to $c\in\Gamma(x,\M)$.
So
\[
\frak f(\M):=\sum e(c):\frak F:=
\bigoplus_{x\in B}\bigoplus_{0\neq c\in\Gamma(x,\M)}\O_{x,c}
\rightarrow \M
\]
is an epimorphism, where each $\O_{x,c}$ is a copy of $\O_x$,
see \cite[(2.23), (3.19)]{ETI}.
Note that for $y\in\Bbb X$, $\Gamma(y,\O_{x,c})=\bigoplus_{s\in\Bbb X(y,x)}
\Gamma(y,\O_{\Bbb X})_{c,s}$, where 
$\Gamma(y,\O_{\Bbb X})_{c,s}$ is a copy of $\Gamma(y,\O_{\Bbb X})$.
Then $\frak f(\M)$ 
is the unique map such that $1\in \Gamma(x,\O_{\Bbb X})_{c,1_x}$ is
mapped to $c$ for each $x$ and $c\neq 0$.

Let $h:\M\rightarrow \N$ be a map in $\Mod(\Bbb X)$.
Then mapping $1\in\Gamma(x,\O_{\Bbb X})_{c,1_x}$ to
$1\in\Gamma(x,\O_{\Bbb X})_{h(c),1_x}$ if $h(c)\neq 0$, and to $0$ if $h(c)=0$, 
we get a map
$\frak F(h):\frak F(\M)\rightarrow \frak F(\N)$ such that 
$\frak F:\PM(\Bbb X)\rightarrow \frak W$ is a functor, where 
$\frak W$ is the full subcategory of $\PM(\Bbb X)$ consisting of
the direct sum of copies of $\O_x$ with $x\in\Bbb X$, 
and that $\frak f:j'\frak F\rightarrow \Id$ is a natural
transformation, where $j': \frak W\hookrightarrow \PM(\Bbb X)$ is the
inclusion.
Moreover, $\frak F(0)=0$.

\begin{lemma}\label{complex-K-flat.thm}
Let $(\Bbb X,\O_{\Bbb X})$ be a ringed site with a small basis
of the topology.
Then there is an endofunctor $\frak F=
\frak F_{\Bbb X}:C(\Mod(\Bbb X))\rightarrow
\L$ and a functorial $\L$-resolution $\frak f=\frak f_{\Bbb X}
: j'\frak F\rightarrow \Id$,
where $\L$ is the full subcategory of $C(\Mod(\Bbb X))$ whose object
set is the set of strongly $K$-flat complexes {\rm\cite[(3.19)]{ETI}}.
\end{lemma}

\begin{proof}
As $\Mod(\Bbb X)$ is a Grothendieck category (so (AB5) is satisfied), 
Lemma~\ref{functorial-resolution.thm} and the discussion above 
are applicable.

By definition, $\O_x$ is strongly $K$-flat for $x\in\Bbb X$.
Now $\frak F(\Bbb F)$ is strongly $K$-flat for $\Bbb F
\in C(\Mod(\Bbb X))$ by construction.
\end{proof}

\section{The restriction and other operations on 
quasi-coherent sheaves}\label{restriction.sec}

\paragraph Let 
$G'$ and $G$ be flat $S$-group schemes, and $h:G'\rightarrow G$ 
a homomorphism of $S$-group schemes.
Let $Z$ be an $S$-scheme on which $G$ acts.
Then $\res^{G}_{G'}:\Mod(G,Z)\rightarrow \Mod(G',Z)$ is
defined to be the inverse image functor $B_h^M(Z)^*$, 
see \cite[(2.45)]{Hashimoto5}.

\begin{lemma}\label{faithful-exact.thm}
$\res_{G'}^{G}$ is a faithful exact functor from $\Qch(G,Z)$ to $\Qch(G',Z)$.
\end{lemma}

\begin{proof}
By \cite[(7.22)]{ETI}, $\res_{G'}^{G}$ is a functor from $\Qch(G,Z)$ 
to $\Qch(G',Z)$.
Then as functors from $\Qch(G,Z)$ to $\Qch(Z)$, 
we have $(?)_0\res_{G'}^{G} \cong (?)_0$,
where the left $(?)_0$ is from $\Qch(G',Z)$ to $\Qch(Z)$, and
the right $(?)_0$ is from $\Qch(G,Z)$ to $\Qch(Z)$.
As the both $(?)_0$ are faithful exact, $\res_{G'}^{G}$ is also faithful exact.
\end{proof}

\begin{lemma}\label{res-vanishing.thm}
If $\M\in\Qch(G,Z)$, then $(L_i\res_{G'}^{G})\M=0$ for $i>0$,
where $L_i$ denotes the $i$th left derived functor $D(G,Z)\rightarrow
\Mod(G',Z)$.
\end{lemma}

\begin{proof}
By \cite[(8.20), (8.21)]{ETI}, we have that $(L_i\res_{G'}^{G})\M$ is
quasi-coherent.
As the restriction functor $(?)_0:\Qch(G,Z)\rightarrow\Qch(Z)$ is
faithful and exact by \cite[(12.12)]{ETI}, it suffices to prove that
$((L_i\res_{G'}^{G})\M)_0=0$ for $i>0$.
But this is $L_i(\id_Z)^*\M_0=0$, by \cite[(8.13)]{ETI}.
\end{proof}

\begin{lemma}\label{res-acyclic.thm}
Let $\Bbb G\in D_{\Qch}(G,Z)$.
Then $\Bbb G$ is \(left\) $\res^{G}_{G'}$-acyclic, in the sense that
for any \(or equivalently, some\) $K$-flat resolution
$\Bbb P\rightarrow \Bbb G$,
the map 
$\res^{G}_{G'}\Bbb P\rightarrow \res^{G}_{G'}\Bbb G$ is a quasi-isomorphism.
In particular, $L\res^{G}_{G'}\Bbb G$ has quasi-coherent cohomology groups.
If $Z$ is locally Noetherian, $G$ and $G'$ are 
locally of finite type, and
$\Bbb G$ has coherent cohomology groups, then $\res^{G}_{G'}\Bbb G$ has
coherent cohomology groups.
\end{lemma}

\begin{proof}
If $\Bbb G\in D_{\Qch}^-(G,Z)$, then the assertion follows from
Lemma~\ref{res-vanishing.thm}.
Consider the general case.
From the bounded-above case, 
it is easy to see that 
the resolution
$\frak g:\frak G(\Bbb G)\rightarrow\Bbb G$ in 
Lemma~\ref{functorial-resolution.thm}
is a $K$-flat resolution such that $\res^G_{G'}\frak g$ is a 
quasi-isomorphism.
\end{proof}

\paragraph
Now we can prove that the (derived) restriction is compatible with 
most of basic operations on sheaves.
For the notation, see \cite{ETI}.

\begin{lemma}\label{delta-isom.thm}
Let $f:(\Bbb X,\O_{\Bbb X})\rightarrow (\Bbb Y,\O_{\Bbb Y})$ be a
morphism of ringed sites.
Assume that for each $x\in\Bbb X$, the category $(I^f_x)\op$ 
\(see {\rm\cite[(2.6)]{ETI}}\) is filtered \(that is, connected and
pseudofiltered, see {\rm\cite[Appendix~A]{Milne}}\).
Then the canonical map
\begin{equation}\label{sheaf-delta.eq}
\Delta:
f^*(\M\otimes_{\Cal O_{\Bbb Y}}\N)
\rightarrow 
f^*\M\otimes_{\Cal O_{\Bbb X}}f^*\N
\end{equation}
is an isomorphism for any $\M,\N\in\Mod(\Bbb Y)$, and the
canonical map
\begin{equation}\label{derived-delta.eq}
\Delta:
Lf^*(\Bbb F\otimes_{\Cal O_{\Bbb Y}}^L \Bbb G)
\rightarrow
Lf^*\Bbb F\otimes_{\Cal O_{\Bbb X}}^L Lf^*\Bbb G
\end{equation}
is also an isomorphism for $\Bbb F,\Bbb G\in D(\Bbb Y)$.
\end{lemma}

\begin{proof}
First consider the corresponding map of presheaves
\begin{equation}\label{presheaf-delta.eq}
\Delta:f_{\PM}^*(\M\otimes_{\Cal O_{\Bbb Y}}^p \N)\rightarrow
f_{\PM}^*\M\otimes_{\Cal O_{\Bbb X}}^p f_{\PM}^*\N
\end{equation}
for $\M,\N\in\PM(\Bbb Y)$.
Then the map between sections at $x\in\Bbb X$ is
\begin{multline*}
\indlim_{x\rightarrow fy}\Gamma(x,\O_{\Bbb X})\otimes_{\Gamma(y,\O_{\Bbb Y})}
(\Gamma(y,\M)\otimes_{\Gamma(y,\O_{\Bbb Y})}\Gamma(y,\N))
\rightarrow\\
(\indlim_{x\rightarrow fy}\Gamma(x,\O_{\Bbb X})\otimes_{\Gamma(y,\O_{\Bbb Y})}
\Gamma(y,\M))
\otimes_{\Gamma(x,\O_{\Bbb X})}
(\indlim_{x\rightarrow fy}\Gamma(x,\O_{\Bbb X})\otimes_{\Gamma(y,\O_{\Bbb Y})}
\Gamma(y,\M)),
\end{multline*}
which is an isomorphism by the assumption that 
$(I^f_x)\op$ is filtered.
Thus (\ref{presheaf-delta.eq}) is an isomorphism.

Now consider $\M,\N\in\Mod(\Bbb Y)$.
It is not so difficult to show that the diagram
\[
\xymatrix{
af^*(q\M\otimes^p_{\O_{\Bbb Y}}q\N) \ar[r]^\Delta \ar[d]^u &
a(f^*q\M\otimes^p_{\O_{\Bbb X}}f^*q\N) \ar[d]^{u\otimes u} \\
af^*qa(q\M\otimes^p_{\O_{\Bbb Y}}q\N) \ar[r]^-{\Delta} &
a(qaf^*q\M\otimes^p_{\O_{\Bbb X}}qaf^*q\N)
}
\]
is commutative.
The top horizontal arrow is an isomorphism by the argument for 
presheaves above, and the vertical arrows are isomorphisms by
\cite[(2.18), (2.34)]{ETI}.
Thus the bottom horizontal arrow, which agrees with the map 
(\ref{sheaf-delta.eq}), is an isomorphism.

Now consider $\Bbb F,\Bbb G\in D(\Bbb Y)$.
Take strongly $K$-flat resolutions $\Bbb P\rightarrow \Bbb F$ and
$\Bbb Q\rightarrow \Bbb G$.
Then the map (\ref{derived-delta.eq}) is nothing but the
composite
\begin{multline*}
Lf^*(\Bbb F\otimes_{\O_{\Bbb Y}}^L\Bbb G)
\cong
Lf^*(\Bbb P\otimes_{\O_{\Bbb Y}}^L\Bbb Q)
\cong
f^*(\Bbb P\otimes_{\O_{\Bbb Y}}\Bbb Q)\\
\xrightarrow{\Delta}
f^*\Bbb P\otimes_{\O_{\Bbb X}}f^*\Bbb Q
\cong
Lf^*\Bbb F\otimes_{\O_{\Bbb X}}^L Lf^*\Bbb G.
\end{multline*}
The second isomorphism comes from the fact that $\Bbb P\otimes_{\O_{\Bbb Y}}
\Bbb Q$ is $K$-flat \cite[(3.21)]{ETI}.
The last isomorphism comes from the fact that $f^*\Bbb P$ and $f^*\Bbb Q$ 
are strongly $K$-flat \cite[(3.20)]{ETI}.
Being the composite of isomorphisms, (\ref{derived-delta.eq}) is an
isomorphism.
\end{proof}

\begin{lemma}\label{diagram-filtered.lem}
Let $J$ be a small category, and $\varphi:X\rightarrow Y$ a morphism
of $J\op$-diagrams of schemes.
Let $f=\varphi^{-1}:\Zar(Y)\rightarrow\Zar(X)$ be the corresponding 
functor between sites {\rm\cite[(5.3)]{ETI}}.
Then for each $(j,U)\in\Zar(X)$ \(where $j\in J$ and $U\in\Zar(X_j)$\),
the category $(I^f_{(j,U)})\op$
\(see {\rm\cite[(2.6)]{ETI}}\) is filtered.
\end{lemma}

\begin{proof}
Left to the reader.
\end{proof}

\begin{lemma}
Let $\Bbb F, \Bbb G\in D(G,Z)$.
Then the canonical map
\[
\Delta:L\res_{G'}^{G}(\Bbb F\otimes_{\O_Z}^L\Bbb G)\rightarrow
(L\res_{G'}^{G}\Bbb F)\otimes_{\O_Z}^L(L\res_{G'}^{G}\Bbb G)
\]
is an isomorphism
\(Note that the $\otimes_{\O_Z}^L$ in the left \(resp.\ right\) hand side
is an abbreviation for $\otimes_{\O_{B_{G}^M(Z)}}^L$ \(resp.\ 
$\otimes_{\O_{B_{G'}^M(Z)}}^L$\)\).
\end{lemma}

\begin{proof}
By Lemma~\ref{diagram-filtered.lem}, 
Lemma~\ref{delta-isom.thm} is applicable.
\end{proof}

\begin{corollary}
Let $\M$ and $\N$ be objects in $\Qch(G,Z)$.
Then
\[
\res^{G}_{G'}\uTor^{\O_Z}_i(\M,\N)
\cong
\uTor^{\O_Z}_i(\res^{G}_{G'}\M,\res^{G}_{G'}\N).
\]
\qed
\end{corollary}

\begin{lemma}\label{ext-res.thm}
Let $X$ be locally Noetherian, and $G$ be locally of finite type.
Let $\Bbb F\in D^-_{\Coh}(G,Z)$ and $\Bbb G\in D^+_{\Lqc}(G,Z)$.
Then the canonical map
\begin{equation}\label{RHom.eq}
L\res^{G}_{G'} R\uHom_{\O_Z}(\Bbb F,\Bbb G)
\rightarrow R\uHom_{\O_Z}(L\res^{G}_{G'}\Bbb F,L\res^{G}_{G'}\Bbb G)
\end{equation}
is an isomorphism.
\end{lemma}

\begin{proof}
By Lemma~\ref{res-acyclic.thm} and \cite[(13.10)]{ETI}, the two
complexes have quasi-coherent cohomology groups.
So in order to prove that the map is an isomorphism, we may discuss
after applying the functor $(?)_0$.
By \cite[(8.13)]{ETI}, the canonical map
\[
\theta: L(\id_Z^*)(?)_0\rightarrow (?)_0\res_{G'}^{G}
\]
is an isomorphism.
So the map (\ref{RHom.eq}) applied $(?)_0$ is identified with
\[
H_0:(R\uHom_{\O_Z}(\Bbb F,\Bbb G))_0
\rightarrow 
R\uHom_{\O_Z}(\Bbb F_0,\Bbb G_0).
\]
It is an isomorphism by \cite[(13.9)]{ETI}.
\end{proof}

\begin{corollary}\label{restriction-extension.thm}
Let $\M\in\Coh(G,Z)$ and $\N\in\Qch(G,Z)$.
Then we have
\[
\res_{G'}^{G}\uExt_{\O_{B_{G}^M(Z)}}^i(\M,\N)\cong
\uExt_{\O_{B_{G'}^M(Z)}}^i(\res_{G'}^{G}\M,\res_{G'}^{G}\N).
\]
\qed
\end{corollary}

\begin{lemma}\label{res-direct.thm}
Let $g: Z'\rightarrow Z$ be a concentrated \(that is, quasi-compact 
quasi-separated\) $G$-morphism of $G$-schemes.
Then the canonical map
\[
\theta: L\res_{G'}^{G}Rf_*\Bbb F\rightarrow Rf_*L\res_{G'}^{G}\Bbb F
\]
is an isomorphism for $\Bbb F\in D_{\Qch}(G,Z')$.
\end{lemma}

\begin{proof}
As the complexes have quasi-coherent cohomology groups by \cite[(8.7)]{ETI}
and Lemma~\ref{res-acyclic.thm}, 
we may discuss after applying the functor $(?)_0$, and the rest is easy.
\end{proof}

\paragraph
Let $V\subset U\subset Z$ be $G$-stable open subsets.
Assume that the inclusions $f:U\hookrightarrow Z$ and $g:
V\hookrightarrow U$ are quasi-compact.
For $\Bbb F\in D_{\Qch}(G,Z)$, 
there is a commutative diagram
\begin{equation}\label{triangle.eq}
\xymatrix{
L\res^{G}_{G'} R\uGamma_{U,V} \Bbb F\ar[r]^\iota &
L\res^{G}_{G'} Rf_*f^* \ar[r]^-u \Bbb F\ar[d]^{d\theta} &
L\res^{G}_{G'} Rf_*Rg_*g^*f^* \Bbb F\ar[r] \ar[d]^{dd\theta\theta} & \\
R\uGamma_{U,V} L\res^{G}_{G'} \Bbb F\ar[r]^\iota &
Rf_*f^* L\res^{G}_{G'} \Bbb F\ar[r]^-u &
Rf_*Rg_*g^*f^* L\res^{G}_{G'} \Bbb F\ar[r] & 
}
\end{equation}
whose rows are distinguished triangles by \cite[(4.10)]{HO2}.
Then, as the derived category is a triangulated category, 
\[
\bar \delta: 
L\res^{G}_{G'} R\uGamma_{U,V} \Bbb F\rightarrow
R\uGamma_{U,V} L\res^{G}_{G'} \Bbb F
\]
which completes (\ref{triangle.eq}) as a map of triangles.
As $d\theta$ and $dd\theta\theta$ are isomorphisms by
Lemma~\ref{res-direct.thm}, $\bar \delta$ is an isomorphism by
\cite[(I.1.1)]{Hartshorne2}.

We can make $\bar \delta$ functorial on $\Bbb F$.
Fix a functorial strictly injective resolution $\Bbb F\rightarrow \Bbb I_{
\Bbb F}$
(in the category $K_{\Qch}(G,Z)$) as in section~\ref{functorial.sec}.
Then the composite
\begin{multline*}
L\res^{G}_{G'} R\uGamma_{U,V}\Bbb F
\cong
\res^{G}_{G'} \uGamma_{U,V}\Bbb I_{\Bbb F}
\cong
\res^{G}_{G'}\Cone(u:f_*f^*\Bbb I_{\Bbb F}\rightarrow f_*g_*g^*f^*
\Bbb I_{\Bbb F})[-1]\\
\cong
\Cone(u:f_*f^*\res^{G}_{G'}\Bbb I_{\Bbb F}\rightarrow f_*g_*g^*f^*\res^{G}_{G'}
\Bbb I_{\Bbb F})[-1]
\cong
R\uGamma_{U,V}\res^{G}_{G'}\Bbb F
\end{multline*}
is the desired functorial $\bar\delta$.

In conclusion,

\begin{lemma}
There is an isomorphism
\[
\bar \delta: L\res^{G}_{G'} R\uGamma_{U,V}\Bbb F
\rightarrow
R\uGamma_{U,V}L\res^{G}_{G'}\Bbb F
\]
which is functorial on $\Bbb F\in D_{\Qch}(G,Z)$,
the diagram
\[
\xymatrix{
\L R\uGamma_{U,V} \ar[r]^\iota \ar[d]^{\bar\delta} &
\L Rf_*f^* \ar[r]^-u \ar[d]^{d\theta} &
\L Rf_*Rg_*g^*f^* \ar[r] \ar[d]^{dd\theta\theta} & 
\L R\uGamma_{U,V} [1]\ar[r] \ar[d]^{\bar\delta[1]} &
\\
R\uGamma_{U,V} \L \ar[r]^\iota &
Rf_*f^* \L \ar[r]^-u &
Rf_*Rg_*g^*f^* \L \ar[r] & 
R\uGamma_{U,V} \L [1] \ar[r] &
}
\]
is commutative, where $\L=L\res^{G}_{G'}$, and
the diagram
\[
\xymatrix{
(?)_0 \L R\uGamma_{U,V}\Bbb F
\ar[d]^{(?)_0\bar\delta} \ar[r]^{\cong} &
(?)_0 R\uGamma_{U,V}\Bbb F \ar[r]^{\cong} &
R\uGamma_{U,V}\Bbb F_0 \\
(?)_0 R\uGamma_{U,V}\L\Bbb F \ar[r]^\cong &
R\uGamma_{U,V}(?)_0\L\Bbb F \ar[ur]^{\cong} 
}
\]
is also commutative.
\end{lemma}

\begin{proof}
Follows easily from the construction above.
\end{proof}

\begin{lemma}\label{restriction-invariance.lem}
Let $h:G'\rightarrow G$ be a homomorphism of $S$-group schemes.
Let $X$ be an $S$-scheme with a trivial $G$-action.
Then for $\M\in\Mod(G,X)$, $\M^{G}\subset \M^{G'}=(\res^G_{G'}\M)^{G'}$.
If $h$ is faithfully flat, then $\M^{G}=\M^{G'}$.
\end{lemma}

\begin{proof}
Let $p$ (resp.\  $p'$) be the second projection $G\times X\rightarrow X$
(resp.\ $G'\times X\rightarrow X$).
Then $\M^G$ is the kernel of the map
\[
\M\xrightarrow{\beta_{\delta_0}-\beta_{\delta_1}}p_*\M_1.
\]
As $p'=p(h\times 1)$, $\M^{G'}$ is the kernel of the map
\[
\M\xrightarrow{\beta_{\delta_0}-\beta_{\delta_1}}p_*\M_1
\xrightarrow{p_*u}p_*(h\times 1)_*(h\times 1)^*\M_1.
\]
So $\M^G\subset \M^{G'}$.
If $h$ is faithfully flat, then $p_*u$ is a monomorphism, and hence
$\M^G=\M^{G'}$.
\end{proof}

\paragraph
Let $h:G'\rightarrow G$ be a homomorphism between 
flat $S$-group schemes of finite type, and 
$X$ a Noetherian $G$-scheme with a $G$-dualizing complex $\Bbb I_X$.

\begin{lemma}\label{restriction-dualizing.lem}
$L\res_{G'}^G\Bbb I_X$ is a $G'$-dualizing complex of $X$.
\end{lemma}

\begin{proof}
Follows easily from 
\cite[(8.20)]{ETI} and \cite[(31.17)]{ETI}.
\end{proof}

\paragraph By abuse of notation, we will sometimes write the 
$G'$-dualizing complex $\res_{G'}^G\Bbb I_X$ by $\Bbb I_X$ or $\Bbb I_X(G')$.

\section{Groups with the Reynolds operators}\label{Reynolds.sec}

\paragraph
Let $G$ be a flat quasi-compact quasi-separated 
$S$-group scheme, and $Y$ an $S$-scheme on which $G$
acts trivially.
Let $\gamma=\gamma_{G,Y}:
(?)^G\rightarrow \Id$ be the inclusion between the functors from
$\Qch(G,Y)$ to itself (although $(?)^G$ is a functor from $\Qch(G,Y)$ to
$\Qch(Y)$, we regard $\M^G$ as a trivial $G$-module, and then
$(?)^G$ can be viewed as a functor from $\Qch(G,Y)$ to itself.
So $\gamma$ in (\ref{G-trivial.par}) is $\gamma_0$ here.
This abuse of notation does not cause a problem).

We say that $G$ has a functorial Reynolds operator on $Y$ if there is a 
natural transformation $p=p_{G,Y}:\Id\rightarrow (?)^G$ such that $p\gamma=\id$.
The natural map $p$ (or sometimes $\frak R:=\gamma p$) 
is called the {\em Reynolds operator}
of $G$ on $Y$.
Note that $\frak R^2=\frak R$.

\paragraph\label{trivial-anti-trivial.par}
Let $G$, $Y$ be as above, and assume that $G$ has a Reynolds operator on $Y$.
For $\M\in\Qch(G,Y)$, we define $U_G(\M)=\Ker p_{G,Y}(\M)=\Image(\id-\frak R)$,
and call it the {\em anti-invariance} of $\M$.
We say that $\M$ is {\em anti-trivial} if $U_G(\M)=\M$, or equivalently, 
$\M^G=0$.
If $\M,\N\in\Qch(G,Y)$, and $\M$ is anti-trivial and $\N$ is trivial, 
then $\Hom_{G,\O_Y}(\M,\N)=0$.
Indeed, if $h\in \Hom_{G,\O_Y}(\M,\N)$, then
$h=\R h(\id-\R)=h\R(\id-\R)=0$.
Similarly, we have $\Hom_{G,\O_Y}(\N,\M)=0$.

\paragraph
Note that $U_G(\M)$ is the sum of all the quasi-coherent submodules 
$\N$ of $\M$ such that $\N^G=0$, and is determined only by $G$, $Y$, and
$\M$, and is independent of the choice of $p$.
As the Reynolds operator 
$p$ is the projection with respect to the direct sum decomposition
$\M=\M^G\oplus U_G(\M)$, it is unique, if exists.
So $\frak R$ is also unique, if exists.

\begin{definition}
We say that an $S$-group scheme $G$ is {\em Reynolds}, if $G$ is 
flat quasi-compact quasi-separated, and for 
any affine open subscheme $U$ of $S$, 
the Reynolds operator of $G$ on $U$ exists.
\end{definition}

\begin{lemma}
Let $G$ be a flat quasi-compact quasi-separated $S$-group scheme, 
$Y$ an $S$-scheme on which $G$ acts trivially, and $U$ its open subset.
If $G$ has a Reynolds operator on $Y$ and
if the inclusion $j:U\hookrightarrow Y$ is quasi-compact, then
$G$ has a Reynolds operator on $U$.
\end{lemma}

\begin{proof}
For $\M\in\Qch(G,U)$, 
define $p_{G,U}:\M\rightarrow \M^G$ to be the composite
\[
\M\xrightarrow{\varepsilon^{-1}} j^*j_*\M\xrightarrow{p_{G,Y}}
j^*(?)^Gj_*\M\xrightarrow{e^{-1}}j^*j_*(?)^G\M\xrightarrow{\varepsilon}
(?)^G\M.
\]
It is easy to see that this is the identity map on $\M^G$.
\end{proof}

\begin{corollary}
Let $G$ be a flat quasi-compact quasi-separated $S$-group scheme.
Assume that $G$ has a Reynolds operator over $S$, and $S$ is quasi-separated.
Then $G$ is Reynolds.
\qed
\end{corollary}

\begin{lemma}\label{Reynolds-invariance-exact.lem}
Let $G$ be a flat quasi-compact quasi-separated $S$-group scheme,
$Y$ an $S$-scheme on which $G$-acts trivially, and assume that
$G$ has a Reynolds operator on $Y$.
Then 
\begin{enumerate}
\item[\bf 1] $(?)^G$ and $U_G$ are exact functors on $\Qch(G,Y)$
which preserve direct sums.
In particular, $H^i(G,\M)=0$ for $\M\in\Qch(G,Y)$ and $i>0$, where 
$H^i(G,?)=R^i(?)^G$, the derived functor of the functor $\Qch(G,Y)\rightarrow
\Qch(Y)$ \(not a functor from $\Mod(G,Y)$\).
\item[\bf 2] The full subcategory of the $G$-trivial quasi-coherent 
$(G,\O_Y)$-modules is closed under extensions and subquotients.
Similarly for the full subcategory of $G$-anti-trivial quasi-coherent
$(G,\O_Y)$-modules.
\end{enumerate}
\end{lemma}

\begin{proof}
{\bf 1} 
As we have $\Id=(?)^G\oplus U_G$, we have that 
both $(?)^G$ and $U_G$ are exact.
Let $(\M_i)_{i\in I} $ be a family of objects in $\Qch(G,Y)$.
Then we have
\begin{equation}\label{direct-sum-triv-decomposition.eq}
\bigoplus_i\M_i=(\bigoplus_i \M_i^G)\oplus (\bigoplus_i U_G(\M_i)).
\end{equation}
As $(?)^G$ is compatible with the direct sum by 
Lemma~\ref{invariance-qcqs.lem}, we have that
$(\bigoplus_i \M_i^G)^G=\bigoplus_i((\M_i)^G)^G=\bigoplus_i \M_i^G$,
and hence $\bigoplus \M_i^G$ is $G$-trivial.
As $(\bigoplus_i U_G(\M_i))^G=\bigoplus_i U_G(\M_i)^G=0$, 
we have that $\bigoplus U_G(\M_i)$ is anti-trivial.
So by the decomposition (\ref{direct-sum-triv-decomposition.eq}), we 
have that $\bigoplus_i \M_i^G=(\bigoplus_i\M_i)^G$ and 
$\bigoplus_i U_G(\M_i)=U_G(\bigoplus_i \M_i)$.

{\bf 2} follows easily by the five lemma and {\bf 1}.
\end{proof}

\begin{lemma}\label{Reynolds-affine-i.lem}
Let the notation be as in {\rm Lemma~\ref{Reynolds-invariance-exact.lem}}.
Assume that $Y=\Spec R$ is affine.
If $V$ is a $G$-trivial $R$-module and $M$ is a $G$-module, then
$(V\otimes_R M)^G=V\otimes_R M^G$ and $U_G(V\otimes_R M)=V\otimes_R U_G(M)$.
\end{lemma}

\begin{proof}
There is a $G$-trivial free $R$-module $F$ and a surjection $F\rightarrow V$.
Then $F\otimes_R M^G$ is a direct sum of copies of $M^G$ as a $(G,R)$-module,
and hence it is $G$-trivial.
Being a homomorphic image of $F\otimes_R M^G$, $V\otimes_R M^G$ 
is $G$-trivial.
Similarly, $F\otimes_R U_G(M)$ is $G$-anti-trivial, and hence so is
$V\otimes_R U_G(M)$.
As we have the decomposition
\[
V\otimes_R M=V\otimes_R M^G \oplus V\otimes_R U_G(M),
\]
we must have $(V\otimes_R M)^G=V\otimes_R M^G$ and
$U_G(V\otimes_R M)= V\otimes_R U_G(M)$.
\end{proof}

\begin{lemma}\label{Reynolds-affine-ii.lem}
Let the notation be as in {\rm Lemma~\ref{Reynolds-affine-i.lem}}.
Let $R'$ be an $R$-algebra \(on which $G$ acts trivially\), 
and set $h:Y'=\Spec R'\rightarrow \Spec R=Y$ 
be the associated map.
Then 
\begin{enumerate}
\item[\bf 1] $G$ also has a Reynolds operator on $Y'$.
\item[\bf 2] $(?)^Gh_*$ and $U_G h_*$ are canonically identified with 
$h_*(?)^G$ and $h_*U_G$, respectively.
\item[\bf 3] $h_*p_{G,Y'}=p_{G,Y}h_*$.
\item[\bf 4] $(?)^G h^*$ and $U_G h^*$ are canonically identified with
$h^*(?)^G$ and $h^*U_G$, respectively.
\item[\bf 5] For a $(G,R)$-module $M$, the diagram
\[
\xymatrix{
M \ar[r]^-{\eta_M} \ar[d]^-{p} & R'\otimes_R M \ar[d]^-{1\otimes p} \ar[dr]^p \\
M^G \ar[r]^-{\eta_{M^G}} & R'\otimes_R M^G \ar[r]^-\cong & (R'\otimes_R M)^G
}
\]
is commutative.
\end{enumerate}
\end{lemma}

\begin{proof}
Let $M$ be a $(G,R')$-module.
Then we can decompose $M=M^G\oplus U_G(M)$ as $(G,R)$-modules.
On the other hand, As $R'\otimes_R M^G$ is $G$-trivial, it is mapped to
$M^G$ by the product $a_M:R'\otimes_R M\rightarrow M$.
Similarly, $R'\otimes_R U_G(M)$ is mapped to $U_G(M)$.
Hence $M^G$ and $U_G(M)$ are $(G,R')$-submodules of $M$, and
the decomposition $M=M^G\oplus U_G(M)$ shows {\bf 1, 2, 3}.

Next, let $M$ be a $(G,R)$-module.
Then by Lemma~\ref{Reynolds-affine-i.lem}, $R'\otimes_R M^G$ is identified
with $(R'\otimes_R M)^G$.
As a $(G,R)$-module, $U_G(R'\otimes_R M)$ is identified with $R'\otimes_R
U_G(M)$.
However, by {\bf 3}, this is an identification also as $(G,R')$-modules.
Now {\bf 4} and {\bf 5} are clear.
\end{proof}

\begin{lemma}\label{covering-Reynolds.lem}
Let $G$ be a flat quasi-compact quasi-separated $S$-group scheme,
$Y$ an $S$-scheme on which $G$ acts trivially.
Let $Y=\bigcup_i U_i$ be an affine open covering such that 
$G$ has a Reynolds operator over $U_i$ for each $i$.
Then for any $Y$-scheme $Y'$ \(with a trivial action\), $G$ has a Reynolds
operator.
\end{lemma}

\begin{proof}
First, we prove that $G$ has a Reynolds operator on $Y$.
Let $\Cal C=\Zar(Y)$, and let $\Cal D$ be the full subcategory of $\Cal C$
consisting of affine open subsets $W$ of $Y$ such that $W\subset U_i$ for
some $i$.
For $W\in\Cal D$, $G$ has a Reynolds operator on $W$ by
Lemma~\ref{Reynolds-affine-ii.lem}.
So defining $\bar p: \M|_{\Cal D}\rightarrow \M^G|_{\Cal D}$ by
\begin{multline*}
\Gamma(W,\M|_{\Cal D})=\Gamma(W,\M|_W)\xrightarrow{p_{G,W}}
\Gamma(W,(\M|_W)^G)
\\
\xrightarrow{\epsilon^{-1}}\Gamma(W,(\M^G)|_W)
=\Gamma(W,\M^G|_{\Cal D}),
\end{multline*}
we get a functorial splitting of $\bar i:\M^G|_{\Cal D}\rightarrow \M|_{\Cal D}$
by Lemma~\ref{Reynolds-affine-ii.lem}, {\bf 4, 5}.
As the restriction from $\Cal C$ to $\Cal D$ gives an equivalence
$\Sh(\Cal C)\rightarrow \Sh(\Cal D)$ by \cite[(4.6)]{Hashimoto5}, 
there is a unique 
splitting $p:\M\rightarrow \M^G$ of $i:\M^G\rightarrow \M$ 
in $\Sh(\Cal C)$ whose restriction
to $\Sh(\Cal D)$ is $\bar p$.
By the uniqueness, it is easy to see that the restriction of $p$ to each
$U_i$ is a Reynolds operator, and hence $p$ is a morphism in 
$\Qch(G,Y)$, and $p$ is the Reynolds operator of $G$ on $Y$.

Next, let $h:Y'\rightarrow Y$ be a $Y$-scheme.
Then for each affine open subset $W$ of $Y'$ such that $h(W)\subset U_i$
for some $i$, $G$ has a Reynolds operator over $W$.
Such $W$ covers $Y'$, and hence $Y'$ has a Reynolds operator by the 
first step.
\end{proof}

\begin{corollary}\label{universal-Reynolds.cor}
Let $G$ be a flat quasi-compact quasi-separated $S$-group scheme.
If $G$ is Reynolds, then for any $S$-scheme $Y$ with the trivial $G$-action,
$G$ has a Reynolds operator over $Y$.
\end{corollary}

\begin{lemma}\label{Reynolds-four-operations.lem}
Let $G$ be a Reynolds $S$-group scheme, and $h:Y\rightarrow S$ a morphism.
Then we have the following.
\begin{enumerate}
\item[\bf 0] The base change $G_Y=Y\times_S G$ is a Reynolds $Y$-group scheme.
\item[\bf 1] The canonical map 
$\epsilon:h^*(?)^G\rightarrow (?)^Gh^*$ \(see 
{\rm(\ref{epsilon.par})}\) is an isomorphism between functors from 
$\Qch(G,S)$ to $\Qch(G,Y)$.
The composite
\[
\Id h^*=h^*=h^*\Id\xrightarrow{p_{G,S}}h^*(?)^G\xrightarrow{\epsilon}(?)^Gh^*
\]
agrees with $p_{G,Y}$.
\item[\bf 2] Let $h$ be quasi-compact quasi-separated.
Then the composite
\[
\Id h_*=h_*=h_*\Id \xrightarrow{p_{G,Y}}h_*(?)^G
\xrightarrow{e}(?)^Gh_*
\]
is $p_{G,S}$.
\item[\bf 3] Let $\Cal V\in\Qch(S)$ be $G$-trivial, and $\M\in\Qch(G,S)$.
Then $1\otimes\gamma:\Cal V\otimes \M^G\rightarrow \Cal V\otimes \M$ is
a monomorphism, and it induces an isomorphism $\gamma'$ 
onto $(\Cal V\otimes \M)^G$.
The composite
\[
\Cal V\otimes \M \xrightarrow{1\otimes p_{G,S}}\Cal V\otimes\M^G
\xrightarrow{\gamma'}(\Cal V\otimes \M)^G
\]
is $p_{G,S}$.
Let $\delta:U_G(\M)\rightarrow \M$ be the inclusion.
Then $1\otimes\delta:\Cal V\otimes U_G(\M)\rightarrow \Cal V\otimes \M$ 
induces an isomorphism $\delta':\Cal V\otimes U_G(\M)
\rightarrow U_G(\Cal V\otimes \M)$.
\item[\bf 4] Let $\M,\N\in\Qch(G,S)$, and assume that $\M$ is trivial and
$\N$ is anti-trivial.
Then $\uHom_{G,\O_S}(\M,\N)=0=\uHom_{G,\O_S}(\N,\M)$.
\item[\bf 5] Let $\Cal V\in\Qch(G,S)$ be $G$-trivial, and $\M\in\Qch(G,S)$.
Then $\uHom_{\O_S}(\Cal V,\M^G)$ is $G$-trivial, and 
$\uHom_{G,\O_S}(\Cal V,U_G(\M))=0$.
In particular, 
\[
\uHom_{\O_S}(\Cal V,\M^G)\rightarrow\uHom_{\O_S}(\Cal V,\M)
\]
is an isomorphism
onto $\uHom_{\O_S}(\Cal V,\M)^G$.
If, moreover, $\uHom_{\O_S}(\Cal V,\M)$ is quasi-coherent, then
\[
\uHom_{\O_S}(\Cal V,U_G(\M)) \rightarrow \uHom_{\O_S}(\Cal V,\M)
\]
is an isomorphism onto 
$U_G(\uHom_{\O_S}(\Cal V,\M))$.
\item[\bf 6] Let $\Cal V\in\Qch(G,S)$ be $G$-trivial, and $\M\in\Qch(G,S)$.
Then $\uHom_{\O_S}(\M^G,\Cal V)$ is $G$-trivial, and 
$\uHom_{G,\O_S}(U_G(\M),\Cal V)=0$.
In particular, 
the monomorphism $p_{G,S}^*:\uHom_{\O_S}(\M^G,V)\rightarrow 
\uHom_{\O_S}(\M,\V)$ is an isomorphism onto $\uHom_{\O_S}(\M,\V)^G$.
If, moreover, $\uHom_{\O_S}(\M,\V)$ is quasi-coherent, then
the monomorphism $q_{G,S}^*: \uHom_{\O_S}(U_G(\M),\V)\rightarrow
\uHom_{\O_S}(\M,\V)$ is an isomorphism onto 
$U_G(\uHom_{\O_S}(\M,\V))$.
\end{enumerate}
\end{lemma}

\begin{proof}
{\bf 0} is trivial by Corollary~\ref{universal-Reynolds.cor}.

{\bf 1} We prove that $\epsilon$ is an isomorphism.
The case that $h$ is an open immersion (from a sufficiently small
affine open subset) follows from Lemma~\ref{covering-Reynolds.lem}.
The case that both $Y$ and $Y'$ are affine follows from 
Lemma~\ref{Reynolds-affine-ii.lem}.
The general case follows from these, using 
Corollary~\ref{epsilon-composition.cor}.

The latter part is obvious, because the composite map is the identity on 
$(?)^Gh^*$ by Lemma~\ref{eta-bar-eta.thm}.

{\bf 2} is proved similarly to {\bf 1}.

{\bf 3} As $\V\otimes\M^G$ is $G$-trivial, it suffices to show that
$\V\otimes U_G(\M)$ is anti-$G$-trivial.
This is checked locally, and we may assume that $Y$ is affine.
Then this is Lemma~\ref{Reynolds-affine-i.lem}.

{\bf 4} It suffices to show that $\Hom_{G,\O_U}(\M|_U,\N|_U)=0=\Hom_{G,\O_U}(
\N|_U,\M|_U)$ for any affine open subset $U$.
As $\M|_U$ is trivial and $\N|_U$ is anti-trivial, 
This is checked in (\ref{trivial-anti-trivial.par}).

{\bf 5, 6} follow from {\bf 4}.
\end{proof}

\begin{lemma}\label{direct-summand-Reynolds.lem}
Let $G$ be a Reynolds $S$-group scheme.
Let $\Cal B$ be a quasi-coherent $(G,\O_S)$-algebra, and 
$\Cal A=\Cal B^G$.
Then the Reynolds operator $p_{G,S}:\Cal B\rightarrow \Cal A$ is
$(G,\Cal A)$-linear.
In particular, $\Cal A$ is a direct summand subalgebra of $\Cal B$.
\end{lemma}

\begin{proof}
Set $Y:=\uSpec_S \Cal A$.
Then $p_{G,S}$ is identified with $p_{G,Y}$, which is $\Cal A$-linear.
\end{proof}

\begin{example}
Let $S=\Spec k$.
We say that an affine algebraic $k$-group scheme $G$ is linearly
reductive if any $G$-module is completely reducible.
In this case, for a $G$-module $M$, letting $U_G(M)$ to be the sum of all
the nontrivial simple $G$-submodules, 
we have a decomposition $M=M^G\oplus U_G(M)$
of $G$-modules.
When we set $p:M\rightarrow M^G$ to be the projection with respect to this
decomposition, we have that $p$ is a Reynolds operator.
Conversely, if $G$ is an affine algebraic $k$-group scheme which is Reynolds, 
then $G$ is linearly reductive.
Assume the contrary.
Then there is a non-semisimple $G$-module $V$.
Let $W$ be its socle.
Then we have that $V/W\neq 0$, and there is a simple submodule $E$ of $V/W$.
Then $0=H^1(G,E^*\otimes W)\cong \Ext^1_G(E,W)\neq 0$, a contradiction.
\end{example}

\begin{example}\label{linearly-reductive-Reynolds.ex}
Let $R$ be a commutative ring, and $H$ a flat commutative $R$-Hopf algebra.
We say that $H$ is Reynolds if there is a decomposition $H=R\oplus U$ 
as an $R$-coalgebra, where $R$ denotes the image $u(R)$ of the unit map $u
:R\rightarrow H$.
Then for any $H$-comodule $M$, the decomposition 
$M=M^H\oplus \ind_H^U M$ is functorial, and so letting $U_G=\ind_H^U$, 
we have that $G=\Spec H$ has a Reynolds operator over $S=\Spec R$.
Hence $G_Y$ is Reynolds for any $R$-scheme $Y$.
Conversely, if $G$ has a Reynolds operator on $S$, then the right regular
representation $H$ is decomposed as $H=R\oplus U$, where $U=U_G(H)$.
As $U$ is a right subcomodule, $\Delta(U)\subset U\otimes H$.
Letting $G$ act trivially on $U$ and right regularly on $H$, 
$U\otimes H$ is a $G$-module, and $\Delta:U\rightarrow U\otimes H$ is
$G$-linear.
So
\[
\Delta(U)\subset 
U_G(U\otimes H)=U\otimes U
\]
by Lemma~\ref{Reynolds-affine-i.lem}.
Being a direct summand of $H$, $U$ is a subcoalgebra of $H$, and
$H=R\oplus U$ is a decomposition as subcoalgebras, and hence $H$ is Reynolds.
\end{example}

\begin{example}\label{finite-Reynolds.ex}
Let $G$ be a finite group of order $n$, and $R=\Bbb Z[n^{-1}]$.
Then $\rho=n^{-1}\sum_{g\in G}g\in RG$ is a central idempotent of $RG$ such
that $g\rho=\rho=\rho g$ for $g\in G$, where $RG$ is the group algebra.
Then we have a decomposition of bimodules $RG=\rho(RG)\oplus (1-\rho)(RG)$.
So we have a decomposition of $R[G]$-bicomodules
\[
R[G]=(RG)^*=(\rho(RG))^*\oplus ((1-\rho)(RG))^*
=\rho R[G]\oplus (1-\rho)R[G].
\]
As $\varepsilon:\rho RG\rightarrow R$ is an isomorphism (since 
$\varepsilon(\rho)=1$), $\rho R[G]=R$,
and $G$ is Reynolds.
The Reynolds operator $p$ is the action of $\rho$, and the element $\rho\in
RG$ is also called the Reynolds operator.
\end{example}

\begin{example}\label{diagonalizable-Reynolds.ex}
Let $\Lambda$ be an additive abelian group, and $R$ a commutative ring.
Then the group ring $R\Lambda=\bigoplus_{\lambda\in\Lambda}Rt^\lambda$ 
is a Hopf algebra, letting each
$t^\lambda$ ($\lambda\in\Lambda$) group-like (that is, $\Delta(t^\lambda)
=t^\lambda\otimes t^\lambda$).
Then as $R\Lambda= R\oplus(\bigoplus_{\lambda\neq 0}R\cdot t^{\lambda})$, 
$R\Lambda$ is Reynolds.
If $\Lambda\cong\Bbb Z^s$, then $G=\Spec R\Lambda$ is a split torus
(of relative dimension $s$), and $G$ is Reynolds.
\end{example}

\paragraph
Let $Y$ be an $S$-scheme, and $G$ an $S$-group scheme.
Let $\kappa$ be an infinite regular cardinal
such that $S$, $G$, and $Y$ are $\kappa$-schemes \cite[(3.11)]{Hashimoto5}.
As in \cite[(3.14)]{Hashimoto5}, we denote the full subcategory of 
the category of $Y$-schemes consisting of $\kappa$-morphisms
by $(\Sch/Y)_\kappa$.
We call a presheaf on $(\Sch/Y)_\kappa$
a $Y$-prefaisceau.
We denote the structure presheaf of $(\Sch/Y)_\kappa$ 
by $\O$.
For an $\O_Y$-module $\M$ (in the Zariski topology), we denote the
associated $Y$-prefaisceau by $\M_a$.
That is, $\M_a$ is the $\O$-module given by $\Gamma(h:Z\rightarrow Y,\M_a)
=\Gamma(Z,h^*\M)$.
For an $\O_Y$-module $\M$, 
the $Y$-prefaisceau of groups $Z\mapsto \End_{\O_Z}\Gamma(Z,\M_a)^\times$ 
is denoted by $\GL(\M)$, and called the 
{\em general linear group} of $\M$.

\paragraph
Let $(\M,\phi)$ be a $G$-linearized $\O_Y$-module.
Then for any $S$-scheme $W$, $y\in Y(W)$ and $\alpha,\beta\in G(W)$, 
the composite
\[
(\alpha\beta y)^*\M\xrightarrow{(\alpha,\beta y)^*\phi} (\beta y)^*\M
\xrightarrow{(\beta,y)^*\phi}y^*\M
\]
agrees with $(\alpha\beta,y)^*\phi$, see \cite[(1.3)]{GIT}.

\paragraph\label{representation-hom-to-GL.par}
Assume that the action of $G$ on $Y$ is trivial.
Then we denote $(\alpha,y)^*\phi:y^*\M\rightarrow y^*\M$ by $h(\alpha)$.
Then by the argument above, $h(\alpha)h(\beta)=h(\alpha\beta)$,
and we get a homomorphism between $Y$-prefaisceaux of groups
$h:G_Y\rightarrow \GL(\M)$, where $G_Y$ is the restriction of $G$ to
$(\Sch/Y)_\kappa$.
For a given $\O_Y$-module $\M$ on an $S$-scheme $Y$ with a trivial 
$G$-action, giving a $G$-linearization $\phi$ and giving a group
homomorphism $G_Y\rightarrow \GL(\M)$ are the same thing.
$\M$ is quasi-coherent if and only if 
$B\otimes_A \Gamma(\Spec A,\M_a)\rightarrow \Gamma(\Spec B,\M_a)$ 
is an isomorphism for any morphism of the form
$\Spec B\rightarrow \Spec A$ in $(\Sch/Y)_\kappa$.

\paragraph
If $G$ is $S$-flat, then we modify the construction above, and we consider
the full subcategory $\E$ of $(\Sch/Y)_\kappa$ consisting of flat 
$Y$-schemes.
If $\M$ is a $G$-equivariant module on $Y$ (in the Zariski topology), 
then we get a homomorphism $h:G_Y\rightarrow \GL(\M)$ of prefaisceaux of
groups on $(\Sch/Y)_\kappa$.
By restriction, we get $h:G|_\E\rightarrow \GL(\M)|_\E$.
As $G$ is flat, it is easy to see that giving such a homomorphism is
the same thing as to give a $G$-linearization on $\M$.

\paragraph
Let $G$ be $S$-flat, $\M$ a $G$-equivariant module on $Y$, and $\N$ its
$\O_Y$-submodule.
Although $\N_a$ may not be a submodule of $\M_a$,
we have that $\N_a|_\E$ is a submodule of $\M_a|_\E$ by flatness.

\begin{lemma}\label{submodule-faisceau.lem}
Let the notation be as above.
Then $\N$ is a $(G,\O_Y)$-submodule of $\M$ if and only if
for each object $U$ in $\E$ which is an affine scheme, 
$\Gamma(U,\N_a)$ is a $G(U)$-submodule of $\Gamma(U,\M_a)$.
\end{lemma}

\begin{proof}
The `only if' part is trivial.
We prove the `if' part.

Let $U$ be any affine open subset of $G\times Y$.
Then $U\xrightarrow{j} G\times Y\xrightarrow{p_2} Y$ lies in $\E$ and
$U$ is affine.
where $j$ is the inclusion, and $p_2$ is the second projection.
Let $g\in G|_\E(U)$ be the map $p_1j:U\rightarrow G$, 
where $p_1:G\times Y\rightarrow G$ is the
first projection.
The action of $g$ on $\Gamma(p_2j:U\rightarrow Y,\M_a)$ is induced by 
$j^*\phi:j^*p_2^*\M\rightarrow j^*p_2^*\M$, where $\phi$ is the linearization
of $\M$.

By assumption, the actions of $g$ and $g^{-1}$ preserve
the submodule $\Gamma(p_2j:U\rightarrow Y,\N_a)$.
As $U$ is arbitrary, $p_2^*\N$ is preserved by the linearization $\phi$ and
its inverse $\phi^{-1}$.
Hence $\N$ is a $G$-equivariant submodule.
\end{proof}

\begin{lemma}\label{invariance-faisceau.lem}
Let $G$ be a flat quasi-compact quasi-separated $S$-group scheme,
$Y$ an $S$-scheme on which $G$ acts trivially, 
and $\M$ a quasi-coherent $(G,\O_Y)$-module.
Then $(\M^G)_a|_\E$ is an $\O$-submodule of $\M_a|_\E$ given by
\begin{multline*}
\Gamma(W,(\M^G)_a)=\{m\in \Gamma(W,\M_a)\mid gm=m \text{ in $\Gamma(W',\M_a)$}
\\
\text{for any morphism $W'\rightarrow W$ in $\E$}, 
\text{ and any $g\in G(W')$}\}.
\end{multline*}
\end{lemma}

\begin{proof}
First we prove the assertion for $W=Y$.
Then the left-hand side is 
\[
\{m\in \Gamma(Y,\M)\mid g_0m=m 
\text{ in $\Gamma(G\times Y,\M)$}
\},
\]
where $g_0\in G(G\times Y)$ is the first projection.
So the right-hand side is contained in the left-hand side.
On the other hand, if $h:W'\rightarrow Y$ is any morphism in $\E$, then
for any $g\in G(W')$, we define $\psi_g:W'\rightarrow G\times Y$ by 
$(g,h)$.
Then by definition, $g=\psi_g(g_0)$.
So if $g_0m=m$ in $\Gamma(G\times Y,\M)$, then $gm=m$ in $\Gamma(W',\M)$,
and the equality was proved.

Next consider general $W\in\E$.
Then replacing $Y$ by $W$ using \cite[(7.5)]{HO2}, the problem is reduced
to the case $Y=W$, and we are done.
\end{proof}

\begin{proposition}\label{enriched-Reynolds.prop}
Let $f:G\rightarrow H$ be a quasi-compact quasi-separated flat 
homomorphism of $S$-group schemes with
$N=\Ker f$.
Then for $\M\in\Qch(G,Y)$, $\M^N$ is 
a quasi-coherent $(G,\O_Y)$-submodule of $\M$.

If, moreover, $N$ is Reynolds, then $U_N(\M)$ is also a 
quasi-coherent $(G,\O_Y)$-submodule of $\M$.
In particular, the Reynolds operator $p_{N,Y}:\M\rightarrow \M^N$ is
$(G,\O_Y)$-linear.
\end{proposition}

\begin{proof}
Although the first assertion can be proved in the same line of 
\cite[(6.18)]{Hashimoto5}, we give a new proof.

For the first assertion,
in view of Lemma~\ref{submodule-faisceau.lem} and 
Lemma~\ref{invariance-faisceau.lem}, it suffices to show that for 
each flat morphism $h:W\rightarrow Y$ and any morphism 
$h':W'\rightarrow W$ such that $hh'$ is flat, 
$m\in\Gamma(W,(\M^G)_a)$, $g\in G(W)$, and
$n\in G(W')$, we have that $ngm=gm$ in $\Gamma(W',\M_a)$.
As $N$ is normal in $G$, $ngm=g(g^{-1}ng)m=gm$.
As the quasi-coherence is trivial, the first assertion has been proved.

Assume that $N$ is Reynolds.
Let $W=\Spec A\rightarrow Y$ be an object of $\E$ such that $W$ is affine,
and $g\in G(W)$.
Set $U=\Gamma(W,U_N(\M)_a)$.
We claim that $gU$ is an $(N,A)$-submodule of $M:=\Gamma(W,\M_a)$.
In order to show this, it suffices to show that for any flat $\kappa$-morphism
$W'=\Spec A'\rightarrow W$ and $n\in N(W')$, we have $n((gU)\otimes_A A')
\subset (gU)\otimes_A A'$ in $M\otimes_A A'$.
This is clear, since
\[
n(gu\otimes a')=g((g^{-1}ng)(u\otimes a'))\in g(U\otimes_A A')=(gU)\otimes_A A'
\]
for $u\in U$ and $a'\in A'$.

Next, we prove that $(gU)^N=0$.
Assume the contrary, and take $u\in U\setminus 0$ such that for each 
flat $\kappa$-morphism $W'\rightarrow W$ and $n\in N(W')$, 
$ngu=gu$.
Then $nu=g^{-1}(gng^{-1})gu=g^{-1}gu=u$, and hence $u\in U\cap M^N=0$, and
this is a contradiction.
Hence $(gU)^N=0$.
That is, $gU\subset U$.
By Lemma~\ref{submodule-faisceau.lem}, we have that $U_N(\M)$ is a 
$G$-equivariant submodule.
Quasi-coherence is trivial.

The last assertion is clear from the fact that the decomposition
$\M=\M^N\oplus U_N(\M)$ is that of a $(G,\O_Y)$-module.
\end{proof}

\begin{lemma}\label{Gabber-Grothendieck.lem}
Let $G$ be a flat quasi-compact quasi-separated $S$-group scheme,
and $X$ a $G$-scheme.
Then $\Qch(G,X)$ is a Grothendieck category.
\end{lemma}

\begin{proof}
In view of \cite[(11.5)]{ETI}, it suffices to show that $\Qch(X)$ is
Grothendieck.
This is Gabber's theorem \cite[(2.1.7)]{Conrad}.
\end{proof}

\begin{lemma}\label{N-injective.thm}
Let $S$ be a scheme, and $G$ a Reynolds group over $S$.
Let $Y$ be an $S$-scheme on which $G$ acts trivially.
Let $\Cal I$ be an injective object of $\Qch(Y)$.
Then $\Cal I$ viewed as an object of $\Qch(G,Y)$
\(formally $\res^e_G \Cal I$\) is an injective object.
\end{lemma}

\begin{proof}
Set $\Cal J:=\res^e_G\Cal I$.
Let $i:\Cal J\hookrightarrow\M$ be a monomorphism in $\Qch(G,Y)$.
By Lemma~\ref{Gabber-Grothendieck.lem}, we have that $\Qch(G,Y)$ has
enough injectives, 
and hence it suffices to show that $i$ splits.
As we have that $\Cal J=\Cal J^G$, the image of $i$ is contained in $\M^G$.
As $\M^G$ is a direct summand of $\M$, it suffices to show that
$i:\Cal J\hookrightarrow \M^G$ splits.
Note that 
$(?)_0:\Qch_G(G,Y)\rightarrow \Qch(Y)$ and $\res^e_N:\Qch(Y)\rightarrow
\Qch_G(G,Y)$ are quasi-inverse each other, where $\Qch_G(G,Y)$ is the
category of $G$-trivial objects in $\Qch(G,Y)$.
So it suffices to show that $\Cal I\hookrightarrow \M^G$ splits in $\Qch(Y)$.
This is obvious, since $\Cal I$ is injective by assumption.
\end{proof}

\begin{lemma}\label{UN.thm}
Let $f:G\rightarrow H$ be an fpqc homomorphism of flat $S$-group schemes
with $N=\Ker f$.
Assume that $N$ is Reynolds.
Let $Y$ be a locally Noetherian $G$-scheme on which $N$ acts trivially.
Let $\Bbb F\in D^-_{\Coh}(G,Y)$ and $\Bbb G\in D^+_{\Qch}(G,Y)$.
If for each $i\in\Bbb Z$, $H^i(\Bbb F)=U_N(H^i(\Bbb F))$ and
$H^i(\Bbb G)=H^i(\Bbb G)^N$,
then $\uExt^i_{\O_{B_M^G(Y)}}(\Bbb F,\Bbb G)^N=0$ for $i\in\Bbb Z$.
\end{lemma}

\begin{proof}
Note that the full subcategory $\Coh_N(G,Y)$ (resp.\ $U_N(G,Y)$) 
consisting of $N$-trivial (resp.\ $N$-anti-trivial) 
coherent (resp.\ quasi-coherent) $G$-modules 
forms a plump subcategory (that is, a full subcategory
which is closed under extensions, kernels, and cokernels) of $\Mod(G,Y)$.
So we may assume that $\Bbb F=\M$ and $\Bbb G=\N$ are single coherent sheaf
and quasi-coherent sheaf, respectively, by the way-out lemma 
\cite{Hartshorne2}.

Note that
$\uExt^i_{\O_{B_M^G(Y)}}(\M,\N)^N=0$
if and only if 
\[
(\res_N^G(\uExt^i_{\O_{B_M^G(Y)}}(\M,\N)))^N
\cong \uExt^i_{\O_{B_N^M(Y)}}(\res_N^G\M,\res_N^G\N)^N=0
\]
by Corollary~\ref{restriction-extension.thm}.
Set $\M':=\res_N^G\M$ and $\N':=\res_N^G\N$.

Let $0\rightarrow \N'_0\rightarrow\Bbb I$ be an injective resolution in
$\Qch(Y)$.
Then applying Lemma~\ref{faithful-exact.thm},
Lemma~\ref{N-injective.thm}, and \cite[(15.2)]{ETI},
we have that $0\rightarrow \res_N^{\{e\}}\N'_0\rightarrow\res_N^{\{e\}}\Bbb I$ 
is an $\uHom_{\O_{B^M_N(Y)}}(\M',?)$-acyclic resolution.
By assumption, $\N'$ is of the form $\res_N^{\{e\}}\L$ for some $\L\in\Qch(Y)$.
Then $\res_N^{\{e\}}\N'_0\cong\res_N^{\{e\}}\L\cong\N'$.
So $\uExt^i_{\O_{B_N^M(Y)}}(\M',\N')^N$ is a subquotient of
$\uHom_{\O_{B_N^M(Y)}}(\M',\Bbb I^i)^N$.
This is zero by Lemma~\ref{Reynolds-four-operations.lem}, {\bf 4}.
\end{proof}

\begin{lemma}\label{linearly-reductive-algebraic-quotient.lem}
Let $G$ be a Reynolds $S$-group scheme,
and $\varphi:X\rightarrow Y$ an algebraic quotient by $G$.
Then $\varphi$ is a universal algebraic quotient.
If, moreover, $\varphi$ is an affine universally
submersive geometric quotient, then $\varphi$ is a universal 
geometric quotient.
\end{lemma}

\begin{proof}
Let $h:Y'\rightarrow Y$ be an $S$-morphism between $S$-schemes on which
$G$ acts trivially, and consider the fiber square
(\ref{fiber.eq}) in Lemma~\ref{eta-inverse.thm}.
As $\varphi$ is affine, $\theta:h^*\varphi_*\O_X\rightarrow
\varphi'_*g^*\O_X$ is an isomorphism.
Now the first assertion follows from
Lemma~\ref{bar-eta-isom.thm}, {\bf 1} and
Lemma~\ref{Reynolds-four-operations.lem}, {\bf 1}.
The second assertion follows immediately.
\end{proof}

\section{Base change of twisted inverse}\label{base-change.sec}

\begin{lemma}\label{independent-square.lem}
Let
\begin{equation}\label{tor-independent.eq}
\xymatrix{
X' \ar[r]^g \ar[d]^{\varphi'} \ar@{}[dr]|{\sigma}  &
X \ar[d]^\varphi \\
Y' \ar[r]^h & Y
}
\end{equation}
be a fiber square of schemes.
Assume that $\varphi$ is quasi-compact quasi-separated.
Then the following are equivalent.
\begin{enumerate}
\item[\bf 1] Lipman's theta 
$\theta:Lh^*R\varphi_*\rightarrow R\varphi'_* Lg^*$ between
the functors $D_{\Qch}(X)\rightarrow D_{\Qch}(Y')$ 
(cf.~\cite[(3.9.1), (3.9.2)]{Lipman}) is an isomorphism.
\item[\bf 2] The square is tor-independent in the sense that
for each $x\in X$ and $y'\in Y'$ such that $\varphi(x)=h(y')=y\in Y$,
$\Tor^{\O_{Y,y}}_i(\O_{X,x},\O_{Y',y'})=0$ for $i>0$.
\end{enumerate}
\end{lemma}

\begin{proof}
As the question is local both on $Y$ and $Y'$, we may assume that both
$Y$ and $Y'$ are affine.
This case is \cite[(3.10.3)]{Lipman}.
\end{proof}

\paragraph
Let $I$ be a small category, and 
(\ref{tor-independent.eq}) be a tor-independent 
fiber square of $I\op$-diagrams of Noetherian schemes 
such that $\varphi:X\rightarrow Y$ is proper.
By Lemma~\ref{independent-square.lem},
Lipman's theta $\theta:Lh^*R\varphi_*\rightarrow R\varphi'_*Lg^*$ is
an isomorphism (between functors from $D_{\Lqc}(X)\rightarrow D_{\Lqc}(Y')$.
Then we define $\zeta(\sigma):
Lg^*\varphi^\times
\rightarrow
(\varphi')^\times h^*
$ as the composite
\[
  Lg^*\varphi^\times
  \xrightarrow u
  (\varphi')^\times R\varphi'_*  Lg^*\varphi^\times
  \xrightarrow{\theta^{-1}}
  (\varphi')^\times Lh^* R\varphi_* \varphi^\times
  \xrightarrow{\varepsilon}
  (\varphi')^\times Lh^*
\]
(cf.~\cite[(19.1)]{ETI}).

\begin{lemma}
Let 
  \begin{equation}\label{horizontal-squares.eq}
  \xymatrix{
  X'' \ar[r]^{g'} \ar[d]^{\varphi''} \ar@{}[dr]|{\sigma'} &
  X' \ar[r]^{g} \ar[d]^{\varphi'} \ar@{}[dr]|{\sigma} &
  X \ar[d]^{\varphi} \\
  Y'' \ar[r]^{h'} &
  Y' \ar[r]^{h} &
  Y
  }
  \end{equation}
  be a commutative diagram of schemes.
  Assume that $\sigma$ is a tor-independent cartesian square.
  Then $\sigma'$ is tor-independent cartesian if and only if
  the whole rectangle $\sigma'+\sigma$ is a tor-independent
  cartesian square.
\end{lemma}

\begin{proof}
  As $\sigma$ is cartesian, $\sigma'$ is cartesian if and only if
  $\sigma'+\sigma$ is cartesian.

  Let $y''\in Y''$ and $x\in X$ such that $hh'(y'')=\varphi(x)$.
  Let $A''=\O_{Y'',y''}$, $A'=\O_{Y',y'}$, $A=\O_{Y,y}$, and $B=\O_{X,x}$,
  where $y'=h'(y'')$ and $y=h(y')=\varphi(x)$.
  There is a spectral sequence
  \[
  E^2_{p,q}=\Tor_p^{A'}(A'',\Tor_q^{A}(A',B))\Rightarrow \Tor_{p+q}^{A}(A'',B).
  \]
  By assumption, $E^2_{p,q}=0$ for $q\neq 0$.
  So $\Tor_n^{A'}(A'',A'\otimes_A B)\cong \Tor_n^A(A'',B)$, and the
  equivalence follows.
\end{proof}

\begin{lemma}\label{hh'.lem}
  Let $I$ be a small category, and (\ref{horizontal-squares.eq})
  be a diagram of $I\op$-diagrams of Noetherian schemes.
  Assume that $\varphi$ is proper, and
  $\sigma$ and $\sigma'$ are tor-independent fiber squares.
  Then the composite
  \begin{multline*}
  L(gg')^* \varphi^\times
  \xrightarrow d
  L(g')^* Lg^*\varphi^\times
  \xrightarrow{\zeta(\sigma)}
  L(g')^*(\varphi')^\times Lh^*\\
  \xrightarrow{\zeta(\sigma')}
  (\varphi'')^\times L(h')^*Lh^*
  \xrightarrow d
  (\varphi'')^\times L(hh')^*
  \end{multline*}
  agrees with $\zeta(\sigma'+\sigma)$.
\end{lemma}

\begin{proof}
  Straightforward, and left to the reader (use \cite[(1.23)]{ETI}).
\end{proof}

\begin{lemma}\label{phi-psi.lem}
  Let
  \begin{equation}\label{vertical-squares.eq}
  \xymatrix{
    X' \ar[r]^f \ar[d]^{\psi'} \ar@{}[dr]|{\sigma'} &
    X \ar[d]^\psi \\
    Z' \ar[r]^g \ar[d]^{\varphi'} \ar@{}[dr]|{\sigma}  &
    Z \ar[d]^\varphi \\
    Y' \ar[r]^h & Y
    }
  \end{equation}
  be a diagram of $I\op$-diagrams of Noetherian schemes.
  Assume that $\varphi$ and $\psi$ are proper, and
  $\sigma$ and $\sigma'$ are tor-independent fiber squares.
  Then the composite
  \begin{multline*}
    Lf^*(\varphi\psi)^\times
    \xrightarrow d
    Lf^*\psi^\times \varphi^\times
    \xrightarrow{\zeta(\sigma')}
    (\psi')^\times Lg^* \varphi^\times\\
    \xrightarrow{\zeta(\sigma)}
    (\psi')^\times (\varphi')^\times Lh^*
    \xrightarrow d
    (\varphi'\psi')^\times Lh^*
  \end{multline*}
  agrees with $\zeta(\sigma'+\sigma)$.
\end{lemma}

\begin{proof}
  Left to the reader (use \cite[(1.22)]{ETI}).
\end{proof}

\paragraph\label{subscheme.par}
Let $G$ be an $S$-group scheme, $Y$ be a $G$-scheme, and $X$ a
$G$-stable subscheme.
That is, $X$ is a subscheme of $Y$ such that $GX\subset X$.
Then $X$ is a closed subscheme of an open subscheme $U$ of $Y$.
Assume that $G$ is universally open (e.g., flat locally of finite presentation)
over $S$.
Then $GU$ is a $G$-stable open subscheme of $Y$, and $X=GU
\setminus G(U\setminus X)$ is a $G$-stable closed subset in $GU$.
Being a $G$-stable subscheme of $Y$, $X$ is a $G$-stable closed
subscheme of $GU$.
Thus if $G$ is universally open, a $G$-stable subscheme is nothing but
a $G$-stable closed subscheme of a $G$-stable open subscheme.

\paragraph\label{zeta-bar.par}
Let $G$ be a flat $S$-group scheme of finite type, and
$\varphi:X\rightarrow Y$ an immersion between Noetherian $G$-schemes.
As we have seen in (\ref{subscheme.par}), 
We can factorize $\varphi$ as
\[
\varphi: X\xrightarrow p
U\xrightarrow i
Y,
\]
where $U$ is a $G$-stable open subscheme of $Y$, $i$ the inclusion,
and $p$ is a $G$-stable closed immersion.

Let $h:Y'\rightarrow Y$ be a $G$-morphism between Noetherian $G$-schemes.
Assume that $\varphi$ and $h$ are tor-independent.
Then $p$ and (the base change of) $h$ are also tor-independent.

For the fiber square (\ref{tor-independent.eq}), 
we define $\bar\zeta(\sigma)$ to be the composite
\[
h^*\varphi^!
=
h^*p^\times i^*
\xrightarrow\zeta
(p')^\times g^* i^*
\xrightarrow d
(p')^\times (i')^*h^*
=
(\varphi')^!h^*,
\]
where
  \begin{equation}\label{vertical-squares.eq}
  \xymatrix{
    X' \ar[r]^f \ar[d]^{p'} \ar@{}[dr]|{\sigma_1} &
    X \ar[d]^p \\
    U' \ar[r]^g \ar[d]^{i'} \ar@{}[dr]|{\sigma_2}  &
    U \ar[d]^i \\
    Y' \ar[r]^h & Y
    }
  \end{equation}
is a commutative diagram with $\sigma_1$ and $\sigma_2$ are cartesian,
and $\varphi'=i'p'$.
Using Lemma~\ref{hh'.lem}, it is easy to see that
the definition of $\bar\zeta(\sigma)$ depends only on $\sigma$, and
is independent of the choice of factorization $\sigma_1$ and $\sigma_2$.

\section{Serre's conditions and the canonical modules}
\label{Serre-canonical.sec}

\paragraph
Let $G$ be an $S$-group scheme.
We say that a $G$-scheme $X$ is {\em $G$-connected} if $X=X_1\coprod X_2$ with
$X_1$ and $X_2$ are $G$-stable open subsets, then either $X_1$ or $X_2$ is
empty.
In this paper, a $G$-connected $G$-scheme is required to be nonempty.
A connected topological space is also required to be nonempty.
If the action of $G$ on $X$ is trivial, then $G$-connected and connected 
are the same thing.

\begin{lemma}\label{categorical-quotient-mono.lem}
Let $G$ be an $S$-group scheme, and $\psi:X\rightarrow W$ and 
$u:W\rightarrow Y$ be $S$-morphisms.
Assume that $\varphi=u\psi:X\rightarrow Y$ is $G$-invariant, and is a 
categorical quotient by $G$.
If $u$ is a monomorphism \(e.g., an immersion\), then $u$ is an isomorphism.
\end{lemma}

\begin{proof}
As $u$ is a monomorphism, it is easy to see that $\psi$ is also 
$G$-invariant.
By the definition of the categorical quotient, there exists some
$v:Y\rightarrow W$ such that $\psi=v\varphi$.
Then $1_Y\varphi=\varphi=u\psi=uv\varphi$.
As $\varphi$ is a categorical quotient, $1_Y=uv$ by the uniqueness.
As $u1_W=u=1_Yu=uvu$ and $u$ is a monomorphism, $1_W=vu$.
Hence $u$ is an isomorphism.
\end{proof}

\begin{lemma}\label{connected.lem}
Let $G$ be an $S$-group scheme, and $\varphi:X\rightarrow Y$ a $G$-morphism.
\begin{enumerate}
\item[\bf 1] If $\varphi$ is dominating and
$X$ is $G$-connected, then $Y$ is $G$-connected.
\item[\bf 2] 
If $\varphi$ is $G$-invariant dominating and $X$ is $G$-connected, 
then $Y$ is connected.
\item[\bf 3] 
Assume that $\varphi$ is a categorical quotient or an algebraic quotient.
Then
$X$ is $G$-connected if and only if $Y$ is connected.
\end{enumerate}
\end{lemma}

\begin{proof}
{\bf 1}.
Assume the contrary, and let $Y=Y_1\coprod Y_2$ with $Y_i$ are
nonempty $G$-invariant open.
Then letting $X_i=\varphi^{-1}(Y_i)$, 
we have that 
$X=X_1\coprod X_2$ with $X_i$ nonempty $G$-invariant open, and this
is a contradiction.

{\bf 2} is obvious from {\bf 1}.

{\bf 3} We prove the \lq only if' part.
First consider the case that
$\varphi$ is a categorical quotient.
Assume that $Y=Y_1\coprod Y_2$ with each $Y_i$ nonempty open.
Then $\varphi$ cannot factors through $Y_1$ or $Y_2$ by 
Lemma~\ref{categorical-quotient-mono.lem}.
Letting $X_i=\varphi^{-1}(Y_i)$ for $i=1,2$, we have that each 
$X_i$ is nonempty $G$-stable open
and $X=X_1\coprod X_2$.
Next, assume that $\varphi$ is an algebraic quotient.
Then it is easy to see that $\varphi$ is dominating and $G$-invariant.
By {\bf 2}, if $X$ is $G$-connected, then $Y$ is connected.

We prove the \lq if' part.
Assume that $X=X_1\coprod X_2$ with $X_i$ are nonempty $G$-stable open.
First consider the case that $\varphi$ is a categorical quotient.
Then the map $h=h_1\coprod h_2:X_1\coprod X_2\rightarrow S\coprod S$,
where $h_i:X_i\rightarrow S$ is the structure map, factors through $Y$.
This shows that $Y$ cannot be connected.
Next, consider the case that $\varphi$ is an algebraic quotient.
Let $V=\Spec A$ be an affine open subset of $Y$.
Then $U=\varphi^{-1}(V)=\Spec B$ is affine.
We have $U=U_1\coprod U_2$ with $U_i=X_i\cap U$.
Set $B_i=\Gamma(U_i,\O_{U_i})$, and $A_i=B_i^G$.
Then $A=A_1\times A_2$, and we can write $V=V_1\coprod V_2$.
This construction is compatible with the localization $A\rightarrow A[a^{-1}]$
for $a\in A$.
So for $y\in Y$ and two affine open neighborhoods $V$ and $W$, $y\in V_1$ if and
only if $y\in W_1$.
So letting $Y_i=\bigcup_V V_i$, we have $Y=Y_1\coprod Y_2$.
As $\varphi^{-1}(Y_i)=X_i$, both $Y_1$ and $Y_2$ are nonempty, and $Y$ 
is disconnected.
\end{proof}

\paragraph
For a subset $Z$ of $X$, 
we say that $Z$ is a $G$-stable closed subset of $X$ if $X\setminus Z$ is
a $G$-stable open subset.
We say that $X$ is $G$-Noetherian if any descending chain of $G$-stable closed
subsets of $X$ eventually stabilizes.

\paragraph
Let $X$ be a $G$-Noetherian $G$-scheme.
A $G$-closed subset $Z$ of $X$ is said to be $G$-irreducible if 
$Z$ is nonempty and if
$Z=Z_1\cup Z_2$, $Z_1$ and $Z_2$ are $G$-closed subsets of $X$, then
either $Z=Z_1$ or $Z=Z_2$.
Any $G$-closed subset $Z$ of $X$ is of the form $Z=\bigcup_{i=1}^r Z_i$ for
some $r\geq 0$ and $G$-irreducible closed subsets $Z_i$.
Thus we can write $X=\bigcup_{i=1}^r X_i$ with $X_i$ $G$-irreducible.
Let $\equiv$ be the equivalence relation on $\{1,2,\ldots,r\}$ 
generated by the relation $X_i\cap X_j\neq \emptyset$.
For any equivalence class $\gamma$ with respect to $\equiv$, 
$X_\gamma:=\bigcup_{i\in\gamma}X_i$ is called a 
{\em $G$-connected component} of $X$.
Obviously, $X_\gamma$ is $G$-stable closed open and $G$-connected, 
and $X=\coprod_{\gamma\in\{1,\ldots,r\}/\equiv}X_\gamma$.
It is easy to see that a $G$-connected component of $X$ is nothing but
a maximal $G$-connected $G$-stable closed subset of $X$.

\paragraph
Let $I$ be a small category, and $X$ an $I\op$-diagram of 
Noetherian schemes.
Then a connected component (see for the definition, \cite[(28.1)]{ETI}) 
$U$ of $X$ is a cartesian closed open subdiagram of schemes.
Indeed, for any $i\in I$, $U_i$ is the union of connected components of 
$X_i$, and for any $\phi:i\rightarrow j$, $X_\phi^{-1}(U_i)= U_j$
by the definition of connected components.
In particular, we have

\begin{lemma}\label{F-connected.thm}
Let $G$ be an $S$-group scheme and $X$ a Noetherian $G$-scheme.
Then a connected 
component of $B_M^G(X)$ is of the form $B_M^G(X_i)$ with $X_i$ a $G$-connected
component of $X$.
Conversely, $B_M^G(X_i)$ with $X_i$ a $G$-connected component of $X$ is
a connected component of $B_M^G(X)$.
\qed
\end{lemma}

\paragraph
Let $G$ be an $S$-group scheme flat of finite type.
Let $X$ be a Noetherian $G$-scheme.

\begin{lemma}
If $\Bbb I_X$ is a $G$-dualizing complex on $X$, $n\in\Bbb Z$, 
and $\L$ a $G$-linearized invertible sheaf on $X$, then
$\Bbb I_X\otimes_{\O_X}\L[n]$ is a $G$-dualizing complex on $X$.
If $\Bbb I_X$ and $\Bbb I'_X$ are two $G$-dualizing complex on $X$
and $X$ is $G$-connected, 
then $R\Hom_{\O_X}(\Bbb I_X,\Bbb I'_X)\cong \L[n]$ for some $n$ and $\L$,
and we have $\Bbb I'_X\cong \Bbb I_X\otimes_{\O_X}\L[n]$.
\end{lemma}

\begin{proof}
This is \cite[(31.12)]{ETI}.
\end{proof}

\paragraph\label{G-canonical.par}
Let $G$ be an $S$-group scheme flat of finite type.
Let $X$ be a Noetherian $G$-scheme.
Let $\Bbb I_X$ be a $G$-dualizing complex on $X$.
Letting $s$ be the smallest integer such that $H^i(\Bbb I_X)\neq 0$, 
we define $\omega_X$ to be $H^s(\Bbb I_X)$.
We call $\omega_X$ the {\em $G$-canonical module} corresponding to $\Bbb I_X$.
If $X=\coprod X_i$ with each $X_i$ $G$-connected, then we define
$\omega'_X$ by $\omega'_X|_{X_i}=\omega_{X_i}$.
We call $\omega'_X$ the {\em componentwise $G$-canonical module} 
(in \cite[(31.13)]{ETI}, we called $\omega'_X$ the $G$-canonical module,
but this is less useful in this paper, and we change the terminology).

A coherent $(G,\O_X)$-module $\omega$ (resp.\ $\omega'$) 
is called a $G$-canonical module
(resp. a componentwise $G$-canonical module)
if there exists some $G$-dualizing complex $\Bbb I$ on $X$ such that
$\omega$ (resp.\ $\omega'$) 
is isomorphic to the $G$-canonical module 
(resp. a componentwise $G$-canonical module)
corresponding to $\Bbb I$.
Thus if $\omega$ (resp.\ $\omega'$) is a $G$-canonical module 
(resp. a componentwise $G$-canonical module)
and $\L$ is a $G$-linearized 
invertible sheaf, then $\omega\otimes_{\O_X}\L$ 
(resp.\ $\omega'\otimes_{\O_X}\L$) 
is a $G$-canonical module
(resp. a componentwise $G$-canonical module).
If $\omega'$ and $\omega''$ are two componentwise 
$G$-canonical modules on $X$, then
there exists some $G$-linerized invertible sheaf $\L$ on $X$ such that
$\omega''\cong\omega'\otimes_{\O_X}\L$.
If $G$ is trivial, then a $G$-canonical module and a 
$G$-componentwise canonical module (with respect to $\Bbb I$) 
are called a canonical module and
a componentwise canonical module (with respect to $\Bbb I$), respectively,
where $\Bbb I$ is a dualizing complex of $X$.

\begin{lemma}
Let $h:G'\rightarrow G$ be a homomorphism between
$S$-group schemes flat of finite type.
Let $X$ be a Noetherian $G$-scheme.
Let $\Bbb I_X(G)$ be a $G$-dualizing complex on $X$, and
set $\Bbb I_X(G'):=L\res^G_{G'}\Bbb I_X(G)$ \(see
{\rm Lemma~\ref{restriction-dualizing.lem}}\).
Let $\omega_X(G)$ be the $G$-canonical module corresponding to $\Bbb I_X(G)$.
Then $\omega_X(G'):=\res^G_{G'}\omega_X(G)$ is the $G'$-canonical module
corresponding to the $G'$-dualizing complex $\Bbb I_X(G')$.
\end{lemma}

\begin{proof}
Let $s=\inf\{i\in\Bbb Z\mid H^i(\Bbb I_X(G))\neq 0\}$.
By Lemma~\ref{res-vanishing.thm}, $\omega_X(G')=H^s(\Bbb I_X(G'))$ and
$H^i(\Bbb I_X(G'))=0$ for $i<s$.
As $\res^G_{G'}:\Qch(G,X)\rightarrow \Qch(G',X)$ is
faithful by Lemma~\ref{faithful-exact.thm},
$\omega_X(G')\neq 0$, and we are done.
\end{proof}

A similar compatibility with the change of groups does not hold for
the componentwise canonical module.

\paragraph
Let $X$ be a locally Noetherian scheme.
A coherent $\O_X$-module $\omega$ is said to be {\em semicanonical} at
$x\in X$ if $\omega_{X,x}$ is either zero, or is 
the canonical module \cite[(1.1)]{Aoyama} 
of the local ring 
$\O_{X,x}$.
We say that 
$\omega$ is semicanonical if it is so at each point.
Being a semicanonical module is a local condition.
That is, if $\omega$ is semicanonical and $U\subset X$ is an
open subset, then $\omega|_U$ is semicanonical on $U$.
If $(U_i)$ is an open covering of $X$ and each $\omega|_{U_i}$ is
semicanonical, then $\omega$ is semicanonical.

\begin{lemma}\label{dualizing-semicanonical.lem}
Let the notation be as in
{\rm(\ref{G-canonical.par})}.
Then the canonical module $\omega_X$ and the componentwise
canonical module corresponding to $\Bbb I_X$ are semicanonical $\O_X$-modules.
\end{lemma}

\begin{proof}
Let $x\in X$.
Then $\Bbb I_{X,x}$ is a dualizing complex in the usual sense for the
local ring $\O_{X,x}$.
So if $\omega_{X,x}\neq 0$, then by the definition of the canonical module
for a local ring \cite[(1.1)]{Aoyama} and the local 
duality \cite[(6.3)]{Hartshorne2}, $\omega_{X,x}$ is the canonical module
of $\O_{X,x}$.
Using this result componentwise, we get a similar result for $\omega'_{X}$.
\end{proof}

\paragraph
For a local ring $(A,\fm)$ and an $A$-module $M$, we define
$\depth_A M=\inf\{n\in\Bbb Z\mid H^n_\fm(M)\neq 0\}$.
For a commutative ring $A$, an ideal $I$, and an $A$-module $M$, we define
\[
\depth_A(I,M)=\depth(I,M):=\inf_{P\in V(I)}\depth_{A_P}M_P.
\]
If $A$ is Noetherian and $M$ is finitely generated, then we have
\[
\depth_A(I,M)=\inf\{n\in\Bbb Z\mid \Ext^n_A(A/I,M)\neq 0\},
\]
which also equals the length of a maximal $M$-sequence in $I$ 
(provided $M\neq IM$),
see \cite[(3.4), (3.6), (3.10)]{LC}.

\begin{lemma}\label{grade-local-cohomology.lem}
Let $A$ be a Noetherian ring, $M$ a finite $A$-module, and $I$ an ideal of $A$.
Then
\[
\depth_A(I,M)=\inf\{n\in\Bbb Z\mid H^n_I(M)\neq 0\}.
\]
\end{lemma}

\begin{proof}
If $\depth_A(I,M)\geq n$, then $\Ext^i_A(A/I^j,M)=0$ for $i<n$ and any 
$j\geq 1$ by 
\cite[(16.6)]{CRT}, and hence $H^i_I(M)=\indlim \Ext^i_A(A/I^j,M)=0$ 
for $i<n$, and the right-hand side is $\geq n$.

We prove the converse.
Let $n\geq 0$.
We want to prove that for a finite $A$-module $M$, 
$H^i_I(M)=0$ for $i<n$ implies that $\depth_A(I,M)\geq n$.
We prove this by induction on $n$.
If $n=0$, this is trivial.
Assume that $n>0$.
Then as $\Hom_A(A/I,M)\subset H^0_I(M)=0$, we have that $\depth_A(I,M)>0$.
So we can find a nonzerodivisor $a\in I$ on $M$.
Then we have a long exact sequence
\[
\cdots\rightarrow H^i_I(M)\xrightarrow a H^i_I(M)\rightarrow H^i_I(M/aM)
\rightarrow H^{i+1}_I(M)\rightarrow \cdots
\]
of the local cohomology.
By assumption, we have that $H^i_I(M/aM)=0$ for $i<n-1$.
By induction, $\depth_A(I,M/aM)\geq n-1$.
Hence $\depth_A(I,M)\geq n$.
\end{proof}

\paragraph
We define $\dim_A M=\dim M$ to be the dimension of the support $\supp_A M$.

\paragraph
Let $X$ be a scheme, $n\geq 0$, 
and $\M$ a quasi-coherent $\O_X$-module.
We say that $\M$ satisfies the $(S'_n)$ (resp.\ $(S_n)$) condition if 
for each $x\in X$, we have $\depth \M_x\geq \min(n,\dim \O_{X,x})$
(resp.\ $\depth \M_x\geq \min(n,\dim \M_x)$
(here we define $\depth 0=\infty>n$)).
Or equivalently, if $\depth\M_x<n$, then $\depth \M_x=\dim \O_{X,x}$ 
(resp.\ $\depth\M_x=\dim \M_x$).
If $\O_{X,x}$ is Noetherian and $\M_x$ is a finite module, then
$\M_x$ is called {\em maximal Cohen--Macaulay} 
(resp.\ {\em Cohen--Macaulay}) 
if $\depth \M_x=\dim \O_{X,x}$ (resp.\ $\depth\M_x=\dim \M_x$).
We say that $X$ satisfies the $(S_n)$ condition if $\O_X$ satisfies the
$(S_n)$ condition (or equivalently, $(S'_n)$ condition).

\begin{lemma}\label{canonical-localization.lem}
Let $B$ be a Noetherian local ring with the canonical module $K$.
Then the associated sheaf $\omega:=\tilde K$ is a semicanonical module on 
$Z=\Spec B$.
In other words, if 
$P$ is a prime ideal of $B$ and $K_P\neq 0$, 
then $K_P$ is the canonical module of $B_P$.
\end{lemma}

\begin{proof}
This is \cite[(4.3)]{Aoyama}.
\end{proof}

\begin{lemma}\label{omega-S_2.lem}
Let $Z$ be a locally Noetherian scheme with a semicanonical module $\omega_Z$.
Then $\omega_Z$ satisfies the $(S'_2)$ condition.
\end{lemma}

\begin{proof}
If $\depth\omega_{Z,z}<2$, then $\omega_{Z,z}\neq 0$, and hence
$\depth\omega_{Z,z}=\dim\omega_{Z,z}$ by
\cite[(1.10)]{Aoyama}.
On the other hand, if $\zeta$ is a minimal element of $\supp\omega_{Z,z}$,
then the dimension of the closure $\bar\zeta$ 
of $\zeta$ in $\Spec \O_{Z,z}$ 
agrees with
$\dim\O_{Z,z}$ by
\cite[(1.7)]{Aoyama}.
Hence $\dim\omega_{Z,z}=\dim\O_{Z,z}$, and $\omega_Z$ satisfies the 
$(S'_2)$ condition.
\end{proof}

\begin{corollary}\label{full-support-omega.cor}
Let $G$ be a flat $S$-group scheme of finite type, and $X$ a
Noetherian $G$-scheme with a $G$-dualizing complex.
If $X$ is locally equidimensional \(e.g., $X$ is $(S_2)$, see
\cite{Ogoma}\), 
then for a componentwise $G$-canonical module $\omega_X'$, we have that
$\supp\omega_X'=X$.
\end{corollary}

\begin{proof}
We may assume that $X$ is $G$-connected.
Let $\Bbb I$ be a $G$-dualizing complex of $X$, and let
$\omega_X'=H^s(\Bbb I)$ with $H^i(\Bbb I)=0$ for $i<s$.
Let $x\in \supp\omega_X'$.
As $\omega'_{X,x}$ satisfies the $(S_1')$-condition
by Lemma~\ref{omega-S_2.lem}, any generalization of $x$ is in $\supp\omega'_X$.
This shows that $\supp\omega_X'$ is $G$-stable closed open.
As $X$ is $G$-connected and $\omega_X'\neq 0$, $\supp\omega_X'=X$.
\end{proof}

\paragraph
Let $X$ be a scheme and $x\in X$.
Then the codimension of $x$ is that of $\{x\}$.
Namely, $\codim_X x=\codim_X\{x\}=\dim \O_{X,x}$.
We denote the set of points of $X$ of codimension $n$ by $X\an n$.
As can be seen easily using \cite[(10.24.4)]{SP}, 
any irreducible closed subset $Z$ of $X$ has a unique generic point $\zeta$,
and obviously we have $\codim_X Z=\codim_X\zeta$.
In particular, $X\an0$ is in one-to-one
correspondence with the set of irreducible components of $X$ by the
correspondence $\xi\mapsto \bar \xi$.

For $n\in\Bbb Z$, a subset of $U$ is said to be {\em $n$-large} if
$\codim_X(X\setminus U)\geq n+1$.
This is equivalent to say that 
$U\supset X\an0\cup\cdots\cup X\an{n}$.
$0$-large is also called {\em strongly dense}, 
and $1$-large is simply called {\em large}.
Note that a strongly dense subset is dense.

If $U\subset V\subset X$ and $U$ is $n$-large
in $X$, then $V$ is $n$-large in $X$.
If $U\subset V\subset X$, $V$ is open in $X$, $U$ is $n$-large
in $V$, and
$V$ is $n$-large in $X$, then $U$ is $n$-large in $X$.

\paragraph\label{LFI.par}
For a scheme $X$ and $n\geq 0$, let $P^n(X)$ be the set of integral closed
subsets of codimension $n$.
Note that $X\an n$ is in one-to-one correspondence with $P^n(X)$.
An element of $P^0(X)$ is nothing but an irreducible component of $X$.
A set $\Lambda$ of subsets of $X$ is said to be {\em locally finite} if for
any affine open subset $U$ of $X$, $U\cap F\neq \emptyset$ for only
finitely many elements $F$ of $\Lambda$.
We say that a scheme $X$ is an LFI-scheme if $P^0(X)$ is
locally finite.
This is equivalent to say that $\{\{\xi\}\mid \xi\in X\an0\}$ 
is locally finite.

A locally Noetherian scheme is LFI.
An open subset $U$ of an LFI-scheme $X$ is dense in $X$ if and only if
it is strongly dense in $X$.

\paragraph
Let $Z$ be a locally Noetherian scheme.
A coherent sheaf $\omega$ on $Z$ is said to be $n$-canonical at $z\in Z$
if the $\O_{Z,z}$-module $\omega_z$ satisfies the $(S'_n)$-condition,
and for each generalization $z'\in Z$ with 
$\codim z' <n$,
$\omega_{z'}$ is either zero, or is the canonical module of the
local ring $\O_{Z,z'}$.
We say that $\omega$ is $n$-canonical if it is $n$-canonical at each point.
Being $n$-canonical is a local condition.

\begin{lemma}\label{1-canonical-equiv.lem}
Let $Z$ be a locally Noetherian scheme with a $1$-canonical module $\omega$.
For $\M\in\Coh(Z)$, consider the following conditions.
\begin{enumerate}
\item[\bf 1] 
$\M$ satisfies the $(S'_1)$ condition, and 
$\supp\M\subset\supp\omega$.
\item[\bf 2] $\M$ satisfies the $(S_1)$ condition, and 
$(\supp\M)\an0\subset (\supp\omega)\an0$.
\item[\bf 3] The canonical map $\M\rightarrow\M^{\vee\vee}$ is monic,
where $(?)^{\vee}=\uHom_{\O_Z}(?,\omega)$.
\item[\bf 4] $\M$ is isomorphic to a submodule of a finite direct sum
of copies of $\omega$.
\end{enumerate}
Then we have 
{\bf 4$\Rightarrow$1$\Leftrightarrow$2$\Leftrightarrow$3}.
A coherent module of the form $\M=\N^\vee$ with $\N\in\Coh(Z)$ satisfies 
{\bf 3}.
If $Z=\Spec B$ is affine, then {\bf 3$\Rightarrow$4} holds.
\end{lemma}

\begin{proof}
As the conditions {\bf 1, 2, 3} are local, and the 
conditions $\M=\N^\vee$ and {\bf 4} localizes, 
we may assume that $Z=\Spec B$ is affine, and we are to prove that
the four conditions are equivalent, and $\N^\vee$ satisfies {\bf 4}
for $\N\in\Coh(Z)$.

Set
$M=\Gamma(Z,\M)$, $K=\Gamma(Z,\omega)$, and $(?)^\vee=\Hom_B(?,K)$.

{\bf 1$\Rightarrow$2}.
As $M$ satisfies the $(S'_1)$-condition, it satisfies the 
$(S_1)$-condition.
Let $P$ be a minimal prime of $M$.
As $\dim M_P=0$, we have that $\depth M_P=0$, and hence $\dim B_P=0$
by the $(S'_1)$-property.
As $K_P\neq 0$, $\dim K_P=0$.

{\bf 2$\Rightarrow$3}.
First, assume that $P\in\supp M$ and $\dim M_P=0$.
Then by assumption, $P\in \supp K$ and $\dim K_P=0$.
As $K_P$ satisfies $(S'_1)$ and $\depth K_P=0$, we have $\dim B_P=0$.

In particular, as $M$ satisfies $(S_1)$, it also satisfies $(S'_1)$.

Let $D:M\rightarrow M^{\vee\vee}$ be the canonical map, and set 
$F:=\supp\Ker D$.
If $\height P=0$, then we have that $K_P=0$ or $K_P$ is the canonical module
of $B_P$.
If $M_P\neq 0$, then 
as the $B_P$-module $M_P$ is maximal Cohen--Macaulay and 
$K_P\neq 0$, $D_P:M_P
\rightarrow M_P^{\vee\vee}$ is an isomorphism by \cite[(4.4)]{Aoyama}.
Hence $\codim_Z F\geq 1$.
If $\Ker D\neq 0$, then taking an associated prime $Q$ of $\Ker D$,
$\height Q\geq 1$ and $Q$ is an associated prime of $M$.
As $\depth M_Q=0$ and $M$ satisfies the $(S'_1)$ condition, 
$\dim B_Q=0$.
This contradicts $\height Q\geq 1$.
Hence $\Ker D=0$, as desired.

Now we prove that $\N^\vee$ satisfies {\bf 4}.
Set $N=\Gamma(Z,\N)$, and take a surjection $B^n\rightarrow N$.
Taking the dual, we have an injection $N^\vee\rightarrow K^n$.

We prove {\bf 3$\Rightarrow$4}.
Set $N:=M^\vee$.
Then there is an injection $M\rightarrow M^{\vee\vee}=N^\vee$ and
an injection $N^\vee\rightarrow K^n$.
So there is an injection $M\rightarrow K^n$.

{\bf 4$\Rightarrow$1}.
As $M\subset K^n$, we have that $\supp M\subset \supp K$.
If $\depth M_P=0$, then $P$ is an associated prime of $M$, and hence
$P$ is an associated prime of $K$, and $\depth K_P=0$.
As $K$ satisfies the $(S'_1)$-condition, $\height P=0$, and hence
$M$ satisfies the $(S'_1)$-condition.
\end{proof}

\begin{lemma}\label{Ischebeck.lem}
Let $B$ be a Noetherian local ring, $N$ a finite $B$-module which satisfies 
the $(S_n)$ condition.
If there is a minimal prime $P$ of $N$ such that $\dim B/P < n$, then
$\dim B/P=\depth N=\dim N$.
If, moreover, $N$ satisfies $(S'_n)$, then $\dim B/P =\dim B$.
\end{lemma}

\begin{proof}
As $\Hom_B(B/P,N)\neq 0$, we have $\depth N<n$ by \cite[(17.1)]{CRT}.
Hence $\depth N=\dim N$ by the $(S_n)$ property.
Hence $\dim B/P=\dim N$ by \cite[(17.3)]{CRT}.
Assume that $N$ satisfies $(S'_n)$.
Since $\depth N=\dim B/P<n$, we have $\dim B/P=\depth N=\dim B$.
\end{proof}

\begin{corollary}\label{Ischebeck2.cor}
Let $B$ be a Noetherian local ring,
and $M$ a finite $B$-module which satisfies $(S_n)$.
Let $N$ be a finite $B$-module which satisfies $(S'_n)$.
If a minimal prime of $M$ is a minimal prime of $N$,
then $M$ satisfies $(S'_n)$.
\end{corollary}

\begin{proof}
Let $P\in\Spec B$ with $\depth M_P<n$.
Then by $(S_n)$ property of $M$, we have $\dim M_P=\depth M_P$.
Let $Q$ be a minimal prime of $M$ such that $Q\subset P$ and
$\dim B_P/QB_P=\dim M_P$.
Then by assumption, $Q$ is a minimal prime of $N$.
As $\dim B_P/QB_P<n$ and $QB_P$ is a minimal prime of $N_P$, 
we have that $\depth M_P=\dim B_P/QB_P=\dim B_P$ by 
Lemma~\ref{Ischebeck.lem} and $(S'_n)$ property of $N$.
\end{proof}

\begin{corollary}
Let $B$ be a Noetherian local ring which satisfies $(S_n)$,
and $M$ a finite $B$-module which satisfies $(S_n)$ and $(S'_1)$.
Then $M$ satisfies $(S'_n)$.
\end{corollary}

\begin{proof}
Apply Corollary~\ref{Ischebeck2.cor} to the case that $N=B$.
\end{proof}

\begin{lemma}\label{2-canonical-double-dual.lem}
Let $Z$ be a locally Noetherian scheme with a $2$-canonical module $\omega$.
For $\M\in\Coh(Z)$, consider the following conditions.
\begin{enumerate}
\item[\bf 1] 
$\M$ satisfies the $(S'_2)$ condition, and 
$\supp\M\subset\supp\omega$.
\item[\bf 2] $\M$ satisfies the $(S_2)$ condition, and 
$(\supp\M)\an0\subset (\supp\omega)\an0$.
\item[\bf 3] The canonical map $\M\rightarrow\M^{\vee\vee}$ is isomorphic,
where $(?)^{\vee}=\uHom_{\O_Z}(?,\omega)$.
\item[\bf 4] There is an exact sequence of the form 
\[
0\rightarrow \M\rightarrow \K^0\rightarrow \K^1
\]
such that $\K^i$ is a finite direct sum of copies of $\omega$ for $i=0,1$.
\end{enumerate}
Then we have {\bf 4$\Rightarrow$1$\Leftrightarrow$2$\Leftrightarrow$3}.
A coherent module of the form $\M=\N^\vee$ with $\N\in\Coh(Z)$ 
satisfies {\bf 3}.
If $Z=\Spec B$ is affine, then {\bf 3$\Rightarrow$4} holds.
\end{lemma}

\begin{proof}
We may assume that $Z=\Spec B$ is affine, and we need to prove that 
the four conditions are equivalent, and $\N^\vee$ satisfies {\bf 4}.
Let $M$, $K$, and $(?)^\vee$ be as in the proof of 
Lemma~\ref{1-canonical-equiv.lem}.

{\bf 1$\Rightarrow$2}.
As $M$ satisfies $(S'_2)$, it satisfies $(S_2)$.
$(\supp\M)\an0\subset (\supp\omega)\an0$ is by 
Lemma~\ref{1-canonical-equiv.lem}, {\bf 1$\Rightarrow$2}.

{\bf 2$\Rightarrow$3}.
By Corollary~\ref{Ischebeck2.cor}, we have that $M$ satisfies $(S'_2)$.
Note that $\supp M\subset \supp K$ by 
Lemma~\ref{1-canonical-equiv.lem}, {\bf 2$\Rightarrow$1}.

Also by Lemma~\ref{1-canonical-equiv.lem}, {\bf 2$\Rightarrow$3}, 
we have that $D:M\rightarrow M^{\vee\vee} $ is monic.
Let $C:=\Coker D$.
Let $P\in\Spec B$ and $\depth M_P<2$.
Then $M_P$ is maximal Cohen--Macaulay by the $(S'_2)$ property of $M$.
We have $K_P\neq 0$ by $\supp M\subset \supp K$.
So $K_P$ is the canonical module of $B_P$, since $K$ is $2$-canonical.
Hence $C_P=0$ by \cite[(4.4)]{Aoyama}.

Now assume that $C\neq 0$, and let $Q$ be a minimal prime of $C$.
Then $C_Q\neq 0$, and hence $\depth M_Q\geq 2$ by the argument above.
Hence $\dim B_Q\geq 2$.
On the other hand, 
\[
0\rightarrow M_Q\rightarrow M_Q^{\vee\vee}\rightarrow C_Q\rightarrow 0
\]
is exact.
As $M_Q^{\vee\vee}$ satisfies $(S'_1)$ by
Lemma~\ref{1-canonical-equiv.lem}, $\depth M_Q^{\vee\vee}\geq 1$.
By the choice of $Q$, $\depth C_Q=0$.
By the depth lemma, $\depth M_Q=1$, and this contradicts $\depth M_Q\geq 2$.
Hence $C=0$, as desired.

Now we prove that $\N^\vee$ satisfies {\bf 4}.
Set $N=\Gamma(Z,\N)$, and take a presentation
\[
F_1\rightarrow F_0\rightarrow N\rightarrow 0
\]
with $F_0$ and $F_1$ finite free.
Dualizing, we get a desired exact sequence.

Now we prove {\bf 3$\Rightarrow$4}.
Set $N:=M^\vee$.
Then $M\cong M^{\vee\vee}\cong N^\vee$.
As $N^\vee$ satisfies {\bf 4}, $M$ satisfies {\bf 4}, too.

{\bf 4$\Rightarrow$1}.
Let $P\in\Spec B$ and assume that $\depth M_P<2$.
By the exact sequence, we must have that $\depth K_P=\depth M_P$.
As $K$ satisfies the $(S'_2)$-condition, $\dim B_P=\depth K_P=\depth M_P$,
and hence $M$ satisfies the $(S'_2)$-condition.
$\supp\M\subset\supp\omega$ is trivial.
\end{proof}

\paragraph
Let $Z$ be a locally Noetherian scheme.
We denote the full subcategory of $\Coh(Z)$ consisting of coherent sheaves
satisfying the $(S'_n)$ condition by $(S'_n)(Z)$.
It is an additive subcategory of $\Coh(Z)$ closed under direct summands,
extensions, and epikernels.

\begin{lemma}\label{S_1'-equiv.lem}
For $\M\in\Coh(Z)$, the following are equivalent.
\begin{enumerate}
\item[\bf 1] $\M\in (S'_1)(Z)$.
\item[\bf 2] For any dense open subset $i:U\rightarrow Z$ of $Z$, 
the canonical map $u:\M\rightarrow i_*i^*\M$ is a monomorphism.
\end{enumerate}
\end{lemma}

\begin{proof}
As the question is local, we may assume that $Z=\Spec B$ is affine.
Set $M=\Gamma(Z,\M)$.

{\bf 1$\Rightarrow$2}.
Let $I$ be an ideal of $B$ such that $U=Z\setminus V(I)$.
Then we have $\height I\geq 1$, and hence
we have that $\depth \M_z\geq 1$ for each $z\in V(I)$ 
by the $(S'_1)$ property.
Thus $\depth_B(I,M)\geq 1$, and
hence $H^0_I(M)=0$ by Lemma~\ref{grade-local-cohomology.lem}.
Hence $M\rightarrow \Gamma(U,\M)$ is injective.

{\bf 2$\Rightarrow$1}.
Assume that $P$ is an associated prime of $M$ with $\height P\geq 1$.
Then letting $U=D(P)=Z\setminus V(P)$, $U$ is dense.
However, as $P$ is an associated prime of $M$, the local cohomology 
$H^0_P(M)$, which is the kernel of $M\rightarrow \Gamma(Z,i_*i^*\M)$,
is nonzero, and this is a contradiction.
Hence if $P\in\Ass M$, then $\height P=0$.
Namely, $\M$ satisfies the $(S'_1)$ condition.
\end{proof}

\begin{lemma}\label{S_2.lem}
Let $Z$ be a locally Noetherian scheme.
For $\M\in\Coh(Z)$, the following are equivalent.
\begin{enumerate}
\item[\bf 1] $\M\in (S'_2)(Z)$.
\item[\bf 2] $\M\in (S'_1)(Z)$, and
for any large open subset $i:U\rightarrow Z$ of $Z$, 
the canonical map $u:\M\rightarrow i_*i^*\M$ is an isomorphism.
\end{enumerate}
\end{lemma}

\begin{proof}
As in the proof of Lemma~\ref{S_1'-equiv.lem}, we may assume that
$Z=\Spec B$ is affine.
Let $M=\Gamma(Z,\O_Z)$.

{\bf 1$\Rightarrow$2}.
Clearly, $\M$ satisfies the $(S'_1)$ condition.

Let $I$ be an ideal of $B$ such that $U=Z\setminus V(I)$.
As we have $\height I \geq 2$, $\depth(I,M)\geq 2$.
So the local cohomology $H^i_I(M)$ vanishes for $i=0,1$.
By the exact sequence
\[
H^0_I(M)\rightarrow M\rightarrow \Gamma(Z,i_*i^*\M)\rightarrow H^1_I(M),
\]
we are done.

{\bf 2$\Rightarrow$1}.
Let $P\in\Spec B$ satisfy $\depth M_P<2$.
If $\depth M_P=0$, then $\dim B_P=0$ by the $(S'_1)$ assumption.
If not, then $\depth M_P=1$.
We want to prove that $\dim B_P=1$.
Assume the contrary.
Then $\dim A_P\geq 2$.
Then letting $U=Z\setminus V(P)$, we have that $U$ is a large open 
subset of $Z$.
As we have an exact sequence
\[
0\rightarrow H^0_P(M)\rightarrow M \xrightarrow{\cong}\Gamma(U,i^*\M)
\rightarrow H^1_P(M)\rightarrow 0
\]
by \cite[(1.9)]{LC}, 
we have that $H^0_P(M)=H^1_P(M)=0$ by assumption.
Hence $H^0_{PB_P}(M_P)=H^1_{PB_P}(M_P)=0$.
Hence $\depth M_P\geq 2$, and this is a contradiction.

Hence $\depth M_P<2$ implies that $\depth M_P=\dim A_P$.
That is, $M$ satisfies the $(S'_2)$ condition.
\end{proof}

\paragraph
Let $Z$ be a scheme with a quasi-coherent sheaf $\M$.
We say that $\M$ is {\em full} if $\supp\M=Z$.

\begin{lemma}
If $\varphi:X\rightarrow Y$ is a flat morphism of schemes 
and $\M$ is a quasi-coherent sheaf on $Y$.
If $\M$ is full, then $\varphi^*\M$ is also full.
If $\varphi^*\M$ is a full and $\varphi$ is 
faithfully flat, then $\M$ is full.
\end{lemma}

\begin{proof}
Easy.
\end{proof}

\begin{lemma}\label{S_2-large-equiv.lem}
Let $Z$ be a locally Noetherian scheme with a full 
$2$-canonical module $\omega$, 
and $U$ its large open subset.
Let $i:U\hookrightarrow Z$ be the inclusion.
If $\M\in (S'_2)(U)$, then $i_*\M\in (S'_2)(Z)$.
Moreover,
$i_*:(S'_2)(U)\rightarrow (S'_2)(Z)$ and
$i^*:(S'_2)(Z)\rightarrow (S'_2)(U)$ are quasi-inverse each other.
\end{lemma}

\begin{proof}
If $i_*\M\in(S'_2)(Z)$ for $\M\in(S'_2)(U)$, then
$i_*:(S'_2)(U)\rightarrow (S'_2)(Z)$ and
$i^*:(S'_2)(Z)\rightarrow (S'_2)(U)$ are well-defined functors.
The counit map $i^*i_*\rightarrow \Id$ is obviously an isomorphism.
On the other hand, $\Id\rightarrow i_*i^*$ is an isomorphism by
Lemma~\ref{S_2.lem}.
So the last assertion follows.

So it suffices to show, 
If $i_*\M\in(S'_2)(Z)$ for $\M\in(S'_2)(U)$.
We can take a coherent subsheaf $\Q$ of $i_*\M$ such that $i^*\Q=i^*i_*\M=\M$
by \cite[Exercise~II.5.15]{Hartshorne}.
Set $\N:=\Q^{\vee\vee}$.
Then $\N\in (S_2')(Z)$ by Lemma~\ref{2-canonical-double-dual.lem}.
So by Lemma~\ref{S_2.lem},
$\N\rightarrow i_*i^*\N\cong i_*(i^*\Q)^{\vee\vee}\cong i_*\M^{\vee\vee}
\cong i_*\M$ is an isomorphism, and hence $i_*\M$ is coherent and is
in $(S_2')(Z)$.
\end{proof}

\begin{example}
Let $Z$ be a locally Noetherian scheme.
In the following cases, $Z$ has a full $2$-canonical module $\omega$.
\begin{enumerate}
\item[\bf 1] $Z$ is normal.
Any rank-one reflexive module (e.g., $\O_Z$) can be used as $\omega$.
\item[\bf 2] $Z$ is locally equidimensional
  Noetherian with a dualizing complex.
The componentwise canonical module is the desired one.
\item[\bf 3] $Z=\Spec B$ with $B$ an equidimensional Noetherian local ring
with a canonical module in the local sense.
The canonical module is the desired one.
\end{enumerate}
\end{example}

\paragraph\label{quasinormal.par}
A locally Noetherian scheme $Z$ is said to be {\em quasi-normal} by $\omega$ if
$\omega$ is a full $2$-canonical module of $Z$.
It is simply called quasi-normal, if it is quasi-normal by some $\omega$.
We say that $Z$ satisfies $(T_n)$ (resp.\ $(R_n)$) if $\O_{Z,z}$ is 
Gorenstein (resp.\ regular) for $z\in Z$ with $\codim z\leq n$,
or equivalently, the Gorenstein (resp.\ regular) locus of
$Z$ is $n$-large.

\begin{lemma}
For a locally Noetherian scheme $Z$, the following are equivalent.
\begin{enumerate}
\item[\bf 1] $Z$ satisfies $(T_1)+(S_2)$.
\item[\bf 2] $Z$ is quasi-normal by $\O_Z$.
\end{enumerate}
In particular, if $Z$ is normal \(or equivalently, satisfies 
$(R_1)+(S_2)$\), then it is quasi-normal by $\O_Z$.
\end{lemma}

\begin{proof}
Easy.
\end{proof}

\begin{lemma}\label{ogoma.lem}
  Let $Z$ be a locally equidimensional
  connected Noetherian scheme with a nonzero 
semicanonical module $\omega$.
Then $\supp\omega=Z$, and we have that $Z$ is universally catenary and 
quasi-normal by $\omega$.
\end{lemma}

\begin{proof}
Let $X_1,\ldots,X_r$ be the irreducible components of $Z$ that are
contained in $\supp\omega$.
Assume that $\supp\omega\neq Z$.
Let $Y_1,\ldots,Y_s$ be the irreducible components of $Z$ that are 
not contained in $\supp\omega$.
By assumption, $r\geq 1$ and $s\geq 1$.
As $Z$ is connected, there exist some $i$ and $j$ such that $X_i\cap Y_j
\neq\emptyset$.
Take $z\in (X_i\cap Y_j)\an0$, and consider the local ring 
$B=\O_{Z,z}$.
As $z\in\supp\omega$, $B$ has a canonical module $K=\omega_z$.
By assumption, $B$ is equidimensional.
By \cite[(1.7)]{Aoyama}, $\supp K=\Spec B$.
Hence $Y_j\subset\supp\omega$, and this is a contradiction.
Hence $\supp\omega=Z$, as desired.

Let $z\in Z$ be any point, and set $B=\O_{Z,z}$.
As $\supp\omega_z=\Spec B$, we have that $\supp\hat\omega_z=\Spec \hat B$,
where $\hat B$ is the completion of $B$.
As $\hat \omega_z$ is a canonical module of $\hat B$, we have that
$\hat B$ is equidimensional by \cite[(1.7)]{Aoyama}.
Hence $B$ is universally catenary by \cite[(31.6)]{CRT}.
Hence $Z$ is universally catenary.
\end{proof}

\begin{corollary}\label{S_2-dualizing-quasinormal.cor}
Let $Z$ be a Noetherian scheme with a dualizing complex.
If $Z$ is locally equidimensional, then the componentwise
canonical module $\omega'$ is a full semicanonical module, 
and $Z$ is quasi-normal by $\omega'$.
\end{corollary}

\begin{proof}
The componentwise canonical module $\omega'$ is full by Lemma~\ref{ogoma.lem}.
\end{proof}

\begin{lemma}\label{S_n-ascent-descent.lem}
Let $\varphi:X\rightarrow Y$ be a flat morphism between locally Noetherian
schemes.
Let $\M$ be a coherent sheaf on $Y$, and $n\geq 0$.
Then
\begin{enumerate}
\item[\bf 1] If $\varphi$ is faithfully flat, 
and $\varphi^*\M$ satisfies the $(S'_n)$ condition
\(resp.\ the $(S_n)$ condition\), then so does $\M$.
\item[\bf 2] If the fibers of $\varphi$ satisfy $(S_n)$ and $\M$ satisfies
$(S'_n)$ \(resp.\ $(S_n)$\), then so does $\varphi^*\M$.
\end{enumerate}
\end{lemma}

\begin{proof}
For the property $(S_n)$, see \cite[(6.4.1)]{EGA-IV-2}.
The assertions for $(S'_n)$ is also proved similarly.
\end{proof}

\begin{lemma}\label{n-canonical-flat.lem}
Let $\varphi:X\rightarrow Y$ be a flat morphism between locally Noetherian
schemes.
Let $\omega$ be a coherent sheaf on $Y$.
\begin{enumerate}
\item[\bf 1] If $\varphi$ is faithfully flat and $\varphi^*\omega$ 
is semicanonical \(resp.\ $n$-canonical\),
then so is $\omega$.
\item[\bf 2] If $X$ is quasi-normal by $\varphi^*\omega$, then
$Y$ is quasi-normal by $\omega$.
\item[\bf 3] If 
$\omega$ is semicanonical and each fiber of $\varphi$ is Gorenstein, 
then $\varphi^*\omega$ is semicanonical.
\item[\bf 4] 
If $\omega$ is $n$-canonical and each fiber of $\varphi$ 
satisfies $(T_{n-1})+(S_n)$, then $\varphi^*\omega$ is $n$-canonical.
\item[\bf 5] If $Y$ is quasi-normal by $\omega$ and each fiber of $\varphi$
satisfies $(T_1)+(S_2)$, then $X$ is quasi-normal by $\varphi^*\omega$.
\end{enumerate}
\end{lemma}

\begin{proof}
{\bf 1}.
The assertion for the semicanonical property follows from 
\cite[(4.2)]{Aoyama}.
As the $(S_n)$ property also descends
by Lemma~\ref{S_n-ascent-descent.lem}, the assertion for the
$n$-canonical property follows easily.

{\bf 2} follows immediately by {\bf 1}.

{\bf 3}.
Let $A\rightarrow B$ a flat Gorenstein local homomorphism between Noetherian
local rings, and $K$ the canonical module of $A$.
It suffices to prove $B\otimes_A K$ is the canonical module.
Let $Q=\fm_B(\hat A\otimes_A B)$, where $\fm_B$ is the maximal ideal of $B$.
Note that $Q$ is a maximal ideal of $\hat A\otimes_A B$.
Set $C:=(\hat A\otimes_A B)_Q$.
Since $\hat K=\hat A\otimes_A K$ is the lowest nonvanishing cohomology group
of the dualizing complex $\Bbb I$ of $\hat A$,
we have that $C\otimes_A K$ is the lowest nonvanishing
cohomology group of $C\otimes_{\hat A}\Bbb I$.
As $\hat A\rightarrow C$ is a flat Gorenstein local homomorphism,
$C\otimes_{\hat A}\Bbb I$ is a dualizing complex by \cite[(5.1)]{AF}.
Hence $C\otimes_A K$ is the canonical module of $C$.
By \cite[(4.2)]{Aoyama}, $B\otimes_A K$ is the canonical module of $B$.

{\bf 4} and {\bf 5} are immediate consequences of {\bf 3}.
\end{proof}

\paragraph\label{initial-canonical-settings.par}
Let $f:G\rightarrow H$ be a quasi-compact 
flat homomorphism between flat $S$-group schemes
of finite type with $N=\Ker f$.
Note that $N$ is also flat of finite type.

\begin{lemma}\label{flat-base-change.thm}
Let $g:Z'\rightarrow Z$ be a $G$-morphism separated of finite type.
Assume that $Z$ is Noetherian.
Then the flat base change map
\[
\bar\zeta:\res^{H}_{G}g^!\rightarrow g^!\res^{H}_{G}
\]
\(see {\rm\cite[Chapter~21]{ETI}}\) is an isomorphism between the functors
$D^+_{\Lqc}(H,Z)\rightarrow D^+_{\Lqc}(G,Z')$
\(it would be better to write $L\res^H_G$ instead of $\res^H_G$, but
as in {\rm\cite{ETI}}, for a left or right derived functor of an exact functor,
we omit $L$ or $R$\).
\end{lemma}

\begin{proof}
This is \cite[(21.8)]{ETI}.
\end{proof}

\paragraph\label{canonical-Y_0.par}
Let $f:G\rightarrow H$, and $N$ be as in 
(\ref{initial-canonical-settings.par}).
Let $Y_0$ be a fixed Noetherian $H$-scheme with a 
fixed $H$-dualizing complex $\Bbb I_{Y_0}=\Bbb I_{Y_0}(H)$.
The restriction $\res^H_G\Bbb I_{Y_0}(H)$ is a $G$-dualizing complex
by \cite[(31.17)]{ETI}.
We denote it by $\Bbb I_{Y_0}$ or $\Bbb I_{Y_0}(G)$.

Let $\Cal F(G,Y_0)$ be the category of $(G,Y_0)$-schemes
separated of finite type over $Y_0$.
For $(h_Z:Z\rightarrow Y_0)\in\Cal F(G,Y_0)$, {\em the} $G$-dualizing 
complex of $Z$ (or better, of $h_Z$) is $h_Z^! \Bbb I_{Y_0}(G)$ by definition,
and we denote it by $\Bbb I_Z=\Bbb I_Z(G)$.

\begin{lemma}
Let $h_Z:Z\rightarrow Y_0$ be an object of $\Cal F(G,Y_0)$.
Assume that the action of $N$ on $Z$ is trivial.
Then $Z\in\Cal F(H,Y_0)$.
When we set $\Bbb I_Z(H):=h_Z^!\Bbb I_{Y_0}(H)$, then we have
$\Bbb I_Z(G)=\res^H_G\Bbb I_Z(H)$.
In particular, 
each cohomology group of $\Bbb I_Z(G)$ belongs to $\Coh_N(G,Z)$, the
full subcategory of $\Qch(G,Z)$ consisting of $N$-trivial coherent
$(G,\O_Z)$-modules.
\end{lemma}

\begin{proof}
The first assertion is by \cite[(6.5)]{Hashimoto5}.
We have
\[
\Bbb I_Z(G)=h_Z^!\Bbb I_{Y_0}(G)=h_Z^!\res^H_G\Bbb I_{Y_0}(H)
\cong \res^H_G h_Z^! \Bbb I_{Y_0}(H)
\cong \res^H_G \Bbb I_Z(H)
\]
by Lemma~\ref{flat-base-change.thm}, and the second assertion holds.
In particular, being restricted from $H$, each cohomology group
of $\Bbb I_Z(G)$ is $N$-trivial.
\end{proof}

\begin{example}
Let $S$ be Noetherian with a fixed dualizing complex $\Bbb I_S$.
Then
$\Bbb I_S(H)=(L_{-1}\Bbb I_S)_{\Delta_M}$ is an $H$-dualizing complex of the
$H$-scheme $S$ by \cite[(31.17)]{ETI}, 
where $L_{-1}:\Mod(S)\rightarrow \Mod(\tilde B_H^M(S))$ is the
left induction, and $(?)_{\Delta_M}: \Mod(\tilde B_H^M(S))\rightarrow
\Mod(H,S)=\Mod(B_H^M(S))$ is the restriction.
See for the notation, \cite{ETI}.
Letting $Y_0=S$, we are in the situation of (\ref{canonical-Y_0.par}).
\end{example}

\paragraph\label{canonical-settings.par}
Let $S$, $f:G\rightarrow H$ and $N$ 
be as in {\rm(\ref{initial-canonical-settings.par})}.
$Y_0$, $\Bbb I_{Y_0}$, and $\Cal F(G,Y_0)$ be as in
{\rm(\ref{canonical-Y_0.par})}.

Let $Z\in\Cal F(G,Y_0)$.
{\em The} $G$-canonical module of 
$Z$, denoted by $\omega_Z$, is defined to be the $G$-canonical module
corresponding to the $G$-dualizing complex $\Bbb I_Z$.
Note that the definition in 
\cite[(31.13)]{ETI} is slightly different, and used the
componentwise $G$-canonical module, see (\ref{G-canonical.par}).

\begin{lemma}\label{codim-two-canonical.thm}
Let $f:G\rightarrow H$, $N$, $Y_0$, $\Bbb I_{Y_0}$, 
and $Z\in\Cal F(G,Y_0)$ be as in {\rm(\ref{canonical-settings.par})}.
Let $U$ be a $G$-stable open subset of $Z$.
If $\omega_Z|_U\neq 0$ \(e.g., $U$ is dense\),
then $\omega_Z|_U\cong \omega_U$ as $(G,\O_U)$-modules.
If, moreover, $U$ is large in $Z$, then
$\omega_Z\cong i_*\omega_U$, where $i:U\hookrightarrow Z$ is the
inclusion.
\end{lemma}

\begin{proof}
Assume that $\omega_Z=H^s(\Bbb I_Z)$.
If $\omega_Z|_U=H^s(\Bbb I_U)\neq 0$, then as $H^i(\Bbb I_U)=0$ for $i<s$, 
we have that $\omega_Z|_U\cong \omega_U$ as $(G,\O_U)$-modules.

Assume that $U$ is large in $Z$.
Then since $\omega_Z$ satisfies $(S'_2)$, we have that
$\omega_Z\rightarrow i_*i^*\omega_Z\cong i_*\omega_U$ is an isomorphism
by Lemma~\ref{S_2.lem}.
\end{proof}

\paragraph 
Let $f:G\rightarrow H$ be as in (\ref{initial-canonical-settings.par}).
Assume that $N$ is smooth over $S$.
Let $\Cal I$ be the defining ideal of the unit element $e$ in $N$.
Then $\Cal I/\Cal I^2$ is a locally free sheaf over $S$ on which $G$ acts
via the conjugation.
We set $\Lie N:=(\Cal I/\Cal I^2)^*$, and $\Theta_N:=\extop \Lie N$.
The following is essentially due to Knop \cite[Lemma~5]{Knop}.

\begin{proposition}\label{canonical-princ.thm}
Assume that $N$ is smooth over $S$.
If $\varphi:X\rightarrow Y$ is a $G$-enriched principal $N$-bundle with 
$Y$ locally Noetherian,
then $\Omega_{X/Y}$ is isomorphic to $h^*_X((\Lie N)^*)$, where $h_X:X
\rightarrow S$ is the structure map.
\end{proposition}

\begin{proof}
Consider the commutative diagram
\begin{equation}\label{knop.eq}
\xymatrix{
X \ar[r]^-{e\times 1_X} \ar[d]^{h_X} &
N\times X \ar[d]^{p_1} \ar[dr]^{p_2} \ar[r]^-\Psi &
X\times_YX \ar[d]^{p_2} \ar[r]^-{p_1} \ar@{}[dr]|{\text{(b)}} &
X \ar[d]^{\varphi} \\
S \ar[r]^e & 
N \ar[dr] \ar@{}[r]|{\text{(a)}} &
X \ar[d] \ar[r]^\varphi \ar[d]^{h_X} &
Y \\
 & & S
}
\end{equation}
of $G$-schemes, where $G$ acts on $N$ by conjugation action,
and $\Psi(n,x)=(nx,x)$.
Then as $\Psi$ is an isomorphism and (a) and (b) are fiber squares,
\begin{multline*}
\Omega_{X/Y}\cong (e\times 1_X)^*\Psi^* p_1^*\Omega_{X/Y}
\cong (e\times 1_X)^* \Psi^*\Omega_{X\times_Y X/X}
\cong \\
(e\times 1_X)^* \Omega_{N\times X/X}
\cong (e\times 1_X)^* p_1^* \Omega_{N/S}
\cong h_X^* e^*\Omega_{N/X}
\cong h_X^*((\Lie N)^*),
\end{multline*}
as desired.
\end{proof}

So $\omega_{X/Y}:=\extop\Omega_{X/Y}=h^*_X(\Theta_N^*)=\Theta_{N,X}^*$,
where $\Theta_{N,X}:=h_X^*\Theta_N$.
We prove a version of this fact which can be used also for the
case where $N$ may not be smooth.

\paragraph
We say that a morphism of schemes $\varphi:X\rightarrow Y$ is of
relative dimension $d$ if $\dim_x\varphi=d$ for each $x\in X$,
see \cite[(17.10.1)]{EGA-IV-4} for the notation.

\paragraph
Let $f:A\rightarrow B$ be a local homomorphism between Noetherian local rings.
Let $\hat f:\hat A\rightarrow \hat B$ be its completion, and
\begin{equation}\label{Cohen-factorization.eq}
\hat A\xrightarrow g C \xrightarrow h \hat B
\end{equation}
a Cohen factorization \cite{AFH} of it.
That is, $g$ is flat with $C/\fm_A C$ regular, and $h$ is surjective.
If so, we define the Avramov--Foxby--Herzog dimension
(AFH dimension for short) of $f$ by
\[
\adim f = \adim_AB=\dim C-\dim A -\height\Ker h
\]
and the depth of $f$ by
$\depth f = \depth B - \depth A$, see \cite{AFH} and \cite{AF2}.
If $f$ is flat, then $\adim f$ is nothing but the dimension of the closed fiber.
We define $\cmd f = \adim f - \depth f$, and call it the Cohen--Macaulay
defect of $f$.
We say that $f$ is Cohen--Macaulay at $\fm_B$ if $\fdim f < \infty$
and $\cmd f=0$.
This is equivalent to say that $\Ker h$ is a perfect ideal.
If $\Ker h$ is a Gorenstein ideal (that is, $\Ker h$ is perfect and
$\Ext^c_C(B,C)\cong B$,
where $c=\height\Ker h$, we say that $f$ is Gorenstein at $\fm_B$.
These definitions are independent of the choice of Cohen factorization
(\ref{Cohen-factorization.eq}) of $\hat f$.

\paragraph
A morphism $\varphi:X\rightarrow Y$ between locally Noetherian schemes
is said to be
Cohen--Macaulay (resp.\ Gorenstein) if $\O_{Y,y}\rightarrow \O_{X,x}$ is
Cohen--Macaulay (resp.\ Gorenstein) at $\fm_x$ for every $x\in X$.
$\adim_x \varphi$ is $\adim_{\O_{Y,y}}\O_{X,x}$.
If $\adim_x \varphi=d$ is independent of $x\in X$, then we say that
$\varphi$ has AFH dimension $d$.

\begin{lemma}
Let $\varphi:X\rightarrow Y$ be a Cohen--Macaulay 
separated morphism of finite type between Noetherian schemes.
Then $\varphi$ has a well-defined AFH dimension on each connected component of $X$.

If $\varphi$ 
is of AFH dimension $d$, then 
$H^i(\varphi^!(\O_Y))=0$ for $i\neq -d$.
If, moreover, $\varphi$ is Gorenstein, then
$\omega_{X/Y}:=H^{-d}(\varphi^!(\O_Y))$ is an invertible sheaf.
If, moreover, $G$ is a flat $S$-group scheme of finite type and $\varphi$
is a $G$-morphism, then $\omega_{X/Y}$ is a $G$-linearized invertible sheaf.
If, moreover, $\varphi$ is smooth, then $\omega_{X/Y}=\ext^d\Omega_{X/Y}$.
\end{lemma}

\begin{proof}
By the flat base change \cite[(4.4.3)]{Lipman}, the question is local both
on $X$ and $Y$, and we may assume that $Y=\Spec A$ and $X=\Spec B$ 
are both affine.
As $B$ is finitely generated, we may write $B=C/I$, where 
$C=A[x_1,\ldots,x_n]$ is a polynomial ring, and $I$ an ideal of $C$.
By assumption, $I$ is a perfect ideal of codimension $h:=n-d$.
We have
\[
j_*\varphi^!=R\uHom_{\O_Z}(j_*\O_X,\psi^*(?))[n],
\]
where $\psi:Z=\Spec C\rightarrow Y$ is the canonical map, 
and $j:X\rightarrow Z$ is the inclusion.
As we have $\Ext^i_C(B,C)=0$ for $i\neq h$, the first assertion follows.
If, moreover, $\varphi$ is Gorenstein, $\Ext^h_C(B,C)$ is rank-one 
projective as a $B$-module, and the second assertion follows.
The third assertion is trivial.
The last assertion follows from \cite[(28.11)]{ETI}.
\end{proof}

\begin{definition}\label{relative-canonical.def}
  Let $G$ be a flat $S$-group scheme of finite type, and
  $\varphi:X\rightarrow Y$ be a $G$-morphism separated of
  finite type between Noetherian   $G$-schemes.
  We denote the lowest non-vanishing cohomology group $H^s(\varphi^!\O_Y)\neq 0$
  ($H^i(\varphi^!\O_Y)=0$ for $i<s$) of
  $\varphi^!\O_Y$ by $\omega_{X/Y}$ or $\omega_\varphi$
  if $X\neq\emptyset$ (if $X=\emptyset$, we define $\omega_{X/Y}=0$),
  and call $\omega_{X/Y}$ the relative canonical
  sheaf of $\varphi$ (or of $X/Y$).
\end{definition}

\begin{lemma}\label{omega-base-change.lem}
  Let $G$ and $\varphi:X\rightarrow Y$ be as in
  Definition~\ref{relative-canonical.def}.
  Assume that $\varphi$ is flat Gorenstein of relative dimension $d$.
  Then for any morphism $h:Y'\rightarrow Y$ with $Y'$ Noetherian,
  we have that $\omega_{X'/Y'}\cong h_X^*\omega_{X/Y}$, where
  $X'=Y'\times_Y X$ and $h_X:X'\rightarrow X$ is the second projection.
\end{lemma}

\begin{proof}
Consider the diagram
\begin{equation}\label{delta-phi.eq}
\xymatrix{
  &  &  \\
  X \ar[dr]^{1_X} \ar[r]^-\Delta \ar `u [urr] `[rr]^\id [rr] &
X\times_YX \ar[d]^{p_2} \ar[r]^-{p_1} & X \ar[d]^{\varphi} \\
& X \ar[r]^\varphi & Y
}.
\end{equation}
Then by the flat base change, 
\[
\O_X\cong \Delta^!p_2^!\varphi^*\O_Y\cong \Delta^!p_1^*\varphi^!
\O_Y\cong \Delta^!p_1^*\omega_{X/Y}[d].
\]
As $\Delta$ is a closed immersion, the description of $\Delta^!$ in
\cite[Chapter~27]{ETI} yields that
there is an isomorphism
\[
\O_X\cong \Delta^!p_1^*\omega_{X/Y}[d]\cong p_1^*\omega_{X/Y}\otimes
\Delta^!\O_{X\times_Y
  X}[d].
\]
Thus $\Delta^!\O_{X\times_Y X}\cong \omega_{X/Y}^{-1}[-d]$.
Note that all the maps in (\ref{delta-phi.eq}) are tor-independent to $h$
and its base change.

The argument above applied to $\varphi':X'\rightarrow Y'$ yields
$(\Delta')^!\O_{X'\times_{Y'}X'}\cong \omega_{X'/Y'}^{-1}[-d]$.
So it suffices to show that the canonical map
$\bar\zeta:Lh_X^*\Delta^!\O_{X\times_Y X}\rightarrow
(\Delta')^!Lh_{X\times_Y X}^*$ is an isomorphism.
To verify this, we may forget the $G$-action,
and we may assume that $G$ is trivial.
Then, as the question is local on $X$, $Y$, and $Y'$, we may assume that
$X=\Spec B$, $Y=\Spec A$, $Y'=\Spec A'$ are all affine.

Then by definition, $\bar\zeta=\zeta$ is identified with the map
\[
H: R\Hom_C(B,C)\otimes_A^L A'\rightarrow R\Hom_{C'}(B',C'),
\]
where $C=B\otimes_A B$, and $B$ is viewed as a $C$-algebra via the product map,
and $B'=A'\otimes_A B$ and $C'=A'\otimes_A C$.

To compute the map $H$, let $\Bbb F$ be a $C$-free resolution of
$B$ whose terms are finite free.
Note that
\[
H^i(\Bbb F^*\otimes_A A')\cong H^i(\Hom_{C'}(A'\otimes_A \Bbb F,C'))\cong
\begin{cases}
  \omega_{X'/Y'} & (i=d) \\
  0 & (otherwise)
\end{cases}.
\]
Letting
\[
\Bbb F^*: 0\rightarrow G^0\rightarrow
\cdots \rightarrow G^i \xrightarrow {\partial^i} G^{i+1}\rightarrow\cdots
\]
and $B^i(\Bbb F^*)=\Image \partial^{i-1}$, we have that
$H^i(\Bbb F^*\otimes A/J)=0$ for any ideal $J$ of $A$ and any $i>d$.
This shows that $\Tor_1^A(B^{d+3},A/J)=0$ for any $J$, and hence $B^{d+3}$ is
$A$-flat, and $B^{d+3}\otimes_A A'\rightarrow B^{d+3}(\Bbb F^*\otimes_A A')$ is
an isomorphism.
So it is easy to see that $Z^d=\Ker \partial^d$ is also $A$-flat and
$Z^d\otimes_A A'\rightarrow Z^d(\Bbb F^*\otimes_A A')$ is an isomorphism.
So considering the exact flat complex bounded above
\[
0\rightarrow G^0\rightarrow\cdots\rightarrow G^{d-1}\rightarrow Z^{d}
\rightarrow Z^d/B^d\rightarrow 0
\]
is compatible with the base change.
We have that the induced map $\omega_{X/Y}\otimes_A A'\rightarrow
\omega_{X'/Y'}$ is an isomorphism, as desired.
\end{proof}

\begin{lemma}
Let $G$ be a flat $S$-group scheme of finite type, and 
$\varphi:X\rightarrow Y$ and $\psi:Y\rightarrow Z$ be 
flat Gorenstein $G$-morphisms separated of finite type between 
Noetherian $G$-schemes with well-defined relative dimensions.
Then $\omega_{X/Z}\cong \varphi^*\omega_{Y/Z}\otimes_{\O_X} \omega_{X/Y}$.
\end{lemma}

\begin{proof}
Let $d$ and $d'$ be the relative dimensions of $\varphi$ and $\psi$,
respectively.
We have
\begin{multline*}
\omega_{X/Z}=H^{-d-d'}((\psi\varphi)^!(\O_Z))
\cong H^{-d-d'}(\varphi^!(\omega_{Y/Z}[d']))\\
\cong H^{-d}(\varphi^*\omega_{Y/Z}\otimes_{\O_X}^L \varphi^!(\O_Y))
\cong \varphi^*\omega_{Y/Z}\otimes_{\O_X}\omega_{X/Y}.
\end{multline*}
\end{proof}

\paragraph
Let $f:G\rightarrow H$, $N$, $Y_0$, $\Bbb I_{Y_0}$, 
and $Z\in\Cal F(G,Y_0)$ be as in {\rm(\ref{canonical-settings.par})}.
Assume that $N$ is separated and has a fixed relative dimension.
We define $\Theta=\Theta_{N,Z}:=e_{N_{Z}}^*\omega_{N_{Z}/Z}^*$,
where $N_{Z}=N\times_S Z$, and $e_{N_{Z}}:Z\rightarrow N_{Z}$ is the
unit element.
Letting $h_Z:Z\rightarrow Y_0$ is the structure map,
$h_Z^*\Theta_{N,Y_0}\cong \Theta_{N,Z}$
by Lemma~\ref{omega-base-change.lem}.
If $S$ is Noetherian, then
letting $\Theta_{N,S}=e_N^*\omega_{N/S}^*$, we have that
$\Theta_{N,Z}=\bar h^*_{Z}\Theta_{N,S}$, where $\bar h_Z:Z\rightarrow S$
is the structure map.

Note that $\Theta$ is a $G$-linearized invertible sheaf on $Z$.
If $N$ is smooth of relative dimension $d$, then 
$\Theta_{N,Z}\cong \bar h_Z^*(\ext^d\Lie N)$, where $\Lie N=\Omega_{N/S}^*$ and
$\bar h_Z:Z\rightarrow S$ is the structure map.

\begin{proposition}
Let $\varphi:X\rightarrow Y$ be a morphism in $\Cal F(G,Y_0)$.
Assume that $N$ is separated and has a relative dimension $d$.
If $\varphi$ is a $G$-enriched principal $N$-bundle, then 
$\omega_{X/Y}\cong \Theta_{N,X}^*$.
\end{proposition}

\begin{proof}
Let us consider the commutative diagram 
\[
\xymatrix{
X \ar[r]^-{e\times 1_X} &
N\times X \ar[dr]^{p_2} \ar[r]^-\Psi &
X\times_YX \ar[d]^{p_2} \ar[r]^-{p_1} \ar@{}[dr]|{\text{(a)}} &
X \ar[d]^{\varphi} \\
& &
X \ar[r]^\varphi & Y
}
\]
in $\Cal F(G,Y_0)$.

Note that (a) is cartesian, and $\varphi$ is flat Gorenstein of
relative dimension $d$.
Now
\begin{multline*}
\omega_{X/Y}[d]\cong \varphi^!\O_Y\cong 
L(e\times 1_X)^*\Psi^*p_1^*\varphi^!\O_Y
\cong
L(e\times 1_X)^*\Psi^*p_2^!\varphi^*\O_Y\\
\cong
L(e\times 1_X)^*p_2^!\O_X
\cong
L(e\times 1_X)^*\omega_{(N\times X)/X}[d]
\cong
\Theta_{N,X}^*[d],
\end{multline*}
and the result follows.
\end{proof}

The following is due to Knop \cite{Knop} when $S=\Spec k$ with $k$ an
algebraically closed field of characteristic zero.

\begin{corollary}\label{canonical-principal.thm}
Let $f:G\rightarrow H$, $N$, $Y_0$, and $\Bbb I_{Y_0}$ 
be as in {\rm(\ref{canonical-settings.par})}.
Let $\varphi:X\rightarrow Y$ be a $G$-enriched principal $N$-bundle
which is a morphism in $\Cal F(G,Y_0)$.
If $N$ is separated and has a fixed relative dimension,
then $\omega_X \cong \varphi^*\omega_Y\otimes_{\O_X}\Theta_{N,X}^*$ as
$(G,\O_X)$-modules, and
$\omega_Y\cong (\varphi_*\omega_X\otimes_{\Cal O_Y}\Theta_{N,Y})^N$ as
$(H,\O_Y)$-modules.
\end{corollary}

\begin{proof}
The first assertion follows immediately from \cite[(28.11)]{ETI} and 
Proposition~\ref{canonical-princ.thm}.
The second assertion follows from the first one and 
\cite[(6.21)]{Hashimoto5}, using the equivariant projection formula 
\cite[(26.4)]{ETI}.
\end{proof}

\begin{lemma}\label{Reynolds-Theta.lem}
Let $f:G\rightarrow H$, $N$, and $Y_0$
be as in {\rm(\ref{canonical-settings.par})}.
If $N$ is finite and Reynolds, then $\Theta_{N,Y_0}\cong\O_{Y_0}$.
In particular, for any object $Y$ of $\Cal F(G,Y_0)$, $\Theta_{N,Y}\cong
\O_Y$.
\end{lemma}

\begin{proof}
  Note that $G$ acts on $N$ by conjugation, and hence
  we can define the semidirect product
$\tilde G:=G\ltimes N$.
  Letting $G$ act on $Y_0$ by the original action and $N$ act on $Y_0$
  as a subgroup of $G$, $\tilde G$ acts on $Y_0$.
The group $\tilde G$ acts on $N$ by $(g,n)n'=gnn'g^{-1}$.
The first projection $p_1: N_{Y_0}=N\times Y_0\rightarrow Y_0$ is 
a $\tilde G$-enriched principal $N$-bundle.
As $(G,\O_{Y_0})$-modules,
\begin{multline*}
\Theta_{N,Y_0}^*=e_{N_{Y_0}}^*\omega_{N_{Y_0}/Y_0}
\cong
e_{N_{Y_0}}^* p_1^* (?)^N \omega_{N_{Y_0}/Y_0}
\cong \omega_{N_{Y_0}/Y_0}^N
\\
\cong R(?)^N R\uHom_{\O_{Y_0}}(\O_{N_{Y_0}},\O_{Y_0})
\cong  \uHom_{\O_{Y_0}}(\O_{N_{Y_0}},\O_{Y_0})^N
\\
\cong \uHom_{\O_{Y_0}}(\O_{N_{Y_0}}^N,\O_{Y_0})
\cong \uHom_{\O_{Y_0}}(\O_{Y_0},\O_{Y_0})\cong \O_{Y_0}.
\end{multline*}
Hence $\Theta_{N,Y_0}\cong \O_{Y_0}$.
The last assertion is trivial.
\end{proof}

\section{Frobenius twists and Frobenius kernels}\label{frob-twist.sec}

\paragraph
In this section, $S$ is an $\Bbb F_p$-scheme, where $p$ is a prime number, and
$\Bbb F_p$ is the prime field of characteristic $p$,
unless otherwise specified.

\paragraph
Let us consider the ordered set $\Bbb Z$ as a category.
For $m,n\in\Bbb Z$, there is a unique morphism from $m$ to $n$ when
$m\leq n$.
Otherwise, $\Bbb Z(m,n)$ is empty.

Let $\FSch_S$ be the category defined as follows.
An object of $\FSch_S$ is a pair $(h_X:X\rightarrow S,n)$ with $h_X:
X\rightarrow S$ an $S$-scheme, and $n\in\Bbb Z$.
The hom-set $\FSch_S((X,n),(Y,m))$ is empty if $n<m$.
If $n\geq m$, then $\FSch_S((X,n),(Y,m))$ is the set of morphisms 
$\varphi:X\rightarrow Y$ (not necessarily $S$-morphisms) such that 
$h_Y\varphi=F_S^{n-m}h_X$, where $F_S^{n-m}:S\rightarrow S$ is the 
$(n-m)$th iteration
of the (absolute) Frobenius morphism.
Note that $\nu:\FSch_S\rightarrow \Bbb Z\op$ given by $\nu(X,n)=n$ is a functor
which makes $\FSch_S$ a fibered category over $\Bbb Z\op$.
An object $\FSch_S$ is called an $(F,S)$-scheme, and 
a morphism of $\FSch_S$ is called an $(F,S)$-morphism.

\paragraph
For $(F,S)$-morphisms 
$\varphi:(X,n)\rightarrow (Y,r)$ and $h:(Y',m)\rightarrow (Y,r)$,
the fiber product $(X,n)\times_{(Y,r)}(Y',m)$ in $\FSch_S$ 
does not exist in general.
However, if $S=\Spec k$ with $k$ a perfect field (of characteristic $p$), 
then it does exist.
If $S$ is general and $m=r$, then it exists.
It is $(X\times_YY',n)$ with the structure map
\[
X\times_Y Y'\xrightarrow{p_1}X\rightarrow S.
\]
Similarly, if $n=r$, then the fiber product exists.

\paragraph If $u:S\rightarrow S'$ is a morphism of $\Bbb F_p$-schemes,
then
\[
(h_X:X\rightarrow S,n)\mapsto(uh_X:X\rightarrow S',n)
\]
is a functor from $\FSch_S$ to $\FSch_{S'}$.
If $v:S'\rightarrow S$ is a morphism of $\Bbb F_p$-schemes,
then
$(X,n)\mapsto (S'\times_S X,n)$ is a functor from $\FSch_S$ to
$\FSch_{S'}$.

\paragraph
A forgetful functor $\FSch_S\rightarrow \Sch/\Bbb F_p$ is given by 
$(X,n)\mapsto X$.
$X$ is called the underlying scheme of $(X,n)$.
Sheaves and modules over $(X,n)$ are those for its underlying scheme $X$.

\paragraph
For $n\in\Bbb Z$,
We denote the fiber $\nu^{-1}(n)$ by $\FSch_{S,n}$.
Note that $i: \Sch/S\rightarrow\FSch_{S,0}$ given by $X\mapsto (X,0)$ is
an equivalence.
We identify $X$ with $i(X)=(X,0)$, and 
$\Sch/S$ with $\FSch_{S,0}$ via $i$, and consider that $\Sch/S$ is
a full subcategory of $\FSch_S$.

\paragraph For $r\in\Bbb Z$, ${}^r(?):\FSch_S\rightarrow\FSch_S$ given by
${}^r(X,n)=(X,n+r)$ and ${}^r\varphi=\varphi$ is an autoequivalence of
$\FSch_S$.
${}^r(?)$ is also denoted by $(?)^{(-r)}$.
Thus we will write $(X,r)$ by ${}^rX$.

In what follows, when we consider Frobenius maps, we work over $\FSch_S$.
The advantage of doing so is, we may consider $X^{(r)}$ for artitrary 
$r\in\Bbb Z$ not only for $S=\Spec k$ with $k$ a perfect field 
(as in \cite[(9.2)]{Jantzen}), but also 
for an arbitrary $\Bbb F_p$-scheme $S$.

\paragraph
A homomorphism $h:A\rightarrow B$ of $\Bbb F_p$-algebras is said to be
{\em purely inseparable} if for each $b\in B$, there exists some $e\geq 0$ 
such that $b^{p^e}\in h(A)$.
A morphism of $\Bbb F_p$-schemes $\varphi:X\rightarrow Y$ is purely inseparable
if it is affine, and for each affine open subset $U$ of $Y$, 
$\Gamma(U,Y)\rightarrow \Gamma(\varphi^{-1}(U),X)$ is purely inseparable.
A purely inseparable morphism is a radical morphism, and hence is an
integral morphism.

\paragraph\label{Frobenius-kernel.par}
Let $Z$ be an $S$-scheme, $r\in\Bbb Z$ and $e\geq 0$.
Note that the absolute Frobenius map $F^e_Z:{}^{e+r}Z\rightarrow {}^rZ$ is an
$(F,S)$-morphism.

An $\O_Z$-module $\M$, viewed as an $\O_{{}^rZ}$-module
(note that ${}^rZ$ is $Z$, when it is viewed 
as a scheme), is denoted by ${}^r\M$.
The structure sheaf ${}^r\O_X$ is also denoted by $\O_{{}^rX}$.

Let $\psi:Z'\rightarrow Z$ be an $S$-morphism.
The map ${}^{e+r}Z'\rightarrow {}^{e+r}Z\times_{{}^rZ} {}^rZ'$ given by 
$z'\mapsto (\psi(z'),
F^e(z'))$ is denoted by $\Phi_e(Z,Z')$ or $\Phi_e(\psi)$ for $e\geq 0$.
Note that $\Phi_e(Z,Z')$ is purely inseparable.
By abuse of notation, we sometimes denote the map
\[
\eta_{\Phi_e(Z,Z')}:\O_{{}^{e+r}Z\times_{{}^rZ} {}^rZ'}
\rightarrow \Phi_e(Z,Z')_*\O_{{}^{e+r}Z'}
\]
by $\Phi_e(Z,Z')$ or $\Phi_e$.
$\Phi_e(S,Z)$ is denoted by $\Phi_e(Z)$, and is called the
{\em $e$th relative Frobenius map} or the {\em $S$-Frobenius map} of $Z$.

$\Phi_e$ is a natural transformation between the functors
from the category of morphisms in $\Sch/S$ to the category 
$\FSch_{S,r+e}$, as can be seen easily.

\begin{lemma}
If $G$ is an $S$-group scheme, then $\Phi_e(G):{}^eG\rightarrow 
{}^eS\times_{S}G$ 
is a homomorphism of ${}^eS$-group schemes.
Similarly, $\Phi_e(G):G\rightarrow S\times_{S^{(e)}}G^{(e)}$ is a
homomorphism of $S$-group schemes.
\end{lemma}

\begin{proof}
The first assertion is the consequence of the commutativity of the diagram
\[
\xymatrix{
 & {}^eG\times_{{}^eS}{}^eG \ar `l [ldd] `[dd]_{\mu_{{}^eG}} [dd]
\ar[d]^\cong \ar[r]^-{(\Phi_e,\Phi_e)} & 
({}^eS\times_S G)\times_{{}^eS}({}^eS\times_S G)
\ar `r [rdd] `[dd]^{\mu_{{}^eS\times_S G}} [dd] \ar[d]^\cong & \\
 & {}^e(G\times_S G) \ar[d]^{{}^e\mu_G} \ar[r]^-{\Phi_e} \ar[d]^{{}^e\mu_G} &
{}^eS\times_S(G\times_S G) \ar[d]^{1\times\mu_G} & \\
 & {}^eG \ar[r]^-{\Phi_e} & {}^eS\times_S G &
}.
\]
The second assertion is also proved similarly.
\end{proof}

\paragraph\label{Frobenius-kernel.par}
For a $G$-scheme $Z$, ${}^eZ$ is an ${}^eG$-scheme.
Also, ${}^eS\times_S Z$ is an ${}^eS\times_S G$-scheme, and hence it is
also an ${}^eG$-scheme through $\Phi_e(G)$.
It is easy to see that $\Phi_e(Z):{}^eZ\rightarrow {}^eS\times_S Z$ is
an ${}^eG$-morphism.

The kernel of $\Phi_e(G):G\rightarrow S\times_{S^{(e)}} G^{(e)}$ 
is denoted by $G_e$, and is called the {\em $e$th Frobenius kernel} of $G$,
see \cite[(I.9.4)]{Jantzen} (for the case that $S=\Spec k$ with $k$
a perfect field).
It is an $S$-subgroup scheme of $G$.
The kernel of $\Phi_e(G):{}^eG\rightarrow {}^eS\times_S G$ is ${}^eG_e$.
We may also call ${}^eG_e$ the Frobenius kernel, by abuse of terminologies.

\begin{lemma}\label{smoth-Frobenius-Noetherian.lem}
Let $V$ and $W$ be locally Noetherian $\Bbb F_p$-schemes,
and $\psi:V\rightarrow W$ a smooth morphism with relative dimension $d$.
Then $\Phi_e(W,V)_*(\O_{{}^eV})$ is a locally free 
sheaf of ${}^eW\times_W V$ of rank $p^{de}$.
\end{lemma}

\begin{proof}
The question is local both on $V$ and $W$,
and we may assume that $V=\Spec B$ and $W=\Spec A$ are both affine,
and that there is a factorization $A\rightarrow C=A[x_1,\ldots,x_d]\rightarrow
B$ such that $C$ is a polynomial ring on $d$ variables over $A$, and
$B$ is \'etale over $C$ (see \cite[(I.3.24)]{Milne}).
It suffices to show that ${}^eB$ is a projective 
${}^eA\otimes_A B$-module of rank $p^{de}$.
By \cite[(33.5)]{ETI}, ${}^eB\cong {}^eC\otimes_C B$, and hence we may
assume that $B=C$.
In this case, 
${}^eA\otimes_A B={}^eA[x_1,\ldots,x_d]$, and
${}^eB={}^eA[{}^ex_1,\ldots,{}^ex_d]$.
So ${}^eB$ is a free ${}^eA\otimes_A B$-module with the basis
$\{{}^ex_1^{i_1}\cdots{}^ex_d^{i_d}\mid 0\leq i_1,\ldots,i_d < p^e\}$.
\end{proof}

\begin{lemma}\label{relative-Frob.thm}
Let $G$ be an $S$-group scheme and $N$ its normal subgroup scheme.
Let $\psi:V\rightarrow W$ be a $G$-enriched principal $N$-bundle, 
and assume that $S$ and $N$ are locally Noetherian, and 
$N$ is regular over $S$ \(that is, flat 
with geometrically regular fibers\).
Then $\Phi_e(W,V):{}^eV\rightarrow {}^eW\times_W V$ is an ${}^eG$-enriched 
principal ${}^eN_e$-bundle.
\end{lemma}

\begin{proof}
It is obvious that $\Phi_e(W,V)$ is an ${}^eG$-morphism.
So it suffices to prove that $\Phi_e(W,V)$ is a principal ${}^eN_e$-bundle,
assuming that $G=N$.
Since $N$ is flat over $S$, $\psi$ is fpqc.

Let $W'$ be the $S$-scheme $V$ with the trivial $N$-action,
and $h:W'\rightarrow W$ be $\psi$.
Then since $\psi$ is a principal $N$-bundle, the base change
$\psi':V'\rightarrow W'$ of $\psi$ by $h$ is a trivial $N$-bundle.
As the base change
\[
\Phi_e(W,V)': W'\times_W {}^eV
\xrightarrow{1_{W'}\times \Phi_e(W,V)} 
W'\times_W ({}^eW\times_W V)
\]
is identified with $\Phi_e(W',V')=\Phi_e(\psi')$ 
(see \cite[Lemma~4.1, {\bf 4}]{Hashimoto6}),
it suffices to prove that $\Phi_e(W',V')$ is a principal ${}^eN_e$-bundle
by \cite[(2.11)]{Hashimoto5}, since $h$ is fpqc.

As $\psi'$ is a trivial bundle, $\Phi_e(\psi')$ is identified 
with $1_{W'}\times \Phi_e(N)$.
By \cite[(2.11)]{Hashimoto5} again, it suffices to
prove that $\Phi_e(N):{}^eN\rightarrow{}^eS\times_S N$
is a principal ${}^eN_e$-bundle.
By the theorem of Radu and Andr\'e \cite{Radu}, \cite{Andre}, 
\cite{Dumitrescu2}, $\Phi_e(W,V)=\Phi_e(N)$ is flat.
Being a homeomorphism, it is fpqc.
Note that $\Ker \Phi_e(N)={}^eN_e$.
Being an fpqc homomorphism, $\Phi_e$ is a principal ${}^eN_e$-bundle
(as in \cite[(6.4)]{Hashimoto5}), as required.
\end{proof}

\section{Semireductive group schemes}\label{semireductive.sec}

\begin{lemma}\label{finite-extension-F-finite.lem}
Let $A\subset B$ be a finite extension of commutative rings.
$A$ is Noetherian if and only if $B$ is Noetherian.
$A$ is Noetherian $F$-finite if and only if $B$ is Noetherian $F$-finite.
\end{lemma}

\begin{proof}
If $A$ is Noetherian, then $B$ is Noetherian by Hilbert's basis theorem.
The converse is known as Eakin--Nagata theorem \cite[Thereom~3.7]{CRT}.

We prove the second assertion.
If $A$ is $F$-finite, then $B$ is $F$-finite, since $B$ is $F$-finite over $A$
\cite[Lemma~2, Example~3]{Hashimoto3}.
We prove the converse.
We have $A^p\subset B^p\subset B$, and $B^p$ is $A^p$-finite and $B$ is 
$B^p$-finite, where $A^p=\{a^p\mid a\in A\}=F_A(A)\subset A$ 
($F_A$ is the Frobenius map).
So $B$ is $A^p$-finite.
As $A^p$ is Noetherian and $A$ is a submodule of the finite $A^p$-module
$B$, $A$ is a finite $A^p$-module, and $A$ is $F$-finite.
\end{proof}

\begin{lemma}\label{finite-mumford-conjecture.lem}
Let $k$ be a field of characteristic $p>0$, and $G$ a finite group.
Then there exists some $e\geq 0$ such that for any
finite dimensional $G$-module $V$ and any $v\in V^G\setminus\{0\}$,
there exists some $h\in\Sym_{p^e}V^*$ such that $h(v)=1$.
\end{lemma}

\begin{proof}
Let $P$ be a Sylow $p$-subgroup of $G$ of order $p^e$.
Let $\sigma_1,\ldots,\sigma_n$ be
a complete set of representatives of $G/P$.
Note that $n$ is invertible in $k$.
Let $\psi\in V^*$ be any element such that $\psi(v)=1$.
Then $h=n^{-1}\sum_{i=1}^r \sigma_i\prod_{g\in P}g\psi$ is the desired element.
\end{proof}

\paragraph\label{semireductive.par}
Let $k$ be a field of arbitrary characteristic, 
and $G$ be an affine algebraic $k$-group scheme.
We say that $G$ is {\em semireductive} if
$\bar G\red^\circ$ is (connected) reductive,
where $\bar k$ is the algebraic closure of $k$, and
$\bar G=\bar k\otimes_k G$.
That is, the radical of
$\bar G\red$ is a torus.
We say that $G$ is a {\em semitorus} if $\bar G\red^\circ$ is a torus.
Note that a semireductive and linearly reductive are equivalent in
characteristic zero.
A linearly reductive affine algebraic $k$-group scheme in characteristic
$p>0$ is a semitorus, as can be seen easily from Nagata's theorem
\cite[Theorem~1]{Nagata}.
See also \cite{Sweedler3}.

\begin{lemma}\label{uniform-git-lemma.lem}
Let $k$ be a field of arbitrary characteristic, 
and $G$ a semireductive affine algebraic $k$-group scheme.
Let $B$ be a $G$-algebra, and $I$ a $G$-ideal.
Then for each $b\in (B/I)^G$, there exists some $r$ such that $b^r\in B^G/I^G$.
More precisely,
\begin{enumerate}
\item[\bf 1] 
If the characteristic of $k$ is zero, $r$ can be taken to be $1$.
\item[\bf 2] 
If the characteristic of $k$ is $p>0$, then $r$ can be taken to be a power of
$p$.
\item[\bf 3] In {\bf 2}, if $G$ is a semitorus, then there exists some $e_0$
which depends only on $G$ and independent of $B$ or $b$, such that
$r$ can be taken to be $p^{e_0}$.
\end{enumerate}
In case {\bf 2}, $(B/I)^G$ is purely inseparable over $B^G/I^G$.
In any case, the canonical map 
$\Spec (B/I)^G\rightarrow \Spec (B^G/I^G)$ is a universal
homeomorphism.
\end{lemma}

\begin{proof}
We may assume that $k$ is algebraically closed.
If the characteristic of $k$ is zero, then $G$ is linearly reductive, and
$B^G\rightarrow (B/I)^G$ is surjective, and the assertion {\bf 1} is obvious.
So we may assume that the characteristic is $p>0$.

Take $c\in B$ such that $c$ modulo $I$ equals $b$.
Let $e\geq 0$ be the number such that $G/G_e$ is reduced (hence is smooth
and is isomorphic to $G^{(e)}\red$).
Note that the connected reductive group 
$G\red^\circ$ over the algebraically closed field $k$ 
is defined over $\Bbb F_p$ (see \cite[(II.1)]{Jantzen} and references
therein), and hence
$(G^{(e)}\red)^\circ=(G\red^\circ)^{(e)}\cong G\red^\circ$ is reductive.
It is easy to see that $c^{p^e}\in B^{G_e}$.
Note that we can take this $e$ depending only on $G$.
Replacing $b$ by $b^{p^e}$, $B$ by $B^{G_e}$, $I$ by $I^{G_e}$, and
$G$ by $G/G_e\cong G^{(e)}\red$, we may assume
that $G$ is smooth.

Then by Haboush's theorem (the Mumford conjecture) 
\cite[(II.10.7)]{Jantzen} 
and \cite[(A.1.2)]{GIT},
we have that there exists some $e$ such that
$b^{p^e}$ is in $(B^{G^\circ}/I^{G^\circ})^{G/G^\circ}$.
The choice of $e$ may depend on $b$ this time, but if {\bf 3} is assumed,
then $G^{\circ}$ is a torus, which is linearly reductive, and we can
take $e=0$, which depends only on $G$.

Then replacing $G$ by $G/G^\circ$, we may assume that $G$ is a finite group.
This case is proved by the same proof as in \cite[(A.1.2)]{GIT},
using Lemma~\ref{finite-mumford-conjecture.lem}.
\end{proof}

\begin{lemma}\label{universally-submersive.lem}
Let $k$ be a field of arbitrary characteristic
and $G$ a semireductive $k$-group scheme,
and $\varphi:X\rightarrow Y$ be an algebraic quotient by $G$.
Then $\varphi$ is a universally submersive categorical quotient.
If, moreover, $\varphi$ is a geometric quotient, then it is universally open.
\end{lemma}

\begin{proof}
We may assume that $k$ is of characteristic $p>0$ by \cite[Theorem~1.1]{GIT}.
In the proof of \cite[Theorem~A.1.1]{GIT} in Appendix to Chapter~1, C., 
it is proved that $\varphi$ is a submersive categorical quotient.
We prove the first assertion.
We only need to prove that $\varphi$ is universally submersive.
So we may assume that $Y=\Spec A$ is affine.
Then $B=\Spec B$ is also affine and $A=B^G$.
It suffices to show that for any $A$-algebra $A'$, 
the base change $X'=\Spec B'\rightarrow \Spec A'=Y'$ is submersive.

There is a sequence of maps
\[
A\xrightarrow \alpha A''\xrightarrow \beta A'
\]
such that $\alpha$ is flat and $\beta$ is surjective.
Indeed, for each $a\in A'$, consider a variable $x_a$, and set $A''=A[x_a\mid
a\in A']$.
Then $(B'')^G=A''$, where $B''=A''\otimes_A B$.
So we know that $\varphi''=X''=\Spec B''\rightarrow\Spec A''=Y''$ is submersive.
So replacing $A$ by $A''$, we may assume that $A\rightarrow A'=A/I$ 
is surjective.

We want to prove that the canonical map 
$\gamma:Z'=\Spec (B')^G\rightarrow \Spec A'=Y'$ is a homeomorphism.
As we know that $\varphi$ is surjective, 
$\varphi':X'=\Spec B'\rightarrow \Spec A'=Y'$ is also surjective.
As $\varphi'$ factors through $\gamma$, we have that $\gamma$ is surjective.
On the other hand, by Lemma~\ref{uniform-git-lemma.lem},
$A'\rightarrow (B')^G$ is purely inseparable.
Thus $\gamma$ is injective and closed, and hence is a homeomorphism.
As $\delta:X'=\Spec B'\rightarrow \Spec (B')^G=Z'$ is known to be submersive, 
$\varphi'=\gamma\delta$ is also submersive.

Now the last assertion follows from Lemma~\ref{GIT-remark.lem}.
\end{proof}

\begin{lemma}\label{F-finite-finite.thm}
Let $k$ be a field, and $G$ an affine algebraic $k$-group scheme.
Let $B$ be a Noetherian $G$-algebra.
Set $A:=B^G$.
Assume either
\begin{enumerate}
\item[\bf a] $k$ is of characteristic zero and $G$ is semireductive; 
\item[\bf b] $k$ is a field of characteristic $p>0$, $G$ is a semitorus,
and $B$ is $F$-finite; or
\item[\bf c] 
$k$ is a field of characteristic $p>0$, $G$ is semireductive, and there is
a Noetherian $k$-algebra $R$ and a $k$-algebra map $R\rightarrow A$ such that
$B$ is of finite type over $R$.
\end{enumerate}
Then
\begin{enumerate}
\item[\bf 1] For any $B$-finite $(G,B)$-module $M$, $M^G$ is a finite
$A$-module.
\item[\bf 2] $A$ is Noetherian.
\item[\bf 3] If $G$ is finite, then $B$ is finite over $A$.
\item[\bf 4] In the case of {\bf b}, $A$ is $F$-finite.
\item[\bf 5] In the case of {\bf c}, $A$ is of finite type over $R$.
\end{enumerate}
\end{lemma}

\begin{proof}
We prove the lemma by the Noetherian induction.
The cases are divided, and when we consider the case {\bf b} (resp.\ {\bf c}),
the inductive hypothesis {\bf 5} (resp.\ {\bf 4}) will never be used.

We may assume that for any nonzero $G$-ideal $I$ of $B$ and any $(G,B/I)$-module
$M$, $M^G$ is $(B/I)^G$-finite, 
$(B/I)^G$ is Noetherian, and if {\bf b} is assumed, $(B/I)^G$ is $F$-finite.
If {\bf c} is assumed, then we may assume that $(B/I)^G$ is of finite
type over $R$.

Consider the case that {\bf b} is assumed.
Note that there exists some $e_0$ such that $((B/I)^G)^{p^{e_0}}\subset
B^G/I^G\subset (B/I)^G$ by Lemma~\ref{uniform-git-lemma.lem}, {\bf 3}.
As $(B/I)^G$ is $F$-finite, $(B/I)^G$ is a finite 
$B^G/I^G$-module for any nonzero $G$-ideal $I$ of $B$.
If {\bf c} is assumed, then we have that $(B/I)^G$ is integral over $B^G/I^G$
by Lemma~\ref{uniform-git-lemma.lem}, {\bf 2}.
As we assume that $(B/I)^G$ is of finite type over $R$, 
$(B/I)^G$ is finite over $B^G/I^G$.
If {\bf a} is assumed, then $B^G/I^G=(B/I)^G$ by the linear reductivity,
and obviously $(B/I)^G$ is $B^G/I^G$-finite.
Thus $(B/I)^G$ is finite over $B^G/I^G$ in either case.

In either case, we have that $B^G/I^G$ is a Noetherian ring by
Lemma~\ref{finite-extension-F-finite.lem}.
Also by the induction hypothesis, $M^G$ is a Noetherian $A$-module if
$M$ is a $B$-finite $(G,B)$-module with $\ann M\neq 0$.

We prove that for any $B$-finite $(G,B)$-module $M$, $M^G$ is a Noetherian
$A$-module.
This proves {\bf 1} and {\bf 2}.

If $B$ is not a $G$-domain 
(see \cite{HM})
and $IJ=0$ for some nonzero $G$-ideals $I$ and $J$,
then 
\[
0\rightarrow (IM)^G\rightarrow M^G\rightarrow (M/IM)^G
\]
is exact and $(IM)^G$ and $(M/IM)^G$ are Noetherian $A$-modules
(since $J$ and $I$ 
respectively annihilate $IM$ and $M/IM$),
$M^G$ is a Noetherian $A$-module.

So we may assume that $B$ is a $G$-domain.
We prove that $M^G$ is a Noetherian $A$-module
by the induction on the length of $M_P$,
where $P$ is any fixed minimal prime ideal of $B$.

By a $G$-torsion submodule of $M$, we mean a $(G,B)$-submodule of $M$
whose annihilator is nonzero.
We define the $G$-torsion part $M\tor$ 
of $M$ to be the sum of all the $G$-torsion submodules of $M$.
Note that $M\tor$ is the largest $G$-torsion submodule of $M$.
As $(M\tor)^G$ is a Noetherian $A$-module and
\[
0\rightarrow (M\tor)^G\rightarrow M^G\rightarrow (M/M\tor)^G
\]
is exact, replacing $M$ by $M/M\tor$, we may assume that $M$ is 
$G$-torsion-free, that is, $M$ does not have a nonzero $G$-torsion submodule.

If $M^G=0$, then $M^G$ is a Noetherian module.
If $M^G\neq 0$, then there is an injective $(G,B)$-linear map
$B\rightarrow M$.
As $(M/B)^G$ is Noetherian by induction,
it suffices to prove that
$A=B^G$ is a Noetherian $A$-module, that is, $A$ is a Noetherian ring.

Let $J$ be an ideal of $A$.
We want to show that $J$ is finitely generated.
If $J=0$, then $J$ is finitely generated.
So we may assume that $J\neq 0$.
Let $a\in J\setminus\{0\}$.
Since $aB$ is a nonzero $G$-ideal of $B$ and $B$ is a $G$-domain, 
we have that $0:_BaB=0$.
That is, $a$ is a nonzero divisor in $B$.
So $B\subset B[a^{-1}]$, and hence $A=B^G=B\cap B[a^{-1}]^G$.
So if $c=ab\in aB\cap A$, then $b=ca^{-1}\in B\cap B[a^{-1}]^G=A$.
So $c\in aA$, and we have that $(aB)^G=aB\cap A=aA$.
Hence $A/aA$ is a Noetherian ring, and $J/aA$ is finitely generated.
Hence $J$ is finitely generated, as desired.

Next, we prove {\bf 3}.
Let $B_0$ be the $k$-algebra $B$ with the trivial $G$-action.
Then the coaction $\omega_B:B\rightarrow B_0\otimes k[G]$ is a
$G$-algebra map.
Thus $B_0\otimes k[G]$ is a $B$-finite $(G,B)$-module.
Clearly, $(B_0\otimes k[G])^G=B_0\otimes k=B_0$ as $A$-algebras, and
$B_0$ is isomorphic to $B$ as $A$-algebras, as can be seen easily.
On the other hand, $B_0=(B_0\otimes k[G])^G$ is $A$-finite by {\bf 1}.
Thus $B$ is $A$-finite, and {\bf 3} has been proved.

So the case {\bf a} has been completed, because {\bf 4} and {\bf 5} are
trivial in this case.

If {\bf c} is assumed, then $\Spec B\rightarrow \Spec A$ is universally
submersive by Lemma~\ref{universally-submersive.lem}.
Now the assertion {\bf 5} follows from {\bf 2}, which has already been
proved, by \cite[(6.2.1)]{Alper}.
So the proof of {\bf c$\Rightarrow$1,2,3,5} has been completed
(as we have emphasized, we have not used the inductive hypothesis {\bf 4}
for the case {\bf c} at all).

If $G$ is finite, then {\bf 4} follows from {\bf 3} and 
Lemma~\ref{finite-extension-F-finite.lem}.

Thus the lemma has also been completely proved for the case that $G$ is finite.

We prove the lemma for the general case.
It suffices to prove {\bf 1, 2, 3, 4} assuming
{\bf b}.

Replacing $k$ by its finite purely inseparable extension, we may assume
that $G\red$ is $k$-smooth.
Then for some $e\geq 0$, $G/G_e\cong k\otimes_{k^{p^e}}G^{(e)}\red$ is 
$k$-smooth, see for the notation on the Frobenius twist, see
section~\ref{Frobenius-pushforwards.sec}.
Note that $\bar k\otimes_k(G/G_e)\cong \bar k\otimes_{k^{p^e}}G^{(e)}\red
\cong \bar G^{(e)}\red$.
As the torus $\bar G^\circ\red$ is defined over $\Bbb F_p$, 
we have that $G/G_e$ is a semitorus.
As the lemma is already proved for the finite group scheme $G_e$,
replacing $G$ by $G/G_e$, we may assume that $G$ is smooth.

Next, replacing $k$ by some finite Galois extension, we may assume that
$G^\circ$ is a split torus.
If the lemma is proved for the split torus $G^\circ$, then as the
lemma is already proved for the finite group scheme $G/G^\circ$, 
the proof of the lemma completes.
Thus we may assume that $G$ is a split torus.
Then replacing $k$ by $\Bbb F_p$, we may assume that $k=\Bbb F_p$, 
which is perfect.
Then the absolute Frobenius map $F:G\rightarrow G^{(1)}$ is a faithfully
flat homomorphism of $k$-group schemes by the theorem of Kunz
\cite[Theorem~2.1]{Kunz}, $G$ acts on $B^{(1)}$, and $(B^{(1)})^G
=(B^{(1)})^{G^{(1)}}=A^{(1)}$ by Lemma~\ref{restriction-invariance.lem}.

Now we repeat the inductive argument above.
Then as above, we reach the situation that 
{\bf 1, 2, 3} are already proved, and we prove {\bf 4}.
Then {\bf 1, 2, 3} are also true for the action of $G^{(1)}$ on $B^{(1)}$.
Hence {\bf 1, 2, 3} are also true for the action of $G$ on $B^{(1)}$
through the Frobenius map $G\rightarrow G^{(1)}$.

As $B$ is a $B^{(1)}$-finite $(G,B^{(1)})$-module by $F$-finiteness,
$A=B^G$ is a finite $A^{(1)}$-module by {\bf 1}, which has already been
proved.
Now by induction, we have proved {\bf 1, 2, 3, 4} for the case {\bf b}.
This finishes the proof of the lemma.
\end{proof}

\section*{\large Chapter~1. Main Results}\label{main.chap}

\section{Almost principal fiber bundles}\label{main.sec}

\paragraph
Let $f:G\rightarrow H$ be a qfpqc homomorphism of $S$-group schemes with
$N=\Ker f$.

\begin{definition}\label{rational-almost-pb.def}
A diagram of $G$-schemes
\[
\xymatrix{
X & U \ar@{_{(}->}[l]_i \ar[r]^\rho & V \ar@{^{(}->}[r]^j & Y
}
\]
is said to be a 
{\em $G$-enriched rational $n$-almost principal $N$-bundle} 
if the following six conditions hold.
\begin{enumerate}
\item[\bf 1] $N$ acts on $Y$ trivially.
\item[\bf 2] $i$ is an open immersion.
\item[\bf 3] $j$ is an open immersion.
\item[\bf 4] $i(U)$ is $n$-large in $X$.
\item[\bf 5] $j(V)$ is $n$-large in $Y$.
\item[\bf 6] $\rho$ is a $G$-enriched principal $N$-bundle.
\end{enumerate}
A $G$-enriched rational $n$-almost
principal $G$-bundle is simply called
a rational $n$-almost principal $G$-bundle.
In these definitions, we may simply say \lq almost' instead of 
saying \lq $1$-almost.'
\end{definition}

\begin{definition}\label{almost-pb.def}
A $G$-morphism $\varphi:X\rightarrow Y$ is said to be a
{\em $G$-enriched $n$-almost principal $N$-bundle} 
with respect to $U$ and $V$ if $U$ is a $G$-stable open subset of $X$ and
$V$ is an $H$-stable open subset of $Y$, $\varphi(U)\subset V$, and
\[
\xymatrix{
X & U \ar@{_{(}->}[l]_i \ar[r]^{\rho} & V \ar@{^{(}->}[r]^j & Y
}
\]
is a $G$-enriched rational $n$-almost principal $N$-bundle,
where $\rho:U\rightarrow V$ is the restriction of $\varphi$,
and $i$ and $j$ are inclusions.
We simply say that $\varphi$ is a $G$-enriched $n$-almost principal 
$N$-bundle, if it is so with respect to some $U$ and $V$.
We may omit the epithet \lq $G$-enriched' if $G=N$.
We may use \lq almost' as a synonym of \lq$1$-almost.'
\end{definition}

\begin{lemma}\label{qfpqc-functor.thm}
Let $S$ be a scheme, and $h:M\rightarrow N$ be a morphism of $S$-schemes.
Then the following conditions are equivalent.
\begin{enumerate}
\item[\bf 1] $h$ is qfpqc.
\item[\bf 2] 
For any $S$-scheme $W$ and an $S$-morphism $\alpha:W\rightarrow N$,
there exists some qfpqc morphism $\beta:W'\rightarrow W$ such that
$\alpha\beta\in N(W')$ is in the image of $h(W'):M(W')\rightarrow N(W')$.
\end{enumerate}
\end{lemma}

\begin{proof}
{\bf 1$\Rightarrow$2}.
Replacing $M$ if necessary, we may assume that $h$ is fpqc.
Let $W'=M\times_N W$, and let $\beta:W'\rightarrow W$ be the second projection, 
and let $\gamma\in M(W')$ be the first projection.
Then $h\gamma=\alpha\beta$ is in the image of $h(W')$.

{\bf 2$\Rightarrow $1}.
Let $\alpha$ be the identity morphism of $N$.
Then there exists some qfpqc morphism $\gamma:W'\rightarrow M$ such that
$\beta=h\gamma:W'\rightarrow N$ is qfpqc.
Thus $h$ is also qfpqc, as desired.
\end{proof}

\begin{lemma}\label{monomorphism-isom.thm}
Let $h: M\rightarrow N$ be a qfpqc monomorphism of $S$-schemes.
Then $h$ is an isomorphism.
\end{lemma}

\begin{proof}
Letting $W=N$ in Lemma~\ref{qfpqc-functor.thm}, {\bf 2}, 
there is a qfpqc morphism $\beta:W'\rightarrow N$ which is
in the image of $h(W'):M(W')\rightarrow N(W')$.
Replacing $\beta$ if necessary, we may assume that $\beta$ is fpqc.
There is $\gamma\in M(W')$ such that $h\gamma=\beta$.
Let $p_i:W'\times_N W'\rightarrow W'$ be the $i$th projection.
Then $h\gamma p_1=\beta p_1=\beta p_2=h\gamma p_2$.
As $h$ is a monomorphism, $\gamma p_1=\gamma p_2$.
By \cite[(2.55)]{Vistoli}, there is $g:N\rightarrow M$ such that
$g\beta=\gamma$.
Then $hg\beta=h\gamma=\beta$.
By \cite[(2.9)]{Hashimoto5} (or by \cite[(2.55)]{Vistoli} again), $hg=1_N$.
So $hgh=1_Nh=h1_M$.
As $h$ is a monomorphism, $gh=1_M$.
So $g=h^{-1}$, and $h$ is an isomorphism, as desired.
\end{proof}

\paragraph\label{four-groups.par}
Let $E$ be an $S$-group scheme, and $G$ its subgroup scheme.
Let $f:G\rightarrow H$ be a qfpqc homomorphism between $S$-group schemes,
and $N:=\Ker f$.
Assume that $G$ and $N$ are normal in $E$.
As $G$ is normal in $E$, $E$ acts on $G$ by the conjugation.
That is, the action is given by $(a,g)\mapsto aga^{-1}$ for $a\in E$ and 
$g\in G$.

\begin{lemma}
Let the notation be as in {\rm (\ref{four-groups.par})}.
There exists a unique action of $E$ on $H$ such that $f$ is an $E$-morphism.
The action is by group automorphisms.
\end{lemma}

\begin{proof}
Consider the following diagram.
\[
\xymatrix{
E\times G\times_{H}G \ar@<1ex>[r]^-{1_E\times p_1}
\ar@<-1ex>[r]_-{1_E\times p_2} &
E\times G \ar[r]^{1_E\times f} \ar[d]^{a_G} &
E\times H \ar@{.>}[d]^{a_H} \\
 & G \ar[r]^f & H
}
\]
We want to find a unique arrow $a_H$ such that the diagram is commutative.
As the representable functor $H=\Hom_{\Sch/S}(?,H)$ is a sheaf 
with respect to the fpqc topology \cite[(2.55)]{Vistoli},
$H(E\times H)$ is the difference kernel of
\[
\xymatrix{
H(E\times G)
\ar@<1ex>[r]^-{(1_E\times p_1)^*}
\ar@<-1ex>[r]_-{(1_E\times p_2)^*} &
H(E\times G\times_H G)
}.
\]
So $a_H$ which makes the diagram commutative is unique.

To show the existence, it suffices to show that 
$fa_G(1_E\times p_1)=fa_G(1_E\times p_2)$.
Let $b\in E$, and $(g_1,g_2)\in G\times_H G$.
Then
\begin{multline*}
(fa_G(1_E\times p_1)(b,g_1,g_2))
(fa_G(1_E\times p_2)(b,g_1,g_2))^{-1}
=\\(fa_G(b,g_1))(fa_G(b,g_2))^{-1}
=f(bg_1b^{-1})f(bg_2b^{-1})^{-1}=f(bg_1g_2^{-1}b^{-1}).
\end{multline*}
As $g_1g_2^{-1}\in N$ and $N$ is normal in $E$, $f(bg_1g_2^{-1}b^{-1})$ is 
trivial, and we have
$fa_G(1_E\times p_1)=fa_G(1_E\times p_2)$.

We show that the action $a_H$ is by group automorphisms.
We want to show that the two maps $\alpha_1,\alpha_2:
E\times H\times H\rightarrow H$
given by $\alpha_1(b,h_1,h_2)=b\cdot (h_1h_2)$ and $\alpha_2(b,h_1,h_2)
= (b\cdot h_1)(b\cdot h_2)$ agree.
As the qfpqc morphism $1_E\times f\times f:
E\times G\times G\rightarrow E\times H\times H$ is an epimorphism,
it suffices to prove that $\alpha_1(1_E\times f\times f)
=\alpha_2(1_E\times f\times f)$.
This is left to the reader.
\end{proof}

\begin{lemma}\label{principal-composition.thm}
Let $E$, $G$, $f:G\rightarrow H$, and $N$ be as in {\rm(\ref{four-groups.par})}.
Let $X$, $Y$, and $Z$ be $E$-schemes.
Let $\varphi:X\rightarrow Y$ be a $G$-morphism which is $N$-invariant.
Let $\psi:Y\rightarrow Z$ be an $H$-invariant morphism.
Consider the following conditions.
\begin{enumerate}
\item[\bf a] $\varphi$ is an $E$-enriched principal $N$-bundle.
\item[\bf b] $\psi$ is an $E$-enriched principal $H$-bundle.
\item[\bf c] $\psi\varphi$ is an $E$-enriched principal $G$-bundle.
\end{enumerate}
Then we have
\begin{enumerate}
\item[\bf 1] {\bf a} and {\bf b} together imply {\bf c}.
\item[\bf 2] {\bf a} and {\bf c} together imply {\bf b}.
\item[\bf 3] If $\varphi$ is an $E$-morphism, then 
{\bf b} and {\bf c} together imply {\bf a}.
\end{enumerate}
\end{lemma}

\begin{proof}
{\bf 1}.
First, we prove that $\Psi_{G}:G\times X\rightarrow X\times_{Z}X$ given by
$\Psi_G(g,x)=(gx,x)$ is a monomorphism.
That is, 
$\Psi_{G}(W): G(W)\times X(W)\rightarrow X(W)\times_{Z(W)}X(W)$
is injective for any $S$-scheme $W$.
Let $(g,x),(g_1,x_1)\in G(W)\times X(W)$ such that 
$\Psi_{G}(W)(g,x)=\Psi_{G}(W)(g_1,x_1)$.
Then $x=x_1$, and $gx=g_1x$.
Letting $g_1^{-1}g=u$, $ux=x$.
As $f(u)\varphi(x)=\varphi(x)$ and hence 
$\Psi_{H}(e,\varphi(x))=\Psi_{H}(f(u),\varphi(x))$, 
we have that $f(u)=e$ by {\bf b}.
That is, $u\in N$.

As $\Psi_{N}(e,x)=\Psi_{N}(u,x)$, we have $u=e$ by {\bf a}, and hence 
$\Psi_{G}$ is injective.

Next, we prove that for any
$(x',x)\in X(W)\times_{Z(W)}X(W)$, there exists some qfpqc morphism
$\alpha:W'\rightarrow W$ such that $(x'\alpha,x\alpha)\in 
X(W')\times_{Z(W')}X(W')$ is in the image of $\Psi_{G}(W')$.

Note that $(\varphi x',\varphi x)\in Y(W)\times_{Z(W)}Y(W)$ is in the
image of $\Psi_{H}(W)$.
That is, there exists some $h\in H(W)$ such that $h\varphi x=\varphi x'$.
As $f:G\rightarrow H$ is qfpqc, there exists some qfpqc morphism
$\beta:W'\rightarrow W$ and $g\in G(W')$ such that $f(g)=h\beta$.
Then we have $\varphi(g(x\beta))=\varphi(x'\beta)$.
As $(g(x\beta),x'\beta)\in X(W')\times_{Y(W')}X(W')$ is in the image
of $\Psi_{N}(W')$,
there exists some $n\in N(W')$ such that $ng(x\beta)=x'\beta$.
This shows that $\Psi_{G}(ng,x\beta)=(x'\beta,x\beta)$, and hence
the image of $(x',x)$ in $(X\times_{Z}X)(W')$ is in the image of
$\Psi_{G}(W')$.
Hence $\Psi_{G}$ is qfpqc by Lemma~\ref{qfpqc-functor.thm}.
Being a qfpqc monomorphism, $\Psi_G$ is an isomorphism
by Lemma~\ref{monomorphism-isom.thm}.
As $\varphi$ and $\psi$ are qfpqc $E$-morphisms, 
$\psi\varphi$ is a qfpqc $E$-morphism by \cite[(2.3)]{Hashimoto5}.
By \cite[(4.43)]{Vistoli}, $\psi\varphi$ is an $E$-enriched 
principal $G$-bundle.

{\bf 2}.
Let $W$ be any $S$-scheme, $y\in Y(W)$, $h\in H(W)$ such that $hy=y$.
Then there is a qfpqc morphism $\beta:W'\rightarrow W$, $x\in X(W')$, 
$g\in G(W')$ such that $\varphi(x)=y\beta$ and $f(g)=h\beta$.
Thus $f(g)(\varphi(x))=\varphi(x)$.
This implies $(gx,x)\in X(W')\times_{Y(W')}X(W')$.
So there exists some $n\in N(W')$ such that $(nx,x)=(gx,x)$.
As $\Psi_{G}$ is a monomorphism, $g=n\in N(W')$.
Hence $h\beta=f(g)=e=e\beta$.
As $\beta$ is an epimorphism, $h=e$, and hence $\Psi_{H}$ is a monomorphism.

Next, let $(y',y)\in Y(W)\times_{Z(W)}Y(W)$.
Take an fpqc morphism $\beta:W'\rightarrow W$, $x'\in X(W')$, $x\in X(W')$
such that $\varphi(x')=y'\beta$ and $\varphi(x)=y\beta$.
As $\Psi_{G}(W')$ is bijective, there exists some $g\in G(W')$ such that
$gx=x'$.
Then $\Psi_{H}(fg,y\beta)=(y'\beta,y\beta)$.
By Lemma~\ref{monomorphism-isom.thm}, $\Psi_{H}$ is an isomorphism.
As $\psi\varphi$ is qfpqc, $\psi$ is qfpqc.
As $\psi\varphi$ is an $E$-morphism and $\varphi$ is a qfpqc $E$-morphism,
it is easy to see that $\psi$ is an $E$-morphism.
Thus $\psi$ is an $E$-enriched principal $H$-bundle.

{\bf 3}.
Consider the diagram
\[
\xymatrix{
G\times X \ar[rr]^{a_{G,X}} \ar[d]^{f\times 1_X} \ar@{}[drr]|{\text{(a)}} & &
X \ar[d]^{\varphi} \\
H\times X \ar[r]^{1_H\times \varphi} \ar[d]^{p_2^X} \ar@{}[dr]|{\text{(b)}} &
H\times Y \ar[r]^-{a_{H,Y}} \ar[d]^{p_2^Y} \ar@{}[dr]|{\text{(c)}} &
Y \ar[d]^{\psi} \\
X \ar[r]^{\varphi} & Y \ar[r]^{\psi} & Z
},
\]
where $p_2^X$ and $p_2^Y$ are the second projections.
It is easy to check that the diagram is commutative.
As $\psi$ is a principal $H$-bundle, the square (c) is cartesian.
Similarly, as $\psi\varphi$ is a principal $G$-bundle, the whole square
((a)+(b)+(c)) is also cartesian.
(b) is also cartesian, and hence it is easy to see that (a) is cartesian.
On the other hand, letting $N$ act on $H\times X$ trivially and on
$G\times X$ by $n(g,x)=(ng,x)$, (a) is a commutative diagram of 
$N$-schemes.

As $f\times 1_X$ is 
a principal $N$-bundle and it is a base change of $\varphi$ by $\psi\varphi$,
which is qfpqc, we have that $\varphi$ is also a principal $N$-bundle by
\cite[(2.11)]{Hashimoto5}.
As we assume that $\varphi$ is an $E$-morphism, we have that
$\varphi$ is an $E$-enriched principal $N$-bundle.
\end{proof}

\paragraph
In Lemma~\ref{principal-composition.thm}, {\bf 3}, the assumption that
$\varphi$ is an $E$-morphism is indispensable.
Let $S=\Spec k=Z$ with $k$ a field, $Y=H=N=E_0=\Bbb Z/2\Bbb Z$ (the
constant group), $X
=G=H\times N$, 
$\varphi=f:G\rightarrow H$ the first projection, and $E=E_0\times G$.
Let $E$ act on $X$ by $(e_0,h,n)(h',n')=(e_0hh',nn')$ for 
$e_0\in E_0$ and $(h,n),(h',n')\in G$.
Let $E$ act on $Y$ by $(e_0,h,n)y=hy$, and on $Z$ trivially.
Then {\bf b} and {\bf c} are satisfied, but $\varphi$ is not an $E$-morhphism.

\begin{lemma}\label{codim-two-flat.thm}
Let $\psi:Z'\rightarrow Z$ be a flat morphism of schemes, and $W\subset Z$ be 
an open subset.
Set $W':=\psi^{-1}(W)$.
\begin{enumerate}
\item[\bf 1] If $W$ is $n$-large 
in $Z$, then $W'$ is $n$-large in $Z'$.
\item[\bf 2] Assume that $\psi$ is qfpqc, and that
$Z'$ is locally Krull.
If $W'$ is large in $Z'$, then $W$ is large in $Z$.
\item[\bf 3] 
Assume that $\psi$ is qfpqc, and 
that $Z'$ is locally Noetherian.
If $W'$ is $n$-large in $Z'$, then $W$ is $n$-large in $Z$.
\end{enumerate}
\end{lemma}

\begin{proof}
{\bf 1}.
Let $z'\in Z'\setminus W'$.
Then $z:=\psi(z')\in Z\setminus W$, and hence $\dim \O_{Z,z}\geq n$.
As the going-down theorem holds between $\O_{Z,z}$ and $\O_{Z',z'}$ by
flatness, we have $\dim \O_{Z',z'}\geq \dim \O_{Z,z}\geq n$,
and hence $W'$ is $n$-large in $Z'$.

{\bf 2, 3}.
The question is local on $Z$, and
we may assume that $Z=\Spec A$ is affine.
Replacing $Z'$, we may assume that $Z'=\Spec A'$ is also affine.
Note that $A'$ is faithfully flat over $A$.

We prove {\bf 2}.
$A'$ is locally Krull.
As $A'$ is a finite direct product of Krull domains, so is $A$ by
\cite[(5.8)]{Hashimoto4}.
So we may further assume $A$ is a domain.
Then the result follows from \cite[(5.13)]{Hashimoto4}.

We prove {\bf 3}.
$A'$ is Noetherian.
Then $A$ is also Noetherian.
Let $P\in \Spec A$ with $\height P<n$.
If $P'$ is a minimal prime of the ideal $PB$,
then $\height P'<n$ by \cite[Theorem~15.1]{CRT}.
So the assertion follows.
\end{proof}

\begin{remark}
Let $A$ be the DVR $k[y]_{(y)}$, $B$ the DVR $k(y)[x]_{(x)}$, and
$A':=A+xB$.
Note that $A'$ is the composite of $B$ and $A$ \cite[section~10]{CRT}.
Set $Z':=\Spec A'$, and $Z:=\Spec A$.
Let $W=Z\setminus\{(y)\}=D(y)\subset Z$.
Let $\psi:Z'\rightarrow Z$ be the morphism associated with the inclusion
$A\hookrightarrow A'$.
Then although $A$ is a DVR and $A'$ is a valuation ring faithfully flat
over $A$, the conclusion of {\bf 2} or {\bf 3} in (\ref{codim-two-flat.thm})
does not hold.
So the Noetherian or Krull hypothesis on $A'$ is indispensable.
\end{remark}

\begin{lemma}
Let $f:G\rightarrow H$ be a qfpqc homomorphism between $S$-group schemes 
with $N=\Ker f$, 
and $\varphi:X\rightarrow Y$ a $G$-morphism which is $N$-invariant.
Let $U\subset X$ and $V\subset Y$ be open subsets.
Let $h:Y'\rightarrow Y$ be a flat $G$-morphism such that $Y'$ is $N$-trivial.
Let $\varphi':X':=Y'\times_Y X\rightarrow Y'$ be the base change,
and set $U'=Y'\times_Y U$, and $V'=Y'\times_Y V$.
Then
\begin{enumerate}
\item[\bf a] If 
$\varphi$ is a $G$-enriched $n$-almost principal $N$-bundle with respect
to $U$ and $V$, then
$\varphi'$ is a $G$-enriched $n$-almost 
principal $N$-bundle with respect to $U'$ and $V'$.
\item[\bf b] Assume that $h$ is fpqc,
and that both
$X'$ and $Y'$ are locally Krull \(resp.\ locally Noetherian\).
If $\varphi'$ is a $G$-enriched almost \(resp.\ $n$-almost\) 
principal $N$-bundle with respect to $U'$ and $V'$,
then 
$\varphi$ is a $G$-enriched almost \(resp.\ $n$-almost\) 
principal $N$-bundle with respect to $U$ and $V$.
\end{enumerate}
\end{lemma}

\begin{proof}
Clearly, $\varphi'$ is also a $G$-morphism which is $N$-invariant.

{\bf a} Let $\rho:U\rightarrow V$ be the restriction of $\varphi$, and
$\rho':U'\rightarrow V'$ be its base change.
Each of the six conditions in Definition~\ref{rational-almost-pb.def} for
$X'$, $\rho':U'\rightarrow V'$ and $U'$ is proved using the corresponding
condition for
$X$, $\rho:U\rightarrow V$ and $U$.
This is trivial for the conditions {\bf 1}, {\bf 2}, and {\bf 3}.
The conditions {\bf 4} and {\bf 5} follow from 
Lemma~\ref{codim-two-flat.thm}, {\bf 1}.
The conditions {\bf 6} follows from \cite[(2.7)]{Hashimoto5}.

{\bf b} 
The image of the composite
\[
G\times U'\hookrightarrow G\times X'\xrightarrow{a} X'\rightarrow X
\]
is contained in $U$ by assumption.
This map agrees with
\[
G\times U'\xrightarrow{1\times h|_{U'}} G\times U\hookrightarrow 
G\times X\xrightarrow{a}X.
\]
As $1\times h|_{U'}$ is faithfully flat and hence is surjective,
$U$ is $G$-stable.
Similarly, $V$ is $H$-stable.

Now we check the six conditions
in Definition~\ref{rational-almost-pb.def} for
$X$, $\rho:U\rightarrow V$ and $U$.
The conditions {\bf 1, 2, 3} are assumed.

{\bf 4} and {\bf 5} are the consequences of Lemma~\ref{codim-two-flat.thm},
{\bf 2} (resp.\ {\bf 3}).
{\bf 6} follows from \cite[(2.11)]{Hashimoto5}.
\end{proof}

\begin{theorem}\label{alg-quot-krull-equiv.thm}
Let $G$ be a quasi-compact quasi-separated flat $S$-group scheme, 
and $\varphi:X\rightarrow Y$ a quasi-compact quasi-separated 
almost principal $G$-bundle.
Assume that $X$ is locally Krull.
Then the following are equivalent.
\begin{enumerate}
\item[\bf 1] The canonical map $\bar\eta:
\O_Y\rightarrow (\varphi_*\O_X)^G$ is an
isomorphism.
\item[\bf 2] $Y$ is a locally Krull scheme.
\end{enumerate}
\end{theorem}

\begin{proof}
By Lemma~\ref{bar-eta-isom.thm}, {\bf 2},
the question is local on $Y$, and hence we may
assume that $Y$ is affine.
So $X$ is quasi-compact quasi-separated, and
{\bf 1$\Rightarrow$2} follows from \cite[(6.3)]{Hashimoto4}.

We prove {\bf 2$\Rightarrow$1}.
Let $\varphi$ be an almost principal $G$-bundle with respect to $i:U
\hookrightarrow X$
and $j:V\hookrightarrow Y$.
Let $\rho:U\rightarrow V$ be the restriction of $\varphi$, which is a
principal $G$-bundle.
Then applying Lemma~\ref{bar-eta-isom.thm}, {\bf 3} to the cartesian square
\[
\xymatrix{
\varphi^{-1}(V) \ar[r]^-{i'} \ar[d]^{\varphi'} & X \ar[d]^\varphi \\
V \ar[r]^j & Y
},
\]
the result follows if $\bar\eta:\O_V\rightarrow
(\varphi'_*\O_{\varphi^{-1}(V)})^G$ is
an isomorphism, where $\varphi':\varphi^{-1}(V)\rightarrow V$ is the
restriction of $\varphi$, and $i':\varphi^{-1}(V)\rightarrow X$ is
the inclusion.
So we may assume that $V=Y$ (and hence $\varphi^{-1}(V)=X$).
As $\eta:\O_X\rightarrow i_*\O_U$ is an isomorphism by 
\cite[(5.28)]{Hashimoto4}, it suffices to show that the composite
\[
\O_Y\xrightarrow{\bar\eta}(\varphi_*\O_X)^G\xrightarrow{\eta}
(\varphi_*i_*\O_U)^G=(\rho_*\O_U)^G,
\]
which equals $\bar\eta$ for $\rho$, is an isomorphism.
As $\rho$ is a principal $G$-bundle, this is \cite[(5.31)]{Hashimoto5}.
\end{proof}

\begin{example}\label{masuda.ex}
Theorem~\ref{alg-quot-krull-equiv.thm} can be used to check that a 
candidate of the invariant subring is certainly the one.

Let $S=\Spec \Bbb C$, $X=\Bbb A^4_{\Bbb C}$, $Y=\Bbb A^3_{\Bbb C}$, and
$G=\Bbb G_a=\Spec \Bbb C[\tau]$.
Let $G$ act on $X$ by $t(x_1,x_2,x_3,x_4)=(x_1+tx_2,x_2,x_3+tx_4,x_4)$.
Let $\varphi:X\rightarrow Y$ be the map given by 
$\varphi(x_1,x_2,x_3,x_4)=(x_2,x_4,x_1x_4-x_2x_3)$.
Let $B=\Gamma(X,\O_X)=\Bbb C[\xi_1,\xi_2,\xi_3,\xi_4]$, and 
$A=\Gamma(Y,\O_Y)=\Bbb C[\xi_2,\xi_4,w]$, where 
$\xi_i(x_1,x_2,x_3,x_4)=x_i$, and $w=\xi_1\xi_4-\xi_2\xi_3$.
Then it is easy to verify that $\varphi$ is $G$-invariant.
Let $V=D(\xi_2,\xi_4)=Y\setminus V(\xi_2,\xi_4)$, and $U=\varphi^{-1}(V)
=X\setminus V(\xi_2,\xi_4)$.
Obviously, $V$ is a large open subset of $Y$, and $U$ is a large 
$G$-stable open subset of $X$.
Let $\rho:U\rightarrow V$ be the restriction of $\varphi$.
Since $B[\xi_2^{-1}]=A[\xi_2^{-1}][-\xi_2^{-1}\xi_1]$ and $t(-\xi_2^{-1}\xi_1)
=-\xi_2^{-1}\xi_1+t$, $\Spec B[\xi_2^{-1}]\rightarrow \Spec A[\xi_2^{-1}]$ is
a trivial $G$-bundle.
Similarly, $\Spec B[\xi_4^{-1}]\rightarrow \Spec A[\xi_4^{-1}]$ is also 
a trivial $G$-bundle, and hence $\rho:U\rightarrow V$ is a principal
$G$-bundle.

Hence $\varphi:X\rightarrow Y$ is an almost principal $G$-bundle
with respect to $U$ and $V$.
By Theorem~\ref{alg-quot-krull-equiv.thm}, we have that $A=B^G$.
So $\varphi$ is an algebraic quotient.

Note that $\varphi$ is not surjective.
Indeed, $(0,0,1)$ is not in the image of $\varphi$.
This also shows that $\varphi$ is not a categorical quotient.
Indeed, if $\varphi$ is a categorical quotient, 
then letting $W=Y\setminus\{(0,0,1)\}$, 
$\psi:X\rightarrow W$ the same as $\varphi$, and $u:W\hookrightarrow Y$ 
the inclusion, we have that $\varphi=u\varphi'$.
By Lemma~\ref{categorical-quotient-mono.lem}, 
$u$ must be an isomorphism, and this is absurd.
\end{example}

\section{The behavior of the class groups and the canonical modules
with respect to rational almost principal bundles}
\label{class-canonical.sec}

\paragraph
Let 
$f:G\rightarrow H$ be an fpqc homomorphism between flat $S$-group schemes
with $N=\Ker f$.

\begin{theorem}\label{main.thm}
Let
\[
\xymatrix{
X & U \ar@{_{(}->}[l]_i \ar[r]^\rho & V \ar@{^{(}->}[r]^j & Y
}
\]
be a $G$-enriched rational almost principal $N$-bundle.
Assume that both $X$ and $Y$ are locally Krull.
Then $i_*\rho^*j^*:\Ref(H,Y)\rightarrow \Ref(G,X)$ is an equivalence,
and $(j_*\rho_*i^*?)^N$ is its quasi-inverse.
This equivalence induces an equivalence $\Ref_n(H,Y)\cong \Ref_n(G,X)$
for each $n\geq 0$, where $\Ref_n$ denotes the category of reflexive 
modules of rank $n$.
It also induces an isomorphism $\Cl(H,Y)\cong \Cl(G,X)$.
\end{theorem}

\begin{proof}
This follows immediately from \cite[(7.4)]{Hashimoto5} and
\cite[(5.31)]{Hashimoto4}.
\end{proof}

\begin{lemma}\label{almost-double-dual.lem}
Let $\varphi:X\rightarrow Y$ be a $G$-enriched almost principal
$N$-bundle with respect to the open subsets $U$ and $V$.
Let $i:U\rightarrow X$ and $j:V\rightarrow Y$ be the inclusion, and
$\rho:U\rightarrow V$ the restriction of $\varphi$.
Assume that $X$ and $Y$ are locally Krull.
Then the equivalence $i_*\rho^*j^*$ agrees with $(?)^{**}\varphi^*$
as functors from $\Ref(H,Y)$ to $\Ref(G,X)$,
and is independent of the choice of $U$ or $V$.
Its quasi-inverse $(j_*\rho_*i^*?)^N$ agrees with $(\varphi_*?)^N$
as functors from $\Ref(G,X)$ to $\Ref(H,Y)$, and 
is also independent of $U$ or $V$.
\end{lemma}

\begin{proof}
As functors from $\Ref(H,Y)$ to $\Ref(G,X)$, 
\[
(?)^{**}\varphi^*\cong 
i_*i^*(?)^{**}\varphi^*
\cong
i_*(?)^{**}i^*\varphi^*
\cong 
i_*(?)^{**}\rho^*j^*
\cong
i_*\rho^*j^*
\]
by \cite[(5.28), (5.20), (5.9)]{Hashimoto4}.
As functors from $\Ref(G,X)$ to $\Ref(H,Y)$, 
\[
(?)^N\varphi_*\cong (?)^N\varphi_*i_*i^*
\cong (?)^Nj_*\rho_*i^*.
\]
\end{proof}

\begin{corollary}\label{almost-main.thm}
If $\varphi:X\rightarrow Y$ is a $G$-enriched almost principal
$N$-bundle, and $X$ and $Y$ are locally Krull,
then $(?)^N\circ \varphi_*:\Ref(G,X)\rightarrow \Ref(H,Y)$ is
an equivalence, and $(?)^{**}\circ\varphi^*:\Ref(H,Y)\rightarrow\Ref(G,X)$ 
is its quasi-inverse.
In particular, $(\varphi_*\O_X)^G\cong\O_Y$ in $\Ref(H,Y)$.
This equivalence also induces an isomorphism $\Cl(H,Y)\cong\Cl(G,X)$.
\qed
\end{corollary}

In the next theorem, consider that $G=N$ and $f:G\rightarrow H=e=S$ is the
trivial homomorphism.

\begin{theorem}\label{four-term.thm}
Let $G$ be a flat $S$-group scheme, and
let
\[
\xymatrix{
X & U \ar@{_{(}->}[l]_i \ar[r]^\rho & V \ar@{^{(}->}[r]^j & Y
}
\]
be a rational almost principal $G$-bundle.
Assume that both $X$ and $Y$ are locally Krull,
$X$ is quasi-compact quasi-separated, and $Y$ is quasi-compact.
Assume that
$\rho$ is quasi-compact
\(e.g., $G\rightarrow S$ is quasi-compact\)
and universally open
\(e.g., $G\rightarrow S$ or $\rho$ is locally of finite presentation\).
Then there is an exact sequence
\[
0\rightarrow H^1\alg(G,\O_X^\times)\rightarrow 
\Cl(Y)\rightarrow \Cl(X)^G\rightarrow H^2\alg(G,\O_X^\times),
\]
where $\Cl(X)^G$ is the subgroup 
of $\Cl(X)$ consisting of $[\M]$ with
$\M\in\Ref_1(X)^G$, where 
$\Ref_1(X)^G$ is the full subcategory of $\Qch(X)$
consisting of rank-one reflexive sheaves $\M$
such that $a^*\M\cong p_2^*\M$ in $\Ref_1(G\times X)$, where 
$a:G\times X\rightarrow X$ is the action, and $p_2:G\times X\rightarrow X$ 
is the second projection.
\end{theorem}

\begin{proof}
Let $\Cal C$ be the set of 
quasi-compact large open subsets of $V$.
Let $\Cal D_1$ be the set of quasi-compact large open 
subsets of $U$.
Let $\Cal D$ be the set of $G$-stable open subsets $Z$ of $U$ 
such that $Z\in\Cal D_1$.

First, for $Z_1\in \Cal D_1$, we have that $\rho(Z_1)\in \Cal C$.
Indeed, as $\rho$ is universally open, $\rho(Z_1)$ is an open subset of $V$.
As $Z_1$ is quasi-compact, so is $\rho(Z_1)$.
As $Z_1$ is large in $U$ and $\rho^{-1}(\rho(Z_1))\supset Z_1$, we have that 
$\rho^{-1}(\rho(Z_1))$ is also large in $U$.
As $U$ is locally Krull and $\rho$ is fpqc,
we have that $\rho(Z_1)$ is large in $V$ by
Lemma~\ref{codim-two-flat.thm}, {\bf 2}.
Thus $\rho(Z_1)\in\Cal C$.

Next, for $W\in \Cal C$, we have that $\rho^{-1}(W)\in\Cal D$.
As $\rho$ is a $G$-invariant morphism, $\rho^{-1}(W)$ is a $G$-stable
open subset of $U$.
As $\rho$ is quasi-compact, $\rho^{-1}(W)$ is quasi-compact.
By Lemma~\ref{codim-two-flat.thm}, {\bf 1}, we have that $\rho^{-1}(W)$ is
large in $U$, and hence $\rho^{-1}(W)\in \Cal D$.

As $\rho$ is a principal $G$-bundle, $\Psi:G\times U\rightarrow U\times_V U$
is a $U$-isomorphism, where $U\times_V U$ is a $U$-scheme via the second
projection.
As $\rho$ is quasi-compact, $U\times_V U$ is quasi-compact over $U$, and
hence so is $G\times U$.
Thus for each $Z_1\in\Cal D_1$, $G\times Z_1$ is quasi-compact.

For $Z_1\in \Cal D_1$, $Z:=\rho^{-1}(\rho(Z_1))$ lies in $\Cal D$ by the 
argument above.
Let $a_1:G\times Z_1\rightarrow Z$ be the action.
As $\Psi:G\times U\rightarrow U\times_V U$ is surjective
(since $\rho$ is a principal $G$-bundle), 
$a_1$ is surjective.
As $a_1$ is flat surjective, $G\times Z_1$ is quasi-compact, and
$Z$ is quasi-separated, we have that $a_1$ is fpqc.

Note that $Z\mapsto \rho(Z)$ and $W\mapsto \rho^{-1}(W)$ gives an
order-preserving bijection between $\Cal D$ and $\Cal C$.
Indeed, as $\rho$ is surjective, $\rho(\rho^{-1}(W))=W$.
On the other hand, for $Z\in\Cal D$, 
as $G\times Z\rightarrow \rho^{-1}(\rho(Z))$ is
surjective and $Z$ is $G$-stable, we have that $\rho^{-1}(\rho(Z))=Z$.

Note that 
\begin{equation}\label{W.eq}
\indlim_{Z\in\Cal D}\Pic(\rho(Z))\cong \indlim_{W\in\Cal C} \Pic(W)\cong \Cl(Y)
\end{equation}
by \cite[(5.33)]{Hashimoto4}.

Let $Z\in\Cal D$.
Then
$i^*:\Cl(X)\rightarrow \Cl(Z)$ induces an isomorphism between 
$\Cl(X)^G$ and $\Cl(Z)^G$, where $i:Z\hookrightarrow X$ is the inclusion.
Indeed, if $\M\in\Ref_1(X)^G$, then
\[
a^*_Zi^*\M\cong(1_G\times i)^*a^*\M
\cong
(1_G\times i)^*p_2^*\M
\cong
(p_2^Z)^*i^*\M, 
\]
and $i^*$ maps $\Cl(X)^G$ to $\Cl(Z)^G$.
Let $\N\in\Ref_1(Z)^G$.
As $i$ is quasi-compact quasi-separated and $a$ and $p_2$ are flat,
\[
a^*i_*\N\cong (1\times i)_*a_Z^*\N\cong
(1\times i)_*(p_2^Z)^*\N\cong
p_2^*i_*\N, 
\]
and $i_*\N\in\Ref_1(X)^G$.
So $(i^*)^{-1}=i_*$ maps $\Cl(Z)^G$ to $\Cl(X)^G$.

On the other hand, for any $\M\in\Ref_1(X)^G$, there exists some 
$Z\in\Cal D$ such that 
$\M|_Z$ is an invertible sheaf.
Indeed, first take a large open subset $Z_1$ of $U$ such that 
$\M|_{Z_1}$ is an invertible sheaf.
This is possible as in the proof of \cite[(5.33)]{Hashimoto4}.
By \cite[(5.29)]{Hashimoto4}, replacing $Z_1$ if necessary, we may assume
that $Z_1$ is quasi-compact, and $Z_1\in\Cal D_1$.
Let $Z=\rho^{-1}(\rho(Z_1))\in\Cal D$.
As $\M|_{Z_1}$ is an invertible sheaf, 
\[
p_2^*(\M|_{Z_1})\cong (p_2^*\M)|_{G\times Z_1}\cong (a^*\M)|_{G\times Z_1}
\cong a_1^*(\M|_Z)
\]
is also an invertible sheaf.
So $\M|_Z$ is also an invertible sheaf, since $a_1$ is fpqc as we have seen.

Combining these, we have that
\begin{equation}\label{X^G.eq}
\indlim_{Z\in \Cal D} \Pic(Z)^G\cong\Cl(X)^G.
\end{equation}

Next, we have that $H^i\alg(G,\O^\times_X)\rightarrow H^i\alg(G,\O^\times_Z)$ is 
an isomorphism for $Z\in \Cal D$.
In order to prove this, 
it suffices to prove the canonical chain map
\[
\xymatrix{
0 \ar[r] & \Gamma(X,\O^\times) \ar[d] \ar[r]^-{d_0-d_1} &
~\Gamma(G\times X,\O^\times)~ \ar[d] \ar[r]^-{d_0-d_1+d_2} &
~\Gamma(G\times G\times X,\O^\times)~ \ar[d]\ar[r] &
\cdots \\
0 \ar[r] & \Gamma(Z,\O^\times) \ar[r]^-{d_0-d_1} &
~\Gamma(G\times Z,\O^\times)~ \ar[r]^-{d_0-d_1+d_2} &
~\Gamma(G\times G\times Z,\O^\times)~ \ar[r] &
\cdots
}
\]
is a chain isomorphism.
To verify this, it suffices to prove that the canonical restriction
$\Gamma(G^i\times X,\O_X)\rightarrow \Gamma(G^i\times Z,\O_Z)$ is an
isomorphism.
Let $i:Z\hookrightarrow X$ be the inclusion.
Then as $Z$ is large in $X$, 
$
\O_X\rightarrow i_*\O_Z
$ is an isomorphism.
As $G^i$ is flat over $S$ and $i$ is quasi-compact quasi-separated,
\[
\O_{G^i\times X}\cong p_2^*\O_X\cong p_2^*i_*\O_Z \cong
(1\times i)_*p_2^*\O_Z\cong (1\times i)_*\O_{G^i\times Z}.
\]
Taking the global section, we get the desired isomorphism.

Thus we have proved that the canonical map 
$H^i\alg(G,\O_X^\times)\rightarrow H^i\alg(G,\O_Z^\times)$ is an isomorphism.
In particular, we have that 
\begin{equation}\label{algebraic-G-cohomology.eq}
H^i\alg(G,\O_X^\times)\cong \indlim_Z H^i\alg(G,\O_{Z}^\times).
\end{equation}

By \cite[(3.14)]{Hashimoto4}, there is an exact sequence
\[
0\rightarrow 
H^1\alg(G,\O_Z^\times)
\rightarrow 
\Pic(G,Z)\rightarrow
\Pic(Z)^G
\rightarrow 
H^2\alg(G,\O_Z^\times).
\]
The proof of \cite[(3.14)]{Hashimoto4} shows that the sequence is
functorial on $Z$.
On the other hand, there is a natural isomorphism
$\Pic(G,Z)\cong \Pic(\rho(Z))$.
This is obvious by \cite[(3.13)]{Hashimoto}.
Taking the inductive limit $\indlim_{Z\in\Cal D}$, 
and using the isomorphisms
(\ref{W.eq}), (\ref{X^G.eq}), and (\ref{algebraic-G-cohomology.eq}),
\[
0\rightarrow H^1\alg(G,\O_X^\times)\rightarrow \Cl(Y)
\rightarrow \Cl(X)^G\rightarrow H^2\alg(G,\O_X^\times)
\]
is exact, as desired.
\end{proof}

\paragraph
Let $f:G\rightarrow H$ be an fpqc homomorphism between flat $S$-group schemes
with $N=\Ker f$.

Let $X$ be a locally Noetherian $G$-scheme.
We denote 
the full subcategory of $\Coh(G,X)$ consisting of $\M\in\Coh(G,X)$ 
which satisfy the $(S'_n)$ condition as an $\O_X$-modules
by $(S'_n)(G,X)$.

\begin{lemma}\label{S_2-equiv-equivariant.lem}
Let $X$ be as above, and $U$ a large open subset of $X$.
If $X$ has a full $2$-canonical module, then
$i_*:(S'_2)(G,U)\rightarrow (S'_2)(G,X)$ is an equivalence
whose quasi-inverse is
$i^*:(S'_2)(G,X)\rightarrow (S'_2)(G,U)$.
\end{lemma}

\begin{proof}
Follows easily from Lemma~\ref{S_2-large-equiv.lem}.
\end{proof}

\begin{proposition}\label{S_2-main.prop}
Let $f:G\rightarrow H$ be an fpqc 
homomorphism between flat $S$-group schemes, and
\[
\xymatrix{
X & U \ar@{_{(}->}[l]_i \ar[r]^\rho & V \ar@{^{(}->}[r]^j & Y
}
\]
be a $G$-enriched rational almost principal $N$-bundle.
Assume that $X$ and $Y$ are locally Noetherian, and 
have full $2$-canonical modules
\(e.g., they are normal; or Noetherian locally equidimensional
and have dualizing complexes\).
If $\rho$ has $(S_2)$ fibers (e.g., $N$ is of finite type or $X$ is $(S_2)$),
then $i_*\rho^*j^*:(S'_2)(H,Y)\rightarrow (S'_2)(G,X)$ 
is an equivalence,
and $(j_*\rho_*i^*?)^N$ is its quasi-inverse.
\end{proposition}

\begin{proof}
By Lemma~\ref{S_2-equiv-equivariant.lem}, $i_*:(S'_2)(G,U)\rightarrow 
(S'_2)(G,X)$ is an equivalence with the quasi-inverse 
$i^*$.
Similarly, $j_*:(S'_2)(H,V)\rightarrow (S'_2)(H,Y)$ is an
equivalence with the quasi-inverse $j^*$.
In view of \cite[(6.21)]{Hashimoto5}, it suffices to show that
for $\M\in\Coh(V)$, 
$\rho^*\M$ satisfies $(S'_2)$ if and only if $\M$ does.
This is proved easily using (\ref{S_n-ascent-descent.lem}).
\end{proof}

\begin{lemma}\label{rank-one-isom.lem}
Let $h:A\rightarrow B$ be a ring homomorphism, and assume that $B$ is
rank-one free as an $A$-module.
Then $h$ is an isomorphism.
\end{lemma}

\begin{proof}
Note that $\Ker h=\ann B=\ann A=0$, and $h$ is injective.
So it suffices to show that $h$ is surjective.
So we may assume that $(A,\fm)$ is local.
By Nakayama's lemma, we may assume that $A$ is a field.
Then $A=B$, since $\dim_AA=\dim_A B=1$.
\end{proof}

\begin{lemma}\label{almost-double-dual-S_2.lem}
Let $\varphi:X\rightarrow Y$ be a $G$-enriched almost principal
$N$-bundle with respect to the open subsets $U$ and $V$.
Let $i:U\rightarrow X$ and $j:V\rightarrow Y$ be the inclusion, and
$\rho:U\rightarrow V$ the restriction of $\varphi$.
Assume that $X$ and $Y$ are locally Noetherian.
\begin{enumerate}
\item[\bf 1] 
The functor $(j_*\rho_*i^*?)^N:
(S'_2)(G,X)\rightarrow \Qch(H,Y)$ agrees with $(\varphi_*?)^N$.
\item[\bf 2] If $\M_X$ is a coherent $(G,\O_X)$-module
which is a full $2$-canonical module as an $\O_X$-module, 
then the functor $i_*\rho^*j^*:(S'_2)(H,Y)\rightarrow (S'_2)(G,X)$ 
agrees with $(?)^{\vee\vee}\varphi^*$, where $(?)^\vee=\uHom_{\O_X}(?,\M_X)$.
\item[\bf 3] Assume that both $X$ and $Y$ are quasi-normal.
If $N$ is of finite type or $X$ is $(S_2)$, 
then $(\varphi_*?)^N:(S'_2)(G,X)\rightarrow (S'_2)(H,Y)$ 
is an equivalence whose quasi-inverse is $(?)^{\vee\vee}\varphi^*$.
\item[\bf 4] In {\bf 3}, if 
$X$ and $Y$ satisfy $(S_2)$, then $\bar\eta:
\O_Y\rightarrow(\varphi_*\O_X)^N$ is an isomorphism.
\item[\bf 5] If either $X$ and $Y$ satisfy $(T_1)+(S_2)$; or
$N$ is of finite type and 
$X$ and $Y$ are Noetherian $(S_2)$ with dualizing complexes,
then $\bar\eta$ in {\bf 4} is an isomorphism.
\end{enumerate}
\end{lemma}

\begin{proof}
{\bf 1} and {\bf 2} are proved similarly to 
Lemma~\ref{almost-double-dual.lem}, and is left to the reader.
{\bf 3} is immediate by {\bf 1} and Proposition~\ref{S_2-main.prop}.

{\bf 4}.
As $X$ satisfies $(S_2)$, we have that 
$\O_X\cong i_*i^*\O_X\cong i_*\O_U\cong i_*\rho^*j^*\O_Y$.
As $\O_Y\in(S'_2)(H,Y)$,
we have that $(\varphi_*\O_X)^N\cong \O_Y$ as $(H,\O_Y)$-modules
by {\bf 3}.
By Lemma~\ref{rank-one-isom.lem}, $\bar\eta$ is an isomorphism.

{\bf 5} is immediate by {\bf 3} and {\bf 4}, 
in view of Corollary~\ref{S_2-dualizing-quasinormal.cor}.
\end{proof}

\paragraph
An $S$-group scheme $G$ is said to be {\em locally finite free} 
({\em LFF} for short), if 
the structure map $h_G:G\rightarrow S$ is finite and $(h_G)_*\O_G$ is 
locally free.
Using the results of \cite[(10.129)]{SP}, it is not so difficult to
show that $G$ is LFF if and only if it is flat finite of finite presentation
(hence is finite syntomic by \cite[(31.14)]{ETI}).

\begin{lemma}\label{DG.thm}
Let $G$ be an LFF $S$-group scheme,
and $\psi:X\rightarrow Y$ be an algebraic quotient by the action of $G$.
Then $\psi$ is a surjective integral universally open morphism which is a  
universally submersive geometric quotient.
If the action of $G$ on $X$ is free, it is a principal $G$-bundle.
\end{lemma}

\begin{proof}
Replacing $Y$ by its affine open subset, we may assume that $Y$ is affine.
Then $X$ is affine, since $\psi$ is assumed to be affine.
Now $\psi$ is integral and the map $\Psi:G\times X\rightarrow X\times_Y X$ 
is surjective by \cite[(III.\S 2,n$\!{}^\circ$4)]{DG}.
By the same theorem, $\psi$ is a principal $G$-bundle if the action is free.

Being an algebraic quotient, it is dominating.
Being integral, it is universally closed.
Being dominating and closed, it is surjective.
Being surjective and universally closed, it is universally submersive.
By Lemma~\ref{GIT-remark.lem}, it is universally open.
Now it is clear that $\varphi$ is a geometric quotient.
\end{proof}

\paragraph
The following generalizes \cite[(32.4), {\bf 3}]{ETI}.

\begin{proposition}\label{linearly-reductive-omega.thm}
Let $f:G\rightarrow H$ be as in {\rm(\ref{initial-canonical-settings.par})}.
Assume that $N$ is finite and Reynolds.
Let $Y_0$, $\Bbb I_{Y_0}$, and $\Cal F(G,Y_0)$ be as in
{\rm(\ref{canonical-Y_0.par})}.
Let $\varphi:X\rightarrow Y$ be a morphism in $\Cal F(G,Y_0)$.
Assume that it is also an algebraic quotient by the action of $N$.
Then $\varphi$ is finite, 
and $(\varphi_*\omega_X)^N\cong \omega_Y$ as $(G,\O_Y)$-modules.
\end{proposition}

\begin{proof}
As $N$ is finite flat and $Y_0$ is Noetherian, the $Y_0$-group scheme 
$N\times_S Y_0$ is LFF.
So $\varphi$ is  integral by Lemma~\ref{DG.thm}.
Being a morphism in $\Cal F(G,Y_0)$, $\varphi$ is of finite type.
So $\varphi$ is finite.

Let $s:=\inf\{i\mid H^i(\Bbb I_Y)\neq 0\}$.
We may assume that $\Bbb I_Y=\Bbb I_Y(G)$ 
consists of injective $\O_{B_G^M(Y)}$-modules,
and $\Bbb I_Y^i=0$ for $i<s$.

Then by definition, $\omega_Y=H^s(\Bbb I_Y)$.
Let $\Bbb I_X=\rho^!(\Bbb I_Y)$ be the $G$-equivariant dualizing complex of $X$.
We may assume that $\Bbb I_X$ is bounded below and 
consists of injective $\O_{B_G^M(X)}$-modules.
Then using $G$-Grothendieck's duality \cite[(29.5)]{ETI},
\begin{multline*}
(\varphi_*\Bbb I_X)^N=(R\varphi_*R\uHom_{\O_X}(\O_X,\varphi^!\Bbb I_Y))^N
\cong
(R\uHom_{\O_Y}(R\varphi_*\O_X,\Bbb I_Y))^N\\
\cong 
\uHom_{\O_Y}(\varphi_*\O_X,\Bbb I_Y)^N
\cong
\uHom_{\O_Y}((\varphi_*\O_X)^N\oplus U_N(\varphi_*\O_X),\Bbb I_Y)^N.
\end{multline*}
By Lemma~\ref{Reynolds-four-operations.lem}, {\bf 4}
and Corollary~\ref{almost-main.thm}, this is 
\[
\uHom_{\O_Y}((\varphi_*\O_X)^N,\Bbb I_Y)\cong \uHom_{\O_Y}(\O_Y,\Bbb I_Y)
\cong \Bbb I_Y.
\]
If $i<s$, then
\[
\varphi_* H^i(\Bbb I_X)\cong \uExt^i_{\O_Y}(\varphi_*\O_X,\Bbb I_Y)=0.
\]
As $\varphi_*:\Qch(X)\rightarrow\Qch(Y)$ is faithful and
$\varphi_*\omega_X$ is nonzero, we must have $\omega_X=H^s(\Bbb I_X)$, and
$(\varphi_*\omega_X)^N=\omega_Y$.
\end{proof}

\paragraph
A Noetherian local ring $A$ is said to be {\em quasi-Gorenstein} 
if $A$ is the
canonical module of $A$.
A locally Noetherian scheme is said to be quasi-Gorenstein if all of its 
local rings are quasi-Gorenstein.

\begin{lemma}\label{quasi-Gorenstein.lem}
Let $Z$ be a locally Noetherian scheme.
Then the following are equivalent.
\begin{enumerate}
\item[\bf 1] $Z$ is quasi-Gorenstein.
\item[\bf 2] $\O_Z$ is a semicanonical module of $Z$.
\item[\bf 3] There exists some invertible sheaf on $Z$ 
which is also a semicanonical module of $Z$.
\item[\bf 4] A coherent sheaf on $Z$ is an invertible sheaf if and only if
it is a full semicanonical module.
\end{enumerate}
\end{lemma}

\begin{proof}
Trivial.
\end{proof}

\begin{lemma}\label{q-Gor-q-nor.lem}
A quasi-Gorenstein locally Noetherian scheme is quasi-normal by $\O_Z$.
\end{lemma}

\begin{proof}
Follows immediately by Lemma~\ref{quasi-Gorenstein.lem} and
Lemma~\ref{omega-S_2.lem}.
\end{proof}

\begin{theorem}\label{canonical.thm}
Let $f:G\rightarrow H$ and $N$ 
be as in {\rm(\ref{initial-canonical-settings.par})}.
$Y_0$, $\Bbb I_{Y_0}$, and $\Cal F(G,Y_0)$ be as in
{\rm(\ref{canonical-Y_0.par})}.
Let
\[
\xymatrix{
X & U \ar@{_{(}->}[l]_i \ar[r]^{\rho} & V \ar@{^{(}->}[r]^j & Y
}
\]
be a $G$-enriched rational almost principal $N$-bundle
which is also a diagram in $\Cal F(G,Y_0)$.
Assume that $N$ is separated and has a fixed relative dimension.
Then there exist isomorphisms of $(G,\O_X)$-modules
\begin{equation}\label{omega_X-isoms.eq}
\omega_X
\cong i_*\rho^*j^*\omega_Y\otimes_{\O_X}\Theta_{N,X}^*
\cong i_*\rho^*j^*(\omega_Y\otimes_{\O_Y}\Theta_{N,Y}^*)
\end{equation}
and isomorphisms  of $(H,\O_Y)$-modules
\begin{equation}\label{omega_Y-isoms.eq}
\omega_Y\cong (j_*\rho_*i^*(\omega_X\otimes_{\O_X}\Theta_{N,X}))^N
\cong ((j_*\rho_*i^*\omega_X)\otimes_{\O_Y}\Theta_{N,Y})^N.
\end{equation}
\end{theorem}

\begin{proof}
We prove the first isomorphism of (\ref{omega_X-isoms.eq}).
By Lemma~\ref{codim-two-canonical.thm}, $j^*\omega_Y\cong \omega_V$, 
and hence we may assume that $V=Y$.
As $i_*\rho^*\omega_Y\otimes_{\O_X}\Theta_{N,X}^*
\cong
i_*(\rho^*\omega_Y\otimes_{\O_U}\Theta_{N,U}^*)
$ 
by the equivariant projection formula \cite[(26.4)]{ETI}
and $i_*\omega_U\cong \omega_X$ by
Lemma~\ref{codim-two-canonical.thm}, we may assume that $U=X$.
Now the assertion follows from Corollary~\ref{canonical-principal.thm}.
The second isomorphism follows easily using the equivariant
projection formula.
We prove 
the first isomorphism of (\ref{omega_Y-isoms.eq}).
As $(?)^N\circ j_*\cong j_* \circ (?)^N$ by \cite[(7.3)]{HO2}
and $j_*\omega_V\cong\omega_Y$, we may assume that $Y=V$.
As $i^*(\omega_X\otimes_{\O_X}\Theta_{N,X})
\cong \omega_U\otimes_{\O_U}\Theta_{N,U}$, we may assume that $X=U$.
Now the assertion follows from 
Corollary~\ref{canonical-principal.thm}.
The second isomorphism follows easily from the equivariant projection formula.
\end{proof}

\begin{corollary}[Watanabe type theorem]\label{abstract-Watanabe.thm}
Let the assumptions be as in {\rm Theorem~\ref{canonical.thm}}.
Then for an $H$-linerized invertible sheaf $\L$ on $Y$,
the following conditions are equivalent.
\begin{enumerate}
\item[\bf a] $\omega_X\cong i_*\rho^*j^*\L\otimes_{\O_X}\Theta_{N,X}^*
\cong i_*\rho^*j^*(\L\otimes_{\O_Y}\Theta_{N,Y}^*)$ 
\(resp.\ $\omega_X\cong i_*\rho^*j^*\L$\) 
in $\Qch(G,X)$, and $Y$ satisfies the $(S_2)$ condition.
\item[\bf b] $\omega_Y\cong\L$ in $\Qch(H,Y)$.
\end{enumerate}
If, moreover, 
$\Theta_{N,X}$ is trivial, then the following are
equivalent.
\begin{enumerate}
\item[\bf c]
$\omega_X\cong \O_X$ and $Y$ is $(S_2)$
\item[\bf d]
$\omega_Y\cong \O_Y$ and $X$ is $(S_2)$.
\end{enumerate}
If these conditions are satisfied, then both $X$ and $Y$ are quasi-Gorenstein.
\end{corollary}

\begin{proof}
{\bf a$\Rightarrow$b}.
We easily have that $\omega_Y\cong j_*j^*\L$ by Theorem~\ref{canonical.thm}
and the assumption.
As $Y$ is $(S_2)$, $\L\cong j_*j^*\L$, and $\omega_Y\cong\L$.

{\bf b$\Rightarrow$a}.
As the semicanonical module $\omega_Y$ is an invertible sheaf,
$Y$ is quasi-Gorenstein.
In particular, $Y$ is $(S_2)$.
The isomorphisms 
follow from Theorem~\ref{canonical.thm} and the assumption.

Now assume that $\Theta_{N,X}$ is trivial.

{\bf c$\Rightarrow$d}.
As the semicanonical module $\omega_X$ is invertible, we have that
$X$ is quasi-Gorenstein and $(S_2)$.
So $\O_X\cong i_*i^*\O_X\cong i_*\O_U\cong i_*\rho^*j^*\O_Y$.
By {\bf a$\Rightarrow$b} above, we have that $\omega_Y\cong\O_Y$.

{\bf d$\Rightarrow$c}.
Then $Y$ is quasi-Gorenstein, and so $Y$ is $(S_2)$.
By {\bf b$\Rightarrow$a} above, we have that
$\omega_X\cong i_*\rho^*j^*\O_Y\cong i_*i^*\O_X\cong \O_X$.
\end{proof}

\begin{corollary}\label{quasi-Gor.cor}
Let the assumptions be as in {\rm Theorem~\ref{canonical.thm}}.
Then the following are equivalent.
\begin{enumerate}
\item[\bf a] $\omega_X\cong i_*\rho^*j^*\L$ for some $H$-linearized 
invertible sheaf $\L$ on $Y$, and $Y$ satisfies the $(S_2)$ condition.
\item[\bf b] $\omega_Y$ is an invertible sheaf.
\end{enumerate}
These conditions imply
\begin{enumerate}
\item[\bf c] $Y$ is quasi-Gorenstein.
\end{enumerate}
If, moreover, $Y$ is connected, then {\bf a, b, c} are equivalent.
\end{corollary}

\begin{proof}
{\bf a$\Rightarrow$b}.
We have
\[
\omega_X\cong i_*\rho^*j^*\L\cong i_*\rho^*j^*(\L\otimes_{\O_Y}\Theta_{N,Y})
\otimes_{\O_X}\Theta_{N,X}^*,
\]
and hence $\omega_Y\cong \L\otimes_{\O_Y}\Theta_{N,Y}$ is an invertible sheaf
by 
Corollary~\ref{abstract-Watanabe.thm}.

{\bf b$\Rightarrow$a,c}.
Letting $\L=\omega_Y\otimes_{\O_Y}
\Theta_{N,Y}^*$, we have
that $\L$ is an $H$-linearized invertible sheaf on $Y$.
We have 
$\omega_X\cong i_*\rho^*j^*\L$ by Corollary~\ref{abstract-Watanabe.thm}.
As $\omega_Y$ is an invertible sheaf, $Y$ is quasi-Gorenstein, and
hence is $(S_2)$.

Now assume that $Y$ is connected and {\bf c} is satisfied.
Then $Y$ is $(S_2)$ and has a dualizing complex.
Hence it is locally equidimensional by Ogoma's theorem \cite{Ogoma}.
By Lemma~\ref{ogoma.lem}, the semicanonical module $\omega_Y$ is full.
As $Y$ is quasi-Gorenstein, $\omega_Y$ is an invertible sheaf and
so {\bf c$\Rightarrow$b} holds.
\end{proof}

\begin{remark}\label{knop-remark.thm}
Let the assumptions be as in Theorem~\ref{canonical.thm}.
In view of Corollary~\ref{abstract-Watanabe.thm},
it is important to know when $\Theta_{N,Y_0}\cong \O_{Y_0}$ 
holds.
\begin{enumerate}
\item[\bf 1] If $N$ is \'etale, then $\Theta_{N,Y_0}\cong \O_{Y_0}$,
since $\Theta_{N,Y_0}\cong q_{Y_0}^*(\ext^0\Omega_{N/S}^*)$, where $q_{Y_0}:Y_0
\rightarrow S$ is the structure map.
\item[\bf 2] If $N$ is finite and Reynolds, then
$\Theta_{N,Y_0}\cong \O_{Y_0}$, see
Lemma~\ref{Reynolds-Theta.lem}.
\item[\bf 3] If
$G=N$ and $N$ is split reductive, then $\Theta_{N,Y_0}\cong\O_{Y_0}$.
To verify this, as $G$ is defined over $\Bbb Z$, we may assume that
$S=Y_0=\Bbb Z$.
As the positive and the negative roots cancel out in $\extop\Lie G$,
$\Theta$ is a rank-one free representation whose weight is zero.
Similarly, if $S=\Spec k$ with $k$ a field, $G=N$, and $N$ is reductive, 
then $\Theta_{N,S}=\O_S$.
\item[\bf 3] If $S=\Spec k$ with $k$ a field and 
$N$ is contained in the center of $G$, then the action of $G$ on 
$N$ is trivial, and hence $\omega_{N/S}$ is $G$-trivial.
Hence $\Theta$ is a $G$-trivial one-dimensional representation, that is,
$\Theta\cong k$.
\item[\bf 4] Even if $S=Y_0=\Spec k$, $G=N$, and 
the identity component $N^\circ$ of $N$ is reductive, 
$\Theta$ may not be trivial.
For example, if the characteristic of $k$ is not two and 
$N=O(2)$, the orthogornal group, then $\Theta$ is not trivial,
see \cite[Bemerkung~4 after Korollar~2]{Knop}.
\end{enumerate}
\end{remark}

\begin{corollary}\label{canonical-cor.thm}
Let $f:G\rightarrow H$ and $N$ 
be as in {\rm(\ref{initial-canonical-settings.par})}.
$Y_0$, $\Bbb I_{Y_0}$, and $\Cal F(G,Y_0)$ be as in
{\rm(\ref{canonical-Y_0.par})}.
Let
$\varphi:X\rightarrow Y$ be a $G$-enriched 
almost principal $N$-bundle which is also a morphism in $\Cal F(G,Y_0)$.
Assume that $N$ is separated with a fixed relative dimension.
Then we have the following.
\begin{enumerate}
  \item[\bf 1]
There exist 
isomorphisms of $(H,\O_Y)$-modules
\begin{equation}\label{knop-almost-1.eq}
\omega_Y\cong (\varphi_*(\omega_X\otimes_{\O_X}\Theta_{N,X}))^N
\cong (\varphi_*\omega_X\otimes_{\O_Y}\Theta_{N,Y})^N.
\end{equation}
If, moreover, $X$ 
has a coherent $(G,\O_X)$-module $\M_X$ which is a full $2$-canonical module,
then there exist
isomorphisms of $(G,\O_X)$-modules
\begin{equation}\label{knop-almost-2.eq}
\omega_X\cong (\varphi^*\omega_Y)^{\vee\vee}\otimes_{\O_X}\Theta_{N,X}^*
\cong (\varphi^*(\omega_Y\otimes_{\O_Y}\Theta_{N,Y}^*))^{\vee\vee},
\end{equation}
where $(?)^\vee=\uHom_{\O_X}(?,\M_X)$.
\item[\bf 2] {\rm (Watanabe type theorem)}
Let $\L$ be an $H$-linearized invertible sheaf on $Y$.
Then the following are equivalent.
\begin{enumerate}
\item[\bf a] $\omega_X\cong \varphi^*\L\otimes_{\O_X}\Theta_{N,X}^*
\cong
\varphi^*(\L\otimes_{\O_Y}\Theta_{N,Y}^*)$ \(resp.\ 
$\omega_X\cong\varphi^*\L$\),
and $Y$ satisfies $(S_2)$.
\item[\bf b] $\omega_Y\cong \L$, and $X$ satisfies $(S_2)$.
\end{enumerate}
\item[\bf 3]
The following are equivalent.
\begin{enumerate}
\item[\bf a] $\omega_X\cong\varphi^*\L$ for some $H$-linearized 
invertible sheaf $\L$ on $Y$, and $Y$ satisfies the $(S_2)$ condition.
\item[\bf b] $\omega_Y$ is an invertible sheaf on $Y$, and $X$ satisfies 
the $(S_2)$ condition.
\end{enumerate}
These conditions imply that both $X$ and $Y$ are 
quasi-Gorenstein, and hence we have
\begin{enumerate}
\item[\bf c] $Y$ is quasi-Gorenstein and $X$ satisfies the $(S_2)$ condition.
\end{enumerate}
If, moreover, $Y$ is connected, then {\bf a, b, c} are equivalent.
\end{enumerate}
\end{corollary}

\begin{proof}
Let $\varphi:X\rightarrow Y$ be a $G$-enriched almost principal
bundle with respect to $U$ and $V$.
Let $i:U\rightarrow X$ and $j:V\rightarrow V$ be the inclusion,
and $\rho:U\rightarrow V$ be the restriction of $\varphi$.

{\bf 1}.
As $\omega_X\otimes_{\O_X}\Theta_{N,X}$ satisfies the $(S'_2)$-condition
by Lemma~\ref{omega-S_2.lem},
the first isomorphism of (\ref{knop-almost-1.eq})
  is immediate from (\ref{omega_Y-isoms.eq}) in 
Theorem~\ref{canonical.thm}, {\bf 1} and 
Lemma~\ref{almost-double-dual-S_2.lem},
{\bf 1}.
The second isomorphism is by the equivariant projection formula
\cite[(26.4)]{ETI}.

The first isomorphism of (\ref{knop-almost-2.eq})
follows from (\ref{omega_X-isoms.eq}) in
Theorem~\ref{canonical.thm}, {\bf 1} and
Lemma~\ref{almost-double-dual-S_2.lem},
{\bf 2}.
The second isomorphism follows easily by the equivariant projection formula.

{\bf 2}.
If {\bf a} is assumed, then 
the semicanonical module $\omega_X$ is an invertible sheaf, and hence 
$X$ is
quasi-Gorenstein.
In particular, $X$ satisfies $(S_2)$.
If {\bf b} is assumed, $X$ is $(S_2)$ by assumption.
So in either case, we have
\[
\varphi^*\L
\cong
i_*i^*\varphi^*\L
\cong
i_*\rho^*j^*\L.
\]
By Corollary~\ref{abstract-Watanabe.thm}, the assertion follows.

{\bf 3} follows easily from {\bf 2}.
\end{proof}

\section{Frobenius pushforwards}\label{Frobenius-pushforwards.sec}

\begin{lemma}\label{finite-depth.lem}
Let $A\rightarrow B$ be a homomorphism between Noetherian rings
whose fibers are zero-dimensional \(e.g., an integral homomorphism\),
and $M$ a \(possibly infinite\) $B$-module.
Then for a prime ideal $\fp$ of $A$, 
\[
\depth_{A_\fp} M_\fp = \inf_{P\cap A=\fp}\depth_{B_P}M_P.
\]
If, moreover, the going-down theorem holds between $A$ and $B$, and
$M$ satisfies the $(S'_n)$ condition as a $B$-module, 
then $M$ satisfies the $(S'_n)$ condition as an $A$-module.
\end{lemma}

\begin{proof}
We have
\[
H^i_{\fp A_\fp}M_\fp=
H^i_{\fp B_\fp}M_\fp
=\bigoplus_{P\cap A=\fp}H^i_{PB_P}M_P,
\]
and the first assertion follows.

We prove the second assertion.
Let $\fp$ be a prime ideal of $A$ such that $\depth_{A_\fp}M_\fp<n$.
Then there exists some $P\in\Spec B$ such that $P\cap A=\fp$ and
$\depth_{B_P}M_P=\depth_{A_\fp}M_\fp<n$.
So $\depth_{B_P}M_P=\dim B_P$ by the $(S'_n)$ property as a $B$-module.
As the fibers are zero dimensional, $\dim A_\fp\geq \dim B_P$.
By the going-down, $\dim A_\fp\leq \dim B_P$.
Hence $\dim A_\fp=\dim B_P=\depth_{B_P}M_P=\depth_{A_\fp}M_\fp$,
and $M$ satisfies the $(S'_n)$ condition as an $A$-module.
\end{proof}

\begin{lemma}\label{invariance-S_2.lem}.
Let $G$ be a flat $S$-group scheme which is quasi-compact over $S$,
and $X$ be a locally Noetherian $S$-scheme on which $G$ acts trivially.
Let $\M$ be a quasi-coherent $(G,\O_X)$-module which satisfies the
$(S'_2)$ condition.
Then $\M$ satisfies the $(S'_2)$ condition.
\end{lemma}

\begin{proof}
This is proved in the same line as \cite[(5.34)]{Hashimoto4}.
\end{proof}

\paragraph
Until the end of this 
section, $S$ is an $\Bbb F_p$-scheme, where $p$ is a prime number, and
$\Bbb F_p$ is the prime field of characteristic $p$,
unless otherwise specified.

\begin{lemma}\label{F-finite.thm}
Let 
\[
\xymatrix{
X & U \ar@{_{(}->}[l]_i \ar[r]^{\rho} & V \ar@{^{(}->}[r]^j & Y
}
\]
be a diagram of $S$-schemes.
Assume that $i$ and $j$ are open immersions whose images are large in
$X$ and $Y$, respectively.
Assume that $Y$ is locally Noetherian and 
$(S_2)$, and ${}^eS\times_S Y$ is locally 
Noetherian with a full $2$-canonical module
for some $e\geq 1$.
Let $X$ be locally Noetherian, and assume that $\rho$ is faithfully flat
and reduced \(that is, flat with geometrically reduced fibers\).
Assume that ${}^{e'}S\times_S X$ is locally Noetherian 
for some $e'\geq 1$.
If $X$ is $F$-finite over $S$
\(that is, $\Phi_1(X)$ is a finite morphism, see {\rm\cite{Hashimoto3}}\), 
then $Y$ is $F$-finite over $S$.
\end{lemma}

\begin{proof}
As the open immersion $i:U\rightarrow X$ is $F$-finite, 
$U$ is also $F$-finite over $S$.
As $\rho:U\rightarrow V$ is faithfully flat and reduced, 
it is easy to see that $V$ is also $F$-finite over $S$ by
\cite[Theorem~21]{Hashimoto3}.
So for each $e\geq 0$, $\Phi_e(V)_*(\O_{{}^eV})$ is coherent.
As $Y$ satisfies the $(S_2)$ condition, ${}^eV$ satisfies the $(S_2)$ condition.
Hence $\Phi_e(V)_*(\O_{{}^eV})$ satisfies the $(S'_2)$ condition by 
Lemma~\ref{finite-depth.lem}.

Now take $e\geq 1$ so that
${}^eS\times_S Y$ is locally
Noetherian with a full $2$-canonical module.
As $V$ is large in $Y$ and ${}^eS\times_S V$ is
large in ${}^eS\times_S Y$,
\[
\Phi_e(Y)_*(\O_{{}^eY})\cong 
\Phi_e(Y)_*{}^ej_*(\O_{{}^eV})\cong
(1_{{}^eS}\times j)_*\Phi_e(V)_*(\O_{{}^eV})
\]
is coherent.
As $\Phi_e(Y)$ is affine, it is finite.
By \cite[Lemma~2]{Hashimoto3}, $Y$ is $F$-finite over $S$.
\end{proof}

\begin{lemma}\label{Eakin-Nagata.thm}
Let $A\rightarrow B$ be an $F$-finite reduced homomorphism between
Noetherian rings of characteristic $p$.
Then ${}^eA\otimes_A B$ is Noetherian for any $e\geq 1$.
\end{lemma}

\begin{proof}
By Dumitrescu's theorem \cite{Dumitrescu}, the relative Frobenius map
$\Phi_e(A,B):{}^eA\otimes_A B\rightarrow {}^eB$ is ${}^eA$-pure.
In particular, it is injective.
It is also finite by assumption.
By Eakin--Nagata theorem \cite[Theorem~3.7]{CRT}, the assertion follows.
\end{proof}

\begin{theorem}\label{Frobenius-pushforward.thm}
Let $f:G\rightarrow H$ be a qfpqc homomorphism between $S$-group schemes
with $N=\Ker f$.
Let $S$ be an $\Bbb F_p$-scheme, 
and assume that $S$ is locally Noetherian and quasi-normal by a
full $2$-canonical module $\M_S$.
Let
\[
\xymatrix{
X & U \ar@{_{(}->}[l]_i \ar[r]^{\rho} & V \ar@{^{(}->}[r]^j & Y
}
\]
be a $G$-enriched rational almost principal $N$-bundle.
Assume that $X$ and $Y$ are locally Noetherian and
flat with $(R_0)+(T_1)+(S_2)$-fibers over $S$, 
and that $X$ is $F$-finite over $S$.
Then
\begin{enumerate}
\item[\bf 1]
${}^eS\times_S X$ is locally Noetherian and quasi-normal by $p_X^*{}^e\M_S$
for each $e\geq 0$,
where $p_X:{}^eS\times_S X\rightarrow {}^eS$ is the first projection.
\item[\bf 2] If $N$ is reduced over $S$, then $Y$ is $F$-finite over $S$.
\item[\bf 3] If $Y$ is $F$-finite over $S$, then 
${}^eS\times_SY$ is locally Noetherian and quasi-normal 
by $p_Y^*{}^e\M_S$ for $e\geq 0$,
where $p_Y:{}^eS\times_S Y\rightarrow {}^eS$ is the first projection.
\item[\bf 4] 
For each $e\geq 0$, 
\[
\xymatrix{
{}^eS\times_S X & {}^eS \times_S U 
\ar@{_{(}->}[l]_{1\times i} \ar[r]^{1\times \rho} & {}^eS\times_S V 
\ar@{^{(}->}[r]^{1\times j} & {}^eS\times_S Y
}
\]
is an ${}^eH\times_H G$-enriched 
rational almost principal ${}^eS\times_S N$-bundle \(where the base
scheme is ${}^eS$, and not $S$\).
\item[\bf 5]
Assume further that $G$ is flat over $S$,
and $f$ is regular \(that is, flat with geometrically regular fibers\).
If, moreover, either $N$ is of finite type; or
${}^eS\times_S X$ and ${}^eS\times_SY$ satisfy $(T_1)+(S_2)$, then
$(S'_2)({}^eH\times_H G,{}^eS\times_S X)$ and 
$(S'_2)({}^eH,{}^eS\times_S Y)$ are equivalent under the equivalence in
{\rm Proposition~\ref{S_2-main.prop}}.
\item[\bf 6] Let the assumptions be as in {\bf 5}.
Let $\M\in (S'_2)(G,X)$, and let $\N$ be the corresponding
sheaf $(j_*\rho_*i^*\M)^N\in(S'_2)(H,Y)$ \(by the correspondence in
{\rm Proposition~\ref{S_2-main.prop}}\).
Then for each $e\geq 0$, ${}^eN_e$ is ${}^eS$-flat and the sheaf
\[
(1_{{}^eS}\times i)_*(1_{{}^eS}\times\rho)^*(1_{{}^eS}\times j)^*
\Phi_e(Y)_*({}^e\N)\in(S'_2)({}^eH\times_H G,{}^eS\times_S X), 
\]
which corresponds to the Frobenius pushforward 
$\Phi_e(Y)_*({}^e\N)$ by the
equivalence in {\rm Proposition~\ref{S_2-main.prop}},
is isomorphic to $(\Phi_e(X)_*({}^e\M))^{{}^eN_e}$.
\end{enumerate}
\end{theorem}

\begin{proof}
{\bf 1} Local Noetherian property follows from Lemma~\ref{Eakin-Nagata.thm}.
Quasi-Normality follows from Lemma~\ref{n-canonical-flat.lem}, {\bf 5},
applied to the map ${}^eS\times_S X\rightarrow {}^eS$.

{\bf 2} follows from {\bf 1} and Lemma~\ref{F-finite.thm}.

{\bf 3} is proved similarly to {\bf 1}.

{\bf 4} and {\bf 5} are trivial.

{\bf 6} 
As $N$ is flat, $f$ is fpqc.
Since $G$ is $S$-flat, $H$ is also $S$-flat.
As $f$ is regular, $N$ is regular over $S$.
Note that
\[
1\rightarrow {}^eN_e\rightarrow {}^eN
\xrightarrow{\Phi_e} {}^eS\times_S N\rightarrow 1
\]
is exact (that is, $\Phi_e$ is qfpqc and ${}^eN_e=\Ker\Phi_e$) by
Lemma~\ref{relative-Frob.thm}.
By the theorem of Radu and Andr\'e \cite{Radu}, \cite{Andre},
\cite{Dumitrescu2}, $\Phi_e$ is flat.
Hence ${}^eN_e$ is flat.
Note that
\[
1\rightarrow {}^eN_e\rightarrow {}^eG\xrightarrow{\Phi_e(H,G)} 
{}^eH\times_H G\rightarrow 1
\]
is exact (that is, $\Phi_e(H,G)$ is qfpqc and ${}^eN_e=\Ker \Phi_e(H,G)$) 
by Lemma~\ref{relative-Frob.thm}.
As ${}^eN_e$ is flat, 
$\Phi_e(H,G)$ is flat.
Being flat and qfpqc, it is fpqc.
In particular, ${}^eH\times_H G$ is flat over ${}^eS$.

Note that $\Phi_e(X)_*({}^e\M)$ satisfies the $(S'_2)$ condition.
So
\begin{multline*}
\Phi_e(X)_*({}^e\M)^{{}^eN_e}
\cong
((1\times i)_*(1\times i)^*\Phi_e(X)_*({}^e\M))^{{}^eN_e}\\
\cong
(1\times i)_*(\Phi_e(U)_*{}^ei^*{}^e\M)^{{}^eN_e}
\cong
(1\times i)_*(\Phi_e(U)_*({}^e(i^*\M)))^{{}^eN_e}.
\end{multline*}
So we may assume that $X=U$.

On the other hand, 
\[
(1\times j)^*\Phi_e(Y)_*({}^e\N)
\cong
\Phi_e(V)_*({}^ej^*{}^e\N)
\cong
\Phi_e(V)_*({}^e(j^*\N)),
\]
as can be seen easily.
So we may assume that $Y=V$, 
and hence $\varphi:X\rightarrow Y$ is a $G$-enriched principal
$N$-bundle.

As ${}^e\M\in (S'_2)({}^eG,{}^eX)$ and 
$\Phi_e$ is a finite homeomorphism, we have that 
$\Phi_e(X)_*({}^e\M)\in (S'_2)({}^eG,{}^eS\times_S X)$
by Lemma~\ref{finite-depth.lem}.
So $(\Phi_e(X)_*({}^e\M))^{{}^eN_e}$ belongs to 
$(S'_2)({}^eH\times_HG,{}^eS\times_S X)$ by
Lemma~\ref{invariance-S_2.lem}.
So in view of Theorem~\ref{main.thm}, it remains to prove that
\[
((1\times\varphi)_*(\Phi_e(X)_*({}^e\M))^{{}^eN_e})^{{}^eS\times_S N}
\cong \Phi_e(Y)_*({}^e\N).
\]
This is clear, since
\begin{multline*}
((1\times\varphi)_*(\Phi_e(X)_*({}^e\M))^{{}^eN_e})^{{}^eS\times_S N}
\cong
(((1\times\varphi)_*\Phi_e(X)_*({}^e\M))^{{}^eN_e})^{{}^eS\times_S N}\\
\cong
(\Phi_e(Y)_*{}^e\varphi_*({}^e\M))^{{}^eN}
\cong
\Phi_e(Y)_*({}^e(\varphi_*\M))^{{}^eN}
\cong
\Phi_e(Y)_*({}^e\N).
\end{multline*}
\end{proof}

\paragraph
Let $S=\Spec k$ with $k$ a perfect field, $H$ and $N$ be 
$S$-group schemes.
Assume that $H$ and $N$ are locally Noetherian and regular.
Let $H$ act on $N$ by the group automorphisms, and let $G$ be the
semidirect product $H\ltimes N$.
Let $X$ be a locally Noetherian $F$-finite $G$-scheme.
We say that $X$ is of {\em finite $(H,N)$-$F$-representation type} by
$\M_1,\ldots,\M_r\in \Coh({}^{e_0}H\ltimes N,X)$ if 
for any $e\geq 1$, we can write 
$(F^e_*\O_{{}^eX})^{{}^eN_e}\cong \N_1\oplus\cdots\oplus \N_u$ by some 
$\N_1,\ldots,\N_u\in \Coh({}^eH\ltimes N,X)$ such that 
for each $j=1,\ldots,u$, there exists some $l(j)$ such that 
$\N_j\cong \M_{l(j)}$ as $(N,\O_X)$-modules (not as $({}^eH\ltimes N,\O_X)
$-modules).
If $X$ is finite $(H,e)$-$F$-representation type, then we say that
$X$ is finite {\em $H$-$F$-representation type}, where $e=\Spec k$ denotes
the trivial group.
If, moreover, $H$ is also trivial, then we say that $X$ is finite 
$F$-representation type.
If $H=\Bbb G_m^s$, the split $s$-torus, and $G=H\times N$, the direct product,
then 
finite $(H,N)$-$F$-representation type is called graded finite
$F$-representation type modulo $N$.
If, moreover, $N$ is trivial, we say that $X$ is graded finite
$F$-representation type.

From Theorem~\ref{Frobenius-pushforward.thm}, we immediately have the
following.

\begin{corollary}\label{ffrt.cor}
Let the assumptions
be as in {\rm Theorem~\ref{Frobenius-pushforward.thm}}, {\bf 5, 6}.
Assume further that $S=\Spec k$ with $k$ a perfect field,
and $G=H\ltimes N$ is a semidirect product.
Then $Y$ is of finite $H$-$F$-representation type by 
$\N_1,\ldots,\N_u$ if and only if 
$X$ is of finite $(H,N)$-$F$-representation type by
$\M_1,\ldots,\M_u$, where $\M_l=i_*\rho^*j^*\N_l$ for $l=1,\ldots,u$.
\qed
\end{corollary}

\section{Global $F$-regularity}\label{global-freg.sec}

\paragraph
Let $X$ be a scheme
and $h:\M\rightarrow \N$ an $\O_X$-linear map between $\O_X$-modules.
We say that $h$ is generically monic 
if $h_\xi:\M_\xi\rightarrow\N_\xi$ is
injective for each $\xi\in X\an 0$.

\paragraph\label{globally-F-reg.settings.par}
Let $S$ be an $\Bbb F_p$-scheme, $G$ an $S$-group scheme, 
and $\varphi:X\rightarrow Y$ a $G$-morphism.
For $e\geq 0$, the scheme $Y\times_{Y^{(e)}}X^{(e)}$ is simply denoted by
$X_Y^{(e)}$.
As $X_Y^{(e)}=Y\times_{Y_S^{(e)}}X_S^{(e)}$, it is a $G$-scheme in a natural
way, and the relative Frobenius map $\Phi_e(Y,X):X\rightarrow 
X_Y^{(e)}$ 
is a $G$-morphism.

\begin{definition}
Let $S$ be an $\Bbb F_p$-scheme, $G$ an $S$-group scheme, 
and $X$ a $G$-scheme.
Assume that $S=\Spec k$ with $k$ a perfect field.

\begin{enumerate}
\item[\bf 1]
We say that $X$ is {\em $G$-globally $F$-regular} 
if for any 
$G$-linearized invertible sheaf 
$\Cal L$ on $X$ and any $G$-invariant 
generically monic section $s:\O_X\rightarrow \L$,
there exists some $e\geq 1$ such that the composite
\begin{equation}\label{sF^e.eq}
sF^e:\O_{X^{(e)}}\xrightarrow{F^e}F^e_*\O_X\xrightarrow{s}
F^e_*\L
\end{equation}
splits as a $(G,\O_{X^{(e)}})$-linear map.
\item[\bf 2] We say that $X$ is {\em $G$-$F$-split} 
if for any (or equivalently, some) $e\geq 1$, 
\[
F^e:\O_{X^{(e)}}\rightarrow F^e_*\O_X
\]
splits as a $(G,\O_{X^{(e)}})$-linear map.
\end{enumerate}
If $G$ is trivial, then we simply say that $X$ is globally $F$-regular
or $F$-split.
\end{definition}

\paragraph
A $G$-globally $F$-regular scheme is
$G$-$F$-split.

\paragraph
Let $\Cal L$ be an ample invertible sheaf on a Noetherian $\Bbb F_p$-scheme $X$.
Assume that 
for any $r\geq 0$ and a monic section $s:\O_{X}\rightarrow \L^{\otimes r}$,
there exists some $e\geq 1$ such that 
$sF^e:\O_{X^{(e)}}\rightarrow F_*^e\L^{\otimes r}$ splits as an
$\O_{X^{(e)}}$-linear map.
Then $X$ is globally $F$-regular.
This is proved similarly to \cite[Theorem~2.6]{Hashimoto9}.

\begin{lemma}\label{smooth-etale-radu-andre.lem}
Let $S$ be an $\Bbb F_p$-scheme, and $Z$ a smooth $S$-scheme.
Then the relative Frobenius map $\Phi_e:Z\rightarrow Z_S^{(e)}$ is affine
and $(\Phi_e)_*(\O_{Z_S^{(e)}})$ is locally free.
In particular, $\Phi_e$ is faithfully flat.
If, moreover, $Z$ is \'etale over $S$, then $\Phi_e$ is an isomorphism.
\end{lemma}

\begin{proof}
We prove that $(\Phi_e)_*(\O_{Z_S^{(e)}})$ is locally free.
As the question is local both on $S$ and $Z$, we may assume that
$S=\Spec R$ and $Z=\Spec A$ are affine.
Then by \cite[(10.131.14)]{SP}, 
there exists some finitely generated $\Bbb F_p$-subalgebra $R_0$ of $R$ 
and a smooth $R_0$ algebra $A_0$ such that $A\cong R\otimes_{R_0}A_0$.
As $R_0\rightarrow A_0$ is a regular homomorphism between Noetherian
$\Bbb F_p$-algebras, we have that $A_0$ is 
$(A_0)_{R_0}^{(e)}:=A_0^{(e)}\otimes_{R_0^{(e)}}R_0$-flat by 
Radu--Andr\'e theorem \cite{Radu}, \cite{Andre}, \cite{Dumitrescu2}.
As $A_0$ is $F$-finite, $A_0$ is a finite projective 
$(A_0)_{R_0}^{(e)}$-module (see Lemma~\ref{smoth-Frobenius-Noetherian.lem}).
Taking the base change $?\otimes_{R_0}R$, we get the desired result.

If, moreover, $Z$ is \'etale, then we can take $A_0$ to be \'etale over $R_0$.
Then by \cite[(33.5)]{ETI}, $(A_0)_{R_0}^{(e)}\rightarrow A_0$ is an 
isomorphism.
By the base change, we have that $\Phi_e$ is an isomorphism.
\end{proof}

\begin{lemma}\label{G-globally-F-reg-descent.lem}
Let $S=\Spec k $ with $k$ a field of characteristic $p>0$, 
$f:G\rightarrow H$ be 
an fpqc homomorphism between $S$-group schemes, and $N=\Ker f$.
Assume that $N$ is smooth over $S$.
Let $\varphi:X\rightarrow Y$ be a $G$-morphism which is $N$-invariant.
Assume that $\varphi(X\an0)\subset Y\an0$.
Assume that $\bar\eta:\O_Y\rightarrow (\varphi_*\O_X)^N$ is an isomorphism.
If $X$ is $G$-globally $F$-regular 
\(resp.\ $G$-$F$-split\), then $Y$ is $H$-globally
$F$-regular \(resp.\ $H$-$F$-split\).
\end{lemma}

\begin{proof}
We prove the assertion for the global $F$-regularity.
The case of $F$-splitting is similar.

Let $\Cal L$ be an $H$-linearized invertible sheaf on $Y$, and 
$s:\O_{Y}\rightarrow \L$ an $H$-invariant generically monic section.
As $\varphi(X\an0)\subset Y\an0$, it is easy to see that
$s:\O_X\rightarrow \varphi^*\L$ is a $G$-invariant generically
monic section.
So there exists some $e\geq 1$ and a $G$-invariant splitting
$\pi: F^e_*(\varphi^*\L)\rightarrow \O_{X^{(e)}}$ of $sF^e:
\O_{X^{(e)}}\rightarrow
F^e_*(\varphi^*\L)$.

Obviously,
\[
\bar\eta:\O_{Y^{(e)}}\rightarrow 
(\varphi^{(e)}_*\O_{X^{(e)}})^{N^{(e)}}
\]
is an isomorphism.
On the other hand, as $N$ is $S$-smooth, the Frobenius map
$F^e:N\rightarrow N^{(e)}$ is faithfully flat, and hence
the restriction $(?)^{N^{(e)}}$ agrees with $(?)^N$ by
Lemma~\ref{restriction-invariance.lem}.

So 
\[
\bar\eta:\O_{Y^{(e)}}\rightarrow 
(\varphi^{(e)}_*\O_{X^{(e)}})^N
\]
is an isomorphism.

On the other hand, 
\[
\varphi^{(e)}_*(F^e_X)_*(\varphi^*\L))^{N}
\cong
(F^e_Y)_*(\varphi_*(\varphi^*\L))^N
\cong
(F^e_Y)_*\L.
\]

So applying $(\varphi^{(e)}_*(?))^N$ to the sequence
\[
\O_{X^{(e)}}\xrightarrow{sF^e} F^e_*(\varphi^*\L)
\xrightarrow{\pi}\O_{X^{(e)}},
\]
we get
\[
\O_{Y^{(e)}}\xrightarrow{sF^e}F^e_*\L
\xrightarrow{\pi}\O_{Y^{(e)}}
\]
whose composite is the identity.
Hence $Y$ is $H$-globally $F$-regular.
\end{proof}

\begin{corollary}[cf.~{\cite[Proposition~1.2, (2)]{HWY}}]\label{HWY.cor}
Let $\varphi:X\rightarrow Y$ be a 
morphism between integral $\Bbb F_p$-schemes
such that $\eta:\O_Y\rightarrow \varphi_*\O_X$ is an isomorphism.
If $X$ is globally $F$-regular \(resp.\ $F$-split\), then so is $Y$.
\end{corollary}

\begin{proof}
Consider $S=\Spec \Bbb F_p=Z$ and the trivial $G$, $H$, and $N$.
Then apply Lemma~\ref{G-globally-F-reg-descent.lem}.
As $\eta$ is an isomorphism, $\varphi$ is dominating and $\varphi(X\an0)
\subset Y\an0$.
The results follow from
Lemma~\ref{G-globally-F-reg-descent.lem} easily.
\end{proof}

\begin{proposition}\label{globally-freg-Reynolds-i.prep}
Let $S=\Spec k$ with $k$ a perfect field of characteristic $p>0$.
Let $H$ and $N$ be $S$-group schemes.
Let $H$ act on $N$ by group automorphisms, and $G:=N\rtimes H$.
Assume that $N$ is a linearly reductive affine algebraic $k$-group scheme.
Let $X$ be an $F$-finite Noetherian 
$H$-globally $F$-regular \(resp.\ $H$-$F$-split\)
$G$-scheme.
Then $X$ is $G$-globally $F$-regular \(resp.\ $G$-$F$-split\).
\end{proposition}

\begin{proof}
We prove the assertion for the global $F$-regularity.
The assertion for the $F$-splitting is similar.

Let $\Cal L$ be a $G$-linearized invertible sheaf on $X$, and 
$s:\O_X\rightarrow \L$ a generically monic section.
Then there exists some $e\geq 1$ and 
an $(H,\O_{X^{(e)}})$-linear map
$\pi:F^e_*\L\rightarrow
\O_{X^{(e)}}$ such that $\pi s F^e=\id$.

As $F^e:X\rightarrow X^{(e)}$ is finite, 
$\H_0=\uHom_{\O_{X^{(e)}}}(\O_{X^{(e)}},F^e_*\L)$ and
$\H_1=\uHom_{\O_{X^{(e)}}}(F^e_*\L,\O_{X^{(e)}})$ are
coherent
$(G,\O_{X^{(e)}})$-modules.

Let $h:X^{(e)}\rightarrow S=\Spec k$ 
be the structure map, which is quasi-compact
quasi-separated by assumption
(if $q:X\rightarrow S$ is the structure map, then $h$ is the composite
\[
X^{(e)}\xrightarrow{q^{(e)}}S^{(e)}\xrightarrow{F^{-e}}S).
\]
So we have a direct sum decomposition of quasi-coherent 
$(G,\O_S)$-modules $h_*\H_1=(h_*\H_1)^N\oplus U_N(h_*\H_1)$
by Proposition~\ref{enriched-Reynolds.prop}.
Applying $\Gamma(S,?)\circ(?)^H$, we get the direct sum decomposition
of
abelian groups 
\begin{multline}\label{hom-Reynolds-decomposition.eq}
\Hom_{H,\O_{X^{(e)}}}(F^e_*\L,\O_{X^{(e)}})\\
=
\Hom_{G,\O_{X^{(e)}}}(F^e_*\L,\O_{X^{(e)}})
\oplus
\Gamma(S,U_N(h_*\uHom_{\O_{X^{(e)}}}(F^e_*\L,\O_{X^{(e)}}))^H).
\end{multline}
By the product
\[
h_*\H_1\otimes_{\O_S}h_*\H_0\rightarrow 
h_*(\H_1\otimes_{\O_{X^{(e)}}}\H_0)
\rightarrow h_*(\uHom_{\O_{X^{(e)}}}(\O_{X^{(e)}},\O_{X^{(e)}})),
\]
$U_N(h_*\H_1)\otimes_{\O_S} (h_*\H_0)^N$ is mapped to 
$U_N(h_*(\uHom_{\O_{X^{(e)}}}(\O_{X^{(e)}},\O_{X^{(e)}})))$ by
Lemma~\ref{Reynolds-four-operations.lem}, {\bf 3}.
Hence when we decompose $\pi=\pi_0+\pi_1$ according to the decomposition
(\ref{hom-Reynolds-decomposition.eq}), 
\[
\pi_0(sF^e)\in \Hom_{G,\O_{X^{(e)}}}(\O_{X^{(e)}},\O_{X^{(e)}})
\]
and
\[
\pi_1(sF^e)\in 
\Gamma(S,U_N(h_*(\uHom_{\O_{X^{(e)}}}(\O_{X^{(e)}},\O_{X^{(e)}})))^H).
\]
As we can decompose the identity of $\O_{X^{(e)}}$ in two ways as
\[
\id=\pi_0(sF^e)+\pi_1(sF^e)=\id+0,
\]
we must have $\pi_0(sF^e)=\id$ and $\pi_1(sF^e)=0$ by the 
uniqueness of the decomposition.
Hence $\pi_0:F^e_*\L\rightarrow \O_{X^{(e)}}$ is the 
desired $(G,\O_{X^{(e)}})$-linear splitting of $sF^e$,
and the proof of the proposition has been completed.
\end{proof}

\begin{corollary}\label{linearly-reductive-globally-F-regular.cor}
Let $S=\Spec k$, $H$, $N$, and $G$ be as in 
{\rm Proposition~\ref{globally-freg-Reynolds-i.prep}}.
Assume that $N$ is smooth.
Let $\varphi:X\rightarrow Y$ be a $G$-morphism which is $N$-invariant.
Assume that $\bar\eta:\O_Y\rightarrow (\varphi_*\O_X)^N$ is an 
isomorphism.
Assume that $X$ is Noetherian normal and $F$-finite.
If $X$ is $H$-globally $F$-regular \(resp.\ $H$-$F$-split\), 
then $Y$ is also $H$-globally $F$-regular \(resp.\ $H$-$F$-split\).
\end{corollary}

\begin{proof}
By \cite[(6.3)]{Hashimoto4}, $\varphi(X\an0)\subset Y\an0$.
The assertion follows easily from
Proposition~\ref{globally-freg-Reynolds-i.prep} and
Lemma~\ref{G-globally-F-reg-descent.lem}.
\end{proof}

\begin{lemma}\label{strongly-F-regular-ample.lem}
Let $\varphi:X\rightarrow Y$ 
be a globally $F$-regular $F$-finite Noetherian $\Bbb F_p$-scheme
with an ample invertible sheaf $\Cal A$.
Then any open subscheme $U$ is also globally $F$-regular.
In particular, $X$ is $F$-regular in the sense that each local ring of 
$X$ is strongly $F$-regular.
In particular, $X$ is Cohen--Macaulay normal.
\end{lemma}

\begin{proof}
Let $r\geq 0$ and $s\in\Gamma(U,\Cal A^{\otimes r})$ a generically
monic section.
Take $r'\geq 0$ and a section $u\in \Gamma(X,\Cal A^{\otimes r'})$ which 
is generically monic such that $X_u\subset U$.
Then by \cite[(27.24.6)]{SP}, there exists some $n\geq 0$ such that 
$u^ns\in\Gamma(X,\Cal A^{\otimes(r+nr')})$.
Then there exists some $e\geq 1$ such that $u^nsF^e:\O_{X^{(e)}}
\rightarrow F^e_*A^{\otimes(r+nr')}$ has a splitting.
Restricting to $U$, $u^nsF^e$ also has a splitting over $U$.
Hence $sF^e$ also has a splitting over $U$.
Thus $U$ is globally $F$-regular.

In particular, any affine open $U=\Spec A$ is globally $F$-regular, and
the $F$-finite Noetherian ring $A$ is strongly $F$-regular \cite{HH}.
So any local ring of $X$ is also strongly $F$-regular.
As an $F$-finite Noetherian ring is excellent \cite{Kunz} and
a strongly $F$-regular ring is weakly $F$-regular \cite[(3.1)]{HH},
$X$ is Cohen--Macaulay normal by \cite[(4.2)]{Huneke}.
\end{proof}

\begin{lemma}
Let $S$ be an $\Bbb F_p$-scheme, 
$G$ an $S$-group scheme, and 
$X$ be an $G$-$F$-split scheme.
If $h:U\rightarrow X$ is an \'etale $G$-morphism, then
$U$ is a $G$-$F$-split $G$-scheme.
\end{lemma}

\begin{proof}
There exists some $e\geq 1$ and a $(G,\O_{X^{(e)}})$-linear splitting $\pi:
F^e_*\O_X\rightarrow \O_{X^{(e)}}$ of $F^e$.
Applying $(h^{(e)})^*$, we have that
\[
\eta_{p_2}=(h^{(e)})^*F^e_X: \O_{U^{(e)}}\rightarrow
(h^{(e)})^*F^e_*\O_X
\cong
(p_2)_*p_1^*\O_X
\cong(p_2)_*\O_{U^{(e)}_X}
\]
has a $(G,\O_{U^{(e)}})$-linear 
splitting, where $p_1:U^{(e)}_X\rightarrow X$ is the first projection,
and $p_2:U^{(e)}_X\rightarrow U^{(e)}$ is the second projection.
As $h$ is \'etale, $\Phi_e(X,U)$ is an isomorphism by
Lemma~\ref{smooth-etale-radu-andre.lem}.
As the composite
\[
U\xrightarrow{\Phi_e(X,U)}U^{(e)}_X\xrightarrow{p_2} U^{(e)}
\]
is $F^e_U$, 
$F^e_U:\O_{U^{(e)}}\rightarrow F^e_*(\O_U)$ has a 
$(G,\O_{U^{(e)}})$-linear splitting, as desired.
\end{proof}

\begin{lemma}\label{large-F-finite.lem}
Let $X$ be a Noetherian quasi-normal $\Bbb F_p$-scheme, and $U$ its large
open subset.
Then $X$ is $F$-finite if and only if $U$ is $F$-finite.
\end{lemma}

\begin{proof}
Assume that $X$ is $F$-finite.
Then $F_X:X\rightarrow X^{(1)}$ is finite.
Taking the base change by $U\rightarrow X$, 
$F_U: U\cong X\times_{X^{(1)}}U^{(1)}\rightarrow U^{(1)}$ is finite.

Assume that $U$ is $F$-finite.
Then $F_*\O_U$ is a coherent $(S'_2)$ $\O_U$-module.
So letting $i:U\hookrightarrow X$ be the inclusion,
$(i^{(1)})_*F_*\O_U\cong F_*i_*\O_U\cong F_*\O_X$ is a coherent sheaf.
\end{proof}

\begin{theorem}\label{globally-F-reg-main.thm}
Let $S=\Spec k$ with $k$ a perfect field of characteristic $p>0$.
Let $G$ be a smooth linearly reductive affine algebraic $k$-group scheme.
Let the diagram
\[
\xymatrix{
X & U \ar@{_{(}->}[l]_i \ar[r]^{\rho} & V \ar@{^{(}->}[r]^j & Y
}
\]
be a rational almost principal $G$-bundle.
Assume that $X$ and $Y$ are Noetherian normal schemes.
Then we have the following.
\begin{enumerate}
\item[\bf 1] $X$ is $F$-finite if and only if $Y$ is $F$-finite.
\item[\bf 2] Assume that $X$ and $Y$ have ample invertible sheaves
and are $F$-finite.
Then $X$ is globally $F$-regular \(resp.\ $F$-split\) if and only if 
$Y$ is globally $F$-regular \(resp.\ $F$-split\).
\end{enumerate}
\end{theorem}

\begin{proof}
{\bf 1} In view of Lemma~\ref{large-F-finite.lem},
we may assume that $X=U$ and $Y=V$.
If $Y$ is $F$-finite, then $X$ is $F$-finite, since $\rho:X\rightarrow Y$ 
is of finite type.
If $X$ is $F$-finite, then $Y$ is $F$-finite, since $\rho$ is an algebraic
quotient by $N$ and 
$N$ is linearly reductive, see 
Lemma~\ref{F-finite-finite.thm}.

{\bf 2}
We only prove the assertion for the global $F$-regularity.
Let $\Cal L$ be an invertible sheaf on $Y$, and $s\in\Gamma(Y,\L)$ a 
generically monic section.
Then assuming that $V$ is globally $F$-regular, 
there exists some $e\geq1$ such that there is a splitting $\pi$ of 
$sF^e$ on $V$.
Note that $\pi\in\Gamma(V,\uHom_{\O_Y^{(e)}}(F_*^e\O_Y,\O_{Y^{(e)}}))$.
As $Y$ is $F$-finite Noetherian normal, 
$\uHom_{\O_Y^{(e)}}(F_*^e\O_Y,\O_{Y^{(e)}})$ is a reflexive sheaf, and hence
$\pi$ is defined over $Y$.
This shows that $Y$ is globally $F$-regular.
On the other hand, if $Y$ is globally $F$-regular, then by
Lemma~\ref{strongly-F-regular-ample.lem}, $V$ is globlly $F$-regular.
Similarly, $X$ is globally $F$-regular if and only if $U$ is so.
Hence we may assume that $X=U$ and $Y=V$.

If $X$ is globally $F$-regular, then by
Corollary~\ref{linearly-reductive-globally-F-regular.cor},
$Y$ is globally $F$-regular.

Let $Y$ be globally $F$-regular.
Let $\L$ be an ample invertible sheaf on $Y$.
We can take some $r\geq 1$ and 
a generically monic section $a\in\Gamma(Y,\L^{\otimes r})$ such that $Y_a$ is
regular and affine.
Replacing $\L$ by $\L^{\otimes r}$ if necessary, we may assume that $r=1$.
As $Y$ is globally $F$-regular, there exists some $e_0\geq 1$ and
$\pi_0:F^{e_0}_*\L\rightarrow \O_{Y^{(e_0)}}$ such that $\pi_0aF^{e_0}=\id$.

As $N$ is smooth, $\rho$ is smooth, and hence 
$X_a=\rho^{-1}(Y_a)$ is regular.
By Lemma~\ref{codim-two-flat.thm}, {\bf 1}, $X_a$ is dense in $X$.
That is, $a\in\Gamma(X,\rho^*\L)$ is generically monic.
By \cite[(28.38.7)]{SP}, $\rho^*\L$ is an ample invertible sheaf on $X$.

Let $n>0$ and $s\in\Gamma(X,\rho^*\L^{\otimes n})$ 
be a generically monic section.
As $X_a$ is affine regular $F$-finite, it is globally $F$-regular, 
and hence there exists some $e_1\geq 1$ and a splitting
$\pi_1:F^{e_1}_*\rho^*\L^{\otimes n}|_{X_a}\rightarrow \O_{X^{(e_1)}_a}$ of 
$sF^{e_1}:\O_{X^{(e_1)}_a}\rightarrow
F^{e_1}_*\rho^*\L^{\otimes n}|_{X_a}$.
As $\uHom_{\O_{X^{(e_1)}}}(F^{e_1}_*\rho^*\L^{\otimes n},\O_{X^{(e_1)}})$ is coherent,
there exists some $e_2\geq 0$ such that 
$\pi_2=
(a^{p^{e_2}})^{(e_1)}\pi_1$ lies in 
$\Hom_{\O_X^{(e_1)}}(F^{e_1}_*\rho^*\L^{\otimes n},(\L^{\otimes p^{e_2}})^{(e_1)})$.
Then $\pi_2sF^{e_1}=(a^{p^{e_2}})^{(e_1)}$.

As $Y$ is $F$-split, there exists some $\pi_3:F^{e_2}_*\O_Y\rightarrow 
\O_{Y^{(e_2)}}$ such that $\pi_3 F^{e_2}=\id$.
Then
\begin{multline}
\pi_0^{(e_1+e_2)}\pi_3^{(e_1)}\pi_2 sF^{e_0+e_1+e_2}
=
\pi_0^{(e_1+e_2)}\pi_3^{(e_1)}(a^{p^{e_2}})^{(e_1)}F^{e_0+e_2}\\
=
\pi_0^{(e_1+e_2)}\pi_3^{(e_1)}F^{e_2}a^{(e_1+e_2)}F^{e_0}
=\pi_0^{(e_1+e_2)}a^{(e_1+e_2)}F^{e_0}=\id,
\end{multline}
and $sF^{e_0+e_1+e_2}$ has a splitting.
This shows that $X$ is globally $F$-regular.
\end{proof}

\section*{\large Chapter~2. Examples and Applications}
\label{ex-app.chap}

\section{Finite group schemes}\label{finite.sec}

\paragraph\label{definition-small.par}
Let $G$ be an $S$-group scheme acting on $X$.
If there is a $G$-stable open subset $U$ of $X$ such that the action of 
$G$ on $U$ is free and $U$ is $n$-large in $X$
then we say that the action of $G$ on $X$ is {\em $n$-small}.
$0$-small is also called 
{\em generically free}.
$1$-small is also called {\em small}.

\begin{lemma}\label{n-small-flat.lem}
Let $G$ be an $S$-group scheme, and 
$\psi:Z'\rightarrow Z$ a flat $G$-morphism.
If the action of $G$ on $Z$ is $n$-small, then the action of $G$ on $Z'$ 
is also $n$-small.
\end{lemma}

\begin{proof}
Obvious from 
Lemma~\ref{stabilizer-basics.lem} and
Lemma~\ref{codim-two-flat.thm}.
\end{proof}

\paragraph
Letting $G$ act on $G\times X$ by $g(g_1,x)=(gg_1g^{-1},gx)$ and
on $X\times X$ diagonally, $\Psi:G\times X\rightarrow X\times X$ 
and the diagonal map
$\Delta:X\rightarrow X\times X$ are $G$-morphisms.
So the structure map $\phi:\Cal S_X\rightarrow X$ is also a $G$-morphism,
where $\Cal S_X$ is the stabilizer of the action of $G$ on $X$.

If there is a separated $G$-invariant morphism $\varphi:X\rightarrow Y$
(e.g., $X$ is $S$-separated or $\varphi$ is affine),
then $\Cal S_X\rightarrow G\times X$ is a closed immersion.
If, moreover, $G$ is finite, then $\phi$ is finite.

\paragraph
Assume that
$\phi$ is finite.
Then the cokernel of the
split monomorphism $\eta:\O_X\rightarrow
\phi_*\Cal O_{\Cal S_X}$ is a quasi-coherent $(G,\O_X)$-module which is
finite-type as an $\O_X$-module.
The complement $\Cal U_X$ of the support 
of $\Coker\eta$ is the largest $G$-stable
open subset of $X$ on which $G$ acts freely.
We call $\Cal U_X$ the free locus of the action.

\begin{example}\label{ps-sm.ex}
Let $G$ be a finite (constant) group,
and $X$ be Noetherian and irreducible.
Assume that there is a separated $G$-invariant
morphism $\varphi:X\rightarrow Y$.
Then $\phi$ is finite, and the free locus exists.
More precisely, for $g\in G$, set $X_g:=\{x\in X\mid gx=x\}$.
Note that $X_g$ is a closed subscheme of $X$.
It is easy to see that $\Cal U_X=X\setminus\bigcup_{g\neq e}X_g$.
So the action is generically free if and only if the action is faithful
(that is, the action of $g$ on $X$ is not the identity if $g\neq e$).
If $\codim_X X_g=1$, then we say that $g$ is a {\em pseudoreflection}.
The action is small if and only if there is no pseudoreflection.
However, see Remark~\ref{ps-small.rem} below.
\end{example}

\paragraph
Let $G$ be a finite group scheme over a field $k$ of characteristic $p>0$.
Then the smallest nonnegative integer $e\geq 0$ such that 
$G_e=G^\circ$ is called the {\em exponent} of $G$.
If $I$ is the nilradical of $k[G]$, 
then the exponent $e$ of $G$ is the smallest nonnegative integer
such that $I^{[p^e]}=0$, where $I^{[p^e]}=I^{(e)}k[G]$.
The exponent is not changed by the extension of the base field.

\begin{proposition}\label{finite-equiv.thm}
Let $k$ be a field, $G$ a finite $k$-group scheme acting on a reduced artinian
$k$-algebra $L$, and set $K=L^G$.
Let $\varphi:X=\Spec L\rightarrow \Spec L^G=Y$ be the canonical map.
Assume that $X$ is $G$-connected.
Then we have that $K$ is a field, and $\dim_K L \leq \dim_k k[G]$ in general.
In particular, $\dim_K L$ is finite.
Moreover, the following are equivalent.
\begin{enumerate}
\item[\bf 1] $\varphi$ is a principal $G$-bundle.
\item[\bf 2] The action of $G$ on $X$ is free.
\item[\bf 3] The action of $G$ on $X$ is generically free.
\item[\bf 4] $\dim_K L=\dim_k k[G]$.
\end{enumerate}
\end{proposition}

\begin{proof}
{\bf 1$\Rightarrow$2} As $\Psi$ is an isomorphism, its base change
$\phi:\Cal S_X\rightarrow X$ is an isomorphism.

{\bf 2$\Rightarrow$1} is Lemma~\ref{DG.thm}.

{\bf 2$\Leftrightarrow$3} is obvious, since $X$ is an artinian scheme.

{\bf 1$\Rightarrow$4}.
Compare the $L$-dimension of the two isomorphic spaces $L\otimes_KL$ and
$k[G]\otimes_k L$.

By \cite[(32.6)]{ETI}, $K$ is a finite direct product of normal domains.
As $X$ is $G$-connected, $Y$ is connected by
Lemma~\ref{connected.lem}.
Hence $K$ is a domain.
As $L$ is an integral extension of $K$ by Lemma~\ref{DG.thm} and 
$L$ is of Krull dimension zero, $K$ is also zero dimensional.
Being a zero dimensional domain, $K$ is a field.

It remains to prove the assertions (i) $\dim_K L\leq
\dim_k k[G]$; and (ii) {\bf 4$\Rightarrow$(1, 2, or 3)}.

Replacing $k$ by $K$, $G$ by $K\otimes_k G$, and not changing $L$,
we may assume that $k=K$.

We claim that 
if $N$ is a closed normal subgroup scheme and the proposition is true for
$N$ and $G/N$, then (i) and (ii), hence the proposition is also true for $G$.
Indeed, $M=L^N=M_1\times\cdots\times M_r$ is a finite direct product of fields.
Applying the proposition for the action of $N$ on $L_i=M_i\otimes_ML$, 
$\dim_{M_i}L_i\leq \dim_k k[N]$.
On the other hand, $\dim_K M\leq \dim_k k[G/N]$.
So 
\begin{multline}\label{KL.eq}
\dim_KL=\sum_i\dim_K L_i=\sum_i \dim_KM_i\dim_{M_i}L_i
\\
\leq
\dim_kk[N]\sum_i \dim_K M_i
=\dim_kk[N]\dim_KM
\\
\leq \dim_k k[N]\dim_kk[G/N]
=\dim_kk[G].
\end{multline}
The equality $\dim_KL=\dim_kk[G]$ holds if and only if
the equality holds everywhere in (\ref{KL.eq}).
If so, $\dim_K M=\dim_k k[G/N]$ and $\dim_{M_i}L_i=\dim_k k[N]$ for each $i$.
As the proposition is assumed to be true for $N$ and $G/N$,
we have that $\Spec M\rightarrow \Spec K=Y$ is a principal $G/N$-bundle, 
and $\Spec L_i\rightarrow \Spec M_i$ is a principal $N$-bundle for each $i$.
In particular, $X=\Spec L\rightarrow \Spec M$ is a $G$-enriched principal
$N$-bundle.
By Lemma~\ref{principal-composition.thm}, $\varphi:X\rightarrow Y$ is 
a principal $G$-bundle, and the claim has been proved.

Assume that $G$ is infinitesimal
of exponent one.
That is, $G$ equals its first Frobenius kernel $G_1$.
As $G$ is geometrically connected, $X$ is also 
connected, and hence $L$ is a field
in this case.

Let $\g:=\Lie G$ be the Lie algebra of $G$.
It is a restricted Lie algebra over $k$.
There is a canonical map
\[
\theta: L\otimes\g\rightarrow \End_k L
\]
given by $(\theta(\alpha\otimes D))(\beta)=\alpha D(\beta)$ for $D\in\g$
and $\alpha,\beta\in L$.
Obviously, the image $\Cal D:=\Image \theta$ 
is contained in $\Der_k(L,L)$, the space of $k$-derivations.
Moreover, we have
\begin{multline*}
[\theta(\alpha \otimes D),\theta(\beta \otimes D_1)]\\
=\theta(\alpha\beta\otimes [D,D_1]+\alpha(D(\beta))\otimes 
D_1-\beta(D_1(\alpha))\otimes D)\in\Cal D
\end{multline*}
and
\[
(\theta(\alpha \otimes D))^p=\theta(\alpha^p\otimes D^p+
((\alpha D)^{p-1}(\alpha))\otimes D)\in\Cal D
\]
by \cite[Exercise~25.1]{CRT} and \cite[(25.5)]{CRT}.
This shows that $\Cal D$ is a restricted $L$-Lie subalgebra of $\Der_k(L,L)$
in the sense of Jacobson \cite[(IV.8)]{Jacobson}.

Note that $\g$ generates $k[G]^*$ as a $k$-algebra.
To verify this, we may assume that $k$ is algebraically closed, and
this case is shown in \cite[(I.9.6)]{Jantzen}.
Let $\Cal A$ be the $k$-subalgebra of $\End_k L$ generated by $\Cal D$.
By \cite[(IV.8), Theorem~19]{Jacobson}, $[L:k]$ is finite, and
$\dim_L \End_k L=[L:k]=\dim_L\Cal A$.
After all, $\Cal A=\End_k L$.
It is easy to see that 
$\tilde\theta:L\otimes k[G]^*\rightarrow \Cal A$ induced by $\theta$ is
surjective, and hence we have $\dim_k L \leq \dim_k k[G]$.
Moreover, if the equality holds (i.e., {\bf 4} holds), then
$\tilde\theta:L\otimes k[G]^*\rightarrow \End_k L\cong \Hom_L(L\otimes_k L,L)$ 
is an isomorphism.
Then
its $L$-dual $L\otimes L\rightarrow L\otimes k[G]$ given by
$\alpha\otimes \beta\mapsto \sum_{(\beta)}\alpha\beta_{(0)}\otimes \beta_{(1)}$ 
is also an isomorphism, where we are using the Sweedler's notation 
\cite[(1.2)]{Sweedler2}.
This is equivalent to say that $\Psi:G\times X\rightarrow X\times_Y X$ is
an isomorphism, and hence {\bf 4$\Rightarrow$1} has been proved.
So the proposition has been proved for the case that $G$ is infinitesimal of
exponent one.

Next, consider the case that $G$ is infinitesimal.
We prove the proposition for this case by the induction on the exponent 
$e$ of $G$.
The case that $e\leq 1$ is already done by above.
Let $e\geq 2$.
Then there is an exact sequence
\[
1\rightarrow G_1\rightarrow G\xrightarrow\pi G/G_1\rightarrow 1.
\]
Note that $\pi^{-1}(G/G_1)_i=\Phi_i^{-1}(k\otimes_{k^{(e)}}G_1^{(e)})$,
where $\Phi_i:G\rightarrow k\otimes_{k^{(e)}}G^{(e)}$ is the realtive 
Frobenius map.
The right-hand side is the whole $G$ for $i=e-1$, and hence the
exponent of $G/G_1$ is at mose $e-1$.
By induction, the proposition is true for $G_1$ and $G/G_1$, and hence
the proposition is true for $G$.

Now we consider the general case.
By the exact sequence
\[
1\rightarrow G^\circ\rightarrow G\rightarrow G/G^\circ\rightarrow 1,
\]
replacing $G$ by $G/G^\circ$, we may assume that $G$ is \'etale, 
since the proposition for $G^\circ$ is already proved by the
infinitesimal case.
Replacing $K=k$ by its suitable finite Galois extension $k'$ and $L$ by $L'
=k'\otimes_k L$, we may assume that $G$ is a constant finite group.

Let $e_1,\ldots,e_r$ be the set of primitive idempotents of $L$.
As $X$ is $G$-connected, $G$ acts transitively on this set.
Let $H$ be the stabilizer of $e_1$.
Then $[G:H]=r$.
Let $\sigma_1,\ldots,\sigma_r$ be the complete set of representatives of 
$G/H$ in $G$, where we choose the index so that $\sigma_i(e_1)=e_i$.
Set $L_i=Le_i=\sigma_i(L_1)$.
The image of $K\rightarrow L\rightarrow L_1$ ($\alpha\mapsto e_1\alpha$)
is contained in $L_1^H$.
On the other hand, $\sum_i\sigma_i$ maps $L_1^H$ to $K$, and 
$e_1:K\rightarrow L_1^H$ has an inverse $\sum_i\sigma_i$.
By the Galois theory, $[L_1:K]\leq \#H$ (note that $H$ need not act on 
$L_1$ effectively, so the equality need not hold).
So $[L:K]=r[L_1:K]\leq [G:H]\cdot \#H=\#G=\dim_k k[G]$,
and in particular, $L$ is $K$-finite.

It remains to prove that if $\dim_KL=\#G$, then the action of $G$ on $X$ is
free.
In order to check this, taking the base change by the separable
closure $k\sep$ of $k$, we may assume that $k$ is separably closed.
Let $e_1,\ldots,e_r$, $H$, $L_i$ be as above.
As $L_1$ is a purely inseparable extension of $k$, we have that $L_1=L_1^H$ 
this time.
So $\dim_k L_i=1$ for each $i$, and hence $r=\#G$ by assumption.
As $[G:H]=r=\#G$, we have that $H$ is trivial.
Then $G$ acts on $e_1,\ldots,e_r$ freely, and hence $G$ acts on $X$ freely, 
and {\bf 4$\Rightarrow$2} has been proved.
\end{proof}

\begin{proposition}\label{finite-main.prop}
Let $G$ be an LFF $S$-group scheme, and
$\varphi:X\rightarrow Y$ an algebraic quotient.
Let $U$ be the free locus, and $V:=\varphi(U)$.
Then 
\begin{enumerate}
\item[\bf 1] $\varphi(X\an n)=Y\an n$ for $n\geq 0$.
\item[\bf 2] $\rho:U\rightarrow V$ is a principal $G$-bundle,
where $\rho$ is the restriction of $\varphi$.
\item[\bf 3] The following are equivalent.
\begin{enumerate}
\item[\bf a] The action of $G$ on $X$ is $n$-small.
\item[\bf a'] $U$ is $n$-large.
\item[\bf b] $V$ is $n$-large.
\item[\bf c] $\varphi$ is an $n$-almost principal $G$-bundle
with respect to $U$ and $V$.
\item[\bf d] $\varphi$ is an $n$-almost principal $G$-bundle.
\end{enumerate}
\end{enumerate}
\end{proposition}

\begin{proof}
{\bf 1} As $\varphi$ is surjective, it suffices to show that $\varphi(X\an n)
\subset Y\an n$ for $n\geq 0$.
To verify this, we may assume that both $X=\Spec B$ and $Y=\Spec A$ are affine.
Let $x\in X$ and $y=\varphi(x)$.
Then as $\varphi$ is open, the going-down theorem holds for the 
map $A\rightarrow B$ \cite[(10.38.2)]{SP}, and hence $\codim x\geq \codim y$.
As $\varphi$ is integral, $\codim x\leq \codim y$.
So the assertion follows.

{\bf 2} 
By Lemma~\ref{GIT-remark.lem}, $V$ is an open subset of $Y$, and
$\rho:U=\varphi^{-1}(V)\rightarrow V$ is an algebraic quotient.
As the action of $G$ on $U$ is free, $\rho$ is a principal $G$-bundle
by Lemma~\ref{DG.thm}.

{\bf 3} Follows easily from {\bf 1} and {\bf 2}.
\end{proof}

\begin{proposition}\label{generically-free.prop}
Let $G$ be an LFF $S$-group scheme with the well defined rank $r$, and 
$\varphi:X\rightarrow Y$ an algebraic quotient by the action of $G$.
Assume that $X$ is reduced and LFI.
Then for each $\eta\in Y\an0$,
\begin{equation}\label{eta-dimension.eq}
\dim_{\O_{Y,\eta}}(\varphi_*\O_X)_\eta\leq r.
\end{equation}
Moreover, 
the action of $G$ on $X$ is generically free if and only if
the equality holds in {\rm(\ref{eta-dimension.eq})} for each point 
$\eta\in Y\an0$.
\end{proposition}

\begin{proof}
We may assume that $Y=\Spec A$ is affine.
Then $X=\Spec B$ is affine and $A=B^G$.
Then for each minimal prime $P$ of $A$, $B_P$ is reduced and zero-dimensional
(since $\kappa(P)=A_P\rightarrow B_P$ is an integral extension).
As $B$ has finitely many minimal primes, $B_P$ has finitely many minimal 
primes, and $\Spec B_P$ is finite.
Then it is easy to see that $B_P$ is a finite direct product of fields.
To prove that (\ref{eta-dimension.eq}) holds, replacing $\varphi$ by
$\varphi_\eta:X_\eta\rightarrow \eta$, $S$ by $\eta$ and $G$ by $G_\eta$, 
we may assume that $S=Y=\Spec k$ is the spectrum of a field (note that
$\varphi_\eta$ is an algebraic quotient, since $A$ is reduced and hence
$\kappa(\eta)$ is merely a localization of $A$).
By Proposition~\ref{finite-equiv.thm}, the inequality follows.

By Proposition~\ref{finite-equiv.thm}, $\eta\in Y\an0$ lies in $V$ if and
only if the equality in (\ref{eta-dimension.eq}) holds.
The assertion follows immediately by 
Proposition~\ref{finite-main.prop} for the case that $n=1$.
\end{proof}

\begin{example}\label{group-scheme-non-generically-free.ex}
Let $V$ be a finite dimensional $k$-vector space, and
let $G$ be an \'etale finite subgroup scheme of $\GL(V)$.
Then the action of $G$ on $V$ is generically free.
In order to check this, we may assume that $k$ is algebraically closed,
and hence $G$ is a constant subgroup.
As $G$ is a subgroup of $\GL(V)$, we have that $g$ is a non-identity for
$g\neq e$, and hence the action is generically free
(see Example~\ref{ps-sm.ex}).

If $G$ is not \'etale, this is not true any more.
Let $k$ be an algebraically closed field of characteristic $p>0$.
Let $V=k^2$, and consider $G=\GL(V)_1$, the first Frobenius kernel of $\GL(V)$.
Let $B=k[V]=\Sym V^*=k[x,y]$, $A=B^G$, $K=Q(A)$, and $L=Q(B)$.
Then $A=k[x^p,y^p]$.
So $\dim_KL=2<\dim_kk[G]=4$.
Hence the action is not generically free by 
Proposition~\ref{generically-free.prop}.
\end{example}

\begin{lemma}\label{projective-coordinate.lem}
Let $S=\Spec R$ be affine.
Let $G=\Spec \Gamma$ be an LFF $S$-group scheme.
Then the coordinate ring $\Gamma$ of $G$ is a projective object as
a $G$-module.
\end{lemma}

\begin{proof}
It is easy to see that there exists some finitely generated 
$\Bbb Z$-subalgebra $R_0$ of $R$ and an LFF $R_0$-group scheme 
$G_0$ such that $R\otimes_{R_0}G_0\cong G$.
Hence we may assume that $R$ is Noetherian.
By \cite[(III.4.1.3)]{Hashimoto7}, we may assume that $R$ is a field.

Note that a $G$-module is nothing but a (right) $\Gamma$-comodule, which is the
same as a (left) $\Gamma^*$-module.
By \cite[(VI.3.6)]{SY}, $\Gamma\cong \Gamma^*$ as a $\Gamma^*$-module,
and we are done.
\end{proof}

\begin{lemma}\label{henselian.lem}
Let $G$ be an LFF $S$-group scheme acting on an $S$-scheme
$X=\Spec B$ which is an affine scheme.
Let $A=B^G$.
If $\varphi:X=\Spec B\rightarrow\Spec A=Y$ is a principal $G$-bundle, 
then $B$ is $A$-finite and $(G,A)$-projective.
If $A$ is a Noetherian Henselian local ring, then $B\cong A[G_A]$ as 
$(G,A)$-modules.
\end{lemma}

\begin{proof}
We may assume that $S=Y=\Spec A$.
As $G$ is flat, $\varphi$ is fpqc.
Let $\Gamma=A[G_A]=A[G]$ be the coordinate ring of $G$.
Let $B'$ be the $A$-algebra $B$ with a trivial $G$-action.
Then $B\otimes_A B'\cong \Gamma\otimes_A B'$.
By the descent argument, $B$ is a finite projective $A$-module.
We prove that $B$ is a projective $G$-module, or a
$\Gamma^*$-module.
There is a surjective $\Gamma^*$-linear map
$\alpha:(\Gamma^*)^n\rightarrow B$, as $B$ is $A$-finite.
We want to prove that this map splits.
This is equivalent to the surjectivity of
\begin{equation}\label{alpha-star.eq}
\alpha_*:\Hom_{\Gamma^*}(B,(\Gamma^*)^n)\rightarrow \Hom_{\Gamma^*}(B,B).
\end{equation}
This is checked after tensoring $B'$ over $A$.
But
\begin{multline*}
\Hom_{\Gamma^*}(B,?)\otimes_AB'=\Hom_A(B,?)^G\otimes_AB'
=(\Hom_A(B,?)\otimes_AB')^G\\
=\Hom_{B'}(B\otimes_AB',?\otimes_AB')^G
=\Hom_{(G,B')}(\Gamma\otimes_AB',?\otimes_AB').
\end{multline*}
This is an exact functor by Lemma~\ref{projective-coordinate.lem}.
So (\ref{alpha-star.eq}) is surjective.

Now assume that $A$ is Noetherian Henselian local.
Since $B'$ is finite projective as an $A$-module and $A$ is a local ring, 
$B'\cong A^n$ for some $n$.
Hence 
$B^n\cong B\otimes_A B'\cong \Gamma\otimes_A B'\cong \Gamma^n$
as $\Gamma^*$-modules.
Since $A$ is Henselian, any finite $\Gamma^*$-module has a 
semiperfect endomorphism ring, and the Krull--Schmidt theorem
holds in the category of finite $\Gamma^*$-modules
(the fact that a mofule-finite algebra over a Noetherian Henselian 
local ring is semiperfect follows easily from \cite[(I.4.2)]{Milne}).
So $B\cong \Gamma$ as $G$-modules, as desired.
\end{proof}

\begin{example}
Let $k$ be a field, and $G$ a finite $k$-group scheme acting on a 
$k$-algebra $B$.
Even if $B$ is a DVR (discrete valuation ring) and the action of $G$ on 
$X=\Spec B$ is generically free, the action may not be free (so it is not
a small action either).
We give an example of a finite group in characteristic zero and 
an infinitesimal group scheme in characteristic $p$.
\begin{enumerate}
\item[\bf 1] If $k=\Bbb C$, $G=\Bbb Z_2=\langle\sigma\rangle$ 
(the cyclic group of order two with the generator $\sigma$), $B=k[[x]]$ with
$\sigma(x)=-x$, then the stabilizer at the vertex $\Spec k=\Spec B/(x)$ is 
$G$, and is nontrivial.
So the action is not free.
The action is generically free by 
Proposition~\ref{generically-free.prop}, 
since $[Q(B):Q(B^G)]=[k((x)):k((x^2))]=2=\#G$.
\item[\bf 2] Let $k$ be a field of characteristic $p$, and $B=k[x]_{(x)}$,
the localization of the polynomial ring $k[x]$ at the prime ideal $(x)$.
Let $D$ be the $k$-derivation $x^p\dfrac{d}{dx}$ of $B$.
Note that $D^p=0$.
So $G=\alpha_p:=(\Bbb G_a)_1$, the first Frobenius kernel of the additive
group $\Bbb G_a$, acts on $X=\Spec B$.
The algebra map $B\rightarrow k[G]\otimes B$ associated with the action
$G\times X\rightarrow X$ is the map 
$k[x]_{(x)}\rightarrow k[t]/(t^p)\otimes k[x]_{(x)}
=k[t,x]_{(x)}/(t^p)$ given by
\[
f \mapsto \exp(D(-t))(f)=\sum_{i=0}^{p-1} D^i(f)(-t)^i/i!.
\]
So $B^G=B^D=\{f\in B\mid Df=0\}=k[x^p]_{(x^p)}$.
As $[Q(B):Q(B^G)]=p=k[G]$, the action is generically free.
It is easy to see that the stabilizer at $\Spec k=\Spec B/(x)$ is $G$,
and the action is not free.
\end{enumerate}
\end{example}

\paragraph
Let $f:G\rightarrow H$ be an fppf finite homomorphism between flat 
$S$-group schemes, and $N=\Ker f$.
Note that $N$ is fppf finite over $S$, that is, LFF.

\begin{lemma}\label{small-almost-principal.thm}
Let the notation be as above.
Let $\varphi:X\rightarrow Y$ be a $G$-morphism which is an algebraic
quotient by the action of $N$.
Then the free locus $U$ of the action of $N$ on $X$ is $G$-stable in $X$.
In particular, $\rho:U\rightarrow V=\varphi(U)$ is a $G$-enriched principal
$N$-bundle.
If, moreover, the action of $N$ on $X$ is $n$-small,
then $\varphi$ is a $G$-enriched $n$-almost principal $N$-bundle.
\end{lemma}

\begin{proof}
In view of Proposition~\ref{finite-main.prop}, it suffices to show that
$U$ is $G$-stable.

Let $G$ act on $X\times_Y X$ diagonally and on $N\times X$ by 
$g(n,x)=(gng^{-1},gx)$.
Then $\Psi:G\times X\rightarrow X\times_Y X$ defined by $\Psi(g,x)=(gx,x)$ 
and the diagonal map $\Delta_X:X\rightarrow X\times_Y X$ are
$G$-morhpisms.
So $\phi:\Cal S_X\rightarrow X$ is also a $G$-morphism, and
hence $U$ is $G$-stable.
\end{proof}

\paragraph
Let $G$ be a flat $S$-group scheme.
For a $G$-scheme $X$, we define the {\em $G$-radical}
of $X$ by
\[
\rad_G(X):=(\bigcap_{\frak M\in\Max(G,X)}\frak M)^*,
\]
the sum of all the quasi-coherent $G$-ideals of $\O_X$ contained 
in $\bigcap_{\frak M\in\Max(G,X)}\frak M$, where $\Max(G,X)$ is the set of
$G$-maximal $G$-ideals of $\O_X$.
We define the {\em $G$-nilradical} of $X$ to be $\sqrt[G]{0}$, the
$G$-radical of the zero ideal, see \cite[(4.25)]{HM}.
Note that $\sqrt[G]{0}\subset\sqrt{0}$ \cite[(4.30)]{HM}.
If $X$ is quasi-compact, then by \cite[(4.27)]{HM}, we have that 
$\rad_G(X)\supset \sqrt[G]{0}$.
Note that even if $S=\Spec k$ and both $G$ and $X$ are $k$-varieties, 
$\rad(X)$ may not contain $\rad_G(X)$.
For example, when we consider the action of 
$G=\Bbb G_m$ on $B=k[x]$ with the grading $\deg x=1$, then the ideal $(x)$ is
the unique $G$-maximal ideal, and so $\rad_G(B)=(x)\not\subset\rad(B)=(0)$.

\begin{lemma}[$G$-Nakayama's lemma]\label{G-nakayama.lem}
Let $G$ be a flat $S$-group scheme and $X$ a quasi-compact $G$-scheme.
Let $\M$ be a quasi-coherent $(G,\O_X)$-module of finite type.
If $\rad_G(X)\M=\M$, then $\M=0$.
\end{lemma}

\begin{proof}
Assume the contrary.
Let $\I=0:_{\O_X}\M$ be the annihilator of $\M$.
Note that $\I$ is a quasi-coherent $G$-ideal (the proof is the same as 
that of \cite[(4.2)]{HM}).
As $\I\neq \O_X$ by assumption, there exists some $\frak M\in\Max(G,X)$ 
containing $\I$ by \cite[(4.28)]{HM}.
Let $\frak m$ be a maximal quasi-coherent ideal of $\O_X$ containing 
$\frak M$, and $x$ the closed point of $X$ corresponding to $\fm$.
Since $\frak m$ contains $\I$ and $\M$ is of finite type, 
$\M_x\neq 0$.
By Nakayama's lemma, $\M_x\otimes_{\O_{X,x}}\kappa(x)\neq 0$.
Similarly, $(\O_X/\rad_G(X))_x\otimes_{O_{X,x}}\kappa(x)\neq 0$,
since $\frak m\supset\rad_G(X)$.
Taking the tensor product, 
$(\M/\rad_G(X)\M)_x\otimes_{\O_{X,x}}\kappa(x)\neq 0$.
Hence $\M\neq \rad_G(X)\M$.
This is a contradiction.
\end{proof}

\begin{lemma}\label{red-free.lem}
Let $f:G\rightarrow H$ be an fppf finite homomorphism between
flat $S$-group schemes, and $N=\Ker f$.
Let $\varphi:X\rightarrow Y$ be a $G$-morphism which is an algebraic
quotient by the action of $N$.
Let $\Cal I$ be a quasi-coherent $G$-ideal of $\O_X$ contained in $\rad_G(X)$,
and $Z=V(\Cal I)$ the corresponding closed $G$-subscheme of $X$.
If the action of $N$ on $Z$ is free, then the action of $N$ on $X$ is also 
free, and hence $\varphi$ is a principal $G$-bundle.
\end{lemma}

\begin{proof}
Let $\Cal C_X$ be the cokernel of $\O_X\rightarrow \phi_*\O_{\Cal S_X}$.
This is a finite-type quasi-coherent $(G,\O_X)$-module.
It suffices to prove that $\Cal C_X=0$.
By Lemma~\ref{G-nakayama.lem}, it suffices to show that 
$j^*\Cal C_X=0$, where $j:Z\hookrightarrow X$ is the inclusion.
As $N$ acts on $Z$ freely, $\Cal C_Z=0$, and hence it suffices to 
show that $\Cal S_Z=\Cal S_X\times_X Z$ in a natural way.
This follows from Lemma~\ref{stabilizer-basics.lem}, 
as $j:Z\rightarrow X$ is a monomorphism.
\end{proof}

\begin{lemma}
Let $G$ be a smooth $S$-group scheme.
\begin{enumerate}
\item[\bf 1] If $X$ is a $G$-scheme, then the reduction $X\red$ has a 
unique $G$-scheme structure such that the inclusion $\mred:
X\red\hookrightarrow X$ 
is a $G$-morphism.
Hence 
$\sqrt{0}=\sqrt[G]{0}$.
If, moreover, $X$ is an LFI-scheme, then the normalization 
$X^\nu$ {\rm\cite[(28.48.12)]{SP}} 
has a unique $G$-scheme structure such that the morphism 
$\nu:X^\nu\rightarrow X$ is a $G$-morphism.
\end{enumerate}
Let $\varphi:X\rightarrow Y$ be a $G$-morphism.
\begin{enumerate}
\item[\bf 2] If $\varphi$ is a $G$-morphism \(resp.\ a principal $G$-bundle\), 
then
$\varphi\red:X\red\rightarrow Y\red$ is a $G$-morphism
\(resp.\ a principal $G$-bundle\).
If, moreover, $\varphi$ is a morphism between LFI-schemes such that
$\varphi(X\an0)\subset Y\an0$, then 
$\varphi^\nu:X^\nu\rightarrow Y^\nu$ is a $G$-morphism
\(resp.\ a principal $G$-bundle\).
\item[\bf 3]
If $\varphi:X\rightarrow Y$ is an algebraic quotient and 
$\varphi\red:X\red\rightarrow Y\red$ is a principal $G$-bundle,
then $\varphi$ is a principal $G$-bundle.
\end{enumerate}
\end{lemma}

\begin{proof}
{\bf 1} As $G\times X\red$ is reduced by \cite[(10.149.6)]{SP}, 
the composite
\[
G\times X\red \xrightarrow{1_G\times\mred}G\times X\xrightarrow{a}X
\]
uniquely factors through $\mred:X\red\rightarrow X$.
As $X\red$ is $G$-stable in $X$, 
$\sqrt{0}=\sqrt{0}^*=\sqrt[G]{0}$ by \cite[(4.30)]{HM}.
The latter part is proved similarly, using \cite[(10.149.7)]{SP} and
\cite[(28.48.15)]{SP}.

{\bf 2} Consider the diagram
\[
\xymatrix{
G\times X\red \ar[r]^{1_G\times\varphi\red}\ar[d]^a \ar@{}[dr]|{\text{(a)}} &
G\times Y\red \ar[r]^{1_G\times\mred_Y}\ar[d]^a \ar@{}[dr]|{\text{(b)}} &
G\times Y \ar[d]^a \\
X\red \ar[r]^{\varphi\red} &
Y\red \ar[r]^{\mred_Y} &
Y
}.
\]
As the square (b) and the whole rectangle (a)$+$(b) commutes and $\mred_Y$ is
a monomorphism, (a) commutes and $\varphi\red$ is a $G$-morphism.
If $\varphi$ is a principal $G$-bundle, then $\varphi$ is smooth,
and hence $X\times_YY\red$ is reduced.
As $1_X\times\mred_Y: X\times_YY\red\rightarrow X\times_YY=X$ is a 
surjective closed immersion, $X\times_YY\red=X\red$, and 
$\varphi\red:X\red\rightarrow Y\red$ is a base change of $\varphi$.
Hence it is a principal $G$-bundle.

The case of normalization is similar and left to the reader.
Note that 
if $\varphi:X\rightarrow Y$ is a principal $G$-bundle (which is a morphism
between LFI-schemes such that $\varphi(X\an0)\subset Y\an0$),
then $X\times_YY^\nu$ is normal, and 
$(1_X\times \nu_Y)\red:
X\times_Y Y^\nu\rightarrow X\red$ is integral and birational
(for the definition of birational morphisms, see \cite[(28.9.1)]{SP}).

{\bf 3}.
We may assume that $Y$ is affine.
Then $X$ is also affine, and hence 
$\sqrt[G]{0}=\sqrt{0}$ is contained in the $G$-radical $\rad_G(X)$ of $X$
by {\bf 1} and \cite[(4.27)]{HM}.
By assumption, the action of $G$ on $X\red$ is free,
and hence 
the action of $G$ on $X$ is also free by Lemma~\ref{red-free.lem}.
So $\varphi$ is a principal $G$-bundle by Lemma~\ref{DG.thm}.
\end{proof}

\begin{lemma}\label{local-criterion-freeness.lem}
Let $G$ be an LFF $S$-group scheme, and $\varphi:X\rightarrow Y$ be
an algebraic quotient.
Let $U$ be the free locus, and $V=\varphi(U)$.
For $y\in Y$, the following are equivalent.
\begin{enumerate}
\item[\bf 1] $y\in V$.
\item[\bf 2] $X_y:=\varphi^{-1}(y)\rightarrow y$ is a principal $G$-bundle.
\item[\bf 3] The action of $G$ on $X_y$ is free.
\end{enumerate}
\end{lemma}

\begin{proof}
{\bf 1$\Rightarrow$2}.
We have that $\rho:U\rightarrow V$ is a principal $G$-bundle, and
its base change $X_y\rightarrow y$ is also a principal $G$-bundle.

{\bf 2$\Rightarrow $3}.
This is trivial.

{\bf 3$\Rightarrow$1}.
We may assume that $Y=\Spec A$ is affine, and then $X=\Spec B$ is affine
and $A=B^G$.
Let $M$ be the coordinate ring of $\Cal S_X$, and let $C$ be the 
cokernel of $B\rightarrow M$.
Let $P$ be the prime ideal corresponding to $y$.
Since $A_P\rightarrow B_P$ is integral, $PB_P$ is contained in the radical of
$B_P$.
As $X_y\rightarrow X$ is a monomorphism, we have that 
$C\otimes_B(B_P/PB_P)=0$ by the assumption that the action of $G$ on 
$X_y$ is free and Lemma~\ref{stabilizer-basics.lem}.
As $C_P$ is a finite $B_P$-module, $PB_P\subset \rad(B_P)$, and
$PC_P=C_P$, we have that $C_P=0$ by Nakayama's lemma.
So no point of $X_y$ supports $C$.
That is, $X_y\subset U$.
Hence $y\in V$.
\end{proof}

\begin{proposition}\label{free-etale.prop}
Let $G$ be an \'etale finite $S$-group scheme \(in particular, LFF\).
Let $\varphi:X\rightarrow Y$ be an algebraic quotient by $G$.
Assume that 
$X$ and $Y$ are locally Noetherian and $\varphi$ is finite
\(for example, let $S=\Spec k$ with $k$ a field, 
$X$ be locally Noetherian, and 
if the characteristic of $k$ is positive, assume further that
$X$ is $F$-finite, see {\rm Lemma~\ref{F-finite-finite.thm}}\).
Let $U$ be the free locus of the action of $G$, and assume that 
the action is generically free.
Then $U$ agrees with the \'etale locus of $\varphi$.
\end{proposition}

\begin{proof}
We may assume that $Y=\Spec A$ is affine and connected.
So $X=\Spec B$ is affine $G$-connected and $A=B^G$.

Set $V=\varphi(U)$.
Let $\rho:U\rightarrow V$ be the restriction of $\varphi$.
Then $\rho:U\rightarrow V$ is a principal $G$-bundle.
As $G$ is \'etale, $\rho$ is \'etale, and hence $U$ is contained in
the \'etale locus.

We prove the opposite incidence.
Let $U'$ be the \'etale locus of $\varphi$.
Then it is a $G$-stable open subset.
We have shown that $U\subset U'$.
So to prove that $U'=U$, replacing $X$ by $U'$, we may assume that
$\varphi$ is \'etale, and we need to prove that $U=X$.
Again, we may assume that $X=\Spec B$ and $Y=\Spec A$ are affine with $A=B^G$,
and $Y$ is connected.
We may assume that $S=Y$.
Note that $B$ is a finite projective $A$-module.
As $Y$ is connected, it has a well-defined rank, say $r$.
On the other hand, the coordinate ring $\Gamma$ of $A$ has a finite projective
module.
Let $r'$ be its rank.
As the action of $G$ on $X$ is generically free, we have $r=r'$ by
Proposition~\ref{generically-free.prop}.
Let $y\in Y$.
Consider the base change of the map
\[
\Psi:G\times X\rightarrow X\times X \qquad((g,x)\mapsto (gx,x))
\]
by $y\rightarrow Y=S$.
It is 
\[
\Psi_y: G_y\times_y X_y\rightarrow X_y\times_y X_y.
\]
This map is surjective, since $\Psi$ is.
The map $\Psi_y$ is a map between affine schemes corresponding to the
$\kappa(y)$-algebra map between \'etale $\kappa(y)$-algebras
\[
I: B(y)\otimes_{\kappa(y)}B(y)\rightarrow \Gamma(y)\otimes_{\kappa(y)}B(y),
\]
where $?(y)$ means $?\otimes_A \kappa(y)$.
This map $I$ is injective, since the algebras are reduced and the 
corresponding map $\Psi_y$ is surjective.
On the other hand, the source and the target of $I$ are both of dimension 
$r^2$ over $\kappa(y)$.
So $I$ must be an isomorphism.
This shows that $X_y\rightarrow y$ is a principal $G$-bundle.
By Lemma~\ref{local-criterion-freeness.lem}, $y\in V$.
As $y$ is an arbitrary point of $Y$, we have that $V=Y$, and hence $U=X$.
\end{proof}

\begin{corollary}
Under the assuptions of {\rm Proposition~\ref{free-etale.prop}}, 
assume further that $X$ is regular and the action of $G$ is small.
Then the singular locus of $Y$ is $Y\setminus \varphi(U)$.
\end{corollary}

\begin{proof}
Set $V=\varphi(U)$.
Let $V'$ be the regular locus of $Y$.
We want to prove that $V=V'$.

As $U$ is regular and $U\rightarrow V$ is faithfully flat, $V$ is regular.
So $V\subset V'$.

Let $y\in V'$, $A=\Cal O_{Y,y}$, and $B=(\varphi_*\O_X)_y$.
Then by the smallness assumption, the branch locus of $B$ over $A$ has
codimension at least two.
By the purity of branch locus \cite[(X.3.1)]{SGA-1},
$B$ is \'etale over $A$.
That is, $y\in V$, and $V\supset V'$.
\end{proof}

\begin{lemma}\label{finite-invariance-S_2.lem}
Let $G$ be an LFF $S$-group scheme, and $\varphi:X\rightarrow Y$ an
algebraic quotient by $G$.
Assume that $X$ and $Y$ are locally Noetherian.
\begin{enumerate}
\item[\bf 1] If $\M$ is a quasi-coherent $\O_X$-module which satisfies 
$(S'_n)$, then $\varphi_*\M$ satisfies $(S'_n)$.
\item[\bf 2] 
If a quasi-coherent $(G,\O_X)$-module $\M$ on 
$X$ satisfies the $(S'_2)$ condition, 
then $(\varphi_*\M)^G$ satisfies the $(S'_2)$ condition.
\end{enumerate}
\end{lemma}

\begin{proof}
{\bf 1} follows from Lemma~\ref{finite-depth.lem} and
Lemma~\ref{DG.thm}.
{\bf 2} follows from {\bf 1} and Lemma~\ref{invariance-S_2.lem}.
\end{proof}

From the results we obtained so far, we can state the following.

\begin{theorem}\label{finite-resume.thm}
Let $S$ be 
a scheme, and $f:G\rightarrow H$ an fppf finite homomorphism between
flat $S$-group schemes, and $N=\Ker f$.
Let $\varphi:X\rightarrow Y$ be a $G$-morphism which is an algebraic
quotient by the action of $N$.
Assume 
that the action of $N$ on $X$ is small.
Let $U$ be the free locus of the action of $N$ on $X$, and $V:=\varphi(U)$.
Then we have the following.
\begin{enumerate}
\item[\bf 0] $\varphi$ is a $G$-enriched almost principal $N$-bundle
with respect to $U$ and $V$.
\item[\bf 1] (cf.~\cite[Theorem~2.4]{HN})
  Assume that $X$ is locally Krull.
Then $Y$ is also locally Krull, and $(\varphi^*?)^{**}
:\Ref(H,Y)\rightarrow \Ref(G,X)$ and
$(\varphi_*?)^N:\Ref(G,X)\rightarrow \Ref(H,Y)$ are quasi-inverse each other.
In particular, for $\L,\M\in\Ref(G,X)$, $\L\cong \M$ if and only if 
$(\varphi_*\L)^N\cong (\varphi_*\M)^N$ in $\Ref(H,Y)$.
$\M$ is indecomposable in $\Ref(G,X)$ if and only if 
$(\varphi_*\M)^N$ is so in $\Ref(H,Y)$.
This equivalence induces an isomorphism between $\Cl(H,Y)$ and $\Cl(G,X)$.
\item[\bf 2] Assume that $X$ is quasi-compact quasi-separated and 
locally Krull.
Then there is an exact sequence
\[
0\rightarrow H^1(N,\O_X^\times)\rightarrow \Cl(Y)\rightarrow \Cl(X)^N
\rightarrow H^2(N,\O_X^\times).
\]
\item[\bf 3] Assume that $G$ is of finite type.
Let $Y_0$ be a fixed Noetherian 
$H$-scheme with a fixed $H$-dualizing complex $\Bbb I_{Y_0}$, and
assume that $\varphi$ is a morphism in $\Cal F(G,Y_0)$.
Then $\varphi$ is finite, and
we have
\[
\omega_Y\cong (\varphi_*(\omega_X\otimes_{\O_X}\Theta_{N,X}))^N
\cong (\varphi_*\omega_X\otimes_{\O_Y}\Theta_{N,Y})^N
\]
as $(H,\O_Y)$-modules.
If, moreover, $X$ 
has a coherent $(G,\O_X)$-module $\M_X$ which is a full $2$-canonical module,
then
we have
\[
\omega_X\cong (\varphi^*\omega_Y)^{\vee\vee}\otimes_{\O_X}\Theta_{N,X}^*
\cong (\varphi^*(\omega_Y\otimes_{\O_Y}\Theta_{N,Y}^*))^{\vee\vee}
\]
as $(G,\O_X)$-modules.
\item[\bf 4] In {\bf 3}, 
If $\Theta_{N,Y_0}\cong\O_{Y_0}$ \(e.g., 
$N$ is \'etale, $N$ is Reynolds, or $S=\Spec k$ with $k$ a field and $G$ centralizes $N$\) ,
then $\omega_Y\cong (\varphi_*\omega_X)^N$ as $(H,\O_Y)$-modules.
If, moreover, $X$ 
has a coherent $(G,\O_X)$-module $\M_X$ which is a full $2$-canonical module,
then we have
$\omega_X\cong (\varphi^*\omega_Y)^{\vee\vee}$ as $(H,\O_X)$-modules.
\item[\bf 5] In {\bf 3}, assume that $\Theta_{N,Y_0}\cong\O_{Y_0}$.
Let $\L$ be an $H$-linearized invertible sheaf on $Y$.
Then
$\omega_X \cong \varphi^*\L$ as $(G,\O_X)$-modules if and only if
$\omega_Y\cong \L$ as $(H,\O_Y)$-modules and $X$ satisfies the 
$(S_2)$ condition.
If so, both $X$ and $Y$ are quasi-Gorenstein.
\item[\bf 6] Let the assumptions be as in {\bf 3}.
Then $\omega_X\cong \varphi^*\L$ for some $H$-linearized invertible sheaf
on $Y$ if and only if $\omega_Y$ is an invertible sheaf and $X$ satisfies
the $(S_2)$ condition.
If so, then both $X$ and $Y$ are quasi-Gorenstein.
If, moreover, $Y$ is connected, then these conditions are 
equivalent to say that
$Y$ is quasi-Gorenstein and $X$ satisfies $(S_2)$.
\end{enumerate}
\end{theorem}

\begin{proof}
{\bf 0} See Lemma~\ref{small-almost-principal.thm}.

{\bf 1} Note that $Y$ is also locally Krull by \cite[(6.3)]{Hashimoto4}.
Now the result follows from Corollary~\ref{almost-main.thm}.

{\bf 2} follows from Theorem~\ref{four-term.thm}.

Considering the fact that $X$ is $(S_2)$ implies $Y$ is $(S_2)$ 
(by Lemma~\ref{finite-invariance-S_2.lem}),
{\bf 3, \bf 4, \bf 5, \bf 6} 
are immediate from Corollary~\ref{canonical-cor.thm}.
For the conditions for $\Theta_{N,Y_0}$ to be trivial, see
Remark~\ref{knop-remark.thm}.
\end{proof}

\paragraph\label{a-invariant.par}
Let $k$ be a field, and $B=\bigoplus_{n\in\Bbb Z}B_n$
be a finitely generated positively graded $k$-algebra (that is, 
$B_n=0$ for $n<0$ and $B_0=k$).
Let $\omega_B$ denote the canonical module of $B$.
The base scheme $S$ is $\Spec k$, the group $G$ is the torus $\Bbb G_m$,
and the $G$-dualizing complex of $S$ is fixed to
$k$ (concentrated in degree zero).
Then $\omega_B$ is a finitely generated
$\Bbb Z$-graded nonzero $B$-module.
So
\[
a=a(B)=-\min\{n\in\Bbb Z\mid \omega_{B,n}\neq 0\}
\]
is well-defined.
The integer $a$ is called the {\em $a$-invariant} of Goto--Watanabe
\cite[(3.1.4)]{GW}.

\paragraph
If $B$ is a quasi-Gorenstein Noetherian $\Bbb Z^n$-graded $k$-algebra
such that there exists some homomorphism $h:\Bbb Z^n\rightarrow \Bbb Z$
such that $B$ is positively graded with respect to the induced $\Bbb Z$-grading.
Then there exists some unique 
$a\in\Bbb Z^n$ such that $\omega_B\cong B(a)$ as
$\Bbb Z^n$-graded $B$-modules.
We also call $a$ the $a$-invariant of $B$.
This definition is consistent with the one in (\ref{a-invariant.par}) 
when $n=1$ and $B$ is positively graded.

\begin{example}\label{finite-graded-canonical.ex}
Let $k$ be a field, and $N$ a finite $k$-group scheme.
Let $G=N\times \Bbb G_m$ and $H=\Bbb G_m$.

Let $B$ be a $G$-algebra, and $\varphi:X=\Spec B\rightarrow \Spec A=Y$ 
be the algebraic quotient, where $A=B^N$.
Assume that the action of $N$ on $X$ is small.
Note that $B$ is a $\Bbb Z$-graded $N$-algebra.
Assume that $B$ is positively graded (that is, $B=\bigoplus_{n\geq 0}B_n$
with $B_0=k$).
Let $k(a)$ be the one-dimensional $G$-module which is the 
one-dimensional $\Bbb G_m$-module concentrated in degree $-a$ and is
trivial as an $N$-module, and set $B(a)=B\otimes k(a)$.

If $N$ is either \'etale; linearly reductive; or abelian, then
$\omega_B^N=\omega_A$.
By Theorem~\ref{finite-resume.thm}, {\bf 5}, 
$\omega_B\cong B(a)$ as $(G,B)$-modules if and only if
$B$ is $(S_2)$ and $\omega_A\cong A(a)$ as graded $A$-modules, that is, 
$B$ is $(S_2)$ and $A$ is quasi-Gorenstein with the $a$-invariant $a$.

Even if $G$ is a general finite $k$-group scheme, 
$\omega_B\cong B$ as $(N,B)$-modules if and only if $A$ is quasi-Gorenstein
and $B$ is $(S_2)$, 
by Theorem~\ref{finite-resume.thm}, {\bf 6}.
The author does not know if the $a$-invariants of $A$ and $B$ agree in general.
\end{example}

\begin{example}\label{finite-polynomial.ex}
Let $B$ be the polynomial ring $k[x_1,\ldots,x_d]$ with $\deg x_i=1$ in
Example~\ref{finite-graded-canonical.ex} above.
As above, assume that $G$ acts on $B$ (that is, $N$ acts on $B$ linearly).
As in Example~\ref{finite-graded-canonical.ex},
assume that 
the action of $N$ on $X$ is small.
Set $A=B^N$.
Then by Theorem~\ref{finite-resume.thm}, 
\begin{enumerate}
\item[\bf 1] $\Cl(A)\cong H^1(N,B^\times)$, since $\Cl(B)=0$.
If, moreover, $N$ is \'etale, then $\Cl(A)\cong\Cal X(N)$ by
\cite[(4.13)]{Hashimoto4}.
\item[\bf 2] $\omega_A\cong (B\otimes \ext^d B_1\otimes\Theta_{N,k})^N$ 
as graded $A$-modules,
and
\[
(B\otimes_A \omega_A)^{**}
\cong B\otimes_k(\ext^d B_1\otimes_k \Theta_{N,k})
\]
as graded 
$(N,B)$-modules.
\item[\bf 3] The following are equivalent.
\begin{enumerate}
\item[\bf a] $\ext^d B_1\cong \Theta_{N,k}^*$ as $N$-modules.
\item[\bf b] $\ext^d B_1\cong \Theta_{N,k}^*\otimes k(-d)$ as $G$-modules.
\item[\bf c] $A$ is quasi-Gorenstein.
\item[\bf d] $A$ is quasi-Gorenstein of the $a$-invariant $-d$.
\end{enumerate}
\item[\bf 4] (cf.~ \cite{Broer}, \cite{Braun}, \cite{FW})
Assume that $N$ is either \'etale; linearly reductive; or abelian.
Then $\Theta_{N,k}$ is trivial, and the following are equivalent.
\begin{enumerate}
\item[\bf a] $N\subset \SL(B_1)$.
\item[\bf b] $A$ is quasi-Gorenstein.
\item[\bf c] $A$ is quasi-Gorenstein with the $a$-invariant $-d$.
\end{enumerate}
\end{enumerate}
\end{example}

\begin{proof}
We only prove {\bf 3}.
{\bf a$\Rightarrow$b$\Rightarrow$d$\Rightarrow$c} is easy.
If $A$ is quasi-Gorentein, then $\omega_A\cong A(a)$ for some $a\in\Bbb Z$ 
as $(H,A)$-modules.
So $B(a)\cong B\otimes_k(\ext^d B_1\otimes_k \Theta_{N,k})$.
Tensoring $B/B_+$, where $B_+$ is the irrelevant ideal of $B$, 
we get $k\cong \ext^d B_1\otimes_k \Theta_{N,k}$ as $N$-modules,
and {\bf c$\Rightarrow$a} follows.
\end{proof}

\section{Multisection rings}\label{multisection.sec}

\paragraph
Let $X$ be a locally Krull scheme.
We define a divisor on $X$ and the $\O_X$-module $\O_X(D)$ for a divisor on $X$.

Recall that 
$P^1(X)$ denotes the set of integral closed subschemes of codimension one
(\ref{LFI.par}).
Set $\Cal F=\prod_{W\in P^1(X)}\Bbb Z\cdot W$.
An element of $\Cal F$ is called a {\em formal divisor}.
$W\in P^1(X)$ as an element of $\Cal F$ is called a prime divisor.
For $D=(a_{D,W}W)_{W\in X}\in\Cal F$, the support $\supp D$ of $D$ is
$\{W\in P^1(X)\mid a_{D,W}\neq 0\}$.
For $D=(a_{D,W}W)$ and $D'=(a'_{D,W}W)$ in $\Cal F$,
we say that $D\geq D'$ if $a_{D,W}\geq a'_{D,W}$ for any $W$.
We say that $D$ is {\em effective} if $D\geq 0$.
If $\supp D$ is locally finite (see (\ref{LFI.par})) in $X$, we say that 
$D$ is a {\em divisor} on $X$.
The set of divisors $\Div(X)$ on $X$ forms a subgroup of $\Cal F$.
If $X$ is quasi-compact, $\Div(X)=\bigoplus_{W\in P^1(X)}\Bbb Z\cdot W$.
For $D\in\Div(X)$ and an open subset $U$ of $X$, we define 
the restriction $D|_U$ of $D$ to to be $(a_{D,\bar W}W)_{W\in P^1(U)}$,
where $\bar W$ is the closure of $W$ in $X$.

\paragraph
Let $X$ be integral.
Then for $f\in K^\times$ 
and $W\in P^1(X)$, we define $a_{f,W}$ to be the order of
$f$ in the DVR $\O_{X,W}\subset K=\kappa(\xi)$,
where $\xi$ is the generic point of $X$, and $K$ is the function field.
We define $\div f:=(a_{f,W}W)\in\Cal F$.
It is easy to see that $\div f$ is a divisor.
Note that $\div:K^\times\rightarrow \Div(X)$ is a homomorphism.
Its image $\div(K^\times)$ is denoted by $\Prin(X)$.
An element of $\Prin(X)$ is called a {\em principal divisor}.

\paragraph
Let $X$ be integral.
Let $\xi$ be its generic point, and $j:\xi\rightarrow X$ the inclusion.
The quasi-coherent sheaf $j_*j^*\O_X$ is denoted by $\K$.
It is the constant sheaf of $K=\kappa(\xi)$.
Let $D=(a_{D,W}W)\in\Div(X)$.

We define an $\O_X$-submodule $\O_X(D)$ of $\K$ by
\[
\Gamma(U,\O_X(D))=\{0\}\cup \{f\in K^\times\mid (\div f+D)|_U\geq 0\}.
\]
Note that $\O_X(D)$ is a rank-one reflexive quasi-coherent sheaf.

\paragraph
In general, a locally Krull scheme $X=\coprod_i X_i$ is the
disjoint union of its irreducible components.
We define
\[
\Prin(X):=\prod_i \Prin(X_i)\subset \prod_i\Div(X_i)=\Div(X).
\]
We define the (geometric) class group $\Cl'(X)$ to be $\Div(X)/\Prin(X)$.
Thus $\Cl'(X)=\prod_i\Cl'(X_i)$.

For $D\in\Div(X)$, $\O_X(D)$ is also defined componentwise.
It is easy to see that $\O_X:D\mapsto \O_X(D)$ gives a homomorphism from
$\Div(X)$ to $\Cl(X)$.

\begin{lemma}
$\O_X$ induces an isomorphism $\O_X:\Cl'(X)\rightarrow \Cl(X)$.
\end{lemma}

\begin{proof}
We may assume that $X$ is integral.
First, we prove that $\O_X$ is surjective.
Let $\M$ be a rank-one reflexive sheaf on $X$.
Let $\xi$ be the generic point of $\xi$, and $j:\xi\rightarrow X$ the
inclusion.
Note that $\K=j_*j^*\O_X$.
As $\M$ is rank-one reflexive, there is a monomorphism
\[
\M\xrightarrow{u} j_*j^*\M\cong j_*j^*\O_X\cong \K,
\]
and we can identify $\M$ with a subsheaf of $\K$.
Then it is easy to see that there exists some $D\in\Div(X)$ such that
$\M\cong\O_X(D)$.
That is, $\O_X:\Div(X)\rightarrow\Cl(X)$ is surjective.

Assume that $f:\O_X\cong\O_X(D)$.
Then $f\in\Hom_{\O_X}(\O_X,\O_X(D))=\Gamma(X,\O_X(D))\subset K$.
As $f$ is an isomorphism, $f\in K^\times$ and $\div f+D=0$,
and hence $D$ is principal.
It is obvious that $\O_X(\div f)=f^{-1}\O_X\cong \O_X$.
So $D\in\Ker(\O_X:\Div(X)\rightarrow \Cl(X))$ if and only if $D\in\Prin(X)$.

Thus the isomorphism $\O_X:\Cl'(X)\rightarrow \Cl(X)$ is induced.
\end{proof}

With this isomorphism, we identify $\Cl'(X)$ and $\Cl(X)$.

\paragraph\label{torus.par}
Let $S$ be a scheme, $\Lambda$ a finitely generated $\Bbb Z$-module,
and $G=\Spec \Bbb Z\Lambda\times_{\Spec \Bbb Z} S$,
where $\Bbb Z\Lambda$ 
is the group algebra $\bigoplus_{\lambda\in \Lambda}\Bbb Z t^\lambda
$ with each $t^\lambda$ group-like.
Let $\varphi:X\rightarrow Y$ be an affine $G$-invariant morphism.
So $X=\uSpec_Y \Cal A$ with $\Cal A=\bigoplus_{\lambda\in\Lambda}\Cal A_\lambda$ 
a graded quasi-coherent $\O_Y$-algebra.

\begin{lemma}\label{principal-torus-equiv.lem}
The following are equivalent.
\begin{enumerate}
\item[\bf 1] $\varphi$ is a principal $G$-bundle.
\item[\bf 2] $\O_Y\rightarrow \Cal A_0$ is an isomorphism, 
each $\Cal A_\lambda$ is an invertible sheaf, and the 
product $\Cal A_\lambda\otimes_{\O_Y}\Cal A_\mu\rightarrow \Cal A_{\lambda+\mu}$ 
is an isomorphism.
\item[\bf 2'] $\O_Y\rightarrow \Cal A_0$ is an isomorphism, 
  and the product
$\Cal A_\lambda\otimes_{\O_Y}\Cal A_{\mu}\rightarrow \Cal A_{\lambda+\mu}$ is 
  surjective for any $\lambda,\mu\in\Lambda$.
\item[\bf 3] 
For each $y\in Y$, there exists some affine open neighborhood $y\in U
=\Spec R$ and a faithfully flat $R$-algebra $R'$ 
such that $A'=\bigoplus_\lambda A_\lambda'$ with 
$A'_\lambda=R'\otimes_R \Gamma(U,\Cal A_\lambda)$ is isomorphic to the group 
algebra $R'\Lambda=\bigoplus_\lambda R't^\lambda$.
\end{enumerate}
If, moreover, $\Lambda$ is torsion-free, then $\varphi$ is a principal
$G$-bundle in the Zariski topology.
That is, we can take $R'=R$ in {\bf 3}.
\end{lemma}

\begin{proof}
{\bf 1$\Rightarrow$2} is obvious from the descent argument.

{\bf 2$\Rightarrow$2'} is trivial.
{\bf 2'$\Rightarrow$3}.
Take any affine open neighborhood $U=\Spec R$ of $y$.
Then $A=\Gamma(\varphi^{-1}(U),\O_X)$ is a graded algebra with $A_0=R$.
We may assume that
\[
\Lambda=\Bbb Z/(m_1)\oplus\Bbb Z/(m_2)\oplus\cdots\oplus\Bbb Z/(m_s)
\]
with $m_i\geq 0$, $m_i\neq 1$.
Let 
$\lambda_i$ be a generator of $\Bbb Z/(m_i)$.
So $\lambda_1,\ldots,\lambda_s$ together generate $\Lambda$.

For each $i$, there exists some expression 
$1=\sum_{l=1}^{m_i} u_{i,l} v_{i,l}$ with $u_{i,l}\in A_{\lambda_i}$ and 
$v_{i,l}\in A_{-\lambda_i}$.
Then contracting $U$ if necessary, we may assume that for each $i$, there
exists some $l_i$ such that $u_{i,l_i}v_{i,l_i}$ is invertible in $A_0$.
So each $t_i:=u_{i,l_i}$ are units of $A$.
If $\Lambda$ is torsion-free, then $m_i=0$ for each $i$, 
and $A=R[t_1^{\pm1},\ldots,t_s^{\pm s}]$.
So the last assertion has been proved.

We consider the case that $\Lambda$ may have a torsion.
We may assume that $m_1,\ldots,m_r\geq 2$ and $m_i=0$ for $i>r$.
Set
\[
R'=R[T_1,\ldots,T_r]/(T_1^{m_1}-t_1^{m_1},\ldots,T_r^{m_r}-t_r^{m_r}),
\]
where $T_1,\ldots,T_r$ are new variables of degree zero.
As $R'$ is a nonzero free $R$-module, $R'$ is faithfully flat over $R$.
Letting $t_i':=t_i\bar T_i^{-1}$, we have
that
\[
A':=R'\otimes_R A=R'[t_1',\ldots,t_r',t_{r+1}^{\pm1},\ldots,t_s^{\pm1}]
/((t_1')^{m_1}-1,\ldots,(t_r')^{m_r}-1)\cong R'\Lambda,
\]
as desired.

{\bf 3$\Rightarrow$1} is trivial.
\end{proof}

\paragraph\label{cl-map.par}
Let $\varphi:X\rightarrow Y$ be a morphism between locally Krull schemes
such that $\varphi(X\an 0)\subset Y\an 0$ and 
$\varphi(X\an 1)\subset Y\an 0\cup Y\an 1$.
We define the pull-back $\varphi^*:\Div(Y)\rightarrow \Div(X)$ by $\varphi$ 
by $\varphi^*(a_{V}V)_{V\in P^1(Y)}=(b_{W}W)_{W\in P^1(X)}$,
where 
$b_W=a_{\overline{\varphi(W)}}\length_{\O_{X,w}}(\O_{X,w}/\fm_{\varphi(w)}\O_{X,w})$
for the generic point $w$ of $W\in P^1(X)$
if $w\in Y\an 1$ ($\fm_{\varphi(w)}$ is the maximal ideal of the DVR 
$\O_{Y,\varphi(w)}$), and $b_W=0$ if $w\in Y\an 0$.
It is easy to see that $\varphi^*$ is a homomorphism, and 
$\varphi^*(\div(f))=\div(\varphi^*f)$ 
if both $X$ and $Y$ are integral and $f\in K^\times$,
where $K$ is the function field of $Y$.
So $\varphi^*:\Cl'(Y)\rightarrow \Cl'(X)$ is induced.

\begin{lemma}
Let $\varphi$ be as in {\rm (\ref{cl-map.par})}.
Then $[\M]\mapsto[(\varphi^*\M)^{**}]$ gives a homomorphism 
$\varphi^*:\Cl(Y)\rightarrow\Cl(X)$.
Moreover, the diagram
\[
\xymatrix{
\Cl'(Y) \ar[r]^{\O_Y} \ar[d]^{\varphi^*} &
\Cl(Y) \ar[d]^{\varphi^*} \\
\Cl'(X) \ar[r]^{\O_X} & \Cl(X)
}
\]
is commutative.
\end{lemma}

\begin{proof}
To prove the commutativity of the diagram, we may assume that both 
$X$ and $Y$ are integral.
It suffices to prove that $(\varphi^*\O_Y(D))^{**}\subset \varphi^*\K=\L$
agrees with $\O_X(\varphi^*D)$, where $\K$ and $\L$ are the constant sheaves
of the rational function fields of $Y$ and $X$, respectively.
As $(\varphi^*\O_Y(D))^{**}$ is reflexive, it suffices to prove that
$(\varphi^*\O_Y(D))^{**}_x=(\varphi^*\O_Y(D))_x\subset L$ agrees with 
$\O_X(\varphi^*D)_x$ for each $x\in X\an 1$, where $L$ is the
function field of $X$.
Let $y=\varphi(x)$.
First consider the case that $\codim y=1$.
If the coefficient of $\bar y$ in $D$ is $a$ and the ramification index
$\length_{\O_{X,x}}(\O_{X,x}/\fm_y \O_{X,x})$ is $b$, then the coefficient of 
$\bar x$ in $\varphi^*D$ is the product $ab$ by definition.
So $\O_X(\varphi^*D)_x=\fm_x^{-ab}$.
On the other hand,
\[
(\varphi^*\O_Y(D))_x=\O_Y(D)_y \O_{X,x}=\fm_y^{-a}\O_{X,x}
=\fm_x^{-ab}=\O_X(\varphi^*D)_x.
\]
Next, consider the case that $\codim y=0$.
Then the coefficient of $\bar x$ in $\varphi^*D$ is zero by definition.
So $\O_X(\varphi^*D)_x=\O_{X,x}$.
On the other hand, 
\[
(\varphi^*O_Y(D))_x=\O_Y(D)_y\O_{X,x}=K\O_{X,x}=\O_{X,x}=\O_X(\varphi^*D)_x,
\]
where $K$ is the rational function field of $Y$.
Hence $(\varphi^*\O_Y(D))^{**}=\O_X(\varphi^*D)$ as subsheaves of $\L$.
In particular, $\varphi^*[\O_Y(D)]=[\O_X(\varphi^*D)]$ in $\Cl(X)$, and
the diagram in problem is commutative.

In the diagram, $\O_Y$ and $\O_X$ are group isomorphisms, and the left
$\varphi^*$ is a homomorphism.
So the right $\varphi^*$ is also a homomorphism.
\end{proof}

\paragraph
Let $X$ be a locally Krull scheme, and $\Sigma$ a subset of $P^1(X)$.
For $D=(a_WW),D'=(a'_WW)\in\Div(X)$, we say that $D\geq_\Sigma D'$
if $a_W\geq a'_W$ for $W\in P^1(X)\setminus \Sigma$.
We define $\O_{X,\Sigma}(D)$ by
\[
\Gamma(U,\O_{X,\Sigma}(D))=\{0\}\cup\{f\in K^\times\mid
(\div f+D)\geq_{\Sigma\cup\{W\in P^1(X)\mid W\cap U=\emptyset\}} 0 \}.
\]
Note that $\O_{X,\Sigma}=\O_{X,\Sigma}(0)$ is a quasi-coherent $\O_X$-algebra.
We define $X_\Sigma=\uSpec \O_{X,\Sigma}$.
As a subintersection of a Krull domain is again a Krull domain 
\cite[(1.5)]{Fossum}, it is easy to see that $X_\Sigma$ is locally Krull.

As in (\ref{cl-map.par}),
the canonical map $j_\Sigma:X_\Sigma\rightarrow X$ induces a
surjective map $j_\Sigma^*:\Div(X)\rightarrow \Div(X_\Sigma)$.
As $j_\Sigma$ is birational, 
$j_\Sigma^*$ maps $\Prin(X)$ surjectively onto $\Prin(X_\Sigma)$.
By the snake lemma, we get the following easily.

\begin{lemma}[{Nagata's theorem \cite[(7.1)]{Fossum}}]\label{nagata.lem}
Let $X$ be a locally Noetherian scheme, and $\Sigma$ a subset of $P^1(X)$.
Then 
$j_\Sigma^*:\Cl'(X)\rightarrow \Cl'(X_\Sigma)$ is a surjective map
whose kernel is generated by the divisors supported in $\Sigma$.
\qed
\end{lemma}

\begin{lemma}\label{surjective.thm}
Let $G=\Bbb G_m^s$ be a split $s$-torus.
Let $\varphi:X\rightarrow Y$ be a principal $G$-bundle.
If $X$ is locally Krull, then the flat pullback $\varphi^*:\Cl'(Y)
\rightarrow \Cl'(X)$ \(see {\rm (\ref{cl-map.par})}\) is surjective.
\end{lemma}

\begin{proof}
We may assume that $Y$ is integral.
Let $W\in P^1(Y)$.
As $G$ is geometrically integral,
we have that $\varphi^{-1}(W)$ is integral.
Applying \cite[(5.13)]{Hashimoto4} to the map of local rings
$\O_{Y,W}\rightarrow \O_{X,\varphi^{-1}(W)}$, we have that $\varphi^{-1}(W)
\in P^1(X)$.
By Lemma~\ref{nagata.lem}, $\Cl'(X)\rightarrow \Coker\varphi^*$ 
factors through the surjection $\Cl'(X)\rightarrow \Cl'(X_\Sigma)$, where
$\Sigma$ is the set of prime divisors of the form $\varphi^{-1}(W)$ with
$W\in P^1(Y)$.
It is easy to see that $X_\Sigma=\varphi^{-1}(\eta)=
\Spec \kappa(\eta)[t_1^{\pm 1},\ldots,t_s^{\pm 1}
]$ by Lemma~\ref{principal-torus-equiv.lem}, 
where $\eta$ is the generic point of $Y$.
As $\Cl'(X_\Sigma)=0$, we are done.
\end{proof}

\paragraph
Let $s\geq 0$.
A $\Bbb Z^s$-graded ring $R=\bigoplus_{\lambda\in \Bbb Z^s}R_\lambda$ is
called a homogeneous DVR if $R$ is a Krull domain with a unique 
graded maximal ideal (that is, a maximal element in the set of 
graded ideals which are not equal to $R$) $P$ such that $\height P=1$.
We say that $(R,P)$ is a homogeneous DVR.
When we set $G$ to be the split torus $\Bbb G_m^s=\Spec \Bbb Z[t_1^{\pm 1},
\ldots,t_s^{\pm 1}]$, then $(R,P)$ is $G$-local.

\begin{lemma}
Let $(R,P)$ be a $\Bbb Z^s$-graded homogeneous DVR.
Then $P$ is a principal ideal.
\end{lemma}

\begin{proof}
As $P$ is generated by homogeneous elements, we can take a homogeneous
element $\alpha\in P\setminus P^{(2)}$, where $P^{(2)}$ is the second
symbolic power of $P$.
Then a minimal prime of $\alpha$ must be height-one homogeneous, and hence
$\alpha$ generates $P$.
\end{proof}

\paragraph
Let $(R,P)$ be a $\Bbb Z^s$-graded homogeneous DVR.
Set $Q$ to be the localization of $R$ by all the nonzero homogeneous elements.
It is the $G$-total ring of quotients of $R$ \cite[(3.1)]{Hashimoto},
as can be seen easily.

\begin{lemma}\label{graded-DVR.lem}
$Q=R[\alpha^{-1}]$ and 
$R=Q\cap R_P$.
\end{lemma}

\begin{proof}
As $R[\alpha^{-1}]$ does not have a nonzero homogeneous prime ideal, any
homogeneous element of $R[\alpha^{-1}]$ is a unit, and we cannot localize
homogeneously any more, and so $Q=R[\alpha^{-1}]$ holds.

Obviously, $R\subset Q\cap R_P$.
Let $\beta\in Q\cap R_P$.
We can write $\beta=b/\alpha^n$ for some $n\geq 0$ and $b\in R$.
Take $n$ as small as possible, and assume that $n>0$.
Then $b=\beta\alpha^n\in R\cap PR_P=P$, and $b$ is divisible by $\alpha$.
This is absurd.
\end{proof}

\paragraph
Let $B$ be a $\Bbb Z^s$-graded Krull domain with the field of fractions
$K$.
We say that $\Lambda$ is a defining family of homogeneous DVR's of $B$
if each element $R\in \Lambda$ is a homogeneous DVR which is
a graded subring of $Q_G(B)$, and $B=\bigcap_{R\in\Lambda}R$,
where $Q_G(B)$ is the localization of $B$ by all the nonzero homogeneous
elements of $B$.

\begin{lemma}\label{graded-defining-Krull.lem}
Let $B$ be a $\Bbb Z^s$-graded Krull domain with the field of fractions
$K$.
Set $\Lambda_1=\{B_{(P)}\mid \text{$P$ is homogeneous and of height one}\}$,
where $B_{(P)}$ denotes the localization of $B$ by the set of homogeneous
elemets of $B\setminus P$.
Then $\Lambda_1$ is a defining family of homogeneous DVR's of $B$.
If $\Lambda$ is a defining family of homogeneous DVR's of $B$, then
$\Lambda\supset \Lambda_1$, and thus $\Lambda_1$ is the smallest 
defining family of homogeneous DVR's of $B$.
\end{lemma}

\begin{proof}
It is easy to see that $(B_{(P)},PB_{(P)})$ is a homogeneous DVR and is 
a graded subring of $Q_G(B)=B_{(0)}$ for 
a homogeneous height one prime $P$ of $B$.
We prove that $B=\bigcap_{R\in\Lambda_1}R=\bigcap_P B_{(P)}$.
$B\subset \bigcap_P B_{(P)}$ is trivial.
Let $a/b\in\bigcap_P B_{(P)}$,
where $a\in B$ and $b$ is a nonzero homogeneous element of $B$.
As a minimal prime of $Bb$ is homogeneous, $a/b\in B_Q$ for any 
height one inhomogeneous prime ideal $Q$.
On the other hand, obviously, $a/b\in B_P$ for any $P$.
Thus $a/b\in( \bigcap_P B_P)\cap (\bigcap_Q B_Q)=B$, since $B$ is a Krull 
domain.

Next, let $\Lambda$ be a defining family of homogeneous DVR's of $B$.
For $R\in\Lambda$, let $\frak m_R$ be the graded maximal ideal of $R$.
Then
\[
B=\bigcap_{R\in\Lambda}R=Q_G(B)\cap \bigcap_R R_{\frak m_R}
\]
by Lemma~\ref{graded-DVR.lem}.

Let $P$ be a homogeneous height-one prime ideal of $B$, and 
assume that $B_P\supset Q_G(B)$.
Set $\frak P=PB_P\cap Q_G(B)$.
Then $\frak P\cap B=PB_P\cap B=P$, and hence $\frak P=PQ_G(B)$ contains $1$,
and this is a contradiction.
So $B_P$ does not contain $Q_G(B)$.
When we express $Q_G(B)=\bigcap_{R'\in \Lambda'}R'$, where $\Lambda'$ is a
set of DVR's whose field of quotients are $K$, then 
$B_P\in \Lambda'\cup\{R_{\frak m_R}\mid R\in\Lambda\}$ by
\cite[(12.3)]{CRT}.
As we know that $B_P\notin \Lambda'$, $B_P=R_{\frak m_R}$ for some $R\in
\Lambda$, and hence $B_{(P)}=Q_G(B)\cap B_P
=Q_G(B)\cap R_{\frak m_R}=R\in\Lambda$.
Hence $\Lambda_1\subset\Lambda$.
\end{proof}

\paragraph
Let $Y$ be a 
locally Krull integral $S$-scheme which is quasi-compact and separated over 
$S$.
Assume that $Y$ has an $h_Y$-ample Cartier 
divisor $D$ (that is, $D$ is a Weil divisor on $Y$ such that 
$\O_Y(D)$ is an $h_Y$-ample invertible sheaf), 
where $h_Y:Y\rightarrow S$ is the structure map.
Set $\Cal A=\Cal R(Y;D):
=(h_Y)_*(\bigoplus_{n\geq 0}\O_Y(nD)T^n)$, where $T$ is a variable.
By assumption, the canonical morphism $u:Y\rightarrow \bar Y:=\uProj\Cal A$
is an open immersion \cite[(27.24.14)]{SP}.
We identify $Y$ with the image $u(Y)$ of $u$, and regard $Y$ as an
open subset of $\bar Y$.

\begin{lemma}
$Y$ is large in $\bar Y$.
\end{lemma}

\begin{proof}
The question is local on $S$, and hence we may assume that $S$ is affine.
The question is also local on $\bar Y$, so it suffices to show that
for any $n>0$ and $0\neq s\in\Cal A_n=\Gamma(Y,\O(nD))$, 
$Y_s$ is large in $D_+(sT^n)$, where
$Y_s=Y\setminus \Supp(\Coker(s:\O\rightarrow\O(nD)))$.
Note that $D_+(sT^n)$ is affine with the coordinate ring
\[
\{0\}\cup\{f\in K^\times \mid \div f+r(\div s+nD)\geq 0 \text{ for some $r$}\}
=\Gamma(Y_s,\O_Y),
\]
and the inclusion $Y_s\rightarrow D_+(sT^n)$ is the obvious map.
Set $R=\Gamma(Y_s,\O_Y)$.
As $R=\Gamma(Y_s,\O_Y)$ and $Y_s$ is locally Krull integral, we have that
$R=\bigcap_{W\in P^1(Y_s)}\Cal O_{\bar Y,W}$ and $R$ is a Krull domain,
where $P^1(?)$ dentoes the set of prime divisors.
Then as $R$ is a Krull domain and is the coordinate ring of $D_+(sT^n)$, 
$R=\bigcap_{W\in P^1(D_+(sT^n))}\Cal O_{\bar Y,W}$.
By \cite[(12.3)]{CRT}, each $W$ in $P^1(D_+(st^n))$ must intersect $Y_s$.
Namely, $Y_s$ is large in $D_+(sT^n)$.
\end{proof}

\paragraph
Let $Y$ and $D$ be as above.
Let $s\geq 1$.
Set $G$ to be the split torus of relative dimension $s$.
That is, $G=\Spec \Bbb Z[t_1^{\pm1},\ldots,t_s^{\pm s}]\times_{\Spec \Bbb Z}S$.
Let $D_1,\cdots,D_s$ be Weil divisors on $Y$, and assume that 
we can write $D=\sum_{i=1}^s\mu_i D_i$ for some $\mu=(\mu_1,\ldots,\mu_s)
\in\Bbb Z^s$.

Set $\Cal D=\bigoplus_{\lambda\in\Bbb Z^s}\O(D_\lambda )t^\lambda$,
where for $\lambda=(\lambda_1,\ldots,\lambda_s)\in\Bbb Z^s$, 
$D_\lambda:=\sum_{i=1}^s \lambda_iD_i$, and $t^\lambda=t_1^{\lambda_1}\cdots
t_s^{\lambda_s}$.
Note that $\Cal D$ is a 
quasi-coherent subalgebra of the constant sheaf of algebra
$K[t_1^{\pm 1},\ldots,t_s^{\pm s}]$ over $Y$, where $K$ is the rational
function field of $Y$.

The (relative) multisection ring of $D_1,\cdots,D_s$ is defined to be
\[
\Cal B=\Cal R(Y;D_1,\ldots,D_s):=(h_Y)_*(\Cal D).
\]

\paragraph We set $X=\uSpec_S\Cal B$, and $Z=\uSpec_Y \Cal D$.
Note that $\Cal D$ is $\Bbb Z^s$-graded, and hence the canonical
map $\pi:Z\rightarrow Y$ is a $G$-invariant morphism.
There is a canonical map $v:Z\rightarrow X$, since $\Cal B=(h_Y)_*(\Cal D)$.
It is a $G$-morphism.

\begin{lemma}
There is a large open subset $V$ of $Y$ such that $D_i|_V$ is Cartier
\(that is, $\O_V(D_i|_V)$ is invertible\) for $i=1,\ldots,s$.
\end{lemma}

\begin{proof}
Let $Y=\bigcup_{j\in J} Y_j$ be an affine open covering with $Y_j$ connected.
Then by the proof of \cite[(5.33)]{Hashimoto4}, for each $j$, 
we can take a large open subset $V_j\subset Y_j$ such that $D_i|_{V_j}$ is
Cartier for $i=1,\ldots,s$.
Now define $V:=\bigcup_{j\in J}V_J$.
\end{proof}

We fix such a $V$.

\begin{proposition}\label{kurano-crutial.prop}
$v:Z\rightarrow X$ above is an open immersion.
We identify $Z$ by $v(Z)$ and regard $Z$ as an open subscheme of $X$.
Then $U:=\pi^{-1}(V)\subset Z$ is large in $X$.
\end{proposition}

\begin{proof}
The question is local on $S$, and we may assume that $S=\Spec R$ is affine
and hence $B=\Cal B=\bigoplus_\lambda B_\lambda$ is a graded $R$-algebra.
Let $n>0$, and $s\in \Gamma(Y,\Cal O_Y(nD))$.
Then the degree $\lambda$ component $B[(st^{n\mu})^{-1}]_\lambda$
of the localization $B[(st^{n\mu})^{-1}]$ is
\[
\bigcup_{r\geq 0} (st^{n\mu})^{-r}B_{\lambda+rn\mu}
= (\bigcup_r \Gamma(Y,\O_Y(D_\lambda+r(\div s+nD))))t^\lambda
=\Gamma(Y_s,\O_Y(D_\lambda)).
\]
So if, moreover, $Y_s$ is also affine, then 
$v$ maps $\pi^{-1}(Y_s)$ isomorphically 
onto the open subset $X_{st^{n\mu}}=\Spec B[(st^{n\mu})^{-1}]$.

We can take a sufficiently divisible $n$ and $s_1,\ldots,s_{m}\in
\Gamma(Y,\O_Y(nD))$ such that each $Y_{s_i}$ is affine and 
$\bigcup_i Y_{s_i}=Y$.
Hence $v$ is an open immersion whose image is
$X\setminus V(J)$,
where
\begin{equation}\label{J.eq}
J=(s_1t^{n\mu},\ldots,s_mt^{n\mu})\subset B.
\end{equation}

To prove that $U$ is large, since $V$ is large in $Y$, replacing $Y$ by 
$V$ (this does not changes $X$), we may assume that $V=Y$ (and $U=Z$).
It suffices to prove that $J$ has height at least two.

Before we finish the proof, we need some constructions.
\let\qed\relax
\end{proof}

\paragraph
Let the notation be as above ($S$ is affine).
For each $W\in P^1(Y)$, define
$P_W=\bigoplus_{\lambda\in \Bbb Z^s}\Gamma(Y,\O_Y(D_\lambda-W))t^\lambda$.
This is a graded prime ideal of $B$.
For $n>0$ and $s\in \Gamma(Y,\O_Y(nD))$, the localization 
$B[(st^{n\mu})^{-1}]\otimes_B P_W$ is $\bigoplus_\lambda\Gamma(Y_s,
\O_Y(D_\lambda-W))t^\lambda$.
As affine  $Y_s$ with $Y_s\cap W\neq \emptyset$ forms a fundamental set
of open neighborhoods of the generic point $w$ of $W$,
we have
\begin{multline}\label{kurano.eq}
B_{\Cal S}=\bigoplus_\lambda(\{0\}\cup\{f\in K^\times\mid \div(f)+D_\lambda
\geq_{P^1(Y)\setminus \{W\}} 0\})\cdot t^\lambda\\
=
\O_{Y,W}[(\alpha^{-c_1}t_1)^{\pm 1},\ldots,(\alpha^{-c_s}t_s)^{\pm 1}]
\end{multline}
and 
\begin{equation}\label{kurano2.eq}
(P_W)_{\Cal S}=\bigoplus_\lambda
(\{0\}\cup\{f\in K^\times\mid \div(f)+D_\lambda
>_{P^1(Y)\setminus \{W\}} 0\})\cdot t^\lambda
=\alpha B_{\Cal S},
\end{equation}
where $\Cal S$ is the homogeneous multiplicatively closed subset of $B$
given by 
\[
\Cal S=\{st^{n\mu}\mid n>0, \; s\in\Gamma(Y,\O_X(nD)),\; \text{$Y_s$ is affine},
\; Y_s\cap W\neq\emptyset\}\cup\{1\},
\]
$\alpha$ is the generator of the maximal ideal of $\O_{Y,W}$,
and $c_i$ is the coefficient of $W$ in the divisor $D_i$.
So $B_{\Cal S}$ 
is a homogeneous DVR with the graded maximal ideal $(P_W)_{\Cal S}$.
In particular, $P_W$ is a homogeneous height one prime of $B$, and 
$B_{\Cal S}$ is the homogeneous localization $B_{(P_W)}$.

\begin{lemma}\label{expression-B.lem}
The map $W\mapsto P_W$ gives a bijection between $P^1(Y)$ and the
set $HP^1(B)$ of homogeneous height one prime ideals of $B$.
We have $B=\bigcap_{W\in P^1(Y)}B_{(P_W)}$.
$B$ is a Krull domain.
\end{lemma}

\begin{proof}
By the first equality of (\ref{kurano.eq}), it is easy to see that
$B=\bigcap_{W\in P^1(Y)}B_{(P_W)}$.
As each $B_{(P_W)}$ is a homogeneous DVR, $B$ is (graded) Krull.
We know that $P_W\in HP^1(B)$ for $P^1(Y)$.
As $Y$ is a separated scheme, the description of (\ref{kurano2.eq}) shows
that $W\mapsto P_W$ is injective.
By Lemma~\ref{graded-defining-Krull.lem} and the fact
$B=\bigcap_{W\in P^1(Y)}B_{(P_W)}$ (with $Q(B)=Q(B_{(P_W)})=K(t_1^{\pm1},\ldots,
t_s^{\pm1})$), 
$W\mapsto P_W$ is surjective.
\end{proof}

The following lemma finishes the proof of Proposition~\ref{kurano-crutial.prop}.

\begin{lemma}
Let $J$ be the ideal of $B$ defined in {\rm(\ref{J.eq})}.
Then the height of $J$ is at least two.
\end{lemma}

\begin{proof}
Assume the contrary.
As $J$ is graded, $J$ is contained in some $P_W$.
Since $Y=\bigcup_i Y_{s_i}$, there is some $i$ such that $Y_{s_i}$ intersects
$W$.
Then $s_it^{n\mu}\in J$ cannot be an element of $P_W$ by the definition of 
$P_W$.
A contradiction.
\end{proof}

Now Proposition~\ref{kurano-crutial.prop} has been proved.
\qed
\medskip

We have a diagram of $S$-schemes
\begin{equation}\label{kurano-almost.eq}
\xymatrix{
X & U \ar@{_{(}->}[l]_i \ar[r]^{\rho} & V \ar@{^{(}->}[r]^j & Y
},
\end{equation}
where $X=\uSpec_S\Cal B$, $V$ is a large open subset of $Y$ such that 
$D_l|_V$ is Cartier for $l=1,\ldots,s$, 
$\pi:Z\rightarrow Y$ and $v:Z\rightarrow X$ are the canonical maps,
$U=\pi^{-1}(V)$, 
$\rho:U=\pi^{-1}(V)\rightarrow V$ the restriction of $\pi:Z\rightarrow Y$,
$i:U\rightarrow X$ the composite $U\hookrightarrow Z \xrightarrow v X$,
and $j:V\rightarrow Y$ the inclusion.

\begin{theorem}\label{kurano-main.thm}
Let $S$ be a scheme, $h_Y:Y\rightarrow S$ an integral locally Krull 
$S$-scheme with an ample Cartier divisor $D$.
Let $G$ be the $s$-torus 
\[
\Spec \Bbb Z[t_1^{\pm1},\ldots,t_s^{\pm1}]\times_{
\Spec \Bbb Z}S
\]
over $S$.
Let $s\geq 1$, and $D_1,\ldots,D_s$ divisors on $Y$ such that $
D\in \sum_i\Bbb Z D_i$.
Let the diagram {\rm(\ref{kurano-almost.eq})} be constructed as above.
Then it is a rational almost principal $G$-bundle.
$X$ is a locally Krull scheme.
\end{theorem}

\begin{proof}
By construction, $\pi:Z\rightarrow Y$ is $G$-invariant, and 
$v:Z\rightarrow X$ is a $G$-morphism.
$V$ is large in $Y$ by construction.
As $U\hookrightarrow Z$ is the base change of $V\hookrightarrow Y$,
$U$ is a $G$-stable open subset of $Z$.
So the diagram is a diagram of $G$-schemes, and 
$G$ acts on $Y$ and $V$ trivially.

$i$ is an open immersion and $i(U)$ in $X$ is large by
Proposition~\ref{kurano-crutial.prop}.

The fact that $\rho$ is a principal $G$-bundle follows from
Lemma~\ref{principal-torus-equiv.lem} easily.

To prove the last assertion, we may assume that $S$ is affine, and
this case is done in Lemma~\ref{expression-B.lem}.
\end{proof}

\begin{corollary}\label{kurano-correspondence.cor}
Let the notation be as above.
Then 
$\gamma=i_*\rho^*j^*:\Ref(Y)\rightarrow\Ref(G,X)$ is an equivalence.
The quasi-inverse is given by $\delta=(?)^Gj_*\rho_*i^*:\Ref(G,X)\rightarrow
\Ref(Y)$.
For a divisor $E$ on $Y$, $\O_Y(E)$ corresponds to 
\[
(h_Y)_*(\bigoplus_{\lambda\in\Bbb Z^s}\O_Y(D_\lambda+E)t^\lambda).
\]
In particular, $\O_Y(D_\nu)$ corresponds to $\O_X(\nu)$, where 
$?(\nu)$ denotes the shift of degree.
$\gamma$ is equivalent to $v_*(?)^{**}\pi^*$, and $\delta$ is equivalent to
$(?)^G\pi_*v^*$, where $(?)^{**}$ denotes the double dual.
\end{corollary}

\begin{proof}
Follows easily from Theorem~\ref{kurano-main.thm} and 
Theorem~\ref{main.thm}.
\end{proof}

\begin{lemma}
Let the notation be as above, and let $\N\in\Ref(Y)$ corresponds to 
$\M\in\Ref(G,X)$.
Namely, set
\[
\M=\bigoplus_{\lambda \in\Bbb Z^s}(h_Y)_*(\N(D_\lambda)))t^\lambda.
\]
Then the $G$-local cohomology $\uH^i_{X\setminus Z}(\M)$ is zero 
for $i=0,1$, and
\[
\uH^i_{X\setminus Z}(\M)\cong \bigoplus_{\lambda\in\Bbb Z^s}
(R^{i-1}h_Y)_*(\N(D_\lambda))t^\lambda
\]
for $i\geq 2$, where $\N(D_\lambda)$ denotes the reflexive sheaf
$(\N\otimes_{\O_Y} \O_Y(D_\lambda))^{**}$.
\end{lemma}

\begin{proof}
From \cite[(4.10)]{HO2}, the sequence
\[
0\rightarrow \uH^0_{X\setminus Z}(\M)\rightarrow \M\xrightarrow u v_*v^*\M
\rightarrow \uH^1_{X\setminus Z}(\M)\rightarrow 0
\]
is exact, and $R^{i-1}v_*(v^*\M)\cong \uH^i_{X\setminus Z}(\M)$ for $i\geq 2$.
As $\M$ is reflexive and $Z$ is large in $X$, $u:\M\rightarrow v_*v^*\M$ is
an isomorphism.
The result follows.
\end{proof}

\paragraph\label{graded-linearization.par}
Let $R$ be a commutative ring, $M$ a finitely generated abelian group, 
and 
$G:=\Spec RM$, where $RM=\bigoplus_{m\in M}R t^m$ is a group algebra
of $M$ over $R$.
Letting each $t^m$ group-like, $G$ is an $R$-group scheme.
A $G$-module is nothing but an $RM$-comodule.
If $V$ is a $G$-module, then $V=\bigoplus_{m\in M}V_m$ as a $G$-module, 
where
\begin{equation}\label{weight-space.eq}
V_m=\{v\in V\mid \omega_V (v)=v\otimes t^m\}.
\end{equation}
Conversely, if $V=\bigoplus_{m\in M}V_m$ as a graded $R$-module,
then letting (\ref{weight-space.eq}) the definition, $V$ is a $G$-module,
and a $G$-module and an $RM$-comodule and an $M$-graded $R$-module are the
same thing.

So a $G$-algebra is an $M$-graded $R$-algebra $B=\bigoplus_m B_m$,
where $\omega_B(b)=b\otimes t^m$ for $b\in B_m$.
We follow the convention that if $G$ acts on an affine $R$-scheme $X=\Spec B$, 
then $B$ is a $G$-module by $(gb)(x)=b(g^{-1}x)$.
That is, $\alpha(b)=t^{-m}\otimes b$ for $b\in B_m$
(since the antipode of $RM$ sends $t^m$ to $t^{-m}$), where $\alpha:B\rightarrow
RM\otimes_R B$ is the map corresponding to the action $G\times X\rightarrow X$.

For a $(G,B)$-module $N$, the $G$-linearization
\[\phi:(RM\otimes_R B){}_\alpha\otimes_B N\rightarrow 
(RM\otimes_R B){}_\beta \otimes_B N
\]
maps $(1\otimes 1)\otimes n$ to $(t^{-m}\otimes 1)\otimes n$ for $n\in N_m$,
where $\beta$ is given by 
$\beta(b)=1\otimes b$, and corresponds to the second projection 
$p_2:G\times X\rightarrow X$.

\begin{proposition}[cf.~{\cite[(1.1), (3)]{EKW}}]\label{multi-class.prop}
In {\rm Theorem~\ref{kurano-main.thm}}, assume that $S$ is quasi-compact
quasi-separated.
Then we have an exact sequence
\[
0\rightarrow \Cal X(G,X)\xrightarrow{\alpha} \Cal X(G)\xrightarrow{\beta}
\Cl(Y)\xrightarrow{\gamma} \Cl(X)\rightarrow 0,
\]
where $\Cal X(G)=\Bbb Z^s$, 
$\Cal X(G,X)=\{\lambda\in\Cal X(G) \mid B^\times\cap B_\lambda\neq\emptyset\}$
\(where $B=\Gamma(X,\O_X)$\), $\alpha$ is the inclusion, 
$\beta(\varepsilon_i)=\O_Y(D_i)$ \($\varepsilon_i
=(0,\ldots,0,1,0,\ldots,0)$ with
$1$ at the $i$th place\), and
$\gamma=(i^*)^{-1}\rho^*j^*$.
\end{proposition}

\begin{proof}
The map $j^*:\Cl(Y)\rightarrow \Cl(V)$ is an isomorphism.
The map $\rho^*:\Cl(V)\rightarrow \Cl(U)$ is surjective by
Lemma~\ref{surjective.thm}.
As the map $i^*:\Cl(X)\rightarrow \Cl(U)$ is an isomorphism,
we have that $\gamma=
(i^*)^{-1}\rho^*j^*:\Cl(Y)\rightarrow \Cl(X)$ is surjective.

Note that 
$\gamma(\O_Y(D_\lambda))=\O_X(\lambda)$ in $\Cl(G,X)$ by
Corollary~\ref{kurano-correspondence.cor}.
In particular, it is zero in $\Cl(X)$.

Note that the map 
$\gamma:\Cl(Y)\rightarrow \Cl(G,X)$ is an isomorphism
by Theorem~\ref{main.thm}, and 
$\gamma(\O_Y(D_i))=t_i^{-1}\O_X=\O_X(\varepsilon_i)$ in $\Cl(G,X)$.

The kernel of the forgetful map 
$r:\Cl(G,X)\rightarrow \Cl(X)$ is the algebraic first $G$-cohomology 
group $H^1\alg(G,\O_X^\times)$, see \cite[Theorem~7.1]{Dolgachev} and
Theorem~\ref{four-term.thm}.
It is explained as follows.
An element of $\Ker r$ is the isomorphism class of a rank-one free
module $\O_X$ equipped with a $G$-linearization $\Phi:
a^*\O_X\rightarrow p_2^*\O_X$.
However, both $a^*\O_X$ and $p_2^*\O_X$ are identified with $\O_{G\times X}$, and
$\Phi$ is nothing but a unit 
element of the ring $C=\Gamma(G\times X,\O_{G\times X})
=\Bbb Z [t_1^{\pm 1},\ldots,t_s^{\pm 1}]\otimes_{\Bbb Z}B$ (by the projection
formula \cite[(3.9.4)]{Lipman}).
As $B$ is a domain, we can write $\Phi=t^\lambda\otimes b$ with $b\in B^\times$
and $\lambda\in \Bbb Z^s$.
From the cocycle condition on $\Phi$, we have that $\Phi=t^\lambda\otimes 1$.
Conversely, $t^\lambda\otimes 1$ satisfies the cocycle condition, and
the group of $1$-cocycles $Z^1\alg(G,\O_X^\times)$ 
is the character group $\Cal X(G)$.
By definition,
\[
B^1\alg(G,\O_X^\times)=\{\phi(gx)/\phi(x)\mid \phi\in B^\times\}\subset 
Z^1\alg(G,\O_X^\times).
\]
As $\phi$ is a homogeneous element, it has a degree, say $\lambda$.
Then $\phi(gx)/\phi(x)=t^\lambda\otimes 1$, and thus
$B^1\alg(G,\O_X^\times)=\Cal X(G,X)$.

Then as in (\ref{graded-linearization.par}), the linearization
of $t_i^{-1}B$ corresponds to $t_i\otimes 1\in Z^1\alg(G,\O_X^\times)$,
and the exact sequence has been proved.
\end{proof}

\begin{proposition}[cf.~{\cite[(1.2), (1.3)]{HK}}]\label{multi-canonical.prop}

Let the notation be as in {\rm Theorem~\ref{kurano-main.thm}}.
Let $S$ be Noetherian with a fixed dualizing complex $\Bbb I_S$,
and assume that $Y$ and $X$ are of finite type over $S$.
Then
\[
\omega_X\cong \bigoplus_{\lambda\in\Bbb Z^s}(h_Y)_*\omega_Y(D_\lambda)t^\lambda.
\]
$\omega_Y\cong \O_Y(D_\lambda)$ as $\O_Y$-modules 
if and only if $\omega_X\cong\O_X(\lambda)$ as $(G,\O_X)$-modules.
\end{proposition}

\begin{proof}
As $G_{\Bbb Z}=\Spec \Bbb Z[t_1^{\pm 1},\ldots,t_s^{\pm 1}]$ 
is a $\Bbb Z$-smooth 
abelian group, $\Lie(G_{\Bbb Z})$ is trivial, and hence
$\Theta_S=h_S^*(\ext^s\Lie(G_{\Bbb Z}))$ is trivial, where $h_S:S\rightarrow
\Spec \Bbb Z$ is the structure map.
Now the result follows from Theorem~\ref{canonical.thm} immediately.
\end{proof}

\paragraph\label{veronese.par}
Let $\Lambda$ be the abelian group $\Bbb Z^s$, and $\Gamma$ its subgroup.
Let $H$ be the torus $\Spec \Bbb Z \Gamma\times_{\Spec\Bbb Z}S$.
The inclusion of group rings $\Bbb Z \Gamma\hookrightarrow \Bbb Z \Lambda$ 
induces an fppf homomorphism of $S$-group schemes $f:G\rightarrow H$.
We set $N=\Ker f=\Spec \Bbb Z (\Lambda/\Gamma)\times_{\Spec \Bbb Z}S$.

Let $Y$ be an $S$-scheme on which $G$ acts trivially, 
and $\M=\bigoplus_{\lambda\in \Lambda} \M_\lambda$ a $(G,\O_Y)$-module.
Then $\M^N=\bigoplus_{\lambda\in\Gamma}\M_\lambda$ 
is nothing but the Veronese submodule
of $\M$.

\paragraph
Let the assumptions be as in Theorem~\ref{kurano-main.thm}.
Let $M$, $\Gamma$, $f:G\rightarrow H$, and $N$ be as above.
Set 
\[
X'=\uSpec_S\Cal B^N
=
\uSpec_S(h_Y)_*(\bigoplus_{\lambda\in\Gamma}\O(D_\lambda)t^\lambda).
\]
If $\lambda_1,\ldots,\lambda_{s'}$ is a $\Bbb Z$-basis of $\Gamma$ and
when we set $D_l':=D_{\lambda_l}$, then we have
\[
\Cal B^N=\Cal R(Y;D_1',\ldots,D_{s'}')
=(h_Y)_*(\bigoplus_{\alpha\in\Bbb Z^{s'}}\Cal O(\alpha_1D_1'+\cdots+
\alpha_{s'}D_{s'}')t^{\sum_i\alpha_i\lambda_i}).
\]
The schemes and morphisms constructed from the divisors $D_1',\ldots,D_{s'}'$
instead of $D_1,\ldots,D_s$ are denoted by
$\pi':Z'\rightarrow Y$ and $\rho':U'=(\pi')^{-1}(V)\rightarrow V$.
Thus $Z'=\uSpec_Y\Cal D^N$.
Let $\tau:Z\rightarrow Z'$ be the map induced by the map of $\O_Y$-algebras
$\Cal D^N\hookrightarrow \Cal D$.
It is an algebraic quotient by $N$.
Note that $\pi'\tau=\pi$ and $\tau^{-1}(U')=U$.
Let $\upsilon:U\rightarrow U'$ be the restriction of $\tau$.
Let $\theta:X\rightarrow X'$ be the map corresponding to the map of
$\O_S$-algebras $\Cal B^N\hookrightarrow \Cal B$.
Thus we get the commutative diagram
\[
\xymatrix{
X \ar[d]^\theta & 
U \ar[d]^\upsilon \ar@{_{(}->}[l]_i \ar[r]^{\rho} & 
V \ar@{^{(}->}[r]^j \ar[d]^{\id_V} & Y \ar[d]^{\id_Y} \\
X' &
U' \ar@{_{(}->}[l]_{i'} \ar[r]^{\rho'} & 
V \ar@{^{(}->}[r]^j & Y
}
\]
whose first and second rows are rational almost principal $G$- and 
$H$-bundles, respectively.

\begin{lemma}\label{Veronese-main.lem}
Let the notation be as above.
Then $\theta:X\rightarrow X'$ is a $G$-enriched almost principal $N$-bundle
with respect to $U$ and $U'$.
\end{lemma}

\begin{proof}
It suffices to show that $\upsilon:U\rightarrow U'$ is a principal 
$N$-bundle.
Let $p: \Lambda\rightarrow \Lambda/\Gamma$ be the canonical projection.
Let us write $\rho_*\O_U=\bigoplus_{\lambda\in\Lambda} \Cal A_\lambda$.
Then the $N$-action on $\rho_*\O_U$ is given by the grading
$\rho_*\O_U=\bigoplus_{\bar\lambda\in \Lambda/\Gamma}\Cal A_{\bar \lambda}$, 
where $\Cal A_{\bar \lambda}=\bigoplus_{\lambda\in p^{-1}(\bar \lambda)}\Cal 
A_\lambda$.
Then $\rho'_*\O_{U'}=(\rho_*\O_U)^N\rightarrow \Cal A_{\bar 0}$ is
an isomorphism, and 
$\Cal A_{\bar \lambda}\otimes_{\Cal A_{\bar 0}}\Cal A_{\bar \mu}
\rightarrow \Cal A_{\bar \lambda+\bar \mu}$ is surjective for 
$\bar\lambda,\bar\mu\in \Lambda/\Gamma$.
By Lemma~\ref{principal-torus-equiv.lem}, we have that
$\upsilon$ is a principal $N$-bundle.
\end{proof}

\begin{lemma}\label{veronese-class-group.lem}
Let the notation be as above.
Then there is an exact sequence
\[
0\rightarrow\Cal X(N,X)\xrightarrow{\bar\alpha} 
\Cal X(N)\xrightarrow{\bar\beta} \Cl(X')
\xrightarrow{\bar\gamma} \Cl(X)\rightarrow 0,
\]
where $\Cal X(N)=\Lambda/\Gamma$, $\Cal X(N,X)=\{\bar\lambda\in
\Cal X(N)\mid B^\times\cap B_{\bar \lambda}\neq\emptyset\}$,
$\bar\alpha$ the inclusion, $\bar\beta(\bar\lambda)=
\gamma'(\O_Y(D_\lambda))$ \(where $p(\lambda)=\bar\lambda$, and this 
definition is independent of the choice of $\lambda\in p^{-1}(\bar\lambda)$\),
and $\bar\gamma(\M)=(\theta^*\M)^{**}$, where $B=\Gamma(X,\O_X)$.
\end{lemma}

\begin{proof}
Using Lemma~\ref{Veronese-main.lem}, we may repeat the proof of 
Proposition~\ref{multi-class.prop}.
Here we give a proof which use the result of Proposition~\ref{multi-class.prop}.
As $B$ is a domain, a unit of $B$ is homogeneous.
It is easy to see that $\Cal X(H,X')=\Cal X(H,X)$.
As
\[
0\rightarrow \Cal X(H)\rightarrow \Cal X(G)\rightarrow \Cal X(N)\rightarrow 0
\]
and
\[
0\rightarrow \Cal X(H,X')\rightarrow \Cal X(G,X)\rightarrow \Cal X(N,X)
\rightarrow 0
\]
are exact, we have that the sequence
\[
0\rightarrow \Cal X(H)/\Cal X(H,X')\rightarrow \Cal X(G)/\Cal X(G,X)
\rightarrow \Cal X(N)/\Cal X(N,X)\rightarrow 0
\]
is exact.
Now the result follows from the commutative diagram
\[
\xymatrix{
0 \ar[r] & \Cal X(G)/\Cal X(G,X) \ar[r]^-\alpha & \Cl(Y) \ar[r]^\gamma &
\Cl(X) \ar[r] & 0 \\
0 \ar[r] & \Cal X(H)/\Cal X(H,X) \ar[u] \ar[r]^-{\alpha'} &
\Cl(Y) \ar[u]^{\id} \ar[r]^{\gamma'} & \Cl (X') \ar[r] \ar[u]^{\bar\gamma} & 0
}
\]
and the snake lemma.
\end{proof}

\begin{lemma}
Let the notation be as above.
Let $S$ be Noetherian with a fixed dualizing complex $\Bbb I_S$, and
assume that $Y$ and $X$ are of finite type.
Then we have 
\[
\omega_{X'}\cong (\theta_*\omega_X)^N
\]
as $(H,\O_{X'})$-modules.
\end{lemma}

\begin{proof}
Let $\Gamma'\subset \Lambda$ be the subgroup such that $\Gamma\subset \Gamma'$,
$\Gamma'/\Gamma$ is a torsion module, and $\Lambda/\Gamma'$ is torsion-free.
Then comparing $X$ and $X''=\Spec B_{\Gamma'}$ and then 
$X''$ and $X'$, we may
assume either that $N$ is a torus or finite.
If $N$ is a torus, we may use Corollary~\ref{canonical-cor.thm}, {\bf 1}
(note that $\Theta_{N,S}$ is trivial in both cases).
\end{proof}

\section{The Cox rings of toric varieties}\label{toric.sec}

We give an example of toric varieties.

Let $M=\Bbb Z^n$ be the free $\Bbb Z$-module of rank $n$.
Let $k$ be a field, and $Y$ a toric variety determined by a 
fan $\Delta$ in $M^*=\Hom_{\Bbb Z}(M,\Bbb Z)$ \cite{Fulton}.
Let $H$ be the torus $\Spec kM$, where
$kM$ is the group algebra of $M$ (with each element of $M$ group-like).
Let $\Delta(1)$ be the set of one-dimensional faces of $\Delta$.
Note that $\Delta(1)$ is in one-to-one correspondence with the set of
$H$-stable prime divisors.
For each $\sigma\in \Delta(1)$, $m^*_\sigma$ denotes the generator of 
$\sigma\cap M^*$.
Let $W=\bigoplus_{\sigma\in\Delta(1)}\Bbb Z D_\sigma$ 
be the $\Bbb Z$-free module with the basis $\Delta(1)$ so that
$W$ is the group of $H$-stable divisors,
where $D_\sigma$ is the $G$-stable prime divisor
corresponding to $\sigma$.
An element of $M$ is a rational function on $Y$, and 
$\div m=\sum_\sigma \langle m,m^*_\sigma\rangle D_\sigma\in W$.
We assume that the map $\div:M\rightarrow W$ 
is injective, and $\Cl(Y)=W/M$ is torsion-free.
This is equivalent to say that $\{m^*_\sigma\mid \sigma\in\Delta(1)\}$ generates
$M^*$.
We set $G=\Spec kW$.
The inclusion $\div:M\hookrightarrow W$ 
induces a surjective map $f:G\rightarrow H$ with
$N:=\Ker f=\Spec k\Cl(Y)$.
Let $B=k[x_\sigma\mid \sigma\in\Delta(1)]$ be the Cox ring of $Y$ \cite{Cox},
where $x_\sigma$ are variables, and $B$ is a polynomial ring.
Letting each $x_\sigma$ of degree $\sigma$, $B$ is $W$-graded, and hence
is a $G$-algebra.
We set $X=\Spec B$.
We choose $\sigma_1,\ldots,\sigma_s\in \Delta(1)$ so that
$[D_1],\ldots,[D_s]$ forms a $\Bbb Z$-basis of $\Cl(Y)$,
where $D_i=D_{\sigma_i}$.
This gives a splitting $\Cl(Y)\rightarrow W$ (given by $[D_i]\mapsto D_i$) 
and the direct docompositions
$W=M\oplus \Cl(Y)$ and $G=N\times H$.
Then by \cite[(1.1)]{Cox}, $B$ is identified with
\[
R(Y;D_1,\ldots,D_s)=\bigoplus_{\lambda\in\Bbb Z^s}\Gamma(Y,\O_Y(D_\lambda))
t^\lambda.
\]
Note that $mt^\lambda\in R(Y;D_1,\ldots,D_s)_\lambda$ corresponds to
$x^{\div m+D_\lambda}$ for $m\in M$ and $\lambda\in\Bbb Z^s$, 
and thus this identification 
$B=R(Y;D_1,\ldots,D_s)$ respects the $W$-grading.

We set $V=Y\reg$.
This particular choice is consistent with our main discussion 
in section~\ref{multisection.sec}.
That is, $V$ is a large open subset of $Y$ such that 
$D_i|_V$ is Cartier for each $i$.
Not only that, $V$ is an $H$-open subset.
Obviously, $\pi:Z\rightarrow V$ is a $G$-morphism which is $N$-invariant.
So we have

\begin{proposition}\label{cox-main.prop}
Let the notation be as above.
Assume that $Y$ is quasi-projective.
Then {\rm(\ref{kurano-almost.eq})} is a $G$-enriched rational almost principal
$N$-bundle.
\end{proposition}

\begin{proof}
We have already seen that the diagram is that of $G$-schemes.
As $Y$ is quasi-projective, it has an ample Cartier divisor $D$.
$D$ lies in $\sum_i\Bbb Z D_i$, because this group is the whole $\Cl(Y)$.
By Theorem~\ref{kurano-main.thm}, the assertion follows.
\end{proof}

\begin{corollary}[cf.~{\cite[(4.3), Proposition]{Fulton}, 
\cite[(I.13.1)]{Stanley}}]\label{toric-canonical.cor}
Let $M=\Bbb Z^n$, and $H=\Spec kM$.
Let $Y$ be a toric variety over a field $k$ defined by a fan $\Delta$ in $M^*$.
Then the $H$-canonical module $\omega_Y$ of $Y$ is isomorphic to 
$\O_Y(-\sum_{\sigma\in\Delta(1)}D_\sigma)$.
\end{corollary}

\begin{proof}
Assume that $\Delta$ is not complete.
Then extending $\Delta$ outside, $Y$ is an $H$-open subscheme of a complete 
toric variety $\bar Y$ determined by $\bar \Delta$.
If $\omega_{\bar Y}\cong \O_{\bar Y}(\sum_{\sigma\in\bar\Delta(1)}D_\sigma)$, then
restricting to $Y$, we get 
\[
\omega_Y\cong \O_{\bar Y}(-\sum_{\sigma\in\bar\Delta(1)}\bar D_\sigma)|_Y
=
\O_{Y}((-\sum_{\sigma\in\bar\Delta(1)}\bar D_\sigma)|_Y)
=\O_Y(-\sum_{\sigma\in\Delta(1)}D_\sigma).
\]
Hence we may assume that $Y$ is complete.

Next, subdividing $\Delta$ if necessary, there is an $H$-equivariant 
birational map between complete toric varieties $g:Y'\rightarrow Y$ 
such that the class group is torsion-free and $Y'$ is projective 
\cite[(2.17)]{Oda}.
Let $W$ be the $H$-stable open subvariety of $Y$ 
obtained by removing the union of all the 
$H$-stable closed subvarieties of codimension grater than or equal to two
($W$ is the toric variety corresponding to the one skelton of $\Delta$).
It is easy to see that $g|_{g^{-1}(W)}:g^{-1}(W)\rightarrow W$ is an isomorphism
($g^{-1}(W)$ also corresponds to the one-skelton of $\Delta$).
If the corollary is true for $Y'$, then it is true also for its 
open subset $g^{-1}(W)$ as above, and then the assertion is also true for
$Y$, because $W$ is large in $Y$.

Thus we are in the situation of Proposition~\ref{cox-main.prop}.
Then as in Proposition~\ref{multi-canonical.prop}, 
it suffices to show that $\omega_X\cong \O_X(-\sum_{\sigma\in\Delta(1)}\sigma)$.
But this is trivial, since $B$ is a polynomial ring with the 
variables $x_\sigma$.
\end{proof}

\begin{lemma}\label{smith-vdb.lem}
Let $S=\Spec k$ with $k$ a field, and $G$ be an affine $S$-group scheme.
Let $(B,\frak m)$ be 
a $G$-local $G$-algebra such that $k\rightarrow B/\frak m$ is bijective.
Let $F$ be a $B$-finite $B$-projective $(G,B)$-module such that 
$F/\frak m F$ is a projective $G$-module.
Then $F$ is a projective $(G,B)$-module, and
$F\cong B\otimes_k (F/\frak m F)$.
\end{lemma}

\begin{proof}
By assumption, the canonical map $\pi: F\rightarrow F/\fm F$ has a 
$G$-linear splitting $i: F/\fm F\rightarrow F$.
We define $\nu:B\otimes_k (F/\fm F)\rightarrow F$ by
$\nu(b\otimes \alpha)=b\cdot i(\alpha)$.
This is $(G,B)$-linear, and is surjective by $G$-Nakayama's lemma,
Lemma~\ref{G-nakayama.lem}.
As $F$ is assumed to be $B$-projective, $\nu$ has a $B$-linear
splitting.
So $K=\Ker \nu$ is a $B$-finite $(G,B)$-module.
As can be seen easily, we have $K/\frak m K=0$, and hence $K=0$ 
by $G$-Nakayama's lemma again.

For a $(G,B)$-module $M$, we have
\[
\Hom_{G,B}(B\otimes_k(F/\frak mF),M)\cong \Hom_G(F/\frak m F,M).
\]
So $\Hom_{G,B}(B\otimes_k(F/\frak mF),?)$ is an exact functor, and
$F\cong B\otimes_k (F/\frak m F)$ is $(G,B)$-projective.
\end{proof}

\begin{proposition}[cf.~{\cite[Theorem~1]{Thomsen}},
{\cite[section~3]{Bruns}}]\label{thomsen.cor}
Let the notation be as in {\rm Corollary~\ref{toric-canonical.cor}}.
Assume that $k$ is a perfect field of characteristic $p>0$.
Then 
$Y$ is of graded finite $F$-representation type by
some rank-one reflexive sheaves $\Cal M_1,\ldots,\Cal M_u$ on $Y$,
with respect to the action of $H$.
\end{proposition}

\begin{proof}
As in the proof of Corollary~\ref{toric-canonical.cor}, we may assume that
$Y$ is complete.
As before, let $g:Y'\rightarrow Y$ be a birational map between complete
toric varieties such that $Y'$ is projective and $\Cl(Y')$ is torsion-free.
Assume that $Y'$ is of graded finite $F$-representation type by 
rank-one reflexive $({}^{e_0}H,\O_{Y'})$-modules
$\M_1,\ldots,\M_u$.
Then for each $e\geq 1$, 
we can write 
$F^e_*(\O_{{}^eY'})=\bigoplus_j \N_j$ as
$({}^eH,\O_{Y'})$-modules such that for each $j$, there exists some
$l(j)$ such that $\N_j\cong \M_{l(j)}$ as $\O_{Y'}$-modules.
Then $F^e_*(\O_{{}^eY})\cong \bigoplus_j g_*\N_j$ as
$({}^eH,\O_{Y})$-modules, since $g_*\O_{Y'}=\O_Y$.
So each $g_*\N_j$ is rank-one reflexive.
Moreover, $g_*\N_j\cong g_*\M_{l(j)}$ as $\O_Y$-modules.
So wasting $\M_l$ which does not appear in the expression at all, if any,
we have that $g_*\M_1,\ldots,g_*\M_u$ are rank-one reflexive 
$({}^{e_0}H,\O_Y)$-modules, and $Y$ is of graded finite $F$-representation
type by $g_*\M_1,\ldots,g_*\M_u$.
Hence we may replace $Y$ by $Y'$, and we are in the situation of 
Proposition~\ref{cox-main.prop}.

By Corollary~\ref{ffrt.cor}, it suffices to show
that there exist some $e_0\geq 0$ and finitely many
rank-one $B$-free $({}^{e_0}H\times N,B)$-modules such that
$({}^{e}B)^{{}^eN_e}$ is a direct sum of copies of these modules as
$(N,B)$-modules,
where $B$ is the Cox ring of $Y$.

By Lemma~\ref{smith-vdb.lem}, we have that 
${}^eB\cong B\otimes ({}^eB/\fm {}^eB)
\cong B\otimes {}^e(B/\fm^{[p^e]})$ as $({}^eG,B)$-modules,
where $\fm^{[p^e]}=\fm^{(e)}B$.
We identify a ${}^eG$-module with a $p^{-e}W$-graded $k$-vector space.
Then we have that ${}^e(B/\fm^{[p^e]})$ is the sum of one-dimensional
representations
\[
{}^e(B/\fm^{[p^e]})=
\bigoplus_{(\alpha_\sigma)\in\Map(\Delta(1),[0,1)\cap p^{-e}\Bbb Z)}
k \cdot x^{\sum_\sigma\alpha_\sigma \sigma}
\cong
\bigoplus_{(\alpha_\sigma)}
k(-\sum_\sigma\alpha_\sigma D_\sigma).
\]
So
\[
{}^eB\cong
\bigoplus_{(\alpha_\sigma)\in\Map(\Delta(1),[0,1)\cap p^{-e}\Bbb Z)}
B(-\sum_\sigma\alpha_\sigma D_\sigma).
\]

Hence $({}^eB)^{{}^eN_e}
\cong \bigoplus_{(\alpha_\sigma)}B(-\sum_\sigma\alpha_\sigma D_\sigma)$,
where the sum is taken over 
$(\alpha_\sigma)\in\Map(\Delta(1),[0,1)\cap p^{-e}\Bbb Z)$ such that
$-\sum_\sigma\alpha_\sigma[D_\sigma]\in p^{-e}\Cl(Y)$ lies in $\Cl(Y)$.
Let $\pi: \Bbb R^{n+s}=\Map(\Delta(1),\Bbb R)\rightarrow \Cl(Y)_{\Bbb R}$ 
be the map given by $\pi(\alpha_\sigma)=\sum_\sigma\alpha_\sigma[D_\sigma]$.
Then $\pi([0,1]^{n+s})\cap \Cl(Y)$ is compact and discrete, and hence is finite.
So we can find some $e_0$ and rank-one free summands $M_1,\ldots,M_u$
of ${}^eB^{{}^eN_e}$ for $e\leq e_0$ ($e$ may vary) such that 
any other rank-one free summand of ${}^eB^{{}^eN_e}$ for any $e$ is
$(N,B)$-isomorphic to some $M_l$.
This is what we wanted to prove.
\end{proof}

The following is well-known.

\begin{proposition}\label{toric-globally-freg.prop}
Let the notation be as in {\rm Corollary~\ref{toric-canonical.cor}}.
Assume that $k$ is a perfect field of characteristic $p>0$.
Then $Y$ is globally $F$-regular.
\end{proposition}

\begin{proof}
By Lemma~\ref{strongly-F-regular-ample.lem} and
Corollary~\ref{HWY.cor},
we may assume that $Y$ is projective and the class group of $Y$ is 
torsion-free.
As the torus $N=\Spec k\Cl(Y)$ is smooth linearly reductive and
the polynomial ring $B$ is strongly $F$-regular,
$Y$ is globaly $F$-regular by 
Theorem~\ref{globally-F-reg-main.thm} and
Proposition~\ref{cox-main.prop}.
\end{proof}

\begin{corollary}
An affine normal semigroup ring over a field of characteristic $p>0$
is strongly $F$-regular.
In particular, it is Cohen--Macaulay.
\end{corollary}

\begin{proof}
Let $A$ be an affine normal semigroup ring over $k$.
By \cite[(3.17)]{Hashimoto10}, we may assume that $k$ is algebraically
closed.
Then by Proposition~\ref{toric-globally-freg.prop},
the associated affine toric variety $\Spec A$ is globally $F$-regular.
That is, $A$ is strongly $F$-regular.
\end{proof}

\section{Surjectively graded rings}\label{surjective.sec}

\paragraph
As we have seen in the last section,
we can construct a rational almost principal bundle from a multisection
ring over a normal quasi-projective variety over a field.
However, given a finitely generated multigraded algebra $B$ over a field $k$,
it seems that it is not so easy to tell if $B$ is a multisection ring.
But this is relatively easy for the case that $B$ is surjectively graded.

\paragraph
Let $\Lambda=\Bbb Z^s$, and $G=\Spec Z\Lambda$, 
the split $s$-torus over $\Bbb Z$.
Let $B$ be an $\Lambda$-graded ring.
Let $\Sigma$ be a subsemigroup (submonoid) of $\Lambda_{\Bbb R}=\Bbb R^s$.
We say that $B$ is {\em $\Sigma$-surjectively graded} 
if for $\lambda,\lambda'\in \Sigma\cap \Lambda$, 
the product 
$B_\lambda\otimes_{\Bbb Z}B_{\lambda'}\rightarrow B_{\lambda+\lambda'}$ is surjective.
By definition, $B$ is $\Sigma$-surjectively graded if and only if
it is $\Sigma\cap\Lambda$-surjectively graded.

The definition is a variant of \cite[(3.5)]{Hashimoto9}.
For a $\Lambda$-graded domain $B$, $\Sigma(B):=\{\lambda\in\Lambda\mid 
B_\lambda\neq 0\}$ is a subsemigroup of $\Lambda$.
We say that $B$ is a surjectively graded domain if $B$ is 
a $\Sigma(B)$-surjectively graded domain.

\begin{lemma}\label{principal-torus-equiv2.lem}
$\Spec B\rightarrow\Spec B_0$ is a principal $G$-bundle if and only if 
$B$ is $\Lambda_{\Bbb R}$-surjectively graded.
\end{lemma}

\begin{proof}
This is Lemma~\ref{principal-torus-equiv.lem}.
\end{proof}

\begin{lemma}
Let $B$ be a $\Sigma$-surjectively graded ring, and $S$ a multiplicatively
closed subset of $B$ consisting of homogeneous elements.
Set $|S|=\{|s|\mid s\in S\}$ be the submonoid of $\Lambda$ of the
degrees of the elements of $S$.
Assume that $|S|\subset\Sigma$.
Then the localization $B_S$ is $\Sigma-|S|$-surjectively graded,
where 
\[
\Sigma-|S|=\{\mu-|s|\mid \mu\in\Sigma,\;s\in S\}.
\]
\end{lemma}

\begin{proof}
Take $\nu_i=\mu_i-|s_i|\in(\Sigma-|S|)\cap \Lambda$ for $i=1,2$,
where $\mu_i\in\Sigma\cap\Lambda$ and $s_i\in S$.
Take $c=bs^{-1}\in (B_S)_{\nu_1+\nu_2}$, where $s\in S$ and $b\in B_{\nu_1+\nu_2
+|s|}$.
Then $bs_1s_2\in B_{\mu_1+\mu_2+|s|}$.
As $B_{\mu_1}\otimes_{\Bbb Z}B_{\mu_2+|s|}\rightarrow B_{\mu_1+\mu_2+|s|}$ is
surjective,
we can write $bs_1s_2=\sum_i u_iv_i$ with $u_i\in B_{\mu_1}$ and
$v_i\in B_{\mu_2+|s|}$.
Then $c=\sum_i (u_is_1^{-1})(v_is^{-1}s_2^{-1})$, 
and $u_is_1^{-1}\in (B_S)_{\nu_1}$ and $v_is^{-1}s_2^{-1}\in (B_S)_{\nu_2}$.
So $(B_S)_{\nu_1}\otimes_{\Bbb Z}(B_S)_{\nu_2}\rightarrow (B_S)_{\nu_1+\nu_2}$ 
is surjective, and $B_S$ is $\Sigma-|S|$-surjectively graded.
\end{proof}

\paragraph
Let $\Sigma$ be a rational convex polyhedral cone in $\Lambda_{\Bbb R}$
with $\Sigma-\Sigma=\Lambda_{\Bbb R}$.
Let $\lambda\in \Sigma^\circ\cap \Lambda$, where $\Sigma^\circ$ is the
interior of $\Sigma$.
Then we have $\Sigma-\Bbb Z_{\geq 0}\lambda=\Lambda_{\Bbb R}$.

Let $B$ be a $\Sigma$-surjectively graded ring.
Let $J(\lambda)$ be the ideal of $B$ generated by $B_\lambda$.
Set $X=\Spec B$, and $U=X\setminus V(J(\lambda))$.

\begin{lemma}[cf.~{\cite[(3.8)]{Hashimoto9}}]\label{surjective-almost.lem}
There is a principal $G$-bundle $\rho:U\rightarrow Y$.
If $\height J(\lambda)\geq 2$, then
\begin{equation}\label{surjective-almost.eq}
\xymatrix{
X & U \ar@{_{(}->}[l]_i \ar[r]^{\rho} & Y \ar@{^{(}->}[r]^{\id_Y} & Y
}
\end{equation}
is a rational almost principal $G$-bundle, where $i:U\rightarrow X$ is
the inclusion.
We have $Y=\Proj \bigoplus_{n\geq 0}B_{n\lambda}t^n$.
$U$ is independent of the choice of $\lambda\in\Sigma^\circ\cap \Lambda$,
and hence $Y$ is also independent of $\lambda$.
\end{lemma}

\begin{proof}
For $\mu\in \Lambda$, let $B(\mu)$ be the rank-one $B$-free 
$(G,B)$-module given by $B(\mu)_\nu=B_{\mu+\nu}$.
The corresponding $G$-linearized invertible sheaf on $X$ is denoted by
$\O(\mu)$.
Let $C$ be the section ring
\[
C=\Gamma_{\geq 0}(X;\O(\lambda))=
\bigoplus_{n\geq 0}\Gamma(X,\O(n\lambda))t^n=\bigoplus_{n\geq 0}B(n\lambda)t^n.
\]
Let $D$ be the ring of invariants $C^G$.
That is, $D=\bigoplus_{n\geq 0}B_{n\lambda}t^n$.
Then we have a sequence of morphisms
\[
U\xrightarrow{\iota} \Proj C\setminus V_+(D_+C)\xrightarrow{\psi} \Proj D,
\]
where $D_+C$ is the ideal of $C$ generated by $D_+=\bigoplus_{n>0}B_{n
\lambda}t^n$.
It is easy to see that $\iota$ is an isomorphism (recall that 
$X=\Proj C$).
On the other hand, it is easy to see that the map $\psi$ induced by
the graded homomorphism of graded rings $D\rightarrow C$ is an algebraic
quotient by $G$.
For $a\in B_{n\lambda}\setminus 0$ for $n\geq 1$, 
$B[a^{-1}]$ is $\Lambda_{\Bbb R}$-surjectively graded.
By Lemma~\ref{principal-torus-equiv2.lem}, $\psi$ is a principal $G$-bundle.
So letting $\rho=\psi\iota$, we are done.

As $B_\lambda^{\otimes n}\rightarrow B_{n\lambda}$ is surjective, 
$J(\lambda)^n=J(n\lambda)$.
If $\mu$ is another element of $\Sigma^\circ\cap \Lambda$, then
$n\lambda-\mu\in\Sigma$ for sufficiently large $n$.
So $B_\mu\otimes_{\Bbb Z}B_{n\lambda-\mu}\rightarrow B_{n\lambda}$ is surjective,
and $J(\mu)\supset J(\lambda)^n$.
Hence $\sqrt{J(\mu)}\supset \sqrt{J(\lambda)}$.
Similarly, $\sqrt{J(\mu)}\subset \sqrt{J(\lambda)}$ is also true, and 
the definition of $U$ is independent of $\lambda$.
As the principal bundle is a categorical quotient and hence is 
unique, $Y$ is also independent of the
choice of $\lambda$.
\end{proof}

\paragraph
Let the assumption be as in Lemma~\ref{surjective-almost.lem}.
Let $\Cal L(\mu)$ be the invertible sheaf on $Y$ corresponding to 
$\O(\mu)$.
Namely, $\Cal L(\mu)=\rho_*(\O_U(\mu))^G$.
Then we have $\rho^*(\Cal L(\mu))\cong \O_U(\mu)$.
Note that $\rho$ is affine, and $\rho_*\O_U$ is a graded $\O_Y$-algebra:
$\rho_*\O_U=\bigoplus_{\mu\in\Lambda}\Cal A_\mu t^\mu$.
As $\Cal L(\mu)=\rho_*\O_U(\mu)^G=(\bigoplus_\nu \Cal A_{\mu+\nu}t^\nu)^G
=\Cal A_\mu$, we have that $U=\uSpec_Y \bigoplus_\mu \Cal L(\mu)t^\mu$.
Hence

\begin{lemma}[cf.~{\cite[(4.4)]{Hashimoto9}}]
  \label{surjective-section-ring.lem}
If $B$ is a Krull domain, then $U$ and $Y$ are locally Krull and integral.
If $B$ is Noetherian and $(S_2)$, then $U$ and $Y$ are 
Noetherian and $(S_2)$.
In both cases, 
$B$ is isomorphic to the multisection ring 
$R(Y;\Cal L_1,\ldots, \Cal L_s)=\bigoplus_{\mu\in \Lambda}\Gamma(Y,\Cal L
(\mu))$.
\end{lemma}

\begin{proof}
Assume that $B$ is a Krull domain.
Being locally Krull and integral is inherited by a nonempty open subset,
and $U$ is locally Krull and integral.
Then by Theorem~\ref{alg-quot-krull-equiv.thm}, $Y$ is also locally 
Krull, and clearly integral.
As $U$ is large, $\O_X\rightarrow i_*\O_U$ is an isomorphism by 
\cite[(5.28)]{Hashimoto4}, and $B$ is isomorphic to 
$R(Y;\Cal L_1,\ldots,\Cal L_s)$.

Next, assume that  $B$ is Noetherian and $(S_2)$.
This property is inherited by the open subset $U$, and then descends to $Y$.
Again, as $U$ is large, $\O_X\rightarrow i_*\O_U$ is an isomorphism by 
Lemma~\ref{S_2.lem}, and we have
$B\cong R(Y;\Cal L_1,\ldots,\Cal L_s)$.
\end{proof}

\begin{proposition}\label{degree-one.prop}
Let $k$ be a field, and let 
$B=\bigoplus_{n\geq 0}B_n$ be a standard graded algebra, that is, 
$B=k[B_1]$ with $\dim_k B_1<\infty$.
Assume moreover that $\dim B\geq 2$.
Then we have
\begin{enumerate}
\item[\bf 1] 
Letting $U=X\setminus 0$ and $Y=\Proj B$, 
{\rm(\ref{surjective-almost.eq})} is a rational almost principal 
$\Bbb G_m$-bundle, where $0$ is the origin of $X$.
\item[\bf 2] $\omega_Y\cong \tilde{\omega}_B$ 
and $\omega_B\cong \bigoplus_{n\in\Bbb Z}\Gamma(Y,\omega_Y(n))t^n$, 
where $\tilde{(?)}$ denotes the
sheaf on $Y$ associated with a graded module.
\item[\bf 3] Let $d>1$. Let $B_{d\Bbb Z}=\bigoplus_{n\geq 0}B_{nd}$ be the
Veronese subring.
Then $(\omega_B)_{d\Bbb Z}\cong \omega_{B_{d\Bbb Z}}$ as $\Bbb Z$-graded
modules.
If, moreover, $B$ has a graded full $2$-canonical module $M$, then 
$\omega_B\cong (B\otimes_{B_{d\Bbb Z}}\omega_{B_{d\Bbb Z}})^{\vee\vee}$,
where $(?)^\vee=\Hom_B(?,M)$.
\item[\bf 4] \(cf.~Goto--Watanabe {\rm\cite[(3.2.1)]{GW}}\) For $r\in\Bbb Z$, 
the following are equivalent.
\begin{enumerate}
\item[\bf a] $B$ is quasi-Gorenstein of $a$-invariant $rd$ \(that is,
$\omega_B\cong B(rd)$\).
\item[\bf b] $\depth B_\fm\geq 2$, and
$B_{d\Bbb Z}$ is quasi-Gorenstein of 
$a$-invariant $rd$ \(that is, $\omega_{B_{d\Bbb Z}}\cong B_{d\Bbb Z}(rd)$\),
\end{enumerate}
where $\fm$ is the irrelevant ideal $B_+$ of $B$.
In particular, $B$ is qusi-Gorenstein and its $a$-invariant is divisible
by $d$ if and only if $\depth B_\fm\geq 2$ and 
the Veronese subring $B_{d\Bbb Z}$ is quasi-Gorenstein.
\item[\bf 5] Assume that $B$ is normal.
Then 
$0\rightarrow \Bbb Z\xrightarrow{\beta}\Cl(Y)\xrightarrow{\gamma}
\Cl(X)\rightarrow 0$ is exact, where $\beta(1)=\O(1)$, and 
$\gamma(\M)=\bigoplus_{n\in \Bbb Z}\Gamma(Y,\M(n))$.
\item[\bf 6] Let $d>1$, and set $X'=\Spec B_{d\Bbb Z}$.
If $B$ is normal, then 
\[
0\rightarrow \Bbb Z/d\Bbb Z\xrightarrow{\bar\beta}\Cl(X')\xrightarrow
{\bar \gamma}\Cl(X)\rightarrow 0
\]
is exact.
\end{enumerate}
\end{proposition}

\begin{proof}
Let $s=1$ and $\Lambda=\Bbb Z$, and $\lambda=1$.
Set $S=Y_0=\Spec k$, $G:=\Spec k\Lambda$,
and $Y_0:=\Spec k$.
Then $J(1)$ is the irrelevant ideal $B_+$ by assumption, and
$\height J(1)\geq 2$, since $\dim B\geq 2$.
Thus {\bf 1} follows from Lemma~\ref{surjective-almost.lem}.

Note that $\Theta_{G,Y_0}$ is $G$-trivial, since $G$ is an abelian group, 
see Remark~\ref{knop-remark.thm}.
So $\omega_Y\cong(\rho_*i^*\omega_X)^G$ by Theorem~\ref{canonical.thm}.
The right-hand side agrees with $\tilde \omega_B$ by definition.
On the other hand, by Theorem~\ref{canonical.thm}, $\omega_X\cong 
i_*\rho^*\omega_Y$, and hence
$\omega_B\cong \bigoplus_{n\in\Bbb Z}\Gamma(Y,\omega_Y(n))t^n$.
So {\bf 2} has been proved.

Set $\Lambda_d=d\Lambda=\Bbb Z d$, and $H:=\Spec \Bbb Z\Lambda_d$.
Let $f:G\rightarrow H$ be the canonical homomorphism induced by 
$\Lambda_d\hookrightarrow \Lambda$, and 
$N:=\Ker f=\Spec \Bbb Z(\Lambda/\Lambda_d)=\mu_d$.
As $G$ acts freely on $U$, $N$ acts on $U$ freely.
So the canonical map $\theta:X\rightarrow X'$ corresponding to $B_{d\Bbb Z}=
B^N \rightarrow B$ is a $G$-enriched almost principal $N$-bundle.
As $N$ is linearly reductive, {\bf 3} follows from 
Corollary~\ref{canonical-cor.thm}.

{\bf 4}.
{\bf a$\Rightarrow$b}.
As $B$ is quasi-Gorenstein, it satisfies $(S_2)$.
As $\dim B_\fm\geq 2$ by assumption, we have that $\depth B_\fm \geq 2$.
Letting $\L=\O_{X'}(rd)$ in Theorem~\ref{finite-resume.thm}, {\bf 5},
$\omega_{B_{d\Bbb Z}}\cong B_{d\Bbb Z}(rd)$ follows.

{\bf b$\Rightarrow$a}.
Let $U=X\setminus 0$, $X=\Spec B_{d\Bbb Z}$ as above, and 
$\theta:X\rightarrow X'$ the canonical map.
Let $U'=\theta(U)$, and $\upsilon:U\rightarrow U'$ be the restriction of 
$\theta$.
As $\upsilon$ is a principal $N$-bundle by Lemma~\ref{Veronese-main.lem},
it is flat with Cohen--Macaulay fibers.
As $U'$ is quasi-Gorenstein, $U$ satisfies the $(S_2)$ condition.
As $U=X\setminus 0$ and $\depth B_\fm\geq 2$, $X$ satisfies the
$(S_2)$ condition.
Now the result follows from
Theorem~\ref{finite-resume.thm}, {\bf 5}.

{\bf 5, 6}.
As we assume that $B$ is normal, $B=\bigoplus_{n\in\Bbb Z}\Gamma(Y,\O_Y(n))$
by Lemma~\ref{surjective-section-ring.lem}.
So {\bf 5} follows from Proposition~\ref{multi-class.prop}.
{\bf 6} follows from Lemma~\ref{veronese-class-group.lem}.
\end{proof}

\begin{example}\label{Gro.ex}
Let $B=k[x,y]$ with $\deg x=\deg y=1$.
Then $X=\Bbb A^2$, $U=\Bbb A^2\setminus 0$, and $Y=\Bbb P^1$.
The category of locally free sheaves on $\Bbb P^1$ is $\Ref(Y)$, which is
equivalent to $\Ref(\Bbb G_m,X)\cong \Ref(\Bbb G_m,B)$.
As $\dim B=2$, a reflexive $(\Bbb G_m,B)$-module is nothing but
a graded finite free $B$-module.
A graded finite free $B$-module is a direct sum of copies of $B(n)$, $n\in\Bbb
Z$, and the Krull--Schmidt theorem holds.
Hence a locally free sheaf on $\Bbb P^1$ is a direct sum of copies of
$\O(n)$, $n\in\Bbb Z$, and the Krull--Schmidt theorem holds.
This is a well-known theorem of Grothendieck, see 
\cite[(4.1)]{HazM}.
\end{example}

\begin{example}\label{A_n-free.ex}
Let $B=k[x,y]$ with $\deg x=1$ and $\deg y=-1$, and $X=\Spec B=\Bbb A^2$.
As $\Bbb G_m$ acts freely on $B[x^{-1}]$ and on $B[y^{-1}]$, we have that
$G$ acts freely on $X\setminus 0$.
In particular, for $n\geq 1$, the subgroup scheme $N=\mu_{n+1}$ acts freely
on $X\setminus 0$.
In particular, $\varphi:X=\Spec B\rightarrow \Spec B^N=Y$ is an
almost principal $N$-bundle.
So the class group of $B^N=k[x^{n+1},xy,y^{n+1}]$ is $\Cal X(N)=
\Bbb Z/(n+1)\Bbb Z$ \cite[Proposition~4]{Waterhouse2}.
\end{example}

\begin{lemma}\label{small-red-circ.lem}
Let $k$ be a perfect field, $G$ a finite $k$-group scheme acting on a 
$k$-scheme $X$.
Assume that there is a separated $G$-invariant morphism $\varphi:X\rightarrow
Y$.
Then the action of $G$ on $X$ is free if and only if the actions of 
$G\red$ and $G^\circ$ on $X$ are free.
\end{lemma}

\begin{proof}
The only if part is trivial.
We prove the if part.
We may assume that $k$ is algebraically closed.
Assume that the actions of $G\red$ and $G^\circ$ are free, but the action
of $G$ is not free.
Then take $x\in X\setminus U$, where $U$ is the free locus.
Then the stabilizer $G_x$ is nontrivial by Nakayama's lemma.
As $(G_x)\red\subset G_x\cap (G\red\otimes_k \kappa(x))=(G\red)_x=e$, 
$G_x$ is contained in $G_x\cap (G\otimes_k \kappa(x))^\circ=(G^\circ)_x=e$,
and $G_x$ is trivial.
A contradiction.
\end{proof}

\begin{lemma}\label{free-red.lem}
Let $k$ be a perfect field, $G$ a finite $k$-group scheme acting on a 
$k$-scheme $X$.
Let $\varphi:X\rightarrow Y$ be an algebraic quotient by $G^\circ$.
Assume that there is a separated $G\red$-invariant 
morphism $\psi:Y\rightarrow Z$.
Then $\Cal S_{G\red,X}=\Cal S_{G\red,Y}\times_Y X$.
The action of $G\red$ on $X$ is free if and only if its action on $Y$ is free.
\end{lemma}

\begin{proof}
We may assume that the characteristic of $k$ is $p>0$.
By Lemma~\ref{stabilizer-basics.lem}, 
$\Cal S_{G\red,X}\subset \Cal S_{G\red,Y}\times_Y X$.
As both of them are finite over $X$, to prove the equality,
it suffices to show that 
$\Cal S_{G\red,x}=\Cal S_{G\red,y}\times_{y} x$ for each point
$x$ of $X$, by Nakayama's lemma, where $y=\varphi(x)$.
Let $a_x:G\red\times x \rightarrow X$ and $a_y:G\red\times y\rightarrow Y$ 
be the actions.
Then
\[
\Cal S_{G\red,y}\times_{y} x
= a_y^{-1}(y)\times_y x=(1_{G\red}\times \varphi)^{-1} a_y^{-1}(y)
=a_x^{-1}(\varphi^{-1}(y)).
\]
As $G\red\times x$ is \'etale over $\kappa(x)$, it is reduced, and hence
any morphism from $G\red\times x$ to $\varphi^{-1}(y)$ factors through
$(\varphi^{-1}(y))\red$.
As $G^\circ$ is infinitesimal, $\varphi$ is purely inseparable.
So $(\varphi^{-1}(y))\red=x$, and
\[
\Cal S_{G\red,y}\times_y x=a_x^{-1}((\varphi^{-1}(y))\red)=a_x^{-1}(x)
=\Cal S_{G\red,x}.
\]
Thus $\Cal S_{G\red,X}=\Cal S_{G\red,Y}\times_Y X$.
If the action of $G\red$ on $Y$ is free, then $\Cal S_{G\red,Y}$ is trivial,
and its base change $\Cal S_{G\red,X}$ is also trivial, and the action on
$X$ is also free.
Conversely, assume that 
the action on $X$ is free and $\Cal S_{G\red,X}$ is trivial.
Let $\Cal C$ be the cokernel of $\O_Y\rightarrow \phi^Y_*\O_{\Cal S_{G\red,Y}}$,
where $\phi^Y:\Cal S_{G\red,Y}\rightarrow Y$ is the structure map.
Then 
\[
\varphi_*\O_X\rightarrow \varphi_*\phi^X_*\O_{\Cal S_{G\red,X}}\rightarrow
\varphi_*\varphi^*\Cal C\rightarrow 0
\]
is exact.
By assumption, $\varphi_*\varphi^*\Cal C=0$.
As $\varphi$ is finite surjective, $\Cal C=0$ by Nakayama's lemma.
This shows that the action of $G\red$ on $Y$ is free.
\end{proof}

\begin{proposition}\label{SL_n-small.prop}
Let $k$ be a field, $n\geq 1$, 
and $G$ be a linearly reductive finite subgroup scheme of $\SL_n$.
Then the canonical action of $G$ on $k[x_1,\ldots,x_n]$ is small.
The action of $G$ on $k[[x_1,\ldots,x_n]]$ is also small.
\end{proposition}

\begin{proof}
We prove that the action of $G$ on $k[x_1,\ldots,x_n]$ is small.
We may assume that $k$ is algebraically closed of 
characteristic $p>0$, and $n\geq 2$.
It suffices to show that the actions of $G^\circ$ and $G\red$ are small
by Lemma~\ref{small-red-circ.lem}.
As $G\red\subset\SL_n$, $G\red$ does not have a diagonalizable 
pseudoreflection.
As $G\red$ is linearly reductive, $G\red$ does not have a transvection
(that is, a pseudoreflection $g\in\GL_n$ such that $1-g$ is nilpotent)
by Maschke's theorem (as the Jordan normal form shows, 
the order of a transvection in characteristic $p$ is $p$).
Thus $G\red$ does not have a pseudoreflection of any kind, and the
action of $G\red$ is small.

So we may assume that $G$ is infinitesimal.
Let $B=k[x_1,\ldots,x_n]$.
As $G$ is also linearly reductive, $G$ is diagonalizable \cite{Sweedler3}.
So $\bigoplus_i kx_i$ is a direct sum of one-dimensional $G$-modules.
Changing variables, we may assume that $G\subset T\cap \SL_n$, 
where $T$ is the subgroup of $\GL_n$ consisting of the invertible 
diagonal matricies.
Considering the action of $T\cap \SL_n$ is to consider a 
$\Bbb Z^n/(\alpha)$-grading, where $\alpha=(1,1,\ldots,1)$.
So when we invert $y_i=(x_1\cdots x_n)/x_i$, then the action of 
the torus $T\cap \SL_n$ on $B[y_i^{-1}]$ is free for $1\leq i
\leq n$, and hence the action of $G$ is also free.
Thus it suffices to show that the ideal $I=(y_1,\ldots,y_n)$ of $B$ 
is height two.
By definition, $I$ is the Stanley--Reisner ideal (the defining ideal
of the Stanley--Reisner ring, see \cite[(II.1.1)]{Stanley}) of 
the $(n-3)$-skelton of the $(n-1)$-simplex.
So $\dim B/I=n-2$, and $\height I=2$.

The last assertion follows from the first assertion and
Lemma~\ref{n-small-flat.lem}.
\end{proof}

\begin{remark}\label{ps-small.rem}
Let $k$ be a field, $V=k^n$, and $\tilde G=\GL_n=\GL(V)$.
We define 
\[
\PR(\tilde G)=\{g\in\tilde G\mid \rank(1_V-g)\leq 1\}.
\]
It is a closed subscheme of $\tilde G$.
For a closed subscheme $F$ of $\tilde G$, we define that 
$\PR(F)=F\cap \PR(\tilde G)$.
We say that a finite subgroup scheme $G$ of $\tilde G$ does not 
have a pseudoreflection if $\PR(G)=\{e\}=\Spec k$, scheme theoretically.
On the other hand, we say that $G$ is small if the action of $G$
on $V$ is small.

By Example~\ref{ps-sm.ex}, if $G$ is \'etale, then $G$ is small
if and only if $G$ does not have a pseudoreflection.
However, 
in general, a small subgroup $G$ of $\tilde G$ may have a pseudoreflection.
For example, if $p=2$ and $G$ is the subgroup scheme of $\SL_2$ of 
type $(A_1)$, then it is easy to see that $G=\PR(G)$.

The author does not have an appropriate way to connect the smallness of the
action and the non-existence of pseudoreflections for non-reduced
finite group schemes.
\end{remark}

\paragraph
Let $k$ be an algebraically closed 
field, and $N$ be a nontrivial finite linearly reductive 
$k$-subgroup scheme of $\SL_2$.
Such $N$ is classified with Dynkin diagrams of type ADE 
\cite{Hashimoto8}.
Let $H=\Bbb G_m$, which acts on $B=k[x,y]$ by $\deg x=\deg y=1$.
Then $G=H\times N$ acts on $B$ in a natural way.
By Proposition~\ref{SL_n-small.prop}, the action of $N$ on $X=\Spec B$ is
small.

The following is well-known for the case that $N$ is \'etale,
see \cite[Chapter~6]{LW}.

\begin{theorem}\label{SL_2.thm}
Let $k$, $N\subset \SL_2$, $H$, $G$, $B=k[x,y]$, and $X$ be as above
\($N$ may not be reduced\).
Set $A=B^N$.
Let $\hat A$ and $\hat B$ respectively 
be the completion of $A$ and $B$ with respect 
to the irrelevant ideal.
Let
$\varphi:X=\Spec B\rightarrow Y=\Spec A$ be the canonical algebraic quotient,
and $\hat\varphi:\hat X=\Spec \hat B\rightarrow \hat Y=\Spec \hat A$ be its
completion.
Then
\begin{enumerate}
\item[\bf 1] The free locus of the action of $N$ on $X$ 
\(resp.\ $\hat X$\) is $X\setminus 0$ \(resp.\ $\hat X\setminus 0$\).
In particular, $\varphi$ and $\hat \varphi$ 
are $G$-enriched almost principal $N$-bundles.
\item[\bf 2] $A$ is strongly $F$-regular Gorenstein of the $a$-invariant $-2$.
\item[\bf 3] The category of $\hat B$-finite $\hat B$-free $(N,
\hat B)$-modules and
the category of maximal Cohen--Macaulay $\hat A$-modules are equivalent.
The Cohen--Macaulay ring 
$\hat A$ has finite representation type, and any
maximal Cohen--Macaulay module of $\hat A$ is isomorphic to 
$\hat M_V:=(\hat B\otimes_k V)^N$ for some finite dimensional $N$-module $V$.
$\hat M_V$ is indecomposable if and only if $V$ is simple.
$\hat M_V\cong \hat M_{V'}$ if and only if $V\cong V'$.
An isomorphism class of simple modules of $N$ corresponds to a vertex of
the corresponding extended Dynkin diagram.
\item[\bf 4] The category of $B$-finite $B$-free $(G,B)$-modules 
\(that is, graded $(N,B)$-modules\) is equivalent to the category of 
maximal Cohen--Macaulay $(H,A)$-modules \(that is, graded maximal
Cohen--Macaulay $A$-modules\).
Any graded maximal Cohen--Macaulay $A$-module is isomorphic to
$M_V=(B\otimes_k V)^N$, where $V$ is a finite dimensional $G$-module.
$M_V$ is indecomposable if and only if $V$ is simple.
So $A$ is of finite representation type in the graded sense \(see
{\rm\cite[Chapter~15]{LW}}\).
$M_V\cong M_{V'}$ if and only if $V\cong V'$.
\item[\bf 5] The class groups of $A$ and $\hat A$ are isomorphic to 
the character group $\Cal X(N)$.
$\Cal X(N)$ is $\Bbb Z/(n+1)\Bbb Z$ for type $(A_n)$, $\Bbb Z/2\Bbb Z
\times \Bbb Z/2\Bbb Z$ for type $(D_n)$, $\Bbb Z/3\Bbb Z$ for type $(E_6)$, 
$\Bbb Z/2\Bbb Z$ for type $(E_7)$, and is trivial for $(E_8)$, and is
independent of the characteristic of $k$.
\end{enumerate}
\end{theorem}

\begin{proof}
{\bf 1} As we have seen, the free locus $U$ of the action of $N$ on $X$ is
large in $X$.
As $U$ is $G$-stable and large, we have that $U=X$ or $U=X\setminus\{0\}$.
However, the origin is a fixed point of the action, and $0\notin U$.
The case of $\hat X$ is similar.

{\bf 2} As $N$ is Reynolds, $A$ is a pure subring of $B$ by
Lemma~\ref{direct-summand-Reynolds.lem}.
Hence $A$ is strongly $F$-regular by \cite[(3.1)]{HH}.
As we have that $N\subset\SL_2$ and linearly reductive, 
$A$ is Gorenstein of $a$-invariant $-2$ by 
Example~\ref{finite-polynomial.ex}, {\bf 4}.

{\bf 3} By {\bf 1}, the categories 
$\Ref(\hat A)$ and $\Ref(N,\hat B)$ are
equivalent.
As $\hat A$ is a two-dimensional Cohen--Macaulay local ring, a reflexive
$\hat A$-module is nothing but a maximal Cohen--Macaulay module.
As $\hat B$ is a two-dimensional regular local ring, any reflexive $\hat 
B$-module is free.
By Lemma~\ref{smith-vdb.lem}, such a module is of the form $\hat B\otimes_k V$
with $V$ a finite dimensional $N$-module.
$V\mapsto \hat B\otimes_k V$ and $F\mapsto F/\frak m F$ is 
a one-to-one correspondence between the set of isomorphism classes of finite
dimensional $N$-modules and the set of isomorphism classes of
$\hat B$-finite $\hat B$-free $(N,\hat B)$-modules, 
and this correspondence respects finite direct sums.
So $V\mapsto \hat M_V$ gives a one-to-one correspondence which respects
the finite direct sums.

It remains to show that the simple $N$-modules are in one-to-one correspondence
with the vertices of the corresponding extended Dynkin diagram.
First, we define the McKay graph $\Gamma_N$ 
of $N\subset\SL_2$ as in the case of usual finite
groups (see \cite[(10.3)]{Yoshino}).
It is a finite quiver defined as follows.
A vertex of $\Gamma_N$ is an isomorphism class of simple $N$-modules.
We draw $n_{ij}=\dim_k \Hom_G(V_i,V\otimes_k V_j)$ arrows from $[V_i]$ to 
$[V_j]$, where $[V_i]$ and $[V_j]$ are vertices.
As $V\cong V^*$, it is easy to see that $n_{ij}=n_{ji}$, and we regard
$\Gamma_N$ as an unoriented graph.
As $N\rightarrow \End(V)$ is a closed immersion, $k[\End(V)]\rightarrow k[N]$
is surjective.
So it is easy to see that any simple $N$-module is a direct summand of some
$V^{\otimes r}$.
So $\Gamma_N$ must be a connected graph (if $V_j$ is a direct summnad of
$V^{\otimes r}$, then starting from the trivial module $[V_0]=[k]$, we reach
$[V_j]$ along a path of the length $r$).
Moreover, when we set $a_j=\dim_k V_j$, we have $2a_j=\dim_k (V\otimes V_j)
=\sum_i n_{ij}a_i$.
If $V_N=\{[V_0],\ldots,[V_n]\}$ is the set of vertices, then $n\geq 1$, 
since $N$ is assumed to be non-trivial.
A connected finite graph with the vertex set $V_N$ 
(with $\#V_N\geq 2$) with $n_{ij}$ arrows
from $[V_i]$ to $[V_j]$ with a function $[V_j]\mapsto a_j$ with the
property $2a_j=\sum_i n_{ij}a_i$ is classified easily, and is one of 
$(A_n)$ $(n\geq 1)$, $(D_n)$ $(n\geq 4)$, $(E_6)$, $(E_7)$, or $(E_8)$ 
displayed in \cite[section~10]{Yoshino} (the symbol $[R]$ there should be
replaced by the trivial representation $[k]$ here), or the graph
\begin{equation}\label{exceptional-graph.eq}
\xymatrix{
[k] \ar@{-}[dr] \\
 & 2 \ar@(dr,ur)@{-}[]
\\ 1\ar@{-}[ur]
},
\end{equation}
which has a self arrow.
$N$ is abelian if and only if $a_j=1$ for all $j$ if and only $\Gamma_N$
is of type $(A_n)$.
So $N$ is of type $(A_n)$ if and only if $\Gamma_N$ is of type $(A_n)$.
If $N$ is of type $(D_n)$ ($n\geq 4$), then $N$ has the Klein
group $\Bbb Z/2\Bbb Z\times \Bbb Z/2\Bbb Z$ as a quotient.
So $\Gamma_N$ is not $(A_{4n-9})$, and $a_j=1$ for at least four $j$.
By dimension counting, $\Gamma_N$ must be $(D_n)$.
If $N$ is of type $(E_6)$ (resp.\ $(E_7)$, $(E_8)$), 
then $N/[N,N]$ is of order $3$ (resp.\ $2$, $1$), and there are
exactly three (two, one) 
one-dimensional representations.
So it is easy to see that $\Gamma_N$ is $(E_6)$ (resp.\ $(E_7)$, $(E_8)$).
After all, (\ref{exceptional-graph.eq}) does not have a corresponding $N$.

{\bf 4} is similar to {\bf 3}.

{\bf 5} follows easily from the discussion in the proof of {\bf 3}.
\end{proof}

\section{Determinantal rings}\label{determinantal.sec}

\begin{lemma}\label{S_2-normal.lem}
Let $S$ be a scheme, and $G$ a flat quasi-compact quasi-separated 
$S$-group scheme.
Let $\varphi:X\rightarrow Y$ be an almost principal $G$-bundle
with respect to $U\subset Y$ and $V\subset X$.
Assume that $X$ is Noetherian and normal, and $Y$ is Noetherian and 
satisfies Serre's condition $(S_2)$.
Then $Y$ is normal, and $\bar\eta:\O_Y\rightarrow (\varphi_*\O_X)^G$
is an isomorphism.
\end{lemma}

\begin{proof}
As $\rho:V\rightarrow U$ is fpqc and $V$ is normal, we have that $U$ is 
normal.
As $U\reg$ is large in $U$ and $U$ is large in $Y$, we have that
$Y$ satisfies Serre's $(R_1)$ condition, and hence $Y$ is normal.
By Theorem~\ref{alg-quot-krull-equiv.thm}, $\bar\eta$ is an isomorphism.
\end{proof}

\paragraph
Let $k$ be a field, and $n\geq m\geq t\geq 2$.
Let $X=\Mat(m,t-1)\times \Mat(t-1,n)$, where $\Mat(a,b)$ denotes the
$ab$-dimensional affine space of the set of $a\times b$ matrices.
Let $Y=Y_t(m,n)$ be the determinantal variety
$\{C\in \Mat(m,n)\mid \rank C<t\}$.
Let $\varphi:X\rightarrow Y$ be the map $\varphi(A,B)=AB$.
Let $U$ be the open set $Y\setminus Y_{t-1}$, and $V=\varphi^{-1}(U)$.
Let $N=\GL(t-1)$, and $G=\GL(m)\times\GL(t-1)\times\GL(n)$.
The proof of \cite[(3.1)]{Hashimoto11} shows the following.

\begin{theorem}\label{dp.thm}
Let the notation be as above.
Then $\varphi$ is a $G$-enriched almost principal $N$-bundle
with respect to $U$ and $V$.
\qed
\end{theorem}

Using this theorem and the fact that $Y$ is
Cohen--Macaulay \cite{HE}, we give short proofs to some well-known
results on determinantal rings.

\begin{corollary}[de Concini--Procesi \cite{DP}, \cite{Hashimoto11}]
$Y$ is normal, and $\varphi$ is an algebraic quotient by the action of $N$
\(as $N$ is reductive, $\varphi$ is also a categorical quotient\).
\end{corollary}

\begin{proof}
Follows immediately from Theorem~\ref{dp.thm}, and
Lemma~\ref{S_2-normal.lem}.
\end{proof}

\begin{corollary}[Bruns \cite{Bruns2}]
$\Cl(Y)=\Bbb Z$.
\end{corollary}

\begin{proof}
As $\Cl(X)=0$, $\Cl(Y)\cong H^1\alg(G,\O_X^\times)$ by 
Theorem~\ref{four-term.thm}.
By \cite[(4.15)]{Hashimoto4}, we have that $\Cl(Y)\cong \Cal X(N)$.
It is well-known that $N/[N,N]\cong \Bbb G_m$, and $\Cal X(N)\cong
\Cal X(\Bbb G_m)\cong \Bbb Z$.
\end{proof}

\begin{corollary}[Svanes \cite{Svanes}]
$Y$ is Gorenstein if and only if $m=n$.
\end{corollary}

\begin{proof}
Let $V=k^n$, $W=k^m$ and $E=k^{t-1}$ be the vector representations of 
$\GL_n$, $\GL_m$, and $\GL_{t-1}$, respectively.
Then letting $B:=\Sym(W^*\otimes E)\otimes\Sym(E^*\otimes V)$  
(so $X=\Spec B$), we have that 
\begin{multline*}
\omega_B=B\otimes_k \extop(W^*\otimes E)\otimes_k \extop(E^*\otimes V)\\
\cong
B\otimes_k (\extop E)^{\otimes(m-n)}\otimes_k 
(\extop W)^{\otimes(1-t)}\otimes_k (\extop V)^{\otimes (t-1)}.
\end{multline*}
In particular, $\omega_B\cong B$ as $(N,B)$-modules if and only if $m=n$.
If $m=n$, then by 
Corollary~\ref{abstract-Watanabe.thm}, $\omega_{A}\cong A$ as $A$-modules,
and hence $A$ is Gorenstein
(note that $\Theta_{N,k}$ is trivial, since $N$ is connected reductive,
see Remark~\ref{knop-remark.thm}).
Conversely, if $A$ is Gorenstein, being a positively graded ring over a
field, $\omega_A\cong A$ as $A$-modules.
So $\omega_B\cong B$ as $(N,B)$-modules by
Corollary~\ref{abstract-Watanabe.thm}, and hence $m=n$.
\end{proof}

\paragraph
We can do a similar discussion also on the invariant subrings under the
action of symplectic groups.

Let $k$ be a field, $t,n\in\Bbb Z$ with 
$4\leq 2t\leq n$, and $X=\Mat(2t-2,n)$.
Let $Y=Y_t$ be the Pfaffian subvariety of $\Alt(n)$, the affine
space of $n\times n$ alternating matrices, defined by 
$2t$-Pfaffians.
That is, when $C=k[x_{ij}]_{1\leq i<j\leq n}$ is the coordinate ring of 
$\Alt(n)$ and $\Gamma=(x_{ij})$ (where $x_{ii}=0$ and $x_{ji}=-x_{ij}$), then
$Y$ is the closed subscheme of $\Alt(n)$ defined by the ideal generated
by all the $2t$-Pfaffians of the alternating matrix $\Gamma$.
We set 
$J=J_{t-1}=(\delta_{i+j,t})_{1\leq i,j<t}\in\GL(t-1)$, where 
$\delta$ denotes Kronecker's delta.
We define
\[
\tilde J=
\tilde J_{t-1}=
\begin{pmatrix}
0 & J \\
-J & 0
\end{pmatrix}\in\GL(2t-2).
\]
The symplectic group is defined as
\[
\Sp_{2t-2}:=\{A\in\GL(2t-2)\mid {}^tA\tilde J A=\tilde J\}.
\]
Let $N=\Sp_{2t-2}$, 
$V=k^n$, $E=k^{2t-2}$, and $G=\GL(n)\times N$.
Note that $G$ acts on $X$ by $(h,n)\cdot A=nAh^{-1}$.

Let $\varphi:X\rightarrow Y$ be the map given by $\varphi(C)={}^tC\tilde J C$.
Almost by definition, $\varphi$ is $N$-invariant.
For $C\in X$, $\varphi(C)$ has rank at most $2t-2$, and hence $2t$-Pfaffians
of $\varphi(C)$ vanish, and $\varphi$ is well-defined.
Set $V=Y\setminus Y_{t-1}$, and $U=\varphi^{-1}(V)$.
Then the discussion in \cite[section~5]{Hashimoto11} shows the following.

\begin{theorem}
Let the notation be as above.
Then $\varphi$ is a $G$-enriched almost principal $N$-bundle with respect to
$U$ and $V$.
$Y$ is Cohen--Macaulay.
\end{theorem}

\begin{corollary}
Let the notation be as above.
\begin{enumerate}
\item[\bf 1] \(De Concini and Procesi {\rm\cite{DP}}\) 
$\varphi$ is an algebraic 
quotient, and $Y$ is a normal variety.
\item[\bf 2] As $N=[N,N]$, we have that $\Cl(Y)$ is trivial.
That is, the coordinate ring of $Y$ is a UFD \(hence is Gorenstein\).
\end{enumerate}
\end{corollary}

\end{document}